\documentclass[12pt]{amsart}
\usepackage{graphicx}
\usepackage{amsmath}
\usepackage{subfigure}
\usepackage{amssymb} 
\usepackage[usenames,dvipsnames]{color}
\usepackage{pinlabel}
\usepackage{mathtools}
\usepackage{txfonts} 
\usepackage{hyperref}
\usepackage{enumitem}
\hypersetup{linktocpage}

\usepackage{float}

\usepackage{tikz}
\usepackage{imakeidx}
\makeindex

  \usetikzlibrary{hobby}
  \usetikzlibrary{patterns}
    \input xy
  \xyoption{all}

\usepackage{fancyhdr}
\newtheorem{thm}{Theorem}[section]  
\newtheorem{cor}[thm]{Corollary}

\newtheorem{defin}[thm]{Definition} 
\newtheorem{lemma}[thm]{Lemma} 
 
\newtheorem{prop}[thm]{Proposition} 
 
\newtheorem{ass}[thm]{Assumption} 
\newtheorem{rrule}[thm]{Rule}
\newtheorem*{defin*}{Definition}

\newcommand{\calA}{\mathcal A}
\newcommand{\calB}{\mathcal B}
\newcommand{\calC}{\mathbb C}
\newcommand{\calD}{\mathcal D}
\newcommand{\calE}{\mathcal E}

\newcommand{\ocalD}{\overline{\calD}}

\newcommand{\diamY}{{\rm diam}(Y)}

\newcommand{\bdd}{\mbox{$\partial$}}
\newcommand{\inter}{\mbox{${\rm int}$}}
\newcommand{\genus}{\mbox{${\rm genus}$}}
\newcommand{\Img}{\operatorname{Img}}
\newcommand{\Diff}{\operatorname{Diff}}
\newcommand{\Mod}{\operatorname{Mod}}
\newcommand{\Emb}{\operatorname{Emb}}
\def\zed{{\mathbb Z}}

\def\R{{\mathbb R}}

\newcommand{\frb}{\mathfrak{b}}
\newcommand{\ofrc}{\mbox{$\overline{\frc}$}}
\newcommand{\frc}{\mbox{$\mathfrak{c}$}} 

\newcommand{\fri}{\mathfrak{i}}
\newcommand{\frs}{\mathfrak{s}}

\begin{document}  

\title{Generating the Goeritz group of $S^3$}   

%

\author{Martin Scharlemann}
\address{\hskip-\parindent
        Martin Scharlemann\\
        Mathematics Department\\
        University of California\\
        Santa Barbara, CA 93106-3080 USA}
\email{mgscharl@math.ucsb.edu}

\thanks{Michael Freedman has been an avid source of support for this project; his notion of a `cycle of weak reductions' underpins it all.  Aoife McCormick inspired the notion of `chamber complex' - the method here for translating Freedman's idea into a recipe for Heegaard surface reduction}

\date{\today}

\begin{abstract}  In 1980 J. Powell \cite{Po} proposed that five specific elements sufficed to generate the Goeritz group for any genus Heegaard splitting of $S^3$.  Here we prove 
that a natural expansion of Powell's proposed generators, to include all eyeglass twists and all topological conjugates of Powell's generators, does suffice. 
\end{abstract}

\maketitle

\setcounter{tocdepth}{2}

\tableofcontents

\section{Introduction} \label{sect:intro}

Suppose that $M$ is a closed orientable 3-manifold and $M = A \cup_T B$ is a Heegaard splitting of $M$.  
Following \cite{JM} the {\em Goeritz group $G(M, T)$} is the group of isotopy classes of diffeomorphisms $(M, T) \to (M, T)$ for which the induced diffeomorphism $M \to M$ is isotopic to the identity.  An element of the Goeritz group can be viewed as the final result of a (possibly non-unique) loop $T_\theta, 0 \leq \theta \leq 2\pi$ of embeddings of $T$ in $M$, that is an element in $\pi_1(\Diff(M)/\Diff(M, T))$ \cite[Theorem 1]{JM}.  That is the viewpoint we will take. 

Little is known about the Goeritz group, even in the case that $M = S^3$ and the Heegaard splittings are fully described \cite{Wa}.  In \cite{Po} J. Powell proposed (indeed he believed he had proven) that five specific isotopies generate the Goeritz group $G(S^3, T)$ of the $3$-sphere.  (In \cite{Sc2} one is found to be redundant, and so the number is reduced to four.)   Figures \ref{fig:bubble} to \ref{fig:eyeglass1} give important examples of Goeritz elements acting on the standard Heegaard splitting of $S^3$.  

Powell's conjecture has been verified for genus $\leq 3$ splittings of $S^3$ \cite{FS1}, but even the question of whether $G(S^3, T)$ is finitely generated remains open when $\genus(T) \geq 4$.  Here we define a generalization, shown in Figure \ref{fig:eyeglass1} and called an {\em eyeglass twist}, of Powell's proposed generator $D_{\theta}$. We ultimately show (Corollary \ref{cor:finale}) that if  $D_{\theta}$ is replaced by the collection of all eyeglass twists, and we include also all topological conjugates of the other three Powell generators (i. e. conjugates by elements of $G(S^3, T)$), the subgroup $\calE \subset G(S^3, T)$ thereby generated is all of $G(S^3, T)$.    

In \cite{Sc3} we show that this leads to the following observation:
Suppose $T$ is the standard genus $g+1$ Heegaard surface in $S^3$.  Any element of $G(S^3, T)$ that acts trivially on a standard genus $1$ summand of $T$ is a consequence of the Powell generators acting on $T$.  In other words, the Powell Conjecture is true, stably.  
\bigskip

The strategy for the proof of Corollary \ref{cor:finale} is rather simple: Suppose $T$ is the standard genus $g$ Heegaard splitting of $S^3$, and $\tau \in G(S^3, T)$.   In \cite{FS1} it is shown that there is a `cycle of weak reductions' capturing $\tau$.  That is, given a loop of embeddings $T_{\theta}, 0 \leq \theta \leq 2\pi$ of $T$ in $S^3$ that represents $\tau$, there is a natural topological way to extract, for generic $\theta$, a pair of properly embedded disjoint essential disks $a_{\theta} \subset A, b_{\theta} \subset B$ that thereby weakly reduce $T_{\theta}$.  Moreover, at the finite number of non-generic points, the chosen weakly reducing pair does not change much.  Since the method of choosing the pair $(a_{\theta}, b_{\theta})$ is topological, it follows relatively easily that $(a_{2\pi}, b_{2\pi}) = (\tau(a_{0}), \tau(b_{0}))$.  A landmark result of Casson and Gordon \cite{CG} shows that a weakly reducing pair gives rise to a reducing sphere for $T$; one might naturally hope that one could track such a reducing sphere $K$ around $\theta$, as was done for the weakly reducing pairs, and thereby be able to meaningfully compare $K$ with $\tau(K)$ and so understand the action of $\tau$.

Sadly, the transition from a weakly reducing pair $(a, b)$ to a reducing sphere involves much choice, so there is no naturally derived reducing sphere $K$ for $T$ as hoped for above.  The program here is to find one, via this natural topological method:  Let $F \subset S^3$ be the surface obtained from $T$ by weakly reducing along the pair $(a, b)$.  There is a natural and oft-used way to sweep out $S^3$ by level $2$-spheres $S_s$ and, given $F$, a natural value of $s$ to pick: a value so that the genus of the part of $F$ lying below $S_s$ matches the genus of the part above.  Could $S_s$ be turned into a viable, topologically defined reducing sphere for $T$ that could play the role of $K$ above?  

It turns out that although $S_s$ itself may not play that role, it does give a recipe for weak reduction that is robust: it doesn't change abruptly as $T$ moves through $S^3$.  And as the weak reduction proceeds (as guided by $S_s$) and $F$ becomes more complicated, its complementary components in $S^3$ (called the {\em chambers}) 
are more likely to contain reducible components.  Reducing spheres in these reducible components naturally cut off summands of the original Heegaard surface, and these (inductively) determine an isotopy to the standard picture that is unique, up to action by $\calE$.  If no complementary component becomes reducible, then at the end of the process we fall back on $S_s$ as the required reducing sphere for $T$ (see the second bullet in the statement of Proposition \ref{prop:balance}).

One can think of the progression Proposition \ref{prop:toyexist} $\to$ Corollary \ref{cor:naturality} $\to$ Corollary \ref{cor:calCcert} $\to$ Corollary \ref{cor:vecCcert} $\to$ Corollary \ref{cor:calCcert2} $\to$ Corollary \ref{cor:finale} as guideposts for the argument. 

Although the program is easy to describe, the technical difficulties encountered below are complex since, in effect, they will involve 4 dimensions of sweep-outs. Much of the machinery is new, and then so is the terminology.  We call the attention of the reader to the Index at the end of the paper for a guide to this new terminology. A plausible strategy for a proof is presented in Section \ref{sect:motiv}, including a a rough overview of the final proof in Subsection \ref{subsect:overview} Some of the argument is not restricted to $M = S^3$, so perhaps some of the methodology can be useful in understanding stabilized Heegaard splittings of other $3$-manifolds as well.  

{\bf General Remarks:}  All manifolds will be orientable and, unless obviously not, compact.  We mostly work in the TOP category (locally flat embeddings, homeomorphisms) to avoid discussion of corner rounding, etc.  But at several points (see especially Section \ref{sect:balisotopy}) we need results from smooth topology.  At those points we will work in the DIFF category.  In these dimensions (two and three) the difference in categories is immaterial, see for example \cite{Moi}.  Unexplained notation in the statement of a Lemma, Proposition, etc. may well be found in the discussion that immediately precedes it.  

\section{The eyeglass subgroup} \label{sect:Heegaard2}

Powell's proposed generators are expressed in terms of {\em standard genus 1 summands} in a canonical picture of a genus $g$ Heegaard surface in $S^3$.  In \cite[Section 2]{FS1} some consequences of these Powell generators are presented that include the Powell generators themselves, but also include more general moves on the standard genus 1 summands.  These more general moves would make sense on any stabilized Heegaard splitting of any closed orientable manifold, if we drop the requirement that the summands are `standard', which is not meaningful in the context of arbitrary stabilized splittings.  Below is a brief description of those generalized moves; more detail can be found in \cite{FS1}.  In addition, note that the last move described (an eyeglass twist) is a significant generalization of Powell's move $D_{\theta}$.

For a Heegaard splitting $M = A \cup_T B$ define a {\em bubble} \index{Bubble} for $T$ to be a $3$-ball $\mathfrak{b}$ in $M$ whose boundary intersects $T$ in a single essential circle.  Via Waldhausen \cite{Wa} we know that any bubble is the boundary sum of a collection of genus $1$ bubbles.  

\begin{enumerate}
\item A {\em bubble move} \index{Bubble move} is an isotopy of $\frb$ along a closed path in $T-\frb$ that  begins and ends at $\frb$, returning $(\frb, \frb \cap T)$ to itself. See Figure \ref{fig:bubble}.  Note that if the closed path is null-homotopic in $M$, which is automatic if $M = S^3$, the bubble move is isotopic to the identity on $M$, so it represents an element of $G(M, T)$. 


 \begin{figure}[ht!]
    \centering
    \includegraphics[scale=0.5]{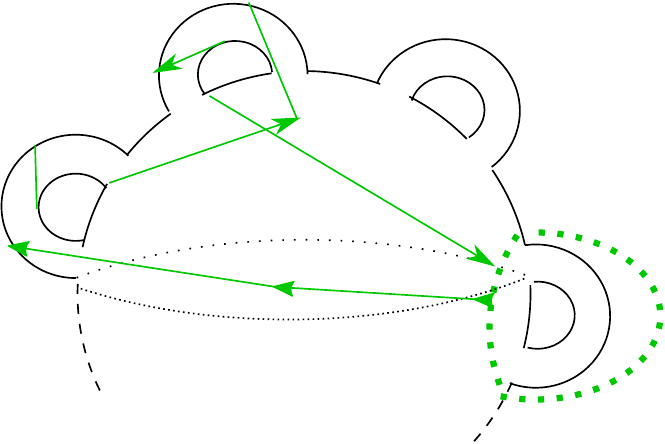}
    \caption{A bubble move} \label{fig:bubble}
    \end{figure}

\item Let $\frb$ be a genus 1 bubble.  A {\em flip} \index{Flip} is the homeomorphism $(\frb, \frb \cap T) \to (\frb, \frb \cap T)$ that reverses orientation of both the meridian and longitude of the summand, as shown in Figure \ref{fig:flip}.  Powell's \cite{Po} label for a flip on the first standard summand is $D_\omega$.  In a genus $1$ splitting, where $T$ is a torus, regard the hyperelliptic involution $(S^3, T) \to (S^3, T)$ as (a degenerate case of) a flip.

 \begin{figure}[ht!]
    \centering
    \includegraphics[scale=0.5]{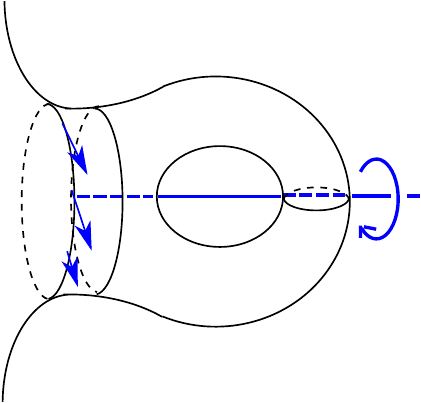}
    \caption{A flip}  \label{fig:flip}
    \end{figure}

\item Let $\frb_1, \frb_2$ be disjoint genus 1 bubbles, and let $v \subset T - (\frb_1 \cup \frb_2)$ be an arc connecting them. Let $\frb$ be the genus $2$ bubble obtained by tubing together the genus 1 bubbles along $v$.  A {\em bubble exchange} \index{Bubble exchange} exchanges the two genus 1 bubbles within $B$, as shown in Figure \ref{fig:switch}.  Powell's label for a bubble exchange on the first two standard summands is $D_{\eta_{12}}$. 

 \begin{figure}[ht!]
    \centering
    \includegraphics[scale=0.5]{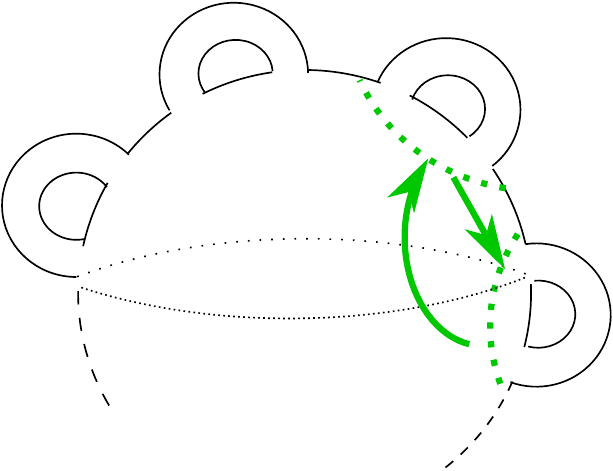}
    \caption{A bubble exchange}  \label{fig:switch}
    \end{figure}

\item An {\em eyeglass} \index{Eyeglass} is the union of two disks, $\ell_a, \ell_b$ ( the {\em lenses} ) with an arc $v$ (the {\em bridge}) connecting their boundaries.  Suppose an eyeglass $\eta$ is embedded in $M$ so that the $1$-skeleton of $\eta$ (called the {\em frame}) lies in $T$, one lens is properly embedded in $A$, and the other lens is properly embedded in $B$.  The embedded $\eta$ defines a natural automorphism $(M, T) \to (M, T)$,  as illustrated in Figure \ref{fig:eyeglass1}, called an {\em eyeglass twist}.  \index{Eyeglass twist} Powell's generator $D_\theta$ is an example, but topologically special because the lenses are primitive disks in each handlebody.  

 \begin{figure}[ht!]
    \centering
    \includegraphics[scale=0.5]{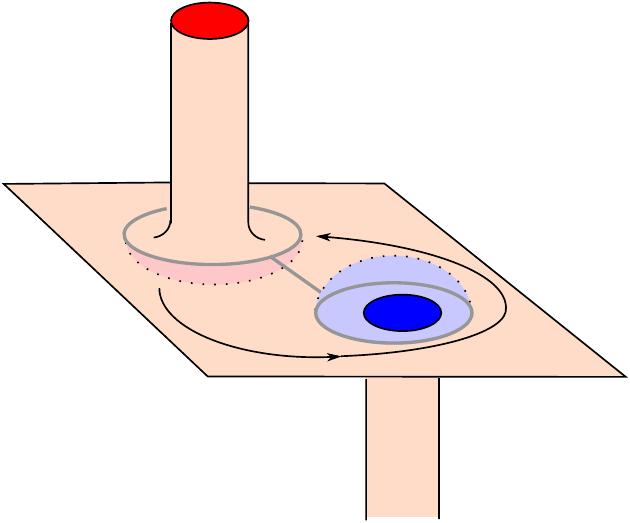}
    \caption{An eyeglass twist}  \label{fig:eyeglass1}
    \end{figure}

\end{enumerate}

Let $\Mod(M, T)$ be the group of path components of $\Diff(M, T)$; it is the Goeritz group $G(M, T)$ when, as is the case with $S^3$, $\Mod(M)$ is trivial \cite{JM}.

\begin{defin}  \label{defin:eyeglass} Suppose $M = A \cup_T B$ is a Heegaard splitting.  The subgroup of $\Mod(M, T)$ generated by the four types of moves just described is called the {\em eyeglass group}  \index{Eyeglass group} $\calE \subset \Mod(M, T)$. 
Any finite composition of these eyeglass generators will be called an {\em eyeglass move}.  

When $M = S^3$ call the subgroup of $G(S^3, T) = \Mod(S^3, T)$ generated by Powell's proposed generators the {\em Powell group}. \index{Powell group}  Any finite composition of these generators will be called a {\em Powell move}.  
\end{defin}

{\bf Remarks:}  It is shown in \cite[Section 2]{FS1} that the Powell group can also be described as the subgroup of $G(S^3, T)$ generated by $D_{\theta}$ together with those flips, bubble moves, and bubble exchanges that act on the standard genus 1 bubbles, not on arbitrary genus 1 bubbles.

It follows that any Powell move (indeed any topological conjugate of a Powell move) is an eyeglass move.  Note also that whereas the Powell group is not known to be normal in $G(S^3, T)$, the eyeglass group is normal, since any topological conjugate of a generator is a generator.
%

The Powell Conjecture is that the Powell group is the entire Goeritz group $G(S^3, T)$; we will eventually show the weaker result that the eyeglass group is the entire Goeritz group.   If one could show that each eyeglass twist and any topological conjugate of a Powell generator is in the Powell group, the Powell Conjecture would follow.  This seems unlikely.  

\bigskip

Let $M = A \cup_T B$ be a Heegaard splitting; we briefly review terminology, and make some elementary observations:

\begin{defin} \label{defin:aligned}  A sphere $S$ in $M$ is {\em aligned} with the Heegaard splitting if $S \cap T$ is at most one circle.  Similarly, a properly embedded disk $(D, \bdd D) \subset (M, \bdd M)$ is aligned \index{Aligned disk or sphere} with the splitting if $D \cap T$ is at most one circle, and, if it is one circle, the annulus component of $D - T$ is a spanning annulus in the compression body in which it lies.  See \cite{Sc1}, \cite{FS2}.
\end{defin}

For example, the boundary of a bubble $\frb$ is an aligned sphere.

Suppose disjoint spheres $S$ and $S'$ are aligned, and $\alpha$ is a properly embedded arc in $T$ so that $\alpha$ intersects $S$ exactly in one end of $\alpha$ and intersects $S'$ only in the other end of $\alpha$.  A thin tubular neighborhood $N \cong D^2 \times I$ of $\alpha$ will intersect $S \cup S'$ in the two disks $D^2 \times \{0, 1\}$.  Delete those disks from $S \cup S'$ and glue on the annulus $\bdd D^2 \times I$. 

\begin{defin} \label{defin:tubesum}  \index{Tube sum} Call the resulting aligned sphere the {\em tube sum} of $S$ and $S'$ along $\alpha$.  

The same construction on $S$ and a disjoint aligned disk $D$ results in another aligned disk, called the tube sum of $D$ and $S$.  

If $\frb$ and $\frb'$ are disjoint bubbles, say of genus $p$ and $q$ then the tube sum of $\bdd \frb$ and $\bdd \frb'$ naturally bound a genus $p + q$ bubble, called the tube sum of $\frb$ and $\frb'$.  
\end{defin}


Since a bubble move is an eyeglass move, the tube sum of a sphere or disk with a bubble, and, in particular, the tube sum of two bubbles, is well-defined 
up to eyeglass moves; it does not depend on the choice of $\alpha$.  

\begin{defin} \label{defin: bubble pass} \index{Bubble pass} For $D$ an aligned disk in $M$ and $\frb$ a disjoint bubble, the tube sum of $D$ and $\frb$ is an aligned disk $D'$ properly isotopic to $D$ in $M$.  Replacing $D$ with $D'$ is called a {\em bubble pass} of $\frb$ through $D$.  
\end{defin} 

\begin{lemma} \label{lemma:nonseppass}  Suppose  $D$ is an aligned {\em non-separating} disk in $M$, $\frb$ is a disjoint bubble, and $D'$ is the disk obtained from $D$ by a {\em bubble pass} of $\frb$ through $D$.  Then there is an eyeglass move which isotopes $D$ to $D'$ in $M$.
\end{lemma}

\begin{proof} Let $\alpha$ be the arc between $D$ and $\bdd \frb$ along which the bubble pass is made, and let $D'$ be the resulting aligned disk.  Since $D$ is assumed non-separating, there is another arc $\beta \subset T$ disjoint from $\alpha$ so that $\bdd \beta = \bdd \alpha$ but the end of $\beta$ at $D$ is on the opposite side of $D$ as the end of $\alpha$ at $D$.  Thus the union $\gamma = \alpha \cup \beta$ is a path in $T - \frb$ that passes exactly once through $D$.  Then a bubble move of $\frb$ via $\gamma$ will isotope $D$ to $D'$ as required.  (See Figure \ref{fig:nonseppass}.  The bubble move of $\frb$ to itself that carries $D$ to $D'$ is along the concatenation loop $\beta\alpha$.) 
\end{proof}

 \begin{figure}[ht!]
\labellist
\small\hair 2pt
\pinlabel  $A$ at 25 80
\pinlabel  $B$ at 25 50
\pinlabel  $T$ at 10 68
\pinlabel  $D$ at 165 80
\pinlabel  $D'$ at 210 80
\pinlabel  $\frb$ at 220 30
\pinlabel  $\beta$ at 280 20
\pinlabel  $\beta$ at 80 20
\pinlabel  $\alpha$ at 195 20
\endlabellist
    \centering
    \includegraphics[scale=1]{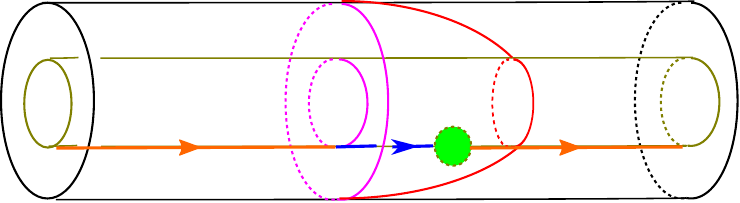}
     \caption{A bubble move of $\frb$ to itself that carries $D$ to $D'$ } \label{fig:nonseppass}
    \end{figure}

\section{The search for a plausible strategy} \label{sect:motiv}

The proof that $G(S^3, T) = \calE$ will be by induction on the genus of $T$; it is known for genus $\leq 3$ by \cite{FS1}.  In this section we lay some groundwork and describe (see Corollary \ref{cor:toy}) a special circumstance that makes the inductive step straightforward.  The goal of the rest of the long and technical argument in Sections  \ref{sect:chamberintro} through \ref{sect:balisotopy} is simply to show that conditions sufficiently similar to this special circumstance always arise.  

\subsection{A technical lemma}

The following technical lemma is needed almost immediately for an inductive step:

Let $(S^3, T)$ be a genus $g \geq 1$ Heegaard splitting and $B_p$ be a small ball around a point $p \in T$ that intersects $T$ in a single disk $D_p$. Let $B_-$ be the $3$-ball that remains when $\inter(B_p)$ is removed, and $T_- = T \cap B_- = T - \inter(D_p)$.  $T_-$ is  once-punctured genus $g$ surface properly embedded in $B_-$. Let $\Mod(B_-, T_-)$ be the group of path components of $\Diff(B_-, T_-)$. (We do not require that a diffeomorphism $f: (B_-, T_-) \to (B_-, T_-) \in \Diff(B_-, T_-)$ be the identity on $\bdd B_p = \bdd B_-$.)

 Echoing the definition of bubble move above, define a {\em $B_p$-bubble move} $\beta_p: (S^3, T) \to (S^3, T)$ to be the result of an isotopy of $(B_p, D_p)$ along a closed path in $T_-$ that returns $(B_p, D_p)$ to itself.  It is understood that $D_p$ remains in $T$ throughout the isotopy, so $T$ is preserved throughout the isotopy.  This means that the final map $\beta_p:(S^3, T) \to (S^3, T)$ 
still represents the identity in $\Mod(S^3, T)$.  However, the restriction $\beta_p|B_-: (B_-, T_-) \to (B_-, T_-)$ may not be isotopic to the identity and so may represent a non-trivial element of $\Mod(B_-, T_-)$
 
Define $\calE_- \subset \Mod(B_-, T_-)$ to be the subgroup generated by $B_p$-bubble moves, together with the flips, bubble moves, bubble exchanges and eyeglass twists that generate $\calE$ and whose associated paths, bubbles, lenses and bridges are disjoint from $B_p$, so they all lie in $B_-$.  

\begin{lemma} \label{lemma:Tminus}  If $\Mod(S^3, T) = \calE$ then $\Mod(B_-, T_-) = \calE_-$.
\end{lemma}

This is a natural statement, and is relevant here because $G(S^3, T) = \Mod(S^3, T)$.

\begin{proof}  
Let $h_-: (B_-, T_-) \to (B_-, T_-)$ represents an element of $\Mod(B_-, T_-)$; after an isotopy we can assume that $h_-$ is the identity near $\bdd B_p = \bdd B_-$.  Extend $h_-$ to a diffeomorphism $h:(S^3, T) \to (S^3, T)$ representing a class in  $\Mod(S^3, T)$ by setting $h|B_p$ to be the identity.  Under the assumption  $\Mod(S^3, T) = \calE$ there is an eyeglass move $h':(S^3, T) \to (S^3, T)$ so that $h'h^{-1}:(S^3, T) \to (S^3, T)$ is isotopic to the identity.  We may as well take the generators of $h'$ to lie in $B_-$, since each of these generators has support near a $1$-complex in $T$ and so, by general position, can be chosen to be disjoint from $B_p$.  For example, for an eyeglass twist choose the frame of the eyeglass disjoint from $B_p$ and it becomes a frame in $B_-$; in a bubble move, regard the bubble as a thin neighborhood in $T$ of the bubble's $1$-skeleton, which we can also take to be disjoint from $B_p$.  Once this is done, the generators give a diffeomorphism $h'_-: (B_-, T_-) \to (B_-, T_-)$ representing an element of $\calE_-$.  (There is no claim that $h'_-$ is well-defined.  In fact a particular choice of $h_-$ depends on how the supporting $1$-complexes of the generators of $h'$ have been isotoped away from $B_p$.) 


By construction, the diffeomorphisms $h'h^{-1}$ and $id_{(S^3,T)}$, both of which are the identity on $B_p$, are isotopic, but such an isotopy might well move $B_p$, so it does not induce an isotopy from $h'_-h_-^{-1}: (B_-, T_-) \to (B_-, T_-)$ to the identity.  Our intention is to fix this, by suitably altering $h'$.  Denote by $\theta_t$ the isotopy from $\theta_0 = h'h^{-1}: (S^3, T) \to (S^3, T)$ to $\theta_1 = id_{(S^3, T)}$ and view $\theta_t$ as a path in $\Diff(S^3, T)$.  The isotopy $\theta_t$ determines a $p$-based loop $\alpha: [0, 1] \to T$ in $T$ via $\alpha(t) = \theta_t(p)$.  Now alter $h'$ to $h'':(S^3, T) \to (S^3, T)$ by post-composing with a $B_p$-bubble move along the loop $\overline{\alpha}$ defined by $\overline{\alpha}(t) = \alpha(1-t)$. By definition of $\calE_-$, the $B_p$-bubble move $h''_-: (B_-, T_-) \to (B_-, T_-)$ still lies in $\calE_-$. Furthermore, the maps $h'$ and $h''$ are isotopic as diffeomorphisms on $(S^3, T)$, but the isotopy $\theta_t$ is replaced by an isotopy $\theta'_t$ from $h''h^{-1}$ to the identity for which the loop $\theta'_t(p)$ in $T$ becomes the concatenation $\alpha \overline{\alpha}$, and so is nullhomotopic in $T$.  

Consider the evaluation map $e_p: \Diff(S^3, T) \to T$ given by $e_p(f) = f(p)$.  We can regard $T$ as the space of embeddings $\Emb(p, T)$ and deduce from \cite{Pa} that $e_p$ is a fibration.  In particular, the null-homotopy of $\alpha \overline{\alpha}$ to $p \in T$ just described lifts to a homotopy rel end points from the arc $\theta'_t$ in $\Diff(S^3, T)$ to an arc in $\Diff(S^3, T)$ that lies entirely over $p$.  This arc then defines an isotopy $\theta''_t$ from $h''h^{-1}:(S^3, T) \to (S^3, T)$ to the identity, but now an isotopy in which the point $p$ is fixed throughout.  That is, for every $0 \leq t \leq 1$ $\theta''_t(p) = p$.

Finally, we will alter below the isotopy $\theta''$ so that for every $0 \leq t \leq 1$, $\theta''_t(B_p) = B_p$.  This will imply that the restriction $\theta''_- = \theta''|(B_-, T_-)$ is an isotopy from $h''_-h_-^{-1}: (B_-, T_-) \to (B_-, T_-)$ to the identity, so $h''_-$ and $h_-$ represent the same element of $\Mod(B_-, T_-)$; since $h'' \in \calE_-$ this will conclude the proof.

Here then is a brief sketch of how, using standard tools of differential topology \cite{GP}, the isotopy $\theta''$ can be altered so that not just $p$, but all of $B_p$ is mapped diffeomorphically to itself during the isotopy.  The first step is to observe that the derivative of any diffeomorphism $f: (S^3, T, p) \to  (S^3, T, p)$ determines a linear automorphism $Df_p$ of the tangent space of $S^3$ at $p$ which also sends the tangent space of $T$ at $p$ to itself.  The space of such linear automorphisms deformation retracts to the subspace in which the automorphism is orthogonal.  (Such a deformation retraction can be easily constructed, for example, from the proof of \cite[Theorem 4.2]{Sp}.)  Via this deformation retraction, continuously alter $\theta''$ so that for any $0 \leq t \leq 1$ the derivative $D(\theta''_t)_p$ is orthogonal.  Once this is done, $\theta''$ can be further altered near $p$ so that each map $\theta''_t$ becomes itself the orthogonal map $D(\theta''_t)_p$ near $p$.  In particular it takes a small ball around $p$ diffeomorphically to itself while continuing to preserve the Heegaard surface $T$ .  The final step is to deform $\theta''$ so that the ball on which this is true contains the original $B_p$.  
%

(Aside:  when $genus(T) = 1$, $B_p$-bubble moves are not needed, per \cite[Theorem 1b]{EE}.)
\end{proof}

A bubble exchange between genus 1 bubbles $\frb_1$ and $\frb_2$, as shown in Figure \ref{fig:switch}, takes place in the neighborhood of $\frb_1 \cup \frb_2 \cup \nu$, where $\nu$ is an embedded arc in $T$ that connects the bubbles. A more general operation, which we will call a {\em generalized bubble exchange}, isotopes $\frb_1$ to $\frb_2$ along an arc $\nu_1 \subset T$ while simultaneously isotoping $\frb_2$ to $\frb_1$ along an arc $\nu_2 \subset T$ disjoint from $\nu_1$.  A generalized bubble exchange is a simple bubble exchange if the arcs $\nu_1, \nu_2$ are parallel in $T$.

\begin{lemma} \label{lemma:genexch}  Any generalized bubble exchange is an eyeglass move.  
\end{lemma}  

\begin{proof}  A generalized exchange of bubbles $\frb_1$ and $\frb_2$ using arcs $\nu_1$ and $\nu_2$ (first column of  Figure \ref{fig:genexch}) can be written (see second column of  Figure \ref{fig:genexch}) as a composition of the simple bubble exchange along $\nu_1$ and a bubble move of $\frb_2$ along the closed path $\nu_1 \cup \nu_2$.  Each of the latter is an eyeglass move.  \end{proof}

Following Lemma \ref{lemma:genexch} we will use the term ``bubble exchange'' to include generalized bubble exchange, unless the distinction is important.

 \begin{figure}[ht!]
\labellist
\small\hair 2pt
\pinlabel  $\frb_1$ at 20 220
\pinlabel  $\frb_2$ at 110 220
\pinlabel  $\nu_1$ at 60 155
\pinlabel  $\nu_2$ at 60 95
\pinlabel  $\nu_1\cup\nu_2$ at 260 140
\endlabellist
    \centering
    \includegraphics[scale=0.75]{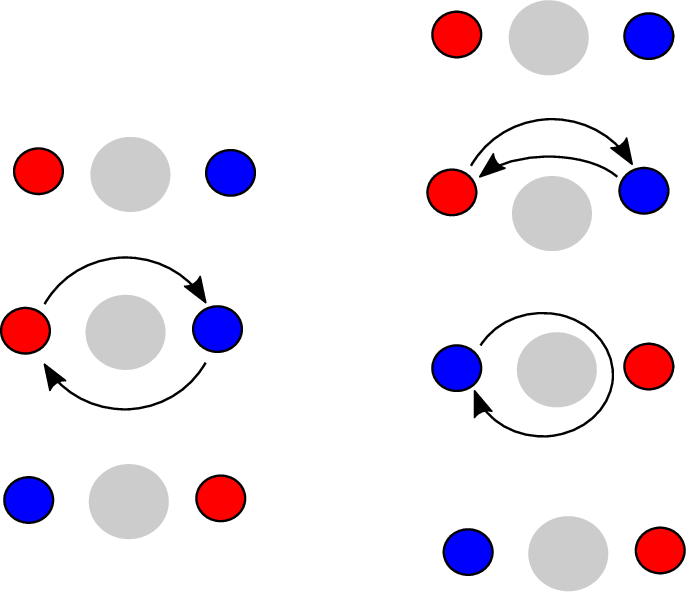}
     \caption{Generalized exchange as an eyeglass move} \label{fig:genexch}
    \end{figure}

\subsection{Resuming the search for a strategy}

We begin by setting up a standard picture for the genus $g$ Heegaard splitting of $S^3$ that is somewhat different than the one used by Powell.  
Let $T_g \subset S^3$ be the standard genus $g$ Heegaard surface in $S^3$, dividing $S^3$ into the genus $g$ handlebodies $A_g$ and $B_g$. 

Let $\{c_1, ..., c_{g-1}\}$ be the disjoint separating circles on $T_g$ shown in Figure \ref{fig:circlesinT}, with each $c_i$ separating the first $i$ standard summands from the last $g-i$ standard summands.  Note that each $c_i$ bounds a disk in both $A$ and $B$ and so defines a reducing sphere $S_i$ for $T_g$.   Let $\frb_i$ be the genus $i$ Heegaard split $3$-ball component of $S^3 - S_i$ containing $S_1$.  Both $\frb_i$ and its complement $S^3 - \frb_i$ are bubbles.  

 \begin{figure}[ht!]
\labellist
\small\hair 2pt
\pinlabel  $c_1$ at 110 140
\pinlabel  $c_2$ at 167 133
\pinlabel  $c_{g-1}$ at 240 135
\endlabellist
    \centering
    \includegraphics[scale=0.75]{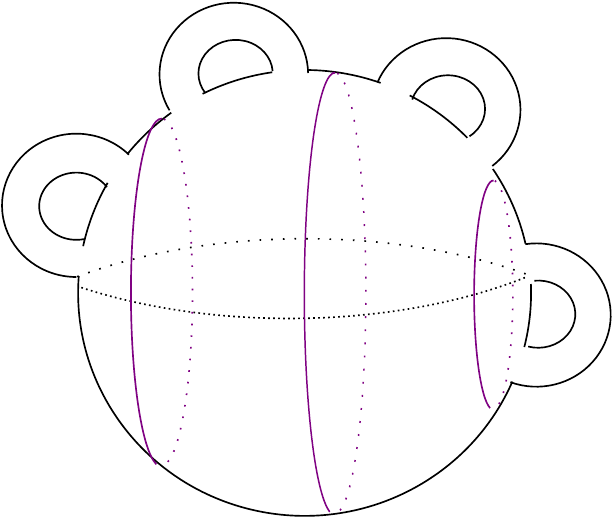}
     \caption{The standard Heegaard surface $T_g \subset S^3$} \label{fig:circlesinT}
    \end{figure}
    
\medskip

It is easy to construct a sequence of (generalized) bubble exchanges (and perhaps a flip)
 that ultimately moves each bubble $\frb_i, 1 \leq i \leq g-1$ to the bubble $S^3 - \frb_{g - i}$. The sequence is best seen in the elaboration Figure \ref{fig:circlesinT2} of Figure \ref{fig:circlesinT} that we now describe.  The unit sphere $S$ in $S^3$ is shown intersecting the $x$-$y$-plane, with the $z$-axis coming out of the page.  Denote by $C$ the unit circle in which $S$ intersects the $x$-$y$-plane.  For $g = 2n$ or $2n-1$ (depending on the parity of $g$) consider $n$ planes $Q_1, ..., Q_n$ parallel to the $x$-$z$ plane (horizontal planes in the figure) starting at $Q_1$ the $x$-$z$ plane itself and ascending from there.  Each $Q_i, 1 \leq i \leq n-1$ intersects $S$ in a circle $d_i$; $d_n = S \cap Q_n$ is either a circle or a point, depending on the parity of $g$.   Place a genus 1 bubble (shown in green in the figure) on $S$ at each point of $d_i \cap C$.  The curves $c_i$ that separate the bubbles, as shown in Figure \ref{fig:circlesinT} may be taken to be the intersection of vertical planes (that is, planes parallel to the $y$-$z$-plane) $P_1, ..., P_{g-1}$.  These planes are distributed symmetrically across the $y$-$z$ plane and are shown in red in Figure \ref{fig:circlesinT2}.  In this picture, simple $\pi$-rotation around the $y$-axis will take each $c_i$ to $c_{g-i}$ and each genus $i$ bubble $\frb_i$ to the bubble $S^3 - \frb_{g - i}$, as we seek.

  \begin{figure}[ht!]
\labellist
\small\hair 2pt
\pinlabel  $P_1$ at 85 40
\pinlabel  $P_3$ at 135 40
\pinlabel  $c_3$ at 157 133
\pinlabel  $P_{g-1}$ at 430 40
\pinlabel  $Q_{1}$ at 60 185
\pinlabel  $Q_{2}$ at 60 235
\pinlabel  $d_{2}$ at 210 200
\pinlabel  $Q_{n}$ at 60 335
\pinlabel  $x$ at 450 185
\pinlabel  $y$ at 260 355
\endlabellist
    \centering
    \includegraphics[scale=0.6]{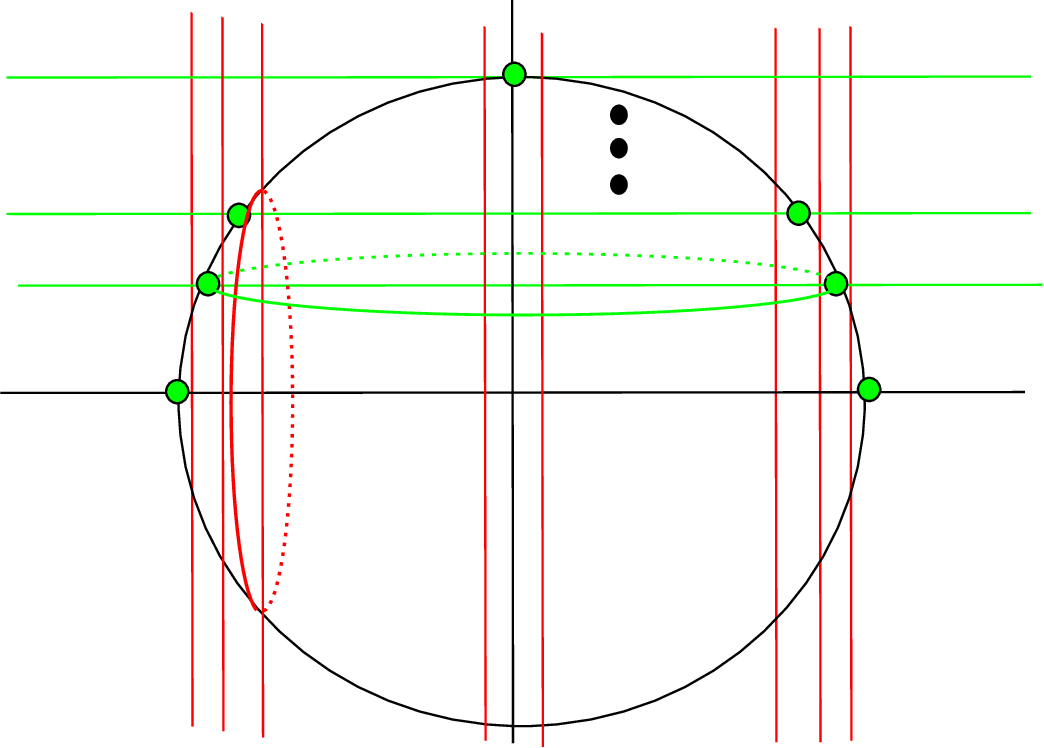}
     \caption{ $h_\rho$ as sequence of bubble exchanges ($g$ odd) } \label{fig:circlesinT2}
    \end{figure}

 It remains to show that (up to pairwise isotopy of ($S^3, T_g$)) this $\pi$ rotation, which we will denote $h_\rho$, can be accomplished by generalized bubble exchanges and perhaps a flip.  This is easy to see; for each $1 \leq i \leq n$ do a generalized bubble exchange between the two bubbles that lie in $d_i$ (or a flip on the single bubble in $d_n$ if $g$ is odd) using the subarc of $d_i$ having $z$ positive (the front face of the sphere $S$) for one arc of the exchange and the subarc of $d_i$ on the back face of $S$ ($z$ negative) for the other arc.  Since each generalized bubble exchange (and the flip) is an eyeglass move, $h_\rho$ is an eyeglass move.  (In fact, since the bubbles are standard, $h_\rho$ is a Powell move.)
    
\medskip


 Suppose $T \subset S^3$ is a genus $g$ Heegaard surface, and $h, h': (S^3, T) \to (S^3, T_g)$ are two orientation-preserving homeomorphisms.  

\begin{defin}  \label{defin:eyeequiv} $h, h'$ are \em{eyeglass equivalent} (written $h \sim h'$) if the composition $h' h^{-1}: (S^3, T_g) \to (S^3, T_g)$ is isotopic in $(S^3, T_g)$ to an eyeglass move.  \index{$\sim$}
\end{defin}

For example, the argument above shows that $h_{\rho}$ is eyeglass equivalent to the identity.  


Throughout the remainder of this section we continue to make this inductive assumption:
\begin{ass} \label{ass:inductive}
$G(S^3, T') = \calE$ whenever $\genus(T') \leq g-1$.
\end{ass}  

\begin{lemma} \label{lemma:preexistunique}  Suppose $S$ is a reducing sphere for a genus $g$ Heegaard splitting $(S^3, T)$.  Then there is an orientation preserving homeomorphism $h_S: (S^3, T) \to (S^3, T_g)$ so that $h_S(S) \in \{S_i, i = 1, ..., g-1\}$.

Suppose that $S'$ is a reducing sphere for $(S^3, T)$ that is disjoint from $S$ and $h_{S'}: (S^3, T) \to (S^3, T_g)$ is similarly defined, so that $h_{S'}(S') \in \{S_i, i = 1, ..., g-1\}$.  Under Assumption \ref{ass:inductive}, the homeomorphisms $h_S$ and $h_{S'}$ are eyeglass equivalent.  In particular $h_S$ is well-defined up to eyeglass equivalence.
%
\end{lemma}

\begin{proof}  The first statement follows almost immediately from \cite{Wa}, as we now describe.  $S$ divides $T$ into surfaces $T_{\pm}$ of genus $g_{\pm} \geq 1$ respectively, with $g_+ + g_- = g$.   Choose any orientation-preserving homeomorphism $f: S^3 \to S^3$ that carries $S$ to $S_{g_+}$ and the complementary ball component of $S$ that contains $T_+$ to the ball $\frb_{g_+}$.  Adjust so that the circle $T \cap S$ is sent to $\frc_{g_+} \subset S_{g_+}$.  Then $f(T_+)$ and $T_g \cap \frb_{g_+}$ each determine a genus $g_+$ Heegaard splitting of $\frb_{g_+}$ and \cite{Wa} then implies that the surfaces are isotopic rel $\frc_{g_+}$.  Similarly isotope $f(T_-)$ to $T_g \cap (S^3 - \frb_{g_+})$ in $(S^3 - \frb_{g_+})$.  The resulting homeomorphism $h_S$ takes $S$ to $S_{g_+}$ and $T$ to $T_g$, as required.  

Note that this construction subtly depends on a choice: if we had reversed the labels $g_{\pm}$ then the requirement that $T_+$ be sent to $\frb_{g_+}$ would have sent $S$ to $S_{g_-}$ instead of $S_{g_+}$.  But this could also be accomplished by composing with the homeomorphism $h_{\rho}$ defined just before this lemma, and we have shown that $h_{\rho} \sim id_{(S^3, T_g)}$. So the choice of labelling makes no difference in the construction of $h_S$, up to eyeglass equivalence.  

Proceeding then with the second statement, we first claim that $h_S$ is well-defined up to eyeglass equivalence.  To that end, suppose another orientation preserving homeomorphism $h': (S^3, T) \to (S^3, T_g)$ has $h'(S) \in \{S_i, i = 1, ..., g-1\}$.  Using possibly $h_{\rho}$ as above, we may as well assume $h_Sh'^{-1}(\frb_{g_+}, T_{g_+}) = (\frb_{g_+}, T_{g_+})$.   Noting that $g_+ < g$, apply Lemma \ref{lemma:Tminus} and Assumption \ref{ass:inductive} to 
the pair $(\frb_{g_+}, T_{g_+})$ with the goal of showing that there is an eyeglass move on $(S^3, T_g)$ which, when composed with $h_Sh'^{-1}
$, is the identity on $\frb_{g_+}$ and remains $h_Sh'^{-1}$ on $S^3 - \frb_{g_+}$ .  Lemma \ref{lemma:Tminus} says that, under the inductive assumption, this is true for some move $h_{g+}:(\frb_{g_+}, T_{g_+}) \to (\frb_{g_+}, T_{g_+})$  in $\calE_-$.  By definition, each flip, bubble move, bubble exchange or eyeglass twist used in the construction of $h_{g+}$ is also an eyeglass move on $(S^3, T_g)$.  On the other hand, a move in $\calE_-$ corresponding to a $B_p$ bubble move in the proof of Lemma \ref{lemma:Tminus} here corresponds to a bubble move on the bubble $S^3 - \frb_{g_+}$ in $(S^3, T_g)$, and this is also an eyeglass move.  So indeed there is an eyeglass move as we seek.  Now apply the same argument on the complementary ball $S^3 - \frb_{g_+}$ and deduce that $h_Sh'^{-1}$ is an eyeglass move on $(S^3, T_g)$, so  $h'$ is eyeglass equivalent to $h_S$, as desired.  


Now consider the reducing sphere $S'$. We may as well choose labels $g_{\pm}$ so that  $h_S(S') \subset \frb_{g_+}$.  Let $\frb_S = h_S^{-1}(\frb_{g_+})$, a genus $g_+$ bubble for $(S^3, T)$.  Apply the first statement again, this time to $h_S|\frb_S:(\frb_S, T_+) \to (\frb_{g_+}, T_{g_+})$, and deduce that  $h_S|\frb_{g_+}$ could have been chosen so that $h_S(S') = S_i$ for some $i \leq g_+$, that is $h_S(S') \in \{S_i, i = 1, ..., g_+\}$.   The same is true, by definition, for the given $h_{S'}$.  The argument we have just given that $h_S$ is well-defined up to eyeglass equivalence, repeated now for $h_{S'}$, shows that $h_S \sim h_{S'}$ as required.  

The last comment, that under the inductive assumption $h_S$ is well-defined up to eyeglass equivalence, follows simply by taking $S'$ to be a parallel copy of $S$ in the previous argument.
\end{proof}

Our goal is to show that for any genus $g$ Heegaard surface $T$ in $S^3$, $G(S^3, T) = \calE$.  Another way of expressing this is that any two orientation preserving homeomorphisms $(S^3, T) \to (S^3, T_g)$ are eyeglass equivalent. 

Suppose $F \subset S^3$ is a possibly disconnected closed surface, dividing $S^3$ into possibly disconnected $3$-manifolds $M_A$ and $M_B$.  Suppose further
\begin{enumerate}
\item Each component $C$ of $M_A$ has a Heegaard splitting $C = A_C \cup_{T_C} B_C$ in which $A_C$ is a  handlebody.  In particular, $\bdd C = \bdd_- B_C$.  Here we follow \cite{FS2} and \cite{Sc1} in allowing sphere components of $\bdd_- B_C$, so the compression body $B_C$ may be reducible.  
\item The symmetric statement is true for each component of $M_B$.
\item The splitting of each ball component of $M_A$ or $M_B$ has genus $\geq 1$.  Hence any component that has a genus 0 splitting is a punctured $3$-sphere (since it has genus 0) with more than one boundary component (since it is not a ball).  
\item  The Heegaard splitting $(S^3, T)$ that is obtained by amalgamating all of these Heegaard splittings is of genus $g$. (See \cite[Section 3]{La} for a description of Heegaard surface amalgamation.)
\end{enumerate}

The {\em defining surface} $F \subset S^3$, together with the Heegaard splittings of $M_A$ and $M_B$ as described above, will be called a {\em Heegaard split chamber complex} that {\em supports} the Heegaard splitting $(S^3, T)$.  Each component of $M_A$ or $M_B$ is called a {\em chamber}.  
Heegaard split chamber complexes are formally defined and explored more extensively in section \ref{sect:Heegaard1}.  (See for example Definition \ref{defin:hsChamcom})

Before proceeding we list some Elementary Facts about surfaces and amalgamation that will be useful:
\begin{itemize}
\item  EF1: A compact surface is non-planar if and only if it contains two simple closed curves that intersect in a single point
\item EF2: Consequently, if $F$ is subsurface of a compact surface $F'$ and $F$ is non-planar, then so is $F'$.  
\item EF3: If $F$ is a non-planar surface and $F'$ is obtained by removing a finite number of disjoint disks from the interior of $F$ (that is $F'$ is a {\em punctured} $F$), then $F'$ is non-planar.  
\item  EF4: Suppose Heegaard splittings $(M_1, T_1)$, $(M_2, T_2)$ of compact orientable 3-manifolds $M_1$ and $M_2$ are amalgamated along a closed surface $F$ that is a boundary component of each.  Then a punctured copy of $F$ lies in each of $T_1, T_2$ and the amalgamated Heegaard surface $T \subset  M = M_1 \cup_F M_2$. See \cite[Figure 12]{La}. In particular, if $F$ is non-planar, so are $T_1, T_2$ and $T$.
\item EF5: Under the amalgamation just described, a punctured copy of each $T_i, i = 1, 2$ lies in $T$.  In particular, if either $T_i$ is non-planar, so is $T$.  Again see \cite[Figure 12]{La}.
\end{itemize}

\begin{lemma}  \label{lemma:Sreduces} Suppose, in a Heegaard split chamber complex in $S^3$ as described above, $S$ is a sphere component of $F$.  Then, after amalgamation, a slight push-off of $S$ becomes a reducing sphere for the splitting $(S^3, T)$, intersecting $T$ in a single circle that is essential in $T$.
\end{lemma}

\begin{proof} We first show that after amalgamation the part of $T$ lying in each of the ball components of $S^3 - S$ must contain a non-planar surface.  There are two cases: 

If every chamber within a ball $B$ bounded by $S$ is a punctured 3-sphere, then every component of $F$ within $B$ is a sphere.  Pick a sphere (possibly $S$ itself) that is innermost among these components.  The ball it bounds contains no other component of $F$ and therefore is a ball chamber.  By the third property of Heegaard split chamber complexes described above, the splitting surface for this ball chamber has genus $\geq 1$ and so is non-planar.  It follows that $T \cap B$ is non-planar, using EF5 above.  

On the other hand, if there is a chamber in $B$ that is not a punctured $3$-sphere then a boundary component of the chamber is non-planar, so again $T \cap B$ is non-planar, using EF4 above.

We now follow the methodology of \cite{CG}:  after amalgamation $S$ becomes a punctured sphere lying in $T$ with some of its boundary components bounding disks in $A$ and others bounding disks in $B$.  Choose a circle $c \subset S$ that divides $S$ into two disks: $D_A$ that contains all disks of $S - T$ that lie in $A$ and $D_B$ that contains all disks of $S - T$ that lie in $B$.  Push $\inter(D_A)$ slightly into $A$ and $\inter(D_B)$ slightly into $B$.  The resulting sphere $S'$ intersects $T$ only in the circle $c$.  As we have just shown, each of the two components of $T-S'$ contains a non-planar surface, namely the part of $T$ lying in a ball component of $S^3 - S$.  Thus neither component of $T-S'$ is planar so, in particular, neither component is a disk.  Hence $S'$ is a reducing sphere for $T$ and $c = S' \cap T$ is essential in $T$.
\end{proof}

\begin{prop} \label{prop:toyexist}  Suppose, for the Heegaard split chamber complex above, $S$ is an incompressible sphere in a chamber of $M_A$ or $M_B$.  Then there is an orientation preserving homeomorphism $h_S: (S^3, T) \to (S^3, T_g)$ so that $h_S(S) \in \{S_i, i = 1, ..., g-1\}$.

Moreover, suppose for $\tau \in G(S^3, T)$ a homeomorphism $h_{\tau(S)}$ is similarly defined, for the Heegaard split chamber complex whose defining surface is $\tau(F)$. Then, under Assumption \ref{ass:inductive}, 
$h_{\tau(S)} \tau \sim h_S.$

%
\end{prop}

\begin{proof} Let $C$ be the chamber in which $S$ lies, say $C \subset M_A$, so $C = A_C \cup_{T_C} B_C$ with $A_C$ a handlebody.  

The main theorem of \cite{Sc1} says that the Heegaard surface $T_C$ in $C$ may be isotoped so that it is aligned with $S$, that is so that it intersects $S$ in at most one circle. 

{\em Claim:} When $S$ is aligned with $T_C$ in $C$ and the Heegaard surface $T \subset S$ is created by amalgamation of the splittings of the chambers, $S$ becomes a reducing sphere for the splitting $(S^3, T)$. 

There are three cases to consider:
\medskip


{\em Case 1:} $T_C$ cannot be isotoped in $C$ to be disjoint from $S$.  

$S$ is aligned with $T_C$, and in this case $T_C \cap S$ cannot be empty, so this intersection must be a single circle $c$.  Furthermore $c$ must be essential in $T_C$:  Indeed, if $c$ bounded a disk $D_T \subset T_C$ then the union of $D_T$ and the disk $S \cap A_C$ would be a sphere in the (irreducible) handlebody $A_C$, and the ball the sphere bounds could be used to isotope $D_T$ through $S \cap A_C$ and so off of $S$, contradicting the hypothesis of this case.  

Since $S$ is separating in $S^3$, the curve $c$ is separating in $T_C$.  Since $c$ is essential in $T_C$, neither component of $T_C - c$ is a disk, so both components are non-planar surfaces.  By the argument of EF5 above, each side of $c$ in $T$ is then non-planar, so $c$ is essential in $T$.  Under amalgamation, the disk $S \cap A_C$ becomes a disk in $A$.  

It is easy to arrange that, similarly, the disk $S \cap B_C$ remains a disk in $B$ after the amalgamation:  isotope the disk $S_B = S \cap B_C$ so that it intersects the defining $B-$disks for the amalgamation $\calB \subset B_C$  only in arcs.  Then $S_B$ intersects $B_C - \eta(\calB)$, which is a collar of $\bdd C$ in $B_C$, only in disks, namely the complement in $S_B$ of the arcs $S_B \cap \calB$.  Finally, properly isotope the disks $S_B - \eta(\calB)$ so that they avoid the $A$-disks of the amalgamation, so $S_B$ ends up entirely in $B$.  Then $S$ is the union of two disks, $S_A \subset A$ and $S_B \subset B$ along their common boundary $c$, an essential circle in $T$.  Hence $S$ is a reducing sphere for $(S^3, T)$.  

 \begin{figure}[ht!]
\labellist
\small\hair 2pt
\pinlabel  $S$ at 160 80
\pinlabel  $\Sigma$ at 280 50
\pinlabel  $\bdd_-B_C$ at 240 -5
\pinlabel  $\bdd_-B_C$ at 240 55
\pinlabel  $\bdd_+B_C=T_C$ at 240 130
\pinlabel  $\Sigma$ at 130 15
\pinlabel  $S_C$ at 115 70

\endlabellist
    \centering
    \includegraphics[scale=0.75]{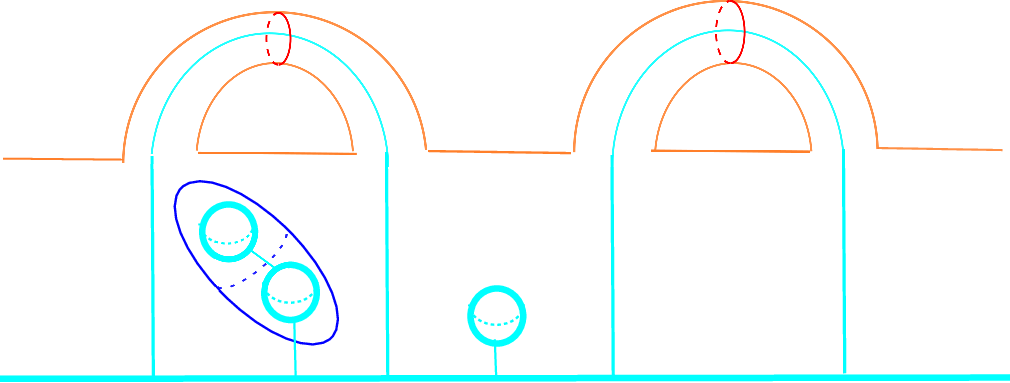}
     \caption{A spine $\Sigma$ (in aqua) of $B_C$ intersecting $S$ in a single point.} \label{fig:toyexist1}
    \end{figure}

{\em Case 2:} $T_C$ can be isotoped in $C$ to be disjoint from $S$ and $\genus(T_C) \geq 1$.

Since the chamber $C \subset M_A$, $A_C$ is a handlebody and so is irreducible.  Hence if $S$ is disjoint from $T_C$ it must lie entirely in the compression body $B_C$.  This implies that $B_C$ is reducible, so $\bdd_- B_C$ contains spheres, some of which lie on the other side of $S$ from $T_C = \bdd_+ B_C$ in $B_C$. Let $S_C \subset \bdd_- B_C$ be one of them.  Since $S$ separates $S_C$ from $T_C$ in $B_C$, any spine of $B_C$ intersects $S$, and there is some spine $\Sigma$ of $B_C$ that intersects $S$ in a single point (as pictured in Figure \ref{fig:toyexist1}).
Now isotope $T_C$ to be a neighborhood of that spine, so it intersects $S$ in a single circle $c$.   Amalgamation happens well away from $S$, so $S$ also ends up intersecting $T$ in $c$.  The circle $c$ divides $T_C$ into a disk, on the side of $S$ containing $S_C$,
and a positive genus surface on the other, since by assumption $\genus(T_C) \geq 1$.  By Lemma \ref{lemma:Sreduces} the part of $T$ that lies in the ball in $S^3$ bounded by $S_C$ (not visible in Figure \ref{fig:toyexist1}) also has genus $\geq 1$.  Hence each of the surfaces into which $S$ divides $T$ has genus $\geq 1$, so $S$ is a reducing sphere for $T$. Moreover the spine $\Sigma$ has been chosen so that the disk $S \cap B_C$ is one of the defining disks for the amalgamation, so it automatically persists as a disk in $B$ after amalgamation.  Thus $S$ becomes a disk intersecting $T$ in the single essential circle $c$ and so a reducing sphere for  $(S^3, T)$.
\medskip

{\em Case 3:} $T_C$ can be isotoped in $C$ to be disjoint from $S$ and $\genus(T_C) = 0$.

By assumption (3) on $F$ above, $C$ is a multiply punctured $3$-sphere.  Since $S$ is incompressible in $T_C$, each side of $S$ in $T_C$ contains some of the punctures, i. e. some sphere components of $\bdd_- C$.  Just as in Case 2, applying Lemma \ref{lemma:Sreduces}, $T_C$ can be isotoped to intersect $S$ in a single circle and the part of $T$ lying on each side of $S$ must have positive genus, so $S$ is a reducing sphere.  

\medskip

This establishes the Claim; the proof of the first statement then follows from \cite{Wa}.  

The proof of the second statement follows from the observation that $h_S \tau^{-1}(\tau(S))\in \{S_i, i = 1, ..., g-1\}$ so for $h_{\tau(S)}$ we could take $h_S\tau^{-1}$.  So, under the inductive assumption, Lemma \ref{lemma:preexistunique} says that any other choice of $h_{\tau(S)}$ is eyeglass equivalent to $h_S\tau^{-1}$, or $h_{\tau(S)} \tau \sim h_S$.
\end{proof}


\begin{prop} \label{prop:toyunique}  Suppose $S$ and $S'$ are not necessarily disjoint spheres, each in possibly different chambers of $M_A$ or $M_B$ and each is incompressible in the chamber in which it lies.  Suppose $h = h_S: (S^3, T) \to (S^3, T_g)$ and $h' = h_{S'}: (S^3, T) \to (S^3, T_g)$ are homeomorphisms as described in Proposition \ref{prop:toyexist}, so, in particular, $S$ and $S'$ are both reducing spheres for $(S^3, T)$.  
Under Assumption \ref{ass:inductive}, the homeomorphisms $h$ and $h'$ are eyeglass equivalent.
\end{prop}

\begin{proof} The proof is an examination of interlocking special cases.

{\bf Case 1:} $S$ and $S'$ are isotopic rel $T$.

In this case $S$ and $S'$ both divide $T$ into components $T_{\pm}$ of genus $g_{\pm} \geq 1$ where $g_+ + g_- = g$.  It follows that $h(S) = S_{g_+}$ or $S_{g_-}$, say the former.  After perhaps composing with the homeomorphism $h_{\rho}: (S^3, T_g) \to (S^3, T_g)$ defined before Lemma \ref{lemma:preexistunique}, we may as well assume that also $h'(S') = S_{g_+}$, so $h' h^{-1}$ leaves $S_{g_+}$ invariant. Apply Lemma \ref{lemma:preexistunique} to two copies of $S_{g_+}$; the homeomorphism $h' h^{-1}$; and the identity homeomorphism $id_{(S^3, T_g)}$.   Deduce that $h' h^{-1} \sim id_{(S^3, T_g)}$.  Hence $h \sim h'$.

\medskip

{\bf Case 2:} $S$ and $S'$ both lie in the same chamber $C$ of $S^3 - F$ and there is an eyeglass move of $T_C$ in $C$ which carries $S$ to $S'$.

As discussed in the proof of \cite[Theorem 5.1]{FS2}, the handle slides that define the eyeglass move in $C$ determine, up to eyeglass twists, parallel handle slides in $T$ after amalgamation.  So in this case there is also an eyeglass move of $T$ in $S^3$ that carries $S$ to a sphere isotopic to $S'$ rel $T$.  The proof of Case 2 then follows from Case 1.
\medskip

{\bf Case 3:}  $S$ and $S'$ are disjoint.

Apply Lemma \ref{lemma:preexistunique}.  

\medskip

{\bf Case 4:} $S$ and $S'$ both lie in the same chamber $C$ and are isotopic in $C$.

The main theorem of \cite{FS2} says that there is a sequence of reducing spheres for $T_C$ starting at $S$ and ending at $S'$ so that each successive pair is either isotopic rel $T_C$, differ by an eyeglass move in $(C, T_C)$, or are disjoint, because they differ by a bubble pass in $T_C$.  The result then follows from Cases 1, 2 and 3.
\medskip

{\bf Case 5:} $S$ and $S'$ are isotopic in $S^3 - F$ to disjoint spheres.

If $S$ and $S'$ lie in different chambers of $S^3 - F$ the result follows from Case 3.  Suppose $S$ and $S'$ lie in the same chamber $C$ of $S^3 - F$ and can be isotoped {\em in $C$}, but not necessarily rel $T_C$, to be disjoint.  According to the main theorem of \cite{Sc1} the Heegaard surface $T_C$ can be aligned with the sphere set $S \cup S'$, so that the $S$ and $S'$ become disjoint reducing spheres for $T_C$.  This new alignment of $S$ and $S'$ may change $h$ and $h'$, but by Case 4 the eyeglass equivalence class of the homeomorphisms will not change.  
The result then follows from Lemma \ref{lemma:preexistunique}.  
\medskip

{\bf The general case}
If $S$ and $S'$ lie in different chambers of $S^3 - F$ the result follows from Case 3.  Suppose they lie in the same chamber $C$.  Ignoring $T_C$ for the moment, recall that in classic 3-manifold theory a standard innermost disk argument shows that there is a sequence of incompressible spheres in $C$ beginning with $S$ and ending with $S'$ so that sequential spheres are isotopic in $C$ to disjoint spheres.  The result then follows from Case 5.
\end{proof}

The following corollary is then a summary:

\begin{cor}   \label{cor:naturality} Suppose in a Heegaard split chamber complex (with notation as in conditions (1)-(4) preceding Lemma \ref{lemma:Sreduces}), 
the manifold $S^3 - F$ contains an incompressible sphere, then under Assumption \ref{ass:inductive}, $F$ determines a natural eyeglass equivalence class $\tilde{h}_F$ of homeomorphisms $(S^3, T) \to (S^3, T_g)$.

Moreover, for $\tau \in G(S^3, T)$, $\tilde{h}_{\tau(F)} \tau = \tilde{h}_F$.
%
\end{cor} 

\begin{proof}  Choose any incompressible sphere $S$ in $S^3 - F$ and let $h_S: (S^3, T) \to (S^3, T_g)$ be a homeomorphism as given by Proposition \ref{prop:toyexist}.  Then Proposition \ref{prop:toyunique} shows that, up to eyeglass equivalence, $h_S$ is independent of any choice, including the choice of incompressible sphere $S$.  The last sentence follows from the last sentence of Proposition \ref{prop:toyexist}.  
\end{proof}

To simplify (but slightly abuse) notation, we will often use $h_F:  (S^3, T) \to (S^3, T_g)$ to denote any representative of the eyeglass equivalence class $\tilde{h}_F$.  With this notation the last sentence of \ref{cor:naturality} would be $h_{\tau(F)} \tau \sim h_F$.
\medskip


\begin{defin}  \label{defin:cocertify} Suppose there are two Heegaard split chamber complexes in $S^3$ with corresponding associated surfaces $F, F' \subset S^3$ and each Heegaard split chamber complex supports the same Heegaard splitting surface $T$ for $S^3$.  Suppose each of $S^3 - F$ and $S^3 - F'$ contains an incompressible sphere, so  (eyeglass equivalence classes of) homeomorphisms $h_F,  h_{F'}:  (S^3, T) \to (S^3, T_g)$ are defined.  If 
$ h_F \sim h_{F'}$ then $F$ and $F'$ {\em cocertify}.
\end{defin}

(This definition will be expanded later, see comments following Definition \ref{defin:flagcertify}.)
\medskip

\begin{lemma} \label{lemma:cocert1}
Suppose, in Definition \ref{defin:cocertify} $F \subset F'$ and a component $C'$ of $S^3 - F'$ contains an incompressible sphere $S$ that is also incompressible in the component $C \supset C'$ of $S^3 - F$ in which it lies.  Then 
$F$ and $F'$ cocertify.  
\end{lemma}

\begin{proof} Via \cite{Sc1} isotope in $C$ the incompressible sphere to a sphere $S \subset C$ that is aligned with $T_C$.  By the Claim in the proof of Proposition \ref{prop:toyexist} $S$ becomes a reducing sphere for the splitting $(S^3, T)$ after amalgamation.  Similarly isotope $S$ in $C'$ to a sphere $S' \subset C'$ that is aligned with $T_{C'}$ and again conclude that $S'$ upon amalgamation becomes a reducing sphere for $(S^3, T)$.  Then by definition $h_{S'} \in \tilde{h}_{F'}$ and $h_{S} \in \tilde{h}_{F}$.  Since $C' \subset C$, the isotopies of both $S$ and $S'$ to alignment take place in $C$ so, by Proposition \ref{prop:toyunique}, $h_{S'}$ also represents $\tilde{h}_{F}$.  Since $ \tilde{h}_F$ and $ \tilde{h}_{F'}$ have the representative $h_{S'}$ in common, they cocertify.
\end{proof}

\begin{cor}  \label {cor:toy} Suppose $F \subset S^3$ is the surface associated with a Heegaard split chamber complex that supports $T$.  Suppose $\tau \in G(S^3, T)$.  If the Heegaard split chamber complexes given by $F$ and $\tau(F)$ cocertify, then $\tau \in \calE$.  
\end{cor}

\begin{proof} Corollary  \ref{cor:naturality} 
shows that $h_F \sim h_{\tau(F)} \tau$.  On the other hand, the assumption that $F$ and $\tau(F)$ cocertify implies that $h_{\tau(F)} \tau \sim h_F \tau$.  Hence $h_F \sim h_F \tau$ or $id_{(S^3, T)} \sim \tau$, as required.  
\end{proof}

\subsection{An overview of the argument} \label{subsect:overview}
Corollary \ref{cor:naturality} suggests a line of attack towards proving $G(S^3, T) = \calE$: For each $\tau \in G(S^3, T)$, find a Heegaard split chamber complex in $S^3$ with associated surface $F$ so that 
\begin{itemize}
\item  the Heegaard split chamber complex supports $T$
\item some chamber contains an incompressible sphere and 
\item $h_F \sim h_{\tau(F)}$.  
\end{itemize}

There are three conditions required for this example to work that are difficult to realize:
\begin{enumerate}
\item The Heegaard splitting of each ball chamber has genus $\geq 1$, as required by the definition of a Heegaard split chamber complex.  This property is needed to show, as in Lemma \ref{lemma:Sreduces} and Proposition \ref{prop:toyexist} above, that an incompressible sphere in a chamber determines a reducing sphere for $T$.
\item Some chamber contains an incompressible sphere
\item  The Heegaard split chamber complexes derived from $F$ and $\tau(F)$ cocertify.
\end{enumerate}

Section \ref{sect:intro} gave a brief description of how the {\bf second requirement} is addressed in the general case: given a position of $T$ in $S^3$, a pair of weakly reducing disks is found for $T$, with weak reduction defining a chamber complex.  Next a particularly useful level sphere (a `guiding sphere', see Section \ref{sect:guiding}) $S$ is found, and innermost disks in $S - T$ are used to decompose the chamber complex into a typically more complex one.  The process is repeated until all circles of intersection are removed.  It will be shown that at some stage in this {\em disk decomposition} process a chamber will contain an incompressible sphere or a sphere that is equally useful for determining an eyeglass equivalence class of homeomorphisms from $(S^3, T) \to (S^3, T_g)$.  (See Definition \ref{defin:flagcertify} of a {\em certificate}.)

This process raises many technical questions:  How to find a guiding sphere whose associated disk decomposition sequence must issue a certificate, Section \ref{sect:balance}.  How to ensure, for a given sphere, that the certificates issued during this process cocertify, Section \ref{sect:certify}.  The most difficult problem is showing that a different choice of such a sphere will result in an equivalent certificate.  This takes many steps, occupying Sections \ref{sect:deflate} through Section \ref{sect:balisotopy}.  The philosophy behind these steps borrows heavily from \cite{FS1} (see next paragraph): the guiding spheres are isotopic in $S^3$ through appropriate guiding spheres (Section \ref{sect:balisotopy}) and such an isotopy between them can be broken into small steps, each of which either adds to, subtracts from, or delays a single disk in the associated disk decomposition sequence.  We then examine the effect of each such minimal change.  The upshot is that, up to cocertification, the certificate issued depends only on the Heegaard split chamber complex itself, and not on the choice of guiding sphere.  

To address the {\bf third requirement}, that is, to show that the Heegaard split chamber complexes derived from $F$ and $\tau(F)$ cocertify, we construct a sequence of Heegaard split chamber complexes beginning with $F$ and ending with $\tau(F)$ so that each successive pair cocertify, see Section \ref{sect:finale}.    The heavy lifting for this part of the argument was done in \cite{FS1}.  Recall that a Heegaard split chamber complex is derived from a Heegaard splitting $(S^3, T)$ by weak reduction, see Section \ref{sect:Heegaard1}.  In \cite{FS1} a series of pairs of weakly reducing disks is found so that, roughly, the resulting Heegaard split chamber complexes are a series beginning with some $F$ and ending with $\tau(F)$, see Section \ref{sect:finale}.  Moreover, successive pairs of weakly reducing disks are related in such a way that Proposition \ref{prop:heegaddisk} guarantees the resulting Heegaard split chamber complexes cocertify.  So in the end $F$ and $\tau(F)$ cocertify and the third requirement above is satisfied.  
%

This leaves the {\bf first requirement}, that each chamber in the chamber complex comes with a splitting surface of genus $\geq 1$.  Establishing this is particularly vexing, and adds great technical complexity to the argument, for reasons we now describe: 

It is necessary to eliminate, whenever they arise, any ball chambers that have genus $0$ splittings.  In the argument below these will be 
the ball chambers (called {\em goneballs}) that are absorbed into the surrounding chamber and so their bounding spheres disappear from $F$.  Goneballs can arise in troublesome ways: a chamber that is a handlebody at one stage can be cut up by a disk decomposition to become a collection of balls; if the handlebody had trivial Heegaard splitting, the resulting balls will be goneballs and their boundaries need to be deleted from $F$.  Since chambers may be deleted in this way, what is to prevent $F$ from disappearing entirely, as pieces are cut up in this manner? We will show that unless the chamber complex is of a particularly simple type, called {\em tiny}, its defining surface will never completely disappear, see Subsection \ref{subsect:tiny}.

There is an additional but related complication: it is a classical result that a non-trivial Heegaard splitting of a ball or handlebody chamber is stabilized, that is it will contain bubbles (standard genus one summands).  It follows from \cite{FS2} that disk decomposition of a Heegaard split chamber gives a Heegaard splitting of the resulting chambers, but it is defined only up to eyeglass moves and passing bubbles through the decomposing disks. As noted above, it will be important to know if a newly created handlebody chamber could have trivial Heegaard splitting, so we want to flag handlebody chambers that, because of how they arise, are known to have non-trivial splittings.  This gives rise to the notion of {\em flagged chamber complexes}: the exact Heegaard splitting of each chamber is typically unknown, but certain handlebody chambers (including ball chambers) are flagged as necessarily having non-trivial splittings.  See Section  \ref{sect:prefdisk}.  

Following this informal overview, we  begin the more formal argument.

\section{Chamber complexes and disk decomposition} \label{sect:chamberintro}

Suppose $M$ is a compact orientable $3$-manifold. 
Let $F \subset S^3$ be a (typically disconnected) closed surface in $\inter(M)$ that divides $M$ into two typically disconnected $3$-manifolds $M_A$ and $M_B$, with each component of $F$ incident on one side to $M_A$ and on the other to $M_B$.  Call the collection $\calC = (F, M_A, M_B)$ a {\em chamber complex} \index{Chamber complex} with defining surface $F = F(\calC)$.  Each component of $M - F$ is called a {\em chamber}, with those in $M_A$ called $A$-chambers and those in $M_B$ called $B$-chambers.   

A Heegaard splitting $M = A \cup_T B$ is a familiar example of a chamber complex in $M$, with $\calC = (T, A, B)$.  A Heegaard splitting belongs to a class of chamber complexes, called {\em tiny}, that, for our purposes, contain too little information to be useful.  However, a Heegaard splitting together with a pair of weakly reducing disks (see Section \ref{sect:Heegaard1} below) does produce a possibly useful chamber complex, through a process called {\em disk decomposition}, or, more formally, {\em chamber complex decomposition}, which we now describe. \index{Disk decomposition} \index{Chamber complex decomposition}
\bigskip  

{\bf Phase one of disk decomposition:}  For $\calC$ a chamber complex in $M$, let $\calD$ be a collection of disjoint properly embedded disks, some perhaps inessential, in the chambers of $\calC$, with $\bdd \calD$ disjoint from $\bdd W$, so $\bdd \calD \subset F$.    We will call $\calD$ a {\em disk set} \index{Disk set} in $\calC$.   Let $F_{\calD} \subset M$ \index{$F_{\calD} \subset M$} be the surface obtained by doing surgery on $F(\calC)$ along the disks $\calD$, and denote by $\hat{\calC}_{\calD}$ \index{$\hat{\calC}_{\calD}$} the associated chamber complex  $M - F_{\calD}$. (The reason for the circumflex will be given shortly.) In particular, the surgery alters $M_A$ by deleting a bicollar neighborhood of each disk in $\calD \cap M_A$ and adding a bicollar neighborhood of each disk in $\calD \cap M_B$.  The symmetric statement applies to $M_B$.  

Each disk $D$ in $\calD$ is the core of a $2$-handle used in the surgery as just described.  The bicollar neighborhood of $\bdd D$ in $F$ will be called the {\em belt annulus} of the surgery; the two copies of $D$ in $F_{\calD}$ that are the result of the surgery are called the {\em scars} of $D$.  \index{Scars, internal and external} Let $C'$ be a chamber of $\hat{\calC}_{\calD}$ and $D \in \calD$ a disk which leaves a scar or two on $\bdd C'$.  If the disk $D$ lies outside the new chamber $C'$ then so does the associated belt annulus and we call the scars {\em external} scars on $\bdd C'$.  If the disk lies inside $C'$ we call the scars {\em internal} scars on $\bdd C'$.  Clearly if $D$ leaves one internal scar on $\bdd C'$ it will leave two, since the associated belt annulus lies in $C'$, but the two internal scars may be on different components of $\bdd C'$.   

One can think of a chamber $\hat{C}$ of $\hat{\calC}_{\calD}$ as obtained in two stages: it starts as the complement $C_-$ of a bicollar of $\calD \cap C$ in a chamber $C$ of $\calD$, a result of surgery on $\calD \cap C$.  Then $2$-handles are added along those disks in $\calD$ that are incident to $\bdd C$ but lie outside $C$.  The chamber $\hat{C}$ is called a {\em remnant} \index{Remnant of chamber} of $C$, even though strictly speaking it doesn't lie entirely inside of $C$.

{\bf Phase two of disk decomposition:}  Among the components of the surgered surface $F_{\calD}$ may be spheres that bound balls in $M$.  In the second phase of disk decomposition a collection of these balls is chosen and, for each ball $G \subset M$ in the collection (henceforth called a {\em goneball}) \index{Goneball} all components of $F_{\calD}$ that lie in $G$ , including the sphere $\bdd G$, are removed from $F_{\calD}$.    Whether a ball in $M$ bounded by a sphere in $F_{\calD}$ will be chosen to be a goneball of the decomposition is a delicate and important part of the theory. (See, for example, Rules \ref{rule:disky1} and \ref{rule:diskyHeeg}.) 
After eliminating all components of $F_{\calD}$ that lie in goneballs, call the resulting chamber complex $\calC_{\calD}$ (note: no circumflex).

Any chamber $C_D$ in the chamber complex $\calC_{\calD}$ is obtained from a chamber $\hat{C}_D$ in $\hat{\calC}_{\calD}$ by attaching balls, each corresponding to a goneball for the decomposition.  We then say that $C_D$ is a remnant of the chamber $C$ in $\calC$ if $\hat{C}_D$ is a remnant of $C$ in $\hat{\calC}_{\calD}$.  Notice that a chamber $C$ in $\calC$ may have no remnants in $\calC_{\calD}$.  For example, $C$ might be handlebody for which $\calD$ contains a complete collection of meridian disks, so the remnants in $\hat{\calC}_{\calD}$ are all balls.  If these all happen to be declared goneballs, then there are no remnants of $C$ in $\calC_{\calD}$.  
\bigskip

%
 
For $\calC, \calD, F, F_{\calD}, \calC_{\calD}$ and $\hat{\calC}_{\calD}$ as above, 
suppose $F'$ is a component of the surface $F_{\calD}$ such that one of the complementary components $W$ of $F'$ in $M$ is a handlebody. (A ball is considered a genus $0$ handlebody).  $W$ is not necessarily a chamber of $\hat{\calC}_{\calD}$, but rather the union of those chambers that lie within it.

\begin{defin} \label{defin:diskyhandle}
The handlebody $W$ is {\em disky} \index{Disky} if each component of $F \cap \inter(W)$ is a disk.  
\end{defin}

\begin{lemma} \label{lemma:disky} If $W$ is disky then 
\begin{enumerate}
\item No component of $F$ lies entirely in the interior of $W$.
\item Each component of $F_{\calD}$ that lies in the interior of $W$ is a sphere.
\item Suppose $D \in \calD$ lies in the interior of $W$.  Then the corresponding two scars do not lie on the same component of the surface $F_{\calD}$.  
\end{enumerate}
\end{lemma}

\begin{proof} Take the statements in order: 

1) Each component of $F$ is a closed surface, so none can be a disk.

2) Suppose $F''$ is a component of $F_{\calD}$ that lies in the interior of $W$, and let $F''_-$ be the compact surface obtained from $F''$ by deleting the interior of all scars of the disk surgery that lie in $F''$.  Then $F''_-$ lies in a component of $F \cap \inter(W)$ and, by assumption, this component is a disk.  Hence the surface $F''_-$ has no genus, so neither can $F''$.  See Figure \ref{fig:disky}.

3) Suppose the two scars left by surgery on a disk $D \in (\calD \cap \inter(W))$ were on the same component $F''$ of $F_{\calD}$.  If $F'' = F'$ then the belt annulus for the surgery is a component of  $F \cap \inter(W)$ that is not a disk, contradicting our assumption that $W$ is disky.  Similarly, suppose $F''$ were in the interior of $W$, and let $\alpha$ be an arc in $F''$ running between the two scars.  Then the union of the belt annulus and a collar of $\alpha$ would be a punctured torus lying in some component of $F \cap \inter(W)$, so that component cannot be a disk, again contradicting our assumption that $W$ is disky. 
\end{proof}

\begin{figure}[ht!]
\labellist
\small\hair 2pt
\pinlabel  $W$ at 120 190
\pinlabel  $F'$ at 200 220
\pinlabel  $\calD_0$ at 180 125
\pinlabel  $\calD$ at 260 150
\pinlabel  $F''_-$ at 380 140
\pinlabel  scars at 350 120
\endlabellist
    \centering
    \includegraphics[scale=0.7]{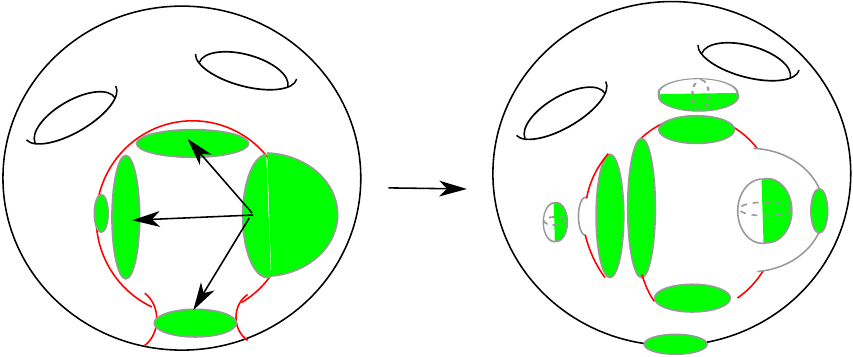}
     \caption{Construction of $F''_-$ in proof of Lemma \ref{lemma:disky} (2)} \label{fig:disky}
    \end{figure}

\begin{lemma} \label{lemma:subdisky} Suppose the handlebody $W$ is disky, and $F''$ is a component of $F_{\calD}$ that lies in the interior of $W$. Then $F''$ bounds a disky ball in $W$.\end{lemma}


\begin{proof}  By Lemma \ref{lemma:disky} any such component $F''$ is a sphere and, since the handlebody $W$ is irreducible, it will bound a ball in $W$.  The issue is to show that such a ball is disky.  By assumption, $F \cap \inter(W)$ consists of a collection $E$ of disks; if $E = \emptyset$ then there are no components of $F_{\calD}$ in the interior of $W$ and there is nothing to prove.  

Otherwise, since $E$ lies inside $W$, the annulus in $E$ adjacent to any circle $c \in \bdd E$ is a belt annulus for a surgery disk $D \in \calD$ that lies in $W$, with one of its scars on $\bdd W$.  Let $\calD_0 \subset \calD$ be the subset of disks that lie in $W$; we now induct on $n = |\calD_0| >0$.

Let $D_0 \in \calD_0$ be a disk whose boundary is innermost in $E$ among the circles $\bdd \calD_0$; let $E_0$ be the subdisk of $E$ that it bounds.  Since $W$ is irreducible, the disks $D_0$ and $E_0$ together bound a ball $B_0$ in $W$.  The ball is disjoint from $F$ by Lemma \ref{lemma:disky}, so we can think of $D_0$ as just being parallel to $E_0$, via the ball $B_0$.  Inductively apply the lemma to $\calD' = \calD_0 - D_0$, creating a collection $S'$ of spheres, each of which, by inductive assumption, bounds a disky ball in $W$.  

Finally, do surgery on the disk $D_0$.  The result of the surgery is two-fold: The ball $B_0$ bounded by $E_0$ and a copy of $D_0$ becomes a ball chamber that is disky because it is empty.  And the sphere $S'_0$ in $S'$ that contains $E_0$ is altered, in effect, by isotoping $E_0$ across $B_0$ to the other copy of $D_0$.  Depending on which side of $S'_0$ the disk $D_0$ lies, this change in $S'_0$ may add the disk $E_0 \subset F$ to the ball that $S'_0$ bounds in $W$. But since $E_0$ is a disk, it does not change the fact that the ball is disky. 
See Figure \ref{fig:subdisky}.  
\end{proof}

\begin{figure}[ht!]
\labellist
\small\hair 2pt
\pinlabel  $D_0$ at 140 70
\pinlabel  $D_0$ at 200 90
\pinlabel  $E_0$ at 340 90
\pinlabel  $B_0$ at 370 50
\endlabellist
    \centering
    \includegraphics[scale=0.7]{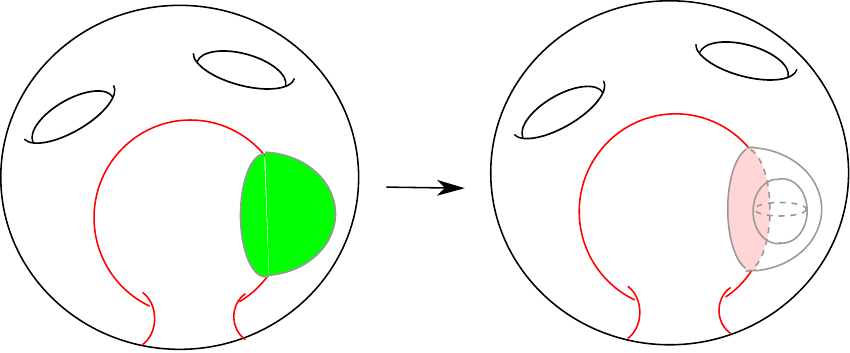}
     \caption{$E_0 \subset F$ in the interior of ball bounded by $S'_0$. (Figure not incorporating decomposition by $\calD'= \calD_0 - D_0$.) } \label{fig:subdisky}
    \end{figure}
    
    The decision process that will be used to decide if sphere components of $F_\calD$ in a chamber complex decomposition bound goneballs will have certain properties.  The ultimate decision process is subtle, but some of those properties can be described already.  At this point, view them as rules that disk decompositions follow; later it will be shown that the decision process that will be used satisfies these rules.  To that end, declare the rule:
    
\begin{rrule} \label{rule:disky1} In a chamber complex decomposition, only disky balls can be goneballs.
    \end{rrule}
    
    Following Lemma \ref{lemma:subdisky} we then have
    
\begin{cor} \label{cor:allballs} In the second phase of a chamber complex decomposition, any component that is removed from $F_{\calD}$ is a sphere that bounds a disky ball.
\end{cor}

Continue with $\calC, \calD, F, F_{\calD}, \hat{\calC}_{\calD}$ and $\calC_{\calD}$ be as above.  Denote the two-stage operation of chamber complex decomposition just described by $$\calC \xrightarrow{\calD} \calC_{\calD}.$$

The next rule limits the amount of disturbance caused by adding or removing a single disk from the chosen disk set.

\begin{rrule} \label{rule:persist} Suppose $\calD$ is a disk set in $\calC$, $D \in \calD$, and $\calD_-$ is the disk set $\calD - D$.   Let $S$ be a sphere component of $F_{\calD}$ that does not contain either scar from $D$.  Then $S$ bounds a goneball in $\calC_{\calD}$ if and only if it bounds a goneball in $\calC_{\calD_-}$.
\end{rrule}
    
    \begin{cor} \label{cor:persist} Suppose $\calD$ consists of a single disk. Then a sphere in $F_{\calD}$ that does not contain a scar of $D$ does not bound a goneball.  
\end{cor}
\begin{proof}  Since $\calD_- = \emptyset$ such a sphere would necessarily be a component of $F = F_{\calD_-}$ and so does not bound a goneball in $\calC = \calC_{\calD_-}$.  By Rule \ref{rule:persist} it does not bound a goneball in $\calC_\calD$.
\end{proof}

An example of a decision process that satisfies both rules is to simply declare each disky ball to be a goneball.  For example, it satisfies Rule \ref{rule:persist} because a sphere $S \subset F_\calD$ not containing either scar will also be in $F_{\calD_-}$ and a ball it bounds is disky in either $\calC_\calD$ or $\calC_{\calD_-}$ if and only if $F$ intersects the ball only in disks.   Unfortunately declaring each disky ball to be a goneball would be too broad for our purposes; too much information would be lost.  But the rules so far do suffice to prove an important property:

\begin{prop}  \label{prop:remnant} Suppose $C$ is a chamber of $\calC$ so that every remnant of $C$ in $\calC_{\calD}$ is a disky handlebody.  Then $C$ is a handlebody.
\end{prop}

\begin{proof} We proceed by induction on $|\calD|$, noting that if $\calD = \emptyset$ there is nothing to prove.  

{\em Case 1}:  There is a disk in $\calD$ whose boundary is inessential in $\bdd C$.

In this case, let $D$ be a disk whose boundary is innermost among all such disks, let $E$ be the disk in $\bdd C - \calD$ bounded by $\bdd D$.  Push $\inter(E)$ slightly into the chamber of $\calC$ on the opposite side of $D$.  That chamber is either $C'$ if $D$ lies in $C$ (the top row of Figure \ref{fig:remnant1}) or $C$ itself, if $D$ lies in a chamber $C'$ adjacent to $C$ (the bottom row of Figure \ref{fig:remnant1}).   
Consider the sphere $S = D \cup E$, lying either in a remnant of $C'$ or a remnant of $C$, and in either case parallel in that remnant to a sphere component $S'$ of $F_{\calD}$.  
\medskip

{\em Claim 1}: With the proper choice of $D$ for the construction above, $S$ bounds a ball $B_S$ in $M$ containing $S'$.  Moreover, if $D$ lies in $C'$, so $S$ is in a remnant of $C$, then the subball of $B_S$ bounded by $S'$ is a goneball.

{\em Proof of Claim 1:}  Suppose first that $S$ lies in a remnant $C'_r$ of $C'$, so $D$ lies inside $C$, as in the top panel of Figure \ref{fig:remnant1}.  Then $S'$ is a sphere in the boundary of a chamber $C_r$ of $\hat{\calC}_{\calD}$ that is a remnant of $C$.  If $C_r$ is a goneball the claim follows.  If instead 
$C_r$ remains as a chamber in $\calC_{\calD}$, $C_r$ must be a ball, since every remnant of $C$ in $\calC_{\calD}$ is a handlebody, proving the claim in this case.  

Suppose instead that $S$ lies in a remnant $C_r$ of $C$, so $D$ lies inside $C'$, as in the bottom panel of Figure \ref{fig:remnant1}.   Then by hypothesis $C_r$ is a handlebody, and so is irreducible. Hence $S$ bounds a ball $B_S$ in $C_r$. If $B_S$ lies on the side of $S$ containing $S'$ the claim is shown.  Suppose $B_S$ lies on the other side of $S$, so $S'$ bounds a ball in $C_r$ and that ball contains the punctured handlebody $\bdd C - E$.  Since each remnant of $C$ is disky, it follows that $\bdd C - E$ must be a disk, so $C$ is a ball.  In this case use instead of $D$ a disk $D'$ in $\calD$ whose boundary is innermost among those in the disk $\bdd C - E$ and repeat the construction.  The new sphere corresponding to $S$ lies in $B_S$ so it bounds a ball entirely in $C_r$, a ball that contains the disk $E'$ in $\bdd C - E$ bounded by $\bdd D'$, i. e. the one that now corresponds to $E$.  This proves Claim 1, using the disk $D'$ instead of $D$.  


\begin{figure}[ht!]
\labellist
\small\hair 2pt
\pinlabel  $C$ at 40 350
\pinlabel  $C'$ at 110 380
\pinlabel  $D$ at 70 300
\pinlabel  $D$ at 160 100
\pinlabel  $E$ at 110 300
\pinlabel  $C_r$ at 420 320
\pinlabel  $C_r$ at 290 120
\pinlabel  $C'_r$ at 310 390
\pinlabel  $S$ at 340 320
\pinlabel  $S'$ at 380 320
\endlabellist
    \centering
    \includegraphics[scale=0.6]{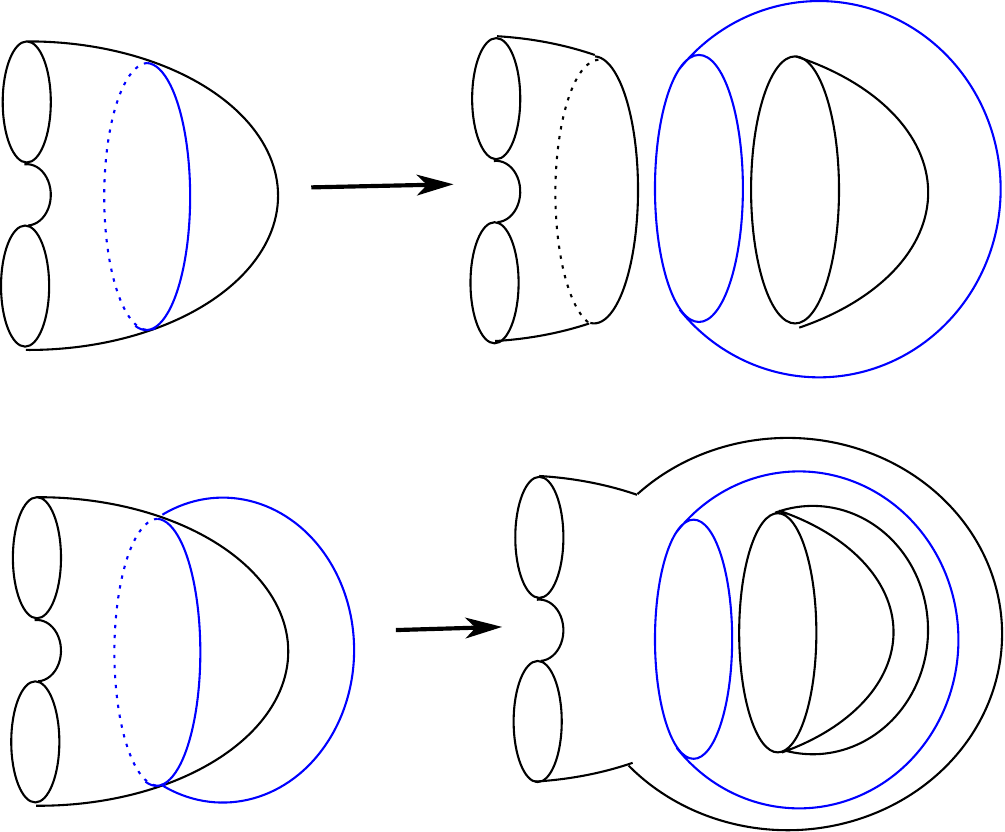}
     \caption{When a disk $D \in \calD$ has $\bdd D$ inessential in $\bdd C$} \label{fig:remnant1}
    \end{figure}

{\em Claim 2}: The disk $D$ of Claim 1 is inessential in the chamber (either $C$ or $C'$) in which it lies.  

{\em Proof of Claim 2:}  Suppose instead that $D$ is essential in the chamber in which it lies, so there are closed components of $F$ lying between $D$ and $E$, i. e. in the ball $B_S$.  By Lemma \ref{lemma:disky} the subball of $B_S$ bounded by $S'$ could not then be a disky ball,  so by Rule \ref{rule:disky1}, $S'$ remains as a sphere component of $F(\calC_{\calD})$.  Claim 1 then implies that $D$ lies in $C$ so $S$ is in a remnant of $C'$.  Then, as in Claim 1, $S'$ is a sphere in the boundary of a chamber of $\hat{\calC}_{\calD}$ that is a remnant of $C$. Since $S'$ remains in $F(\calC_{\calD})$ the ball it bounds in $B_S$ must be a ball remnant $B'$ in $\calC_{\calD}$ since every remnant of $C$ is a handlebody.  But even then there is a contradiction, for every remnant of $C$ is assumed to be disky, and here $B'$ contains closed components of $F$, contradicting Lemma \ref{lemma:disky}, and so proving Claim 2.  


\medskip

Let $\calD_- = \calD - D$.  Observe that since $D$ is inessential by Claim 2, $\calD_-$ still satisfies the hypotheses of Proposition \ref{prop:remnant}:  Indeed, from Case 1 and Rule \ref{rule:persist}, the only difference in the remnants of $C$ if we add $D$ back to $\calD_-$ before decomposing is:
\begin{itemize}
\item possibly adding a disky ball bounded by $S'$ to the set of remnants of $C$ or
\item isotoping $E$ to $D$, if the ball that $S'$ bounds is a goneball.  This is the case when $D \in C'$ and may be the case when $D \in C$.
\end{itemize}
Neither will affect whether each remnant of $C$ is a disky handlebody.  
So, by inductive assumption, replacing $\calD$ with $\calD_-$ leads to the conclusion of Proposition \ref{prop:remnant}, namely $C$ is a handlebody. This concludes the proof of Proposition \ref{prop:remnant} for Case 1.

{\em Case 2:}  Any disk incident to $\bdd C$ has essential boundary in $\bdd C$.  (In particular, any disk incident to $\bdd C$ is essential.)  

First consider a remnant $C_r$ of $C$ in $\calC_{\calD}$, which by hypothesis of the Proposition is a disky handlebody.  
Since $C_r$ is a remnant of $C$, each scar on $\bdd C_r$ is incident to $\bdd C$, so by the hypothesis of this case its boundary is essential in $\bdd C$.  
But since $C_r$ is disky, this implies it can have no internal scars, i. e. $F$ is disjoint from $\inter(C_r)$, since the belt annulus for an internal scar lies on a component of $\bdd C \cap \inter(C_r)$ and these are all disks.  This in turn implies that $\inter(C_r)$ contains no goneballs, so $C_r$ is a remnant of $C$ in $\hat{\calC}_{\calD}$, before goneballs are absorbed. Thus the remnants of $C$ in $\hat{\calC}_{\calD}$ consist of a union of disky handlebodies, though among these may be balls that become goneballs in $\calC_{\calD}$. See left side of Figure \ref{fig:remnant2}.

\begin{figure}[ht!]
\labellist
\small\hair 2pt
\pinlabel  $C$ at 280 220
\pinlabel  $D$ at 380 240
\pinlabel  $\calD_-$ at 265 295
\pinlabel  $\calD_-$ at 360 220
\pinlabel  $\calD_-$ at 290 240
\pinlabel  $\calC$ at 220 200
\pinlabel  $\calD$ at 220 160
\pinlabel  $\calD_-$ at 360 160
\pinlabel  $\calC_{\calD}$ at 40 40
\pinlabel  goneball at 210 40
\pinlabel  $\calC_{\calD_-}$ at 360 40
\pinlabel  $C_r$ at 100 90
\endlabellist
    \centering
    \includegraphics[scale=0.7]{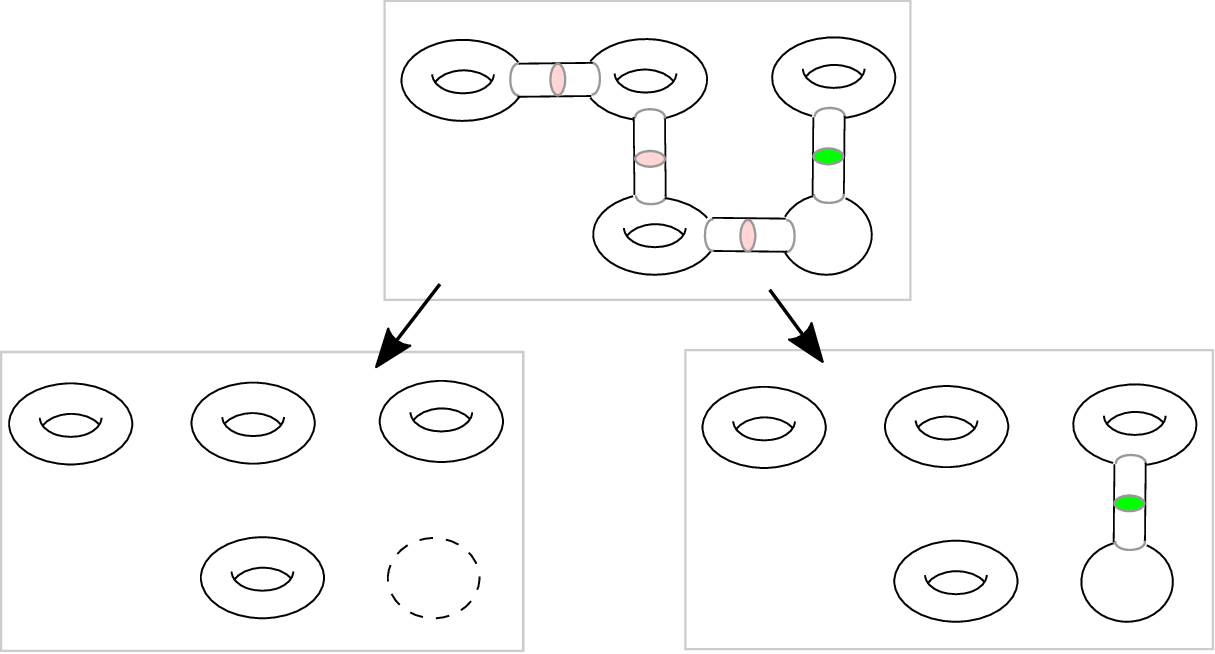}
     \caption{When each disk $D \in \calD$ has $\bdd D$ essential in $\bdd C$} \label{fig:remnant2}
    \end{figure}

Let $D$ be any disk in $\calD$ and, as in Case 1, let $\calD_- = \calD - D$.  Inductively, it suffices to show that $\calD_-$ still satisfies the hypotheses of the Proposition, namely that every remnant of $C$ in $\calC_{\calD_-}$ is a disky handlebody.  If $D$ is not incident to $\bdd C$ then removing $D$ from $\calD$ has no effect on the remnants of $C$ in $\hat{\calC}_{\calD_-}$, so the remnants are the same as those in $\hat{\calC}_{\calD}$.  By Rule \ref{rule:persist},   the remnants of $C$ in $\calC_{\calD_-}$, are then also the same as those in $\calC_{\calD}$, namely a union of disky handlebodies. 
Thus in this case $\calD_-$ satisfies the hypothesis of Proposition \ref{prop:remnant} as required.  

What remains is the case that $D$ is incident to $\bdd C$.   See right side of Figure \ref{fig:remnant2}. Since there are no internal scars on remnants of $C$, the two scars in $\hat{\calC}_{\calD}$ left by $D$ are external scars, so $D$ lies in $C$.  Thus $\hat{\calC}_{\calD_-}$ is obtained from $\hat{\calC}_{\calD}$ by simply attaching a $1$-handle dual to $D$ to a component, or between two components, of $\hat{\calC}_{\calD}$.  Since each such component is a disky handlebody, the result of adding the $1$-handle is also a handlebody, and it is disky because its interior is still disjoint from $F$.  Hence $\hat{\calC}_{\calD_-}$ is a union of disky handlebodies.  $\calC_{\calD_-}$ is obtained from $\hat{\calC}_{\calD_-}$ by removing some ball components (the goneballs).  Hence $\calC_{\calD_-}$ is also a union of disky handlebodies, so it satisfies the hypothesis of Proposition \ref{prop:remnant} as required.
\end{proof}

\begin{defin}  \label{defin:tiny1} A chamber complex $\calC$ in $M$, with defining surface $F = F(\calC)$, is {\em tiny} \index{Tiny chamber complex} if either
\begin{itemize}
\item $F = \emptyset$ or
\item there is a chamber $C$ of $\calC$ so that $M - C$ consists of handlebody chambers.  These chambers are called the {\em designated handlebody chambers}. \index{Designated handlebody chamber}
\end{itemize}
\end{defin}


A Heegaard surface in which at least one complementary component is a handlebody is a familiar example of a defining surface of a tiny chamber complex.  On the other hand, a weak reduction of a Heegaard surface will yield a chamber complex that is not tiny.  That fact (see Proposition \ref{prop:weaknottiny}) will be the entryway to our study of Heegaard splittings below.  The word ``tiny" is used because the tree dual to a chamber complex is particularly small for tiny chamber complexes: at its most complicated, it is a star graph with central vertex corresponding to the chamber $C$. 

\bigskip

\begin{prop}[Tinyness pulls back]  \label{prop:tiny1} Suppose $\calC \xrightarrow{\calD} \calC_{\calD}$ is a chamber complex decomposition in which 
the chamber complex $\calC_{ \calD}$ is tiny.  Suppose further, if the defining surface of $\calC_D$ is not empty,  that each of the handlebody chambers named in Definition \ref{defin:tiny1} is disky.  Then $\calC$ is tiny.
\end{prop}

\begin{proof}  We first consider the case in which the defining surface of $\calC_D$ is not empty.  That is, per Definition \ref{defin:tiny1}, there is a single chamber $C'$ of $\calC_{ \calD}$ so that $M - C'$ consists of a union $W'$ of disky handlebody chambers.  
Without loss, assume $C'$ is a $B$-chamber, so each component of $W'$ is an $A$-chamber adjacent to $C'$ in $\calC_{\calD}$.  See Figure \ref{fig:tiny1}.

Since each component of $W'$ is a disky handlebody, any disk $D \in \calD$ that lies in $\inter(W')$ lies on a disk in $F$ and so has inessential boundary on $F$; let $D$ be one whose boundary is innermost on $F$.  Surgery on $D$ creates a ball in $\inter(W')$, which must be a goneball since $W'$ is irreducible.  The effect then of surgery on $D$ followed by declaring the ball gone is merely to isotope $F \cap \inter(W')$; it has no effect on either the hypothesis of the proposition or the conclusion.  Hence we may inductively assume that $\calD \cap \inter(W') = \emptyset$, so $W'$ has only external scars.

Let $B_{\calC}$ denote the collection of $B$-chambers in $\calC$.  $C'$ is a $B$-chamber in $\calC_{\calD}$, in fact the only $B$-chamber in $\calC_{\calD}$. Recall that it is obtained from $B_{\calC}$ in three steps: 
\begin{itemize}
\item $B_{\calC}$ is cut along those disks in $\calD$ that lie in $B_{\calC}$.
\item Two-handles are added along those disks in $\calD$ that lie in the $A$-chambers
\item Goneballs are eliminated.
\end{itemize}

Towards understanding the third step, eliminating goneballs, notice that no goneball can lie in $\inter(W')$ since $W'$ has no disks of $\calD$ in its interior and, in particular, only external scars.  Suppose then that $U$ is a maximal goneball in the $B$-chamber $C'$.  Since $U$ is maximal and lies in the $B$-chamber $C'$, the chamber of $\hat{\calC}_{\calD}$ immediately inside $U$ is an $A$-chamber. Rule \ref{rule:disky1} says that $U$ is disky.  Applying the same argument as just applied to $W'$, we may as well then assume that the only scars on $\bdd U$ are external scars, so $U$ itself is an $A$-chamber in $\hat{\calC}_{\calD}$.  With this simplification, so that the only goneballs of the decomposition are $A$-chambers, none of the three steps in the process above decreases the number of $B$-chambers (though the first step may increase it).  Since, by hypothesis, there is only one $B$-chamber in $\calC_{\calD}$ we conclude that $B_{\calC}$ is also a single chamber in $\calC$.   

\begin{figure}[ht!]
\labellist
\small\hair 2pt

\pinlabel  $\calD$ at 280 70
\pinlabel  $B_{\calC}$ at 40 40
\pinlabel  $C'$ at 360 40
\pinlabel  $C$ at 120 120
\pinlabel  $\calC$ at 120 -10
\pinlabel  $\calC_{\calD}$ at 430 -10
\pinlabel  $W'$ at 430 48
\pinlabel  $W'$ at 430 123
\pinlabel  $W'$ at 340 123
\pinlabel  $W'$ at 520 123
\endlabellist
    \centering
    \includegraphics[scale=0.7]{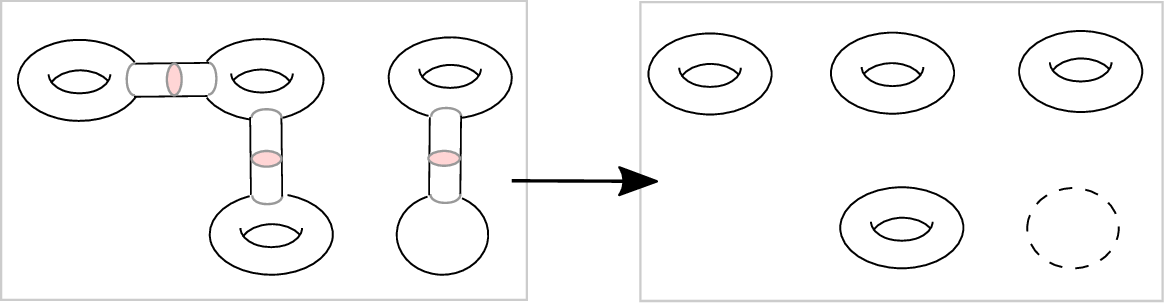}
     \caption{Tininess when the defining surface of $\calC_D$ is not empty} \label{fig:tiny1}
    \end{figure}


Consider any  chamber $C$ of $\calC$ other than $B_{\calC}$. $C$ is then an $A$-chamber in $\calC$ and so has the property that each of its remnants in $\calC_{\calD}$ is in $W'$, so each is a disky handlebody.  It then follows from Proposition \ref{prop:remnant} that $C$ itself is a handlebody.  Hence $\calC$ consists of a union of handlebodies (the $A$-chambers) whose complement $B_{\calC}$ is a single $B$-chamber.  That is, $\calC$ is tiny.
\medskip

The other possibility from Definition \ref{defin:tiny1} is that $\calC_{\calD}$ is tiny because the defining surface for $\calC_{\calD}$ is empty.  This means that in $\calC_{\calD}$, $M$ is entirely a $B$-chamber, say.  This implies that either $F$ was empty (completing the argument) or surgery on $\calD$ resulted only in spheres bounding goneballs.    In that case, the argument is basically the same: Rule \ref{rule:disky1} says that each goneball is disky.  Then, examining maximal goneballs as above we can assume without loss that each goneball is an $A$-chamber with only external scars. 
This implies that every $A$-chamber in $\calC$ can be recovered from a collection of balls (the maximal goneballs) by attaching $1$-handles, dual to the disks that leave external scars.  The result is visibly a collection of handlebodies.   So again $\calC$ is tiny.  \end{proof}

\section{Heegaard splittings and chamber complexes} \label{sect:Heegaard1}

\subsection{Review of weak reduction and amalgamation} Chamber complexes are naturally relevant to weakly reducible Heegaard splittings, as we now describe.   This first subsection is essentially a review of relevant known results; see, for example, \cite[Section 3]{La} to which we refer for relevant notation.  


Suppose, in a Heegaard splitting $M = A \cup_T B$ of a compact 3-manifold $M$, $\calA, \calB$ are disjoint families of 
disks, properly embedded in $A, B$ respectively, with $\bdd \calA, \bdd \calB \subset T$ and at least one member of each disk family an essential disk.  Such a pair of families is called a {\em weakly reducing} disk family, \index{weakly reducing disk family} a notion with a long and important history in the study of Heegaard splittings (cf \cite{CG}).  Surger $T$ along $\calA \cup \calB$ and call the resulting surface $\hat{F}$.  

$\hat{F}$ defines a chamber complex $\hat{\calC}$ in $M$ as follows.  One set of chambers, called the $A$-chambers, 
are derived from $A$ in two stages: first bicollars of disks in $\calA$ are removed (so these disks lie {\em outside} the resulting $A$-chambers) and then bicollar neighborhoods of $\calB$ are attached (so these are disks that lie {\em inside} the resulting $A$-chambers).  The remaining chambers, the $B$-chambers are derived from $B$ in the symmetric fashion.  The result is the perhaps counter-intuitive fact that after the surgery, each $A$-chamber may contain some disks from $\calB$, but none from $\calA$, and symmetrically for the $B$-chambers.  In any pair of adjacent chambers, one is an $A$-chamber and the other is a $B$-chamber, since the surface between them is (except for some disks) a submanifold of $T$.  

Each $A$-chamber $C$ (and symmetrically for $B$-chambers) inherits a natural Heegaard splitting $C = A_C \cup_{T_C} B_C$.  See Figure \ref{fig:TtoTC}.  This can be demonstrated by considering the two-stage construction just described, see Figure \ref{fig:TtoTC} :  In the first stage, (NE to SW in Figure \ref{fig:TtoTC}) the surface $T$ is compressed in $A$ along $\calA$ and a component $C'$ is chosen.  Since $A$ was a compression body, so is $C'$.  Let $T_C \subset C'$ be the end of a collar of $\bdd_+ C'$ in $C'$, so $T_C \subset \inter(C')$ (shown in red in the figure).  $T_C$ is a trivial Heegaard surface for $C'$, dividing it into the collar of $\bdd_+ C'$ and a copy $A_C$ of $C'$ lying entirely in $\inter(C')$.  In the second stage of the construction of $C$ (bottom row of Figure \ref{fig:TtoTC}) $2$-handles lying in $\calB$ are attached to $C'$ along $\bdd_+ C'$, turning the collar of $\bdd_+ C'$ into a compression body $B_C$.  Thus $T_C$ remains a Heegaard surface for $C$.  (Note that, as in the outermost surface in the final panel of Figure \ref{fig:TtoTC}, $C$ may have spherical boundary components, so $A_C$ and $B_C$ would classically be described as punctured compression bodies, as discussed on \cite{Sc1} and \cite{FS2}.). The new boundary components of $C$, that is $\bdd C - \bdd M = \bdd C - \bdd_- A$ are, except for the scars of the surgery, subsurfaces of $T$.  

\begin{figure}[ht!]
\labellist
\small\hair 2pt
\pinlabel  $\bdd M$ at 120 380
\pinlabel  $A$ at 120 410
\pinlabel  $T$ at 12 400
\pinlabel  $B$ at 220 410
\pinlabel  $\in\calA$ at 375 350
\pinlabel  $\in\calB$ at 450 340
\pinlabel  $C'$ at 220 10
\pinlabel  $T_C$ at 123 160
\pinlabel  $\bdd_+C'$ at 220 190
\pinlabel  $C$ at 510 10
\pinlabel  $B_C$ at 420 13
\pinlabel  $\rm{scar}$ at 350 40
\endlabellist
    \centering
    \includegraphics[scale=0.6]{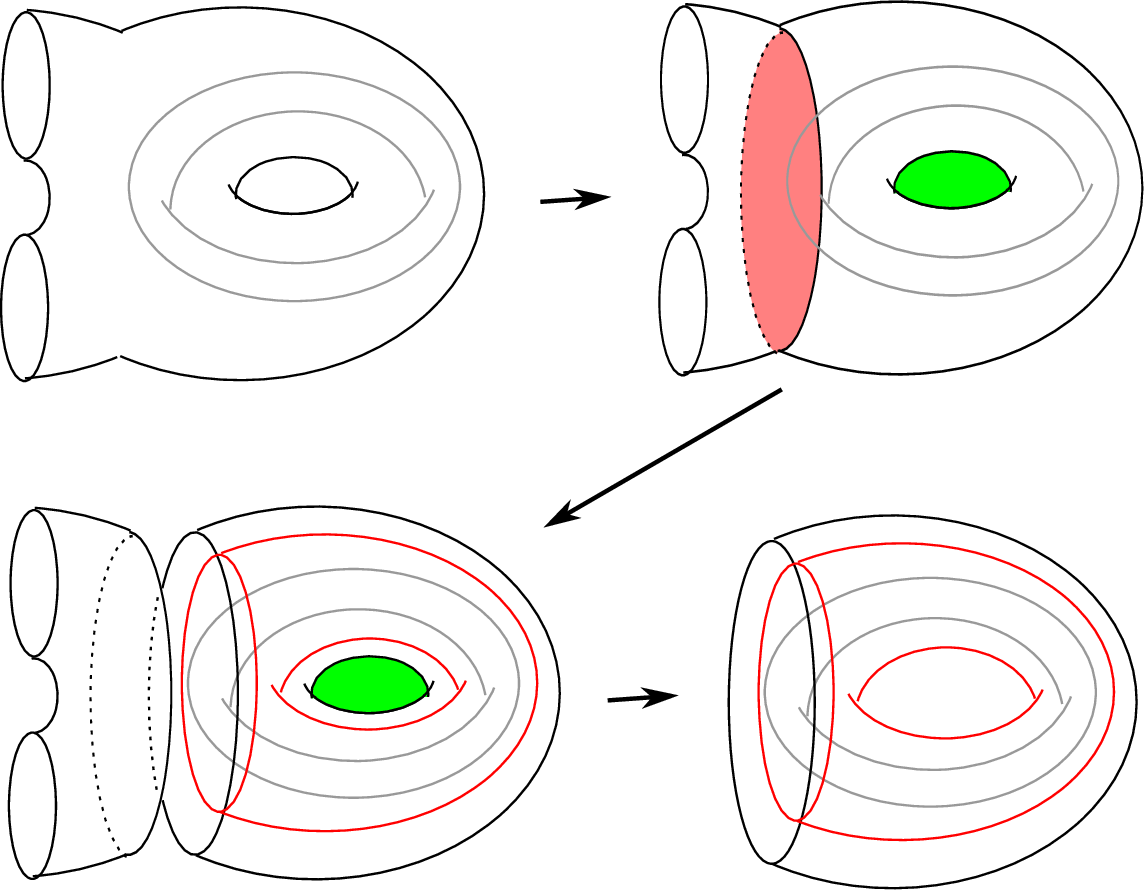}
     \caption{Inherited Heegaard splittings in disk decomposition} \label{fig:TtoTC}
    \end{figure}

Note that in this construction, the result is technically not a Heegaard split chamber complex as defined before Lemma \ref{lemma:Sreduces} because that definition requires that in each $A$-chamber $C$, $A_C$ is a handlebody, whereas in the construction above $\bdd_- A_C$ may contain components of $\bdd_-A$.  We will be mostly concerned with a special form of Heegaard splitting in which this is not an issue:

\begin{defin} \label{defin:pure}  A Heegaard splitting $M = A \cup_T B$ is called a {\em pure} \index{Pure Heegaard splitting} Heegaard splitting if all the components of $\bdd M$ lie in $\bdd B$ (or all in $\bdd A$).  Put another way, it is pure if either $A$ or $B$ is a handlebody.
\end{defin}

Obviously if $\bdd M$ has at most one boundary component then the splitting is pure.  Moreover,

\begin{prop} \label{prop:weak=pure} Suppose $\calC$ is the chamber complex resulting from weak reduction on a Heegaard splitting of a closed $3$-manifold.   Then the Heegaard splitting of each chamber in $\calC$ is pure, so the result of the weak reduction is a Heegaard split chamber complex.  
\end{prop}

\begin{proof} 
Since $M = A \cup_T B$ is closed, both $A$ and $B$ are handlebodies. 
Suppose $C$ is an $A$-chamber and consider the construction of its Heegaard structure above.  (In Figure \ref{fig:TtoTC} take $\bdd M = \emptyset$.)
 In the first stage, $A$ is cut up by $\calA$ into handlebodies and a component $C'$ of the result is chosen.  Next $T_C$ is defined as the end of a collar of $\bdd C'$ in $C'$.  $A_C$ is the complement of the collar, so in particular the surface $\bdd C'$ is entirely disjoint from the handlebody $A_C$.  That doesn't change when $B_C$ is created by adding 2-handles to $\bdd C'$ to create $\bdd C$.  Thus $\bdd C$ lies entirely in $\bdd_- B_C$, as required.
\end{proof}
\medskip

There is a natural construction that is inverse to weak reduction:  Suppose we are given a chamber complex $\hat{\calC}$ with defining surface $\hat{F}$ in which we are able to alternately label the chambers $A$-chambers and $B$-chambers, for example when each component of $\hat{F}$ is separating in $M$.  Pick a pure Heegaard splitting $C = A_C \cup_{T_C} B_C$ for each chamber $C$.
These Heegaard splittings induce a Heegaard splitting $M = A \cup_T B$ by amalgamating the Heegaard splittings along $\hat{F}$, and the amalgamated Heegaard splitting is unique up to isotopy, cf \cite[Proposition 3.1]{La}.  

As described in \cite{La}, 
amalgamation in this setting proceeds as follows: for each $A$-chamber $C = A_C \cup_{T_C}B_C$ choose a complete set of meridian disks $\calB_C$ for the compression body $B_C$.  That is, if $B_C$ is cut along $\calB_C$ the result is a collar of $\bdd_- B_C = \bdd C - \bdd M$.  The disks $\calB_C$ define a natural spine $\gamma_C$ for $B_C \subset C$, namely the union of the surface $\bdd C - \bdd M$ and a collection of arcs  in the interior of $C$, specifically the arc cocores of the $2$-handle neighborhoods of $\calB_C$.  The arcs are then extended down through the collar using its product structure.  Do the dual construction in each $B$-chamber.  Then $A$ can be recovered from $\hat{\calC}$ by deleting from each $A$-chamber $C$ a neighborhood of the graph $\gamma_C$ and attaching a neighborhood of the graph constructed in each $B$-chamber.  (And, as usual, symmetrically for $B$.)  

An important point for our purposes is that it is possible to choose a {\em different} spine $\gamma'_C$ for $B_C$  in each $A$-chamber $C$, by originally making a different choice $\calB'_C$ of meridian disks.  This will result in a different set of arcs and so a different set of of $1$-handles 
but will not change the recovered Heegaard splitting (and dually for $B$-chambers).  As argued more formally in \cite{La} one can get from the arcs of $\gamma_C$ to those of $\gamma'_C$ by edge slides, which correspond to handle-slides in $B$ that replace $\calB_C$ with $\calB'_C$ (and dually for $B$-chambers).  

{\em Two cautionary notes on pure Heegaard splittings}: 

a). The result of amalgamating pure Heegaard splittings may not be pure.  For example, suppose $F$ is a closed surface and $F \times I$ is given a pure Heegaard splitting.  Take two copies of this pure Heegaard split $F \times I$ and attach them along one end of each to get a combined copy of $F \times I$.  The Heegaard splitting of $F \times I$ obtained by amalgamating the two splittings will not be pure.  See top panel of Figure \ref{fig:caution}.

\begin{figure}[ht!]
\labellist
\small\hair 2pt
\pinlabel  $A$ at 80 208
\pinlabel  $B$ at 80 190
\pinlabel  $A$ at 80 245
\pinlabel  $B$ at 80 265
\pinlabel  $A$ at 250 240
\pinlabel  $B$ at 250 190

\endlabellist
    \centering
    \includegraphics[scale=0.9]{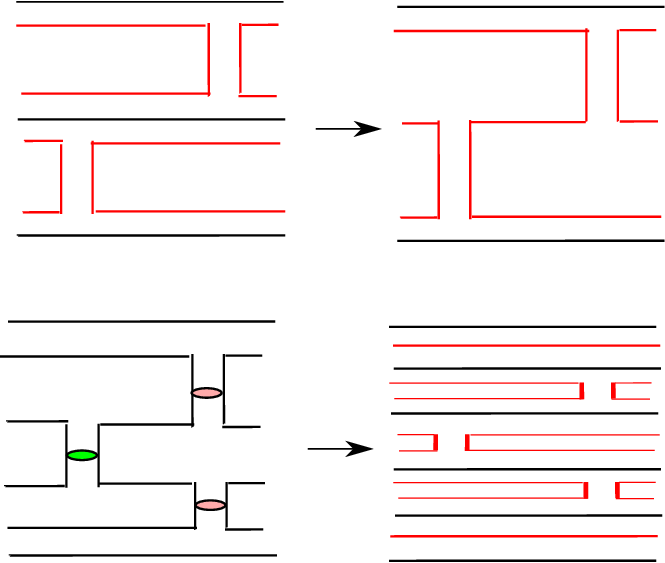}
     \caption{Amalgamation and weak reduction may destroy purity} \label{fig:caution}
    \end{figure}

b) Weak reduction, even of a pure Heegaard splitting, may not result in pure Heegaard splittings in every chamber.  For example, in $M = F \times I$ consider the following natural splitting: begin with four horizontal copies of $F$ in the interior of $M$ labeled in order $F_1, ..., F_4$, breaking $M$ into 5 copies of $F \times I$, which we will call collars and which we alternately label $A$ and $B$, so that both the top and bottom collar are labeled $A$. In each of the three collars between $F_i$ and $F_{i+1}, i = 1, 2, 3$ tube $F_i$ to $F_{i+1}$ by a vertical tube.  The result is a pure Heegaard splitting $M = A \cup_T B$ of genus $4 \cdot \genus(F)$ in which each collar labeled $A$ becomes part of $A$, and symmetrically.  Now weakly reduce along the meridians of the 3 tubes,  two of them disks in $A$ and one of them in $B$, creating a chamber complex $\calC$ with 5 chambers, alternately $A$-chambers and $B$-chambers and each a copy of $F \times I$.  The middle 3 chambers all have pure splittings, but the chambers incident to $\bdd_- M$ are each copies of $F \times I$ split into two copies by a single copy of $F$, and so are not pure.  See bottom panel of Figure \ref{fig:caution}.

%
\bigskip 

Here are some useful preliminary results.  We use the standard convention that the genus of a surface is the sum of the genera over all components.  

\begin{lemma}  \label{lemma:genusFvsT} Suppose $M = A \cup_T B$ is a pure Heegaard splitting of a $3$-manifold $M$.  
Then for any collection of components $F \subset \bdd M$
\begin{itemize}
\item $\genus(T) \geq \genus(F)$.
\item  $\genus(T) = \genus(F)$ only if each component of $\bdd M - F$ is a sphere.  
\end{itemize}
\end{lemma}

\begin{proof}  With no loss assume $\bdd M = \bdd_- B$.  
A spine of $B$ consists of the union of $\bdd_- B = \bdd M$ and a graph $\gamma$ that intersects $\bdd_- B$ exactly in some valence 1 vertices.  Since $B$ is connected there has to be at least one such end vertex in each component of $\bdd_- B$ and, after edge slides, we can assume there is exactly one in each, and that $\gamma$ is connected. Recall that for $\eta$ a regular neighborhood  of a graph $\gamma$ in a 3-manifold, $\chi(\bdd\eta) = 2\chi(\gamma)$.   $T$ is the boundary of a regular neighborhood of the spine, so it follows immediately from this construction that if $\bdd M = \bdd_-B$ has $n$ boundary components
\begin{eqnarray*}
\chi(T) = (\chi(\bdd_- B) - n) + (2\chi(\gamma) - n) & = & (\chi(\bdd_- B) - 2n) + 2\chi(\gamma) \\ & = &  -2\genus(\bdd_- B) +  2\chi(\gamma) . 
\end{eqnarray*}
Hence \[2\genus(T) = 2 - \chi(T) = 2 + 2\genus(\bdd_- B) - 2\chi(\gamma) \geq 2\genus(\bdd_- B)\] (with equality if and only if $\chi(\gamma) = 1$, i. e. $\gamma$ is a tree).  Hence
\[\genus(T) \geq \genus(\bdd_- B) = \genus(\bdd M) = 
\genus(F) + \genus(\bdd M - F)\] 
so
$\genus(T) \geq \genus(F) $ with equality only if $\genus(\bdd M - F) = 0,$ that is $\bdd M - F$ consists of spheres.  
\end{proof}

\begin{lemma} \label{lemma:chiamalg} Suppose a closed surface $F$ divides a 3-manifold $M$ into two components $C_1$ and $C_2$, and each is given a Heegaard splitting $C_i = A_i \cup_{T_i} B_i$.  Let $T$ be the Heegaard surface for  the splitting of $M$ obtained by amalgamating the two splittings.  Then \[ \chi(T) = \chi(T_1) + \chi(T_2) - \chi(F).\]  In particular, if $F$ is connected, then 
\[\genus(T) = \genus(T_1) + \genus(T_2) - \genus(F)\]
\end{lemma}

\begin{proof}  With no loss let $A_1$ (resp $B_2$) be the compression body in the splitting by $T_1$ (resp $T_2$) that is incident to $F$.  In amalgamation, the construction of the splitting surface $T$ of the amalgamated splitting replaces a $p$-punctured copy of $F$ in $T_1$ and a $q$-punctured copy of $F$ in $T_2$ with a single $(p +q)$-punctured copy of $F$.  The $p$ punctures correspond to scars parallel to $F$ of the chosen set of meridian disks in $A_1$, after $F_1$ is compressed along these disks.  A similar and symmetric statement is true for the $q$ punctures.  See \cite[Figure 12]{La}.

Computing Euler characteristics from this description we have 
\[ \chi(T) = \chi(T_1) + \chi(T_2) - (\chi(F) + p) - (\chi(F) + q) + (\chi(F) + p + q)\] or, more succinctly
\[ \chi(T) = \chi(T_1) + \chi(T_2) - \chi(F).\]  Then, if $F$ is connected, \[\genus(T)  + \genus(F) = \genus(T_1) + \genus(T_2)\] follows since all surfaces are connected.
\end{proof}

%

\begin{prop} \label{prop:genusamalg}  Suppose $M$ is 
a $3$-manifold with connected boundary in which each closed surface separates.  Suppose $\calC$ is a chamber complex in $M$ with defining surface $F$ and each chamber in $\calC$ has a pure Heegaard splitting.  Let $C = A_C \cup_{T_C} \cup B_C$ be the chamber whose boundary contains $\bdd M$.  Suppose $M = A \cup_T B$ is the Heegaard splitting of $M$ obtained by amalgamating the Heegaard splittings of all the chambers.  Then 
\begin{enumerate}[label=\alph*)]
\item $\genus(T_C) \leq \genus(T)$
\item If $\genus(T_C) = \genus(T)$ then $T_C$ is isotopic to $T$ and 
\item If $\genus(T) = 
\genus(\bdd M)$ then $M = A \cup_T B$ is a trivially split handlebody, and $F$ consists entirely of spheres bounding balls in $M$.  Moreover, amalgamating the Heegaard splittings in each ball bounded by a component of $F$ gives a trivial splitting of the ball. 
\end{enumerate}
\end{prop}

\begin{proof} 
With no loss, assume $\bdd M \subset \bdd_- B_C$.   

It will be useful to consider the natural graph $\Gamma$ associated to the chamber complex $\calC$, in which each vertex corresponds to a chamber in $\calC$ and two such vertices are connected by an edge if they are adjacent, that is there is a component of the defining surface $F = F(\calC)$ that is incident to both.  Since every closed surface in $M$ separates, $\Gamma$ can contain no cycles, so it is a tree; we can take the vertex $v_C$ corresponding to $C$ as the root of the tree.
\medskip

{\em Case 1:} $M$ is a single chamber, so $A$ is a handlebody and $\Gamma$ is a single vertex.

In this case $T_C$ is $T$, so a) and b) are automatic.  For c), assume $\genus(T) = \genus(\bdd M)$.  Then no $2$-handles could be attached in the construction of $B_C$ from $\bdd_+ B_C = T$, for if any attachment circle were non-separating then $\genus(T) > \genus(\bdd M)$ and if any were separating on $T$ then each component of the result would be the Heegaard surface for a separate chamber, contradicting the assumption that $M$ has a single chamber.  We conclude that $B_C$ is a collar of $\bdd M$.  Since $A_C$ is a handlebody, so then is $M$.
\medskip

{\em Case 2:} $M$ has two chambers, $C \supset \bdd M$ and $C'$. 

Since each component of $F$ separates, $F = C \cap C'$ is connected.  By Lemma \ref{lemma:chiamalg}
\[\genus(T) = \genus(T_C) + \genus(T_{C'}) - \genus(F).\] 
By Lemma \ref{lemma:genusFvsT} $\genus(T_{C'}) - \genus(F) \geq 0$, so $\genus(T) \geq \genus(T_C)$, establishing a).  Moreover, if $\genus(T) = \genus(T_C)$ then $\genus(T_{C'}) = \genus(F)$ and, by Case 1 applied to $C'$, $C'$ is a trivially split handlebody.  But if $C'$ is a trivially split handlebody then amalgamating it with $T_C$ leaves $T_C$ unchanged, so $T$ is isotopic to $T_C$, establishing b) in this case.  

For c) observe that we so far have the inequalities 
\[\genus(T) \geq \genus(T_C) \geq \genus(\bdd M).\] 
Hence if $\genus(T) = \genus(\bdd M)$ then both inequalities are equalities.  The first equality, by the second claim, shows $T_C$ is isotopic to $T$.  The second equality, by Lemma \ref{lemma:genusFvsT} (applied to $C$ and reversing the roles of $F$ and $\bdd M$) says $F$ is a sphere, and we have already established that $C'$ is a trivially split handlebody.  Thus $C'$ is a trivially split ball, as claimed.  Moreover, Case 1 applied to the amalgamated splitting of $M$ shows $M$ is a trivially split handlebody, completing the proof of c) in this case.

\medskip

{\em Case 3:}  Every chamber in $M$ other than $C$ is adjacent to the chamber $C$.  That is, $\Gamma$ is a star graph: a collection of edges each of which has a single end incident to $v_C$.  

The proof is by induction on the number $n$ of chambers not $C$.  Cases 1 and 2 above apply when $n = 0, 1$, so we can assume $n \geq 2$, that is there are at least 2 chambers other than $C$.  Let $C_1, C_2$ be two of them, separated from $C$ by surfaces $F_1, F_2$ respectively.  The chamber complex $\calC_1$ obtained from $\calC$ by amalgamating along $F_1$ satisfies the hypotheses of Case 3, but with $n$ reduced.  So we inductively assume that the Proposition is true for $\calC_1$.  But further amalgamation of all of $\calC_1$ gives Heegaard surface $T$ and does not change $\bdd M$ so conclusion c) applies: If $\genus(T) = \genus(\bdd M)$ then $M$ is a trivially split handlebody and each  $C_i, i \neq 1$ is a trivial split ball.  Repeating the argument for the chamber complex $\calC_2$ obtained from $\calC$ by amalgamating along $F_2$ shows that $\calC_1$ is also a trivially split ball.  Thus c) is true in this case.

To prove a) observe that the inductive hypothesis applied to $\calC_1$ shows that the Heegaard surface $T_1$ of the chamber of $\calC_1$ containing $\bdd M$ has $\genus(T_1) \leq \genus(T)$.  But that chamber is obtained from $C$ and $C_1$ by amalgamation along $F_1$ so by Case 2) $\genus(T_C) \leq \genus(T_1)$.  Hence $\genus(T_C) \leq \genus(T)$ as required to prove a).  

Finally, if $\genus(T_C) = \genus(T)$ then we have just shown $\genus(T_1) = \genus(T)$ and $\genus(T_C) = \genus(T_1)$.  The first equality shows, by inductive hypothesis, that $T_1$ and $T$ are isotopic; the second shows, by Case 2), that $T_C$ is isotopic to $T_1$.  Thus $T_C$ and $T$ are isotopic, proving b) in this case.

\medskip

{\em Case 4:} There is an edge in $\Gamma$ that is not incident to $v_C$.  

There is then such an edge 
so that deleting the edge from $\Gamma$ leaves two components: one containing $v_C$ and the other a star graph, as in Case 3.  (For example, in a path from $v_C$ to a most distant vertex in $\Gamma$ choose the second to last edge traversed.). The component $F'$ of $F$ corresponding to that edge then bounds a submanifold $N \subset M$ containing a chamber complex $\calC_N$ that satisfies the hypothesis of Case 3. That is, the chamber adjacent to $F'$ in $N$ is adjacent to every other chamber in $\calC_N$.   
Let $T_N$ by the Heegaard surface obtained by amalgamating the chambers of $N$.  The resulting splitting of $N$ is pure, since $F' = \bdd N$ is connected.  

We induct on the number of chambers in $\calC$.  Let $\calC'$ be the chamber complex in $M$ obtained by amalgamating the chambers of $N$. The inductive hypothesis applies then to $\calC'$ and, since $F' \neq \bdd M$, the chamber $C$ is unchanged by the amalgamation within $N$.  Moreover, the result of amalgamating the Heegaard splittings of all the chambers in $\calC'$ is the same as amalgamating all those in $\calC$, so by inductive hypothesis $\genus(T) \geq \genus(T_C)$ and equality implies that $T$ is isotopic to $T_C$.  This verifies a) and b).

For c), suppose $\genus(T) = \genus(\bdd M)$.  Again applying the inductive hypothesis to $\calC'$ we deduce that $M = A \cup_T B$  is a trivially split handlebody, and all components of $F$, other than those in the interior of $N$, are spheres bounding balls in which the amalgamated Heegaard splitting is trivial.  In particular, $N$ is a ball trivially split by $T_N$.  Finally, apply Case 3 to $\calC_N$: all components of $F$ in the interior of $N$ are also spheres bounding trivially split balls.
\end{proof}

\subsection{From weak reduction to chamber complex decomposition}

There is a lot of inefficiency in the chamber complex $\hat{\calC}$ obtained by surgery on the weakly reducing family $\calA, \calB$.  For example, if two disks in $\calA$ are parallel, then surgery as above turns the collar between them into its own $A$-chamber, one that is simply a ball with trivial Heegaard splitting.  This problem is easily addressed by viewing surgery on $\calA, \calB$ as the first step in a disk-decomposition of the chamber complex $A \cup_T B$ along $\calA, \calB$, and adopting the following rule:

\begin{rrule} \label{rule:diskyHeeg} Following the weak reduction of $A \cup_T B$ along $\calA, \calB$, let $W$ be a ball in $M$ bounded by a component of $\hat{F}$.  
\begin{itemize}
\item Declare $W$ to be a goneball if and only if the Heegaard splitting of $W$ obtained by amalgamating the Heegaard splittings of all the chambers of $\hat{\calC}$ in $\inter(W)$ gives a trivial Heegaard splitting of $W$.   
\item For each goneball $G$, amalgamate the Heegaard splittings of all the chambers of $\hat{\calC}$ in $G$ (including amalgamation along $\bdd G$).  
\end{itemize}
\end{rrule}



Denote the resulting chamber complex $\calC$ (circumflex removed) and say that it is obtained from the splitting $A \cup_T B$ by {\em weak reduction} \index{Weak reduction} along $\calA \cup \calB$.   In Proposition \ref{prop:Heegdecomp} below we show that Rules \ref{rule:disky1} and \ref{rule:diskyHeeg} do not conflict, so weak reduction can be viewed as disk decomposition of the chamber complex determined by $T$. Each chamber in $\calC$ is Heegaard split, and furthermore no chamber in $\calC$ is a trivially split ball.   

When a goneball $G$ is absorbed into the adjacent $A$-chamber $C$ (or symmetrically), the effect on the spine $\gamma_C$ of $B_C$ is to replace a collar of $\bdd G$ in $C$ by a single vertex whose neighborhood is the ball $G$.  As a result, each $A$ chamber $C$ in $\calC$ continues to inherit a natural spine $\gamma_C$ for $B_C$, but instead of just a collection of arcs attached to $\bdd C - \bdd M$ the spine is a graph attached to $\bdd C - \bdd M$, containing vertices as well as arcs.  As before, slides of edges in the graph  over other edges and over $\bdd C - \bdd M$ can be used to replace one spine by another and so a chamber-derived alteration of the disks $\calB$.

%

\begin{prop}  \label{prop:triv=disky} Following the weak reduction of $A \cup_T B$ along $\calA, \calB$, let $W$ be a handlebody in $M$ bounded by a component of $\hat{F}$.  Suppose the Heegaard splitting $W = A_W \cup_{T_W} B_W$ obtained by amalgamating the splittings of all the chambers in $\inter(W)$ is trivial.  Then 
\begin{itemize}
\item $T \cap \inter(W)$ consists only of disks and
\item If $W$ is not a ball after the absorption of the goneballs given by Rule \ref{rule:diskyHeeg}, then it becomes a single chamber in $\calC$.  If $W$ is a ball, it becomes a goneball in $\calC$.  
\end{itemize}
\end{prop}

\begin{proof}   With no loss assume that the chamber in $W$ adjacent to $\bdd W$ is an $A$-chamber, so $\bdd W = \bdd_- B_W$.  Since weak reduction is inverse to amalgamation, it is easy to derive this natural description of the splitting surface $T_W$:  Delete from $\bdd W$ the $p$ disks corresponding to internal scars and call the resulting surface $\bdd W_-$.  Then attach to $\bdd W_-$ the surface $T \cap \inter(W)$, which has $p$ boundary components, each corresponding to an internal scar in $\bdd W$.  Then push the resulting surface into $\inter(W)$ via a collar of $\bdd W$.  

The compression body $B_W$ in $W$ has a natural spine: the union of a graph $\gamma \subset \inter(W)$ with $p$ end-points attached to $\bdd_- B_W = \bdd W$ at the centers of the internal scars.  We then calculate as in the proof of  Lemma \ref{lemma:genusFvsT}:
\[\chi(T_W) = (\chi(\bdd W) - p) + (2\chi(\gamma) - p)  =  \chi(\bdd W) + (2\chi(\gamma) - 2p)\]

By assumption $T_W$ and $\bdd W$ are homeomorphic, and this implies $\chi(\gamma) = p$.  Let $n$ be the number of components of $\gamma$.  Since $T_W$, hence $B_W$, hence the spine is connected, each component of $\gamma$ must contain one of the end points of $\gamma$ on $\bdd W$.  This implies $n \leq p$.  On the other hand, each component of $\gamma$ has Euler characteristic $\leq 1$.  Hence $p = \chi(\gamma) \leq n$.  Together, these inequalities that $p = n$ and each component of $\gamma$ has Euler characteristic $= 1$.  In other words, each component of $\gamma$ is a tree.  This in turn implies that each component of $T \cap \inter(W)$ is a disk, verifying the first assertion of the Proposition.  

The last assertion c) of Proposition \ref{prop:genusamalg} further says that each component of $\hat{F}$ lying in the interior of $W$ bounds a goneball, establishing the second claim.
\end{proof}

\begin{cor} \label{cor:triv=disky} The interior of each goneball, as defined in Rule \ref{rule:diskyHeeg}, intersects $T$ only in disks.   \end{cor}

\begin{proof} Rule \ref{rule:diskyHeeg} says that each goneball has trivial Heegaard splitting, so Proposition \ref{prop:triv=disky} applies. \end{proof}

The discussion above leads to

\begin{defin} \label{defin:hsChamcom} A {\em Heegaard split chamber complex} $\calC$ for a compact manifold $M$ is a chamber complex so that:
\begin{itemize}
\item Each chamber is labeled as either an $A$-chamber or $B$-chamber.
\item Adjacent chambers are labeled $A$ and $B$
\item Each chamber has a pure Heegaard splitting
\item The Heegaard splitting of each ball chamber is non-trivial.
\end{itemize}
\end{defin}
 
Note that the last property guarantees that a Heegaard split chamber complex satisfies the troublesome third requirement of the chamber complexes discussed preceding Lemma \ref{lemma:Sreduces}. 

\begin{prop} \label{prop:Heegdecomp}  Let $M = A \cup_T B$ be a Heegaard split closed manifold and  $\calA, \calB$ be a weakly reducing family of disks for the splitting.  Let $\calC_T$ denote the chamber complex in $M$ defined by $T$.  That is, the chambers are $A$ and $B$.  Then surgery on $\calA \cup \calB$ followed by goneball removal,  as just described, defines a chamber complex decomposition $$\calC_T \xrightarrow{\calA \cup \calB} \calC,$$ and $\calC$ is a Heegaard split chamber complex.
\end{prop}

\begin{proof} Surgery on $\calA \cup \calB$ is the first step in chamber complex decomposition.  The next step is goneball absorbtion, where Rule \ref{rule:disky1} requires that any ball that is declared a goneball must be disky.  That Rule \ref{rule:diskyHeeg} implies this follows immediately from Corollary \ref{cor:triv=disky}.  The construction above describes how each chamber is endowed with a pure Heegaard splitting. 

 It remains to show that the Heegaard splitting of each ball chamber in $\calC$ is non-trivial.  Suppose $S$ is a sphere in $\hat{F}$ that bounds a ball chamber  $C$ in $\calC$.  Since the ball contains no other chambers, the only components of $\hat{F}$ that lie in the interior of $C$ must be goneballs.  Rule \ref{rule:diskyHeeg} says that these goneballs are absorbed, so their boundary spheres are removed, by amalgamation of Heegaard splittings.  In particular, the resulting Heegaard splitting of $C$ is obtained by amalgamation of all the chambers in $\inter(C)$.  Since $C$ is not itself a goneball, it follows again from Rule \ref{rule:diskyHeeg} that this Heegaard splitting is non-trivial, as required. 
 \end{proof} 

\subsection{Heegaard split chamber complex decomposition} \label{subsec:diskdecomp}

Suppose $\calC$ is a Heegaard split chamber complex in $M$ and $M = A \cup_T B$ is the Heegaard splitting of $M$ obtained by amalgamation of the Heegaard splittings of the chambers of $\calC$.  Then we say that $\calC$ {\em supports} the splitting $M = A \cup_T B$.  So a weak reduction of $ A \cup_T B$ as above produces a Heegaard split chamber complex that supports it.  (See the proof of Proposition \ref{prop:prepersist} below.)


The process of decomposing chamber complexes, as described in Section \ref{sect:chamberintro}, extends naturally to Heegaard split chamber complexes, as we now describe.  Suppose $\calC$ is a Heegaard split chamber complex with defining surface $F$, and $\calD$ is a 
disk set in $\calC$ with $\bdd \calD \subset F$.  In each $A$-chamber $C = A_C \cup_{T_C}B_C$ (and dually for the $B$-chambers), isotope the splitting surface $T_C$ until it is aligned with the disk set $\calD_C = \calD \cap C$, as described in \cite{Sc1}.  That is, each disk $D \in  \calD_C$ will either
\begin{itemize}
\item  intersect $T_C$ in a single circle, and the complementary component $D \cap B_C$ of the circle in $D$ will be a vertical annulus in $B_C$, or 
\item lie entirely in $B_C$, as $D_1, D_2$ do in Figure \ref{fig:toyexist2}. 
\end{itemize}
($D$ cannot lie in $A_C$ since in an $A$-chamber $A_C$ is disjoint from $F$.)  In either case $B_C$ has a spine $\Sigma$ (illustrated in Figure \ref{fig:toyexist2}) that intersects the disks $\calD_C$ exactly in $\bdd \calD_C$. 
$B_C$ is isotopic to a thin regular neighborhood of the spine $\Sigma$, and after the isotopy every 
disk $D \in  \calD_C$ intersects $T_C$ in a single circle, with the annulus $D \cap B_C$ vertical in $B_C$. 
We say then that $\calD$ is {\em aligned} with $\calC$.  

 \begin{figure}[ht!]
\labellist
\small\hair 2pt
\pinlabel  $D_1$ at 165 85
\pinlabel  $\Sigma$ at 280 80
\pinlabel  $\bdd_-B_C$ at 240 -5
\pinlabel  $D_2$ at 240 65
\pinlabel  $\bdd_+B_C=T_C$ at 240 130
\pinlabel  $\Sigma$ at 130 15

\endlabellist
    \centering
    \includegraphics[scale=0.75]{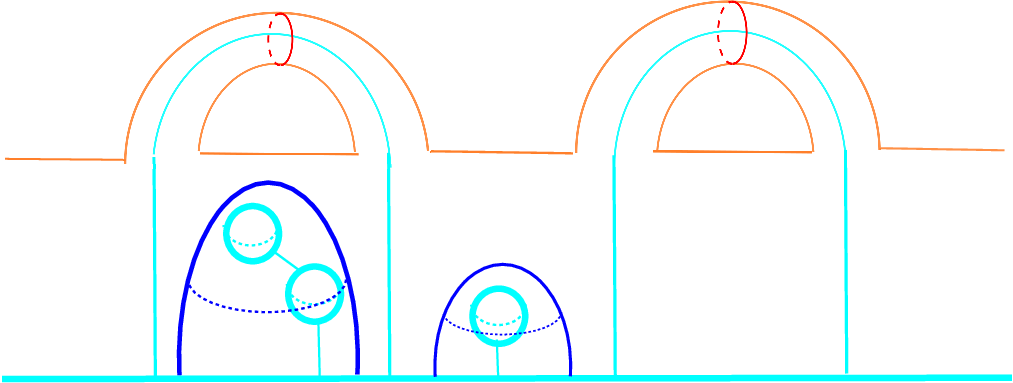}
     \caption{Alignment of disks in $\calD_C$ that lie entirely in $B_C$} \label{fig:toyexist2}
    \end{figure}

The first step in chamber complex decomposition along $\calD$ is to surger $F$ along $\calD$, making the surface $F_{\calD}$.  When the Heegaard splittings of the chambers are aligned with $\calD$ as above, this has two effects on an $A$-chamber $C$: 
\begin{itemize}
\item Surgery on the $B$-disks incident to $\bdd C$, lying in adjacent $B$-chambers, adds $2$-handles to $C$ and $B_C$ along $\bdd_- B_C$; call the results respectively $C_+$ and $B_{C_+}$.  The manifold $B_{C_+}$ is still a compression body so, in particular, $T_C$ remains a Heegaard splitting surface for $C_+$.  Note that a $B$-disk incident to $\bdd C$, lying in an adjacent $B$-chamber and inessential in that chamber may become part of an essential sphere in $C_+$, cutting off a punctured ball from $C_+$
\item Further surgery on the 
disks $\calD_C$ creates a surface $T_{\hat{C}_D}$, each of whose components is a Heegaard surface for one of the chambers $\hat{C}_{D}$ obtained from $C_+$ by the surgery on $\calD_C$.  For essential disks, this observation is familiar in Heegaard theory as $\bdd$-reduction of a Heegaard surface (see \cite{Sc1}). 
Thus $T_{\hat{C}_D}$
 is a collection of Heegaard surfaces for the remnants of $C$ in $\hat{\calC}_{\calD}$.  
\end{itemize}

The result of the operation just described, done simultaneously on all the chambers of $\calC$ defines a pure Heegaard splitting of each chamber in $\hat{\calC}_{\calD}$.  

\begin{defin}   \label{defin:Heegdecomp} Let $\calC_{\calD}$ be the Heegaard split chamber complex in $M$ obtained from  $\hat{\calC}_{\calD}$ by the above process, following Rule \ref{rule:diskyHeeg} to declare goneballs.  Then $\calC_{\calD}$ is obtained from $\calC$ by {\em Heegaard split chamber complex decomposition} and we write $$\calC \xrightarrow{\calD} \calC_{\calD}.$$
\index{Heegaard split chamber complex decomposition}
\end{defin}

This choice of language implies that the process is indeed a chamber complex decomposition, that is that  following Rule \ref{rule:diskyHeeg} to declare goneballs is consistent 
%
%
%
with Rules \ref{rule:disky1} and \ref{rule:persist}.  For example, we need to show that if $G$ is a goneball as declared by Rule \ref{rule:diskyHeeg} then each component of $F \cap G$ is a disk, as was shown for the case $F = T$ in Corollary \ref{cor:triv=disky}.  That indeed both rules \ref{rule:disky1} and \ref{rule:persist} are satisfied will follow from the next two more general propositions, see Corollaries \ref{cor:disky1} and \ref{cor:prepersist}.  

\begin{prop}  \label{prop:Heegdisky2} Let $\calC$ be a Heegaard split chamber complex in a 3-manifold $M$ and $F = F(\calC)$ be its defining surface.  Suppose $\calD$ is a disk set in $\calC$ and $$\calC \xrightarrow{\calD} \calC_{\calD}$$ is a Heegaard split chamber complex decomposition.  Suppose a component of $F_{\calD}$ bounds a handlebody $W$ in $M$. If the Heegaard splitting of $W$ obtained by amalgamating the splittings of all the chambers contained in $W$ is trivial then each component of $F_W = F \cap \inter(W)$ is a disk.
\end{prop}

\begin{proof}  
Denote by $ M = A \cup_T B$ the Heegaard splitting of $M$ obtained by amalgamating all the splittings of $\calC$.  Following the argument in Proposition \ref{prop:triv=disky} we 
know that $T_W = T \cap \inter(W)$ consists of disks.  

We now exploit the fact that when Heegaard splittings of adjacent chambers are amalgamated to the defining surface $F$, 
$F$ can be viewed as a subsurface of $T$ to which disks are attached.  Hence $F_W$ is a subsurface of $T_W$ with possibly disks attached, that is $F_W$ is a planar surface with at most one boundary circle per component.  $F_W$ then consists of properly embedded disks and spheres. 

An only slightly more complicated argument rules out the existence of spheres.  Suppose there were closed components of $F_W$, necessarily all spheres.  In this case, let $F_0$ be an innermost such sphere, so that it bounds a ball chamber $C$ containing no other components of $F_W$.  Since $\calC$ is a Heegaard split chamber complex, $C$ has non-trivial Heegaard splitting, so its Heegaard surface $T_C$ contains a non-separating circle.  But when the Heegaard splittings of the chambers are amalgamated, a punctured copy of $T_C$ persists into $T$, and the non-separating circle can be taken to be disjoint from those punctures.  Again this would contradict the fact that $T_W$ consists of disks.  We deduce that $F_W$ consists entirely of disks.
\end{proof}

\begin{cor} \label{cor:disky1} Let $\calC$ be a Heegaard split chamber complex in $M$, $\calD$ be a 
 disk set in $\calC$ and 
 $\hat{\calC}_{\calD}$ be the chamber complex obtained by the process described before Definition \ref{defin:Heegdecomp}.  Then following Rule \ref{rule:diskyHeeg} to declare goneballs is consistent with Rule \ref{rule:disky1}.  
 \end{cor}
 
 \begin{proof} Suppose $G$ is a ball  in $M$ that is bounded by a sphere in $F_{\calD}$.   If $G$ is a goneball under Rule \ref{rule:diskyHeeg} then the Heegaard splitting obtained by amalgamating all the splittings of chambers in $\hat{\calC}_{\calD}$  that lie in $G$ is trivial.  But by Proposition \ref{prop:Heegdisky2} this implies that $F = F(\calC)$ intersects $\inter(G)$ only in disks, that is $G$ is disky.  Thus $G$ is also a goneball under Rule \ref{rule:disky1}.  
\end{proof}

\begin{prop} \label{prop:prepersist} Suppose $\calC$ is a Heegaard split chamber complex for $M$ that supports the Heegaard splitting $M = A \cup_T B$.  Suppose $\calD$ is an aligned
disk set in $\calC$ .
Let $\hat{\calC}_{\calD}$ be the chamber complex obtained by the process described before Definition \ref{defin:Heegdecomp}. Then amalgamating all the induced Heegaard splittings of the chamber complex $\hat{\calC}_{\calD}$ yields the original Heegaard splitting $M = A \cup_T B$.
\end{prop}

\begin{proof}  We are given that amalgamating all the Heegaard splittings of the chambers in $\calC$ gives $M = A \cup_T B$.  Our assumption is that each disk in $\calD$ is aligned with the splitting of the chamber in which the disk lies.  Suppose $D \in \calD$ lies, without loss, in an $A$-chamber $C = A_C \cup_{T_C} B_C$ and $C'  = A_{C'} \cup_{T_{C'}} B_{C'}$ is the adjacent $B$-chamber on whose boundary $\bdd D$ lies.  The process described before Definition \ref{defin:Heegdecomp} affects a bicollar neighborhood $D \times [-1, 1]$ of $D$ as follows (see Figure \ref{fig:prepersist}):
\begin{itemize}
\item The ball $\inter(D) \times (-1, 1)$ is moved from $\inter(C)$ to $\inter(C')$.
\item The annulus $\bdd D \times [-1, 1]$ is deleted from the spines of $A_{C'}$ and $B_C$.
\item The disks $D \times \{\pm 1\}$ are added to the spines of $A_{C'}$ and $B_C$.
\item The arc $\{0\} \times [-1, 1]$ is added to the spine of $A_{C'}$.
\end{itemize}
On the other hand, amalgamation of the splittings of $\hat{\calC}_{\calD}$ does exactly the opposite near $D \times [-1, 1]$: a tube around $\{0\} \times [-1, 1]$, which we can take to be $\bdd D \times [-1, 1]$, is added to $F_\calD$ after deleting the disks $D \times \{\pm 1\}$, undoing each of the items above.  Extending this observation now to all the disks in $\calD$, the result of amalgamating all the Heegaard splittings after the process described before Definition \ref{defin:Heegdecomp} is the same as amalgamating before the process, namely $M = A \cup_T B$.
\end{proof}

 \begin{figure}[ht!]
\labellist
\small\hair 2pt
\pinlabel  $D$ at 90 110
\pinlabel  $C$ at 60 140
\pinlabel  $C'$ at 60 170
\pinlabel  $C'$ at 400 170
\pinlabel  $C$ at 310 170
\pinlabel  $\subset {\rm spine}(A_{C'})$ at 410 100
\endlabellist
    \centering
    \includegraphics[scale=0.75]{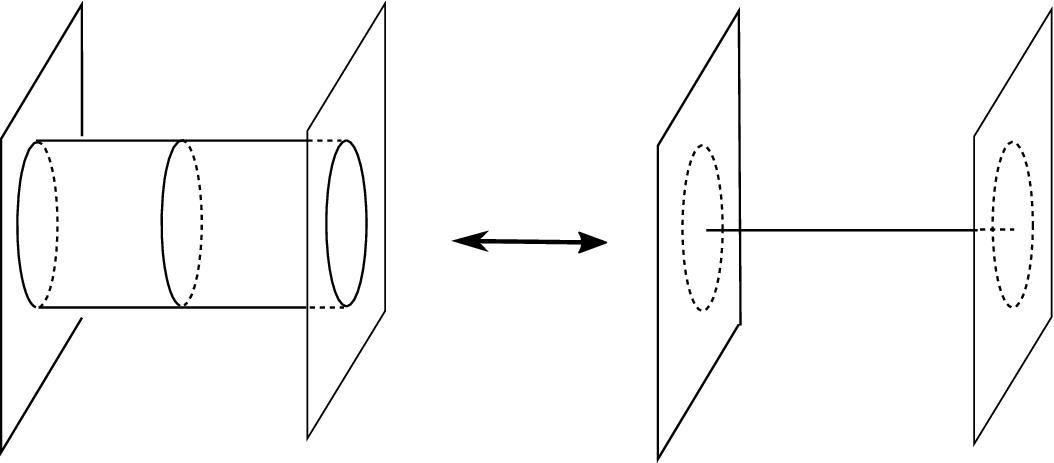}
     \caption{Decomposition near $D \in \calD$} \label{fig:prepersist}
    \end{figure}

\begin{cor} \label{cor:prepersist}  Let $\calC$ be a Heegaard split chamber complex in $M$, $\calD$ be a 
 disk set in $\calC$ and 
 $\hat{\calC}_{\calD}$ be the chamber complex obtained by the process described before Definition \ref{defin:Heegdecomp}.  Then following Rule \ref{rule:diskyHeeg} to declare goneballs is consistent with Rule \ref{rule:persist}.  
 \end{cor}       
 
 \begin{proof}  Suppose $G$ is a ball  in $M$ that is bounded by a sphere in $F_{\calD}$.  Suppose further that the disk $D \in \calD$ leaves no scar on $\bdd G$, and define $\calD_- = \calD - D$.   Since $D$ leaves no scar on $\bdd G$, the sphere is also a component of $F_{\calD_-}$. Moreover, if we let $\calD_{\bdd} \subset \calD_-$ be the set of disks in $\calD$ that {\em do} leave scars on $\bdd G$, then $\bdd G$ is also a component of $F_{\calD_{\bdd}}$ and, as described before Definition \ref{defin:Heegdecomp}, $G$ inherits a chamber complex structure $\hat{\calC}_{\bdd}$, in which each chamber is Heegaard split.   Let $T_G$ be the Heegaard splitting of the ball $G$ obtained by amalgamating all these splittings.  That is, $\hat{\calC}_{\bdd}$ supports $T_G$.  
 
 According to Proposition \ref{prop:prepersist} the chamber complexes in $G$ obtained from $\hat{\calC}_{\bdd}$ by decomposition along $\calD \cap G$ or along $\calD_- \cap G$ also support $T_G$.  In particular, under Rule \ref{rule:diskyHeeg},  $G$ is a goneball in $\calC_{\calD}$ if and only if $T_G$ is a trivial splitting of $G$ and this is true if and only if $G$ is also a goneball in $\calC_{\calD_-}$.  Thus Rule \ref{rule:persist} holds.
 \end{proof}   
 
 \begin{thm} \label{thm:Heegdisky}   Suppose $\calC$ is a Heegaard split chamber complex for $M$ that supports the Heegaard splitting $M = A \cup_T B$ and $\calD$ is a
disk set in $\calC$.
After perhaps a proper isotopy of $\calD$ within the chambers, not moving $T$, the chamber complex decompositon \[ \calC \xrightarrow{\calD} \calC_{\calD}\] is a Heegaard split chamber complex decomposition and the resulting Heegaard split chamber complex $\calC_{\calD}$ also supports $T$.  
\end{thm}

\begin{proof}  This would seem to be a straightforward consequence of Proposition \ref{prop:prepersist} and the discussion at the beginning of Subsection \ref{subsec:diskdecomp}.  Namely:
\begin{enumerate}
\item In each chamber $C$ isotope the splitting surface $T_C$ until it is aligned with $\calD$, as is shown to be possible in \cite{Sc1}.
\item Mimic the handle slides used in the isotopies within each chamber by handleslides on $T$ itself, as discussed for example in \cite{La}.
\item Now that $T$ and $\calD$ are aligned, apply Proposition \ref{prop:prepersist}.
\end{enumerate}
This (ultimately successful) strategy has an obvious conceptual weakness:  The process requires isotoping $T$ by various handleslides, so $T$ will typically end up in a different position in $M$ than it began.
We will show that the process above can be reframed so that $T$ is fixed throughout and only the disks $\calD$ and 
 the chambers $\calC$ are allowed to move.  We must further ensure that the disks $\calD$ stay disjoint, in particular that the boundaries of disks in $\calD$ that lie in adjacent chambers don't end up intersecting, after their proper isotopies, in the component of $F = F(\calC)$ that lies between them.  That is the goal of the argument that follows. 
\bigskip

{\em Claim:} There is an isotopy of $\calD$, fixed on $\bdd \calD \subset F$, so that afterwards $\calD$ is aligned with the splitting surface of each chamber.  

The proof of the Claim is straightforward and well-known:  Let $C$ be a chamber, $\calD_C$ be the set of disks in $\calD$ that lie in $C$, and $T_C \subset \inter(C)$ be the Heegaard splitting surface of the chamber.  Per \cite{Sc1} there is an isotopy of embeddings $\phi_t: T_C \to C$ so that $\phi_0$ is the original embedding and $\phi_1(T_C)$ is aligned with $\calD_C$.  By the isotopy extension theorem, $\phi_t$ can be extended to an isotopy $\theta_t: C \to C$ in which $\theta_0$ is the identity.  Since $T_C \subset \inter(C)$, we can also take $\theta_t$ to be fixed on $\bdd C$.  

Define an isotopy of embeddings $\rho_t: \calD_C \to C$ by $\rho_t = (\theta_t)^{-1}|\calD_C$.  Observe that $\rho_0$ is the original embedding (since $\theta_0$ is the identity).  Moreover, it is easy to check that $\rho_1(\calD)$ is aligned with $T_C$.  Namely, observe that \[\theta_1(\rho_1(\calD_C) \cap T_C) = \calD_C \cap \theta_1(T_C) = \calD_C \cap \phi_1(T_C)\] and the last term consists, by construction, of at most one circle in each component of $\calD_C$.  Hence  $\rho_1(\calD_C) \cap T_C$ consists of at most one circle in each component of $\calD_C$, that is the disks $\rho_1(\calD_C)$ are aligned with $T_C$.  This proves the claim.
\bigskip

We continue our effort to bring $\calD$ into alignment with $T$, not by handleslides that move $T$ (mimicking handleslides within chambers), but by proper isotopies of $\calD$.  In doing so, we do not use the Claim {\em per se}, but rather note that the proof of the claim implies that we only need to show how to keep $T$ fixed, during the alignment, in a collar of the boundary of each chamber (i. e. near the defining surface $F$).  Then a way to keep $T$ also fixed  outside the collar, that is away from $F$, is provided by the proof of the Claim.  So we focus on how  handle slides of the original splitting $M = A \cup_T B$ that mimic handle slides in the  splitting of a chamber $C = A_C \cup_{T_C} B_C$ can be replaced {\em near $F$} by proper isotopies of $\calD$.  

Consider a bicollar neighborhood $F_0 \times [-1, 1]$ of a component $F_0$ of $F$ that separates an $A$-chamber $C$ of $\calC$ from a $B$-chamber $C'$.  Here we take $F_0 \times [0, 1] \subset C$ and $F_0 \times [-1, 0] \subset C'$. $F_0$ itself is a subsurface of $T$ whose boundary is capped off by disk scars coming from surgery on a collection of disks $(\calB, \bdd \calB) \subset (B_C, \bdd B_C)$ and disk scars coming from surgery on a collection of disks $(\calA, \bdd \calA) \subset (A_{C'}, \bdd A_{C'})$.  $B$ intersects $F_0 \times [0,1]$ in vertical cylinders, the ``legs" of $1$-handles in $B_C$, one for each scar of $\calB$ on $F_0$.  We denote the leg correspoinding to a scar $b$ by $b \times [0, 1]$.  (The symmetric statement is true for $A \cap (F_0 \times [-1,0])$.)


Let $\calD_C = \calD \cap C$ and $\calD_{C'} = \calD \cap C'$.  Also spanning the collar $F_0 \times [0, 1]$ are vertical annuli, each of them a collar neighborhood of the boundary of a disk in $\calD_C$.  The symmetric statement is true in  the collar $F_0 \times [-1, 0]$.  By general position in the surface $F_0$ we can take take the legs to be disjoint from the annuli, and, since $\calD$ is embedded, we have that $\bdd \calD_C$ and $\bdd \calD_{C'}$ are disjoint in $F_0$.    

Consider how a handle slide in $B$ that mimics a handle slide of $B_C$ in $C$ appears in this collar.  As it begins, the end of a leg $b_1 \times [0, 1]$ is slid via a path $\gamma \subset F_0$ to the end of a leg $b_2 \times [0, 1]$ and then up and out of view.  At the end of the handle slide the same thing occurs elsewhere in the collar.  So it is only this initial move that needs to be understood.  By general position we may assume that the path $\gamma$ is disjoint from all scars.  Here are possibilities:
\begin{enumerate}
\item the path $\gamma$ is disjoint from both $\bdd \calD_C$ and $\bdd \calD_{C'}$, as in the left side of Figure \ref{fig:Heegdisky1}. 
\item the path $\gamma$ is disjoint from $\bdd \calD_C$ but not $\bdd \calD_{C'}$.
\item the path $\gamma$ intersects $\bdd \calD_C$ and possibly also $\bdd \calD_{C'}$.
\end{enumerate}

In the first case, where $\gamma$ is disjoint from $\bdd \calD$, this first stage of the handle slide can be accomplished, leaving $T$ fixed, by just replacing the disk in $\calB$ whose scar is $b_2$ by the disk $b'$ in $B$ obtained by band-summing $b_1$ to $b_2$ along $\gamma$.  This new disk is disjoint from other scars and, by the assumption of this case, still disjoint from $\bdd D$.  See right side of Figure \ref{fig:Heegdisky1}.  

 \begin{figure}[ht!]
\labellist
\small\hair 2pt
\pinlabel  $F_0$ at 170 70
\pinlabel  $F_0\times \{-1\}$ at 170 20
\pinlabel  $a$ at 65 75
\pinlabel  $b_2$ at 85 60
\pinlabel  $b_1$ at 155 60
\pinlabel  $b'$ at 400 50
\pinlabel  $\gamma$ at 120 60
\endlabellist
    \centering
    \includegraphics[scale=0.75]{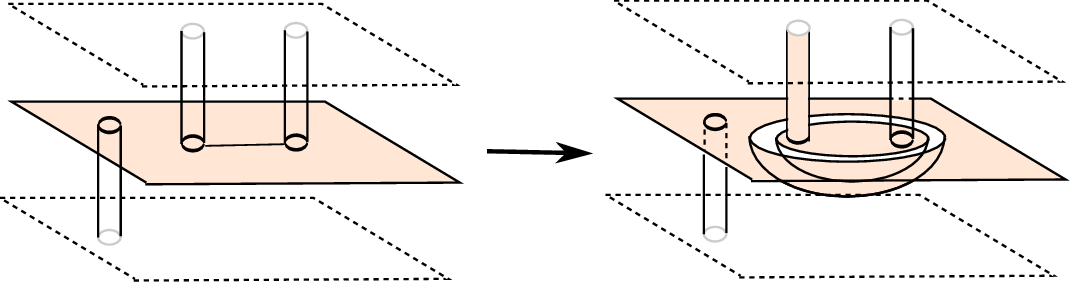}
     \caption{When $\gamma$ is disjoint from both $\bdd \calD_C$ and $\bdd \calD_{C'}$} \label{fig:Heegdisky1}
    \end{figure}

In the second case, $\gamma$ may intersect only  $\bdd \calD_{C'}$.  These intersection points can be removed by isotoping $\bdd \calD_{C'}$ along $\gamma$ towards and then across the disk $b_1 \subset F_0 \subset \bdd C'$.  See Figure \ref{fig:Heegdisky2}.  (Another description is that we band sum $\calD_{C'}$ at the points it crosses $\gamma$ to copies of $b_1$ by bands around subsegments of $\gamma$).  Such a proper isotopy of $\calD_{C'}$ is allowed, and reduces this case to the previous case.  

 \begin{figure}[ht!]
\labellist
\small\hair 2pt
\pinlabel  $D\in \calD_{C'}$ at 130 30
\endlabellist
    \centering
    \includegraphics[scale=0.75]{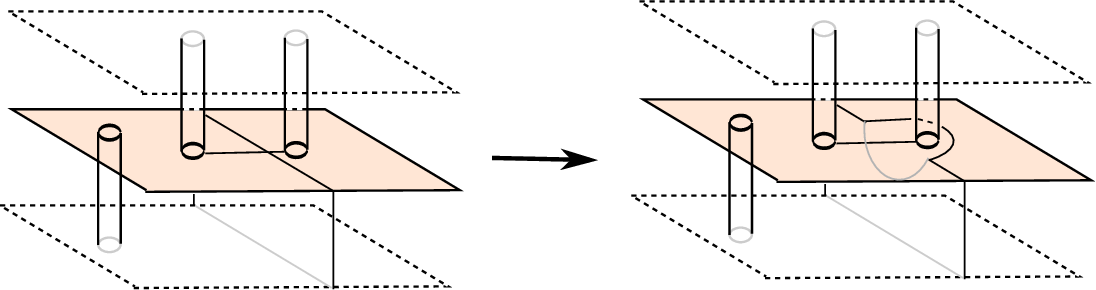}
     \caption{When $\gamma$ is disjoint from $\bdd \calD_C$ but not $\bdd \calD_{C'}$} \label{fig:Heegdisky2}
    \end{figure}

In the third case, the slide of $b_1 \times [0, 1]$ can be broken up into a series of slides along $\gamma$, each successive one across a single intersection of $\gamma$ with $\bdd \calD_C$.  For example, let $\gamma_1$ be the segment of $\gamma$ lying between $b_1$ and the closest point $p$ of $\gamma \cap \bdd \calD_C$. Just as in step 2, $\calD_{C'}$ can be properly isotoped across $b_1$ so that it is disjoint from $\gamma_1$.  The slide of the leg of $b_1 \times [0, 1]$ along $\gamma_1$ creates a disk intersection of the leg with $\calD_C$; as the leg is then straightened from below the disk intersection moves upwards out of view.  There is an obvious reinterpretation of this step that keeps the leg (hence $T$) fixed: instead slide $p \in \bdd \calD_C$ and neighboring points of $\calD_C$ along $\gamma_1$ and across the leg $b_1 \times [0, 1]$. See Figure \ref{fig:Heegdisky3}. Again this creates an intersection disk and again, as the annulus that contains $p$ is made vertical from below, the disk ascends out of view.  Note that because we have already cleared all points of $\bdd \calD_{C'}$ from $\gamma_1$, $\bdd \calD_{C'}$ and $\bdd \calD_C$ remain disjoint, as required.  Continue the slide of the leg along the rest of $\gamma$, proceeding as just described across each successive segment of $\gamma - \bdd \calD_C$, eventually reducing this case to the second case, which we have already considered.   
\end{proof}

 \begin{figure}[ht!]
\labellist
\small\hair 2pt
\pinlabel  $D\in \calD_{C}$ at 105 110
\pinlabel  $\gamma_1$ at 128 75
\pinlabel  $p$ at 120 55
\pinlabel  $p$ at 463 73
\endlabellist
    \centering
    \includegraphics[scale=0.75]{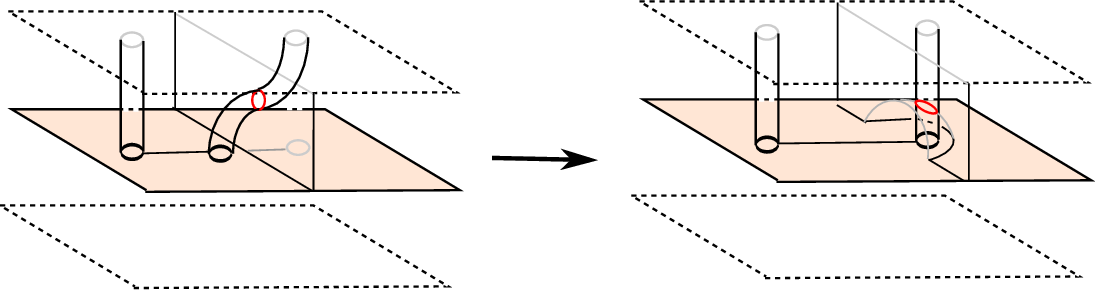}
     \caption{When $\gamma$ intersects $\bdd \calD_C$ and possibly also $\bdd \calD_{C'}$} \label{fig:Heegdisky3}
    \end{figure}

\subsection{Tiny Heegaard split chamber complexes} \label{subsect:tiny}

Proposition \ref{prop:tiny1} has an analogue in Heegaard split chamber complex decompositions, as we now describe:

\begin{defin}  \label{defin:tinyHeeg} A Heegaard split chamber complex $\calC$ is {\em tiny} \index{Tiny Heegaard split chamber complex} if $\calC$ is tiny as a chamber complex, and each designated handlebody has a trivial Heegaard splitting.
%
\end{defin}

Note that in a tiny Heegaard split chamber complex no designated handlebody can be a ball, since, by definition, no chamber in a Heegaard split chamber complex is a trivially split ball.  Also, since each 
designated handlebody has only one boundary component, each must be adjacent to the unique chamber $C$ that is not a designated handlebody.  Hence if $C$ is a $B$-chamber, all the designated handlebodies are $A$-chambers, and symmetrically.  

\begin{prop} \label{prop:weaknottiny}  Suppose $\calC$ is a Heegaard split chamber complex obtained from a Heegaard splitting $M = A \cup_T B$ by weak reduction.  Then $\calC$ is not tiny.
\end{prop}
\begin{proof}  Let $\calA, \calB$ be the weakly reducing family of disks, and $D_A$ be an essential disk in $\calA$.  Then after surgery on $\calA \cup \calB$, $D_A$ lies in a $B$-chamber $C$, with $\bdd D_A$ on a component of $T \cap \inter(C)$.  But that component can't be a disk, since $\bdd D_A$ is essential on $T$.  Hence $C$ is not a disky handlebody and so, by 
Proposition \ref{prop:triv=disky} it is not a trivially split handlebody. To summarize: some $B$-chamber of $\calC$ is not a trivially split handlebody.  

The symmetric argument on a disk $D_B \in \calB$ shows that there is some $A$ chamber of $\calC$ that is not a trivially split handlebody.  This immediately implies that the defining surface $F(\calC) \neq \emptyset$ and also contradicts the consequence of  Definition \ref{defin:tinyHeeg} that either all $A$-chambers or all $B$-chambers are designated handlebodies, and so must be trivially split.  
%
\end{proof}

\begin{prop}[Tinyness pulls back]  \label{prop:tinyHeeg} Let $\calC \xrightarrow{\calD} \calC_{\calD}$ be a Heegaard split chamber complex decomposition as given in Definition \ref{defin:Heegdecomp}, If $\calC_{ \calD}$ is a tiny Heegaard split chamber complex, so is $\calC$.
\end{prop}

\begin{proof}  Proposition \ref{prop:tiny1} and most of its proof apply in this situation: We are given that the designated handlebodies of $\calC_{\calD}$ have trivial Heegaard splitting so, per Proposition \ref{prop:Heegdisky2}, each such chamber is disky, and Proposition \ref{prop:tiny1} applies.  That is, $\calC$ is tiny as a chamber complex.  What remains to be shown is that the designated handlebodies of $\calC$ have trivial Heegaard splittings.  The proof of Proposition \ref{prop:tiny1} is also valid in this setting; that proof invokes Proposition \ref{prop:remnant} to conclude that all but one chamber of $\calC$ is a handlebody, and these become the designated handlebodies of $\calC$.  Then this strengthened version of Proposition \ref{prop:remnant} implies, in our setting, that each such designated handlebody is trivially split:

\begin{lemma}  \label{lemma:trivremnant} Suppose $C$ is a chamber of $\calC$ so that every remnant of $C$ in $\calC_{\calD}$ is a handlebody with trivial Heegaard splitting.  Then $C$ is a handlebody with trivial Heegaard splitting.
\end{lemma}

\begin{proof}  Proposition \ref{prop:remnant} suffices to conclude that $C$ is a handlebody and the proof of that proposition continues to apply here.  According to that proof, all that is required to complete the proof of Lemma  \ref{lemma:trivremnant} is this additional claim:
\medskip

{\em Claim:} Suppose $C$ is a handlebody chamber in $\calC$ and the disks $\calD_C$ in $\calD$ that are incident to $\bdd C$ all lie within $C$. 
Then $C$ is trivially split.  

{\em Proof of Claim:}  Without loss suppose $C$ is an $A$-chamber and denote by $C_r$ the collection of remnants of $C$, by hypothesis each a trivially split handlebody $A$-chamber.  Let $F = \bdd C$ and let $T_C$ be the chamber's Heegaard surface.  Then, by Propositiono \ref{prop:genusamalg}c),  $\chi(T_C) \leq \chi(F)$, with equality only if the Heegaard splitting of $C$ is trivial.

Let $\hat{T}_{\calD}$ be the union of the Heegaard surfaces for the chambers 
of $\hat{\calC}_{\calD}$ that are remnants of $C$ (some of which may be goneballs).  When aligned, as described at the beginning of Subsection \ref{subsec:diskdecomp}, each disk in $\calD_C$ intersects $F$ in a single circle (its boundary) and intersects $T_C$ either in a single circle or not at all.  In the latter case, the disk lies entirely in $B_C$ and, since $\bdd_- B_C$ is incompressible in the compression body $B_C$, the boundary of the disk is inessential in $\bdd_-C$.  Hence surgery on such a disk creates a new sphere component of $F_\calD$ bounding a trivially split ball, so such a surgery adds a sphere component both to $\hat{T}_{\calD}$ and to $F_{\calD}$ .  Similarly, surgery on a disk in $\calD$ that intersects $T$ in a single circle raises the Euler characteristic of both $\hat{T}_{\calD}$ and $F$ by 2.  So in the end, surgery on $\calD$ raises the Euler characteristic of both $\hat{T}_{\calD}$ and $F$ by $2|\calD_C|$.  (That is, $\chi(F_\calD) - \chi(F) = 2|\calD_C|$.)
Consequently, 
$\chi(\hat{T}_{\calD}) \leq \chi(F_{\calD})$, again with equality only if $C$ is trivially split.

Consider next what happens when the bounding sphere of a goneball is removed from $F_{\calD}$, and $\hat{T}_{\calD}$ is altered by amalgamation along that sphere.  Since goneballs are exactly those balls that are trivially split, the result is to remove exactly a sphere from both $F_{\calD}$ and $\hat{T}_{\calD}$, lowering the Euler characteristic of both by $2$.  Once all spheres bounding goneballs are removed, we then have $\chi(T_{\calD}) \leq \chi(\bdd C_r)$, again with equality only if $\calC$ is trivially split.  (Here $T_{\calD}$ is the union of all the Heegaard surfaces of the chambers of $C_r$.)  By assumption, each component of $\calC_r$ is trivially split.  This means that $\chi(T_{\calD}) = \chi(\bdd C_r)$, so $\chi(T_C) = \chi(F)$ and indeed $C$ is trivially split.  This proves the Claim, hence the Lemma and so the Proposition.
\end{proof}
\end{proof}

\section{Sequences of aligned chamber complex decompositions in $S^3$} \label{sect:sequence}

There is an important caveat about the structure of a Heegaard split chamber complex decomposition:  Typically the Heegaard splitting of a specific chamber of $\calC_{\calD}$ will not be well-defined.  Indeed, even the genus of the chamber may be ambiguous, since, in the process described before Definition \ref{defin:Heegdecomp} 
 there is a choice of how to align $\calD_C$ with $T_C$ in a chamber $C$ of $\calC$.  One aspect of this choice (see \cite{FS2}) is that a bubble may be moved across a disk in $\calD_C$; if, for example, the disk is separating, this changes the genus of the chambers on each side of the disk.  One result of this ambiguity is that $\calC_{\calD}$ may not be well-defined even as a chamber complex: depending on alignment, a ball chamber in $\hat{\calC}_{\calD}$ may or may not be trivially split, and therefore may or may not appear as a chamber in $\calC_{\calD}$.
 
Soon (see Section \ref{sect:prefdisk}) this ambiguity will be directly addressed, by describing a {\em preferred} way of aligning disks, a way that will eventually suffice to certify unambiguously a chamber complex obtained by weak reduction of a Heegaard splitting of $S^3$.  In this section we avoid much of this ambiguity by oversimplifying the full theory.
The section is motivational and transitional: We present an application of the theory of disk decompositions developed in Sections \ref{sect:chamberintro} and \ref{sect:Heegaard1}, one that extends Proposition \ref{prop:toyunique} and so connects to our previous discussions.  It also is a model for later arguments we will need in a more complicated setting.  We begin with a definition that will only be used in this section, as a way of illustrating and simplifying the full argument that will eventually follow.

  \bigskip
 

\begin{defin}   \label{defin:aligndecomp} In the setting of Definition \ref{defin:Heegdecomp}, suppose $\calC_{\calD}$ is a Heegaard split chamber complex obtained from  $\hat{\calC}_{\calD}$ by declaring as goneballs all  disky balls.  Then $$\calC \xrightarrow{\calD} \calC_{\calD}$$ is an {\em aligned chamber complex decomposition}.  \index{Aligned chamber complex decomposition}
\end{defin}

\begin{defin} \label{defin:bubquot}   Suppose $\calC$ is a Heegaard split chamber complex, 
that supports a genus $g$ Heegaard splitting $M = A \cup_T B$; 
 $C$ is a chamber in $\calC$; and $\frb$ is a genus $g' < g$ 
bubble in the Heegaard splitting $C = A_C \cup_{T_C} B_C$.  Let $\calC/\frb$ be the chamber complex, with Heegaard splittings of each chamber, obtained by destabilizing $C = A_C \cup_{T_C} B_C$ along $\frb$, replacing $\frb$ with a neighborhood of a point $*$ in the splitting surface.  
\end{defin}

Amalgamating all Heegaard splittings of $\calC/\frb$ gives a genus $g - g'$ Heegaard splitting of $M$ obtained by destabilizing $T$ along $\frb$.

Although every chamber in $\calC/\frb$ is Heegaard split, $\calC/\frb$ may fail to be a Heegaard split chamber complex for a somewhat technical reason: when $C$ is a ball and $\frb$ is a maximal bubble in $C$, the chamber $C$ in $\calC/\frb$ becomes a trivially split ball chamber and this is not allowed in a Heegaard split chamber complex.  Nonetheless, any disk set $\calD$ in $\calC/\frb$ can be aligned with the Heegaard splittings of the chambers, the construction described in Subsection \ref{subsec:diskdecomp} carried out, and all disky balls declared goneballs.  Denote this aligned chamber complex decomposition $$\calC/\frb \xrightarrow{\calD} (\calC/\frb)_{\calD}.$$

\begin{lemma} \label{lemma:bubquot}  Suppose, in the setting of Definition \ref{defin:bubquot}, $\calD$ is a disk set in $\calC$.  Then $\calD$ can be aligned in $\calC$ and in $\calC/\frb$ so that $\frb$ remains a bubble in a chamber of $\calC_{\calD}$ and $\calC_{\calD}/\frb = (\calC/\frb)_{\calD}$.
\end{lemma}
\begin{proof}  This is immediate from the definitions: Align the disks $\calD$ in $\calC/\frb$ so that, by general position, they do not contain the point $*$.
This defines also an alignment of $\calD$ in $\calC$.  Following surgery on $\calD$ absorb goneballs in both $\hat{\calC}$ and $\hat{\calC/\frb}$.  They are the same goneballs because in both cases the goneballs are exactly the disky balls.  These are defined at the level of chamber complexes, and do not depend on the Heegaard splittings of the chambers.  
\end{proof}

Call an alignment as in Lemma \ref{lemma:bubquot} a {\em $\frb$-wise alignment}. \index{$\frb$-wise alignment}


Suppose a Heegaard split chamber complex $\calC_0$ supports $T$.  For example $\calC_0$ could be the Heegaard splitting $A \cup_T B$ itself.  Let  \[
\vec{\calC}:\quad 
\calC_0 \xrightarrow{\calD_0} \calC_1 \xrightarrow{\calD_1} \calC_2 \xrightarrow{\calD_2} ... \xrightarrow{\calD_{n-1}}\calC_n\] be a sequence of chamber complex decompositions.

\begin{prop} \label{prop:seqcertifyexist}
Suppose chambers $C_0 \in \calC_0$ and $C_n \in \calC_n$ contain, respectively, incompressible spheres $S$ and $S'$. 
Then, for iteratively $0 \leq i \leq n-1$ each disk set $\calD_i$ can be  aligned in $\calC_i$,
so that $\calC_0$ and the resulting Heegaard split chamber complex structure on $\calC_n$ cocertify.


\end{prop}

\begin{proof}  We induct on $n \geq 1$. 

A standard innermost disk argument shows that there is an incompressible sphere $S_0$ in the chamber $C_0$ that is disjoint from the disks $\calD_0$.  This implies that $S_0$ will lie in a single chamber $C_{1}$ of the chamber complex $\calC_1$. 
  Align the disks $\calD_0$ as in Theorem \ref{thm:Heegdisky} so that $$\calC_0 \xrightarrow{\calD_0} \calC_{1}$$ is an aligned disk decompositon. More specifically, in chamber $C_0$ choose an alignment, per \cite{Sc1}, so that the incompressible sphere $S_0$ is also aligned with the Heegaard splitting surface of the chamber and hence, after amalgamation, also aligned with $T$, per Proposition \ref{prop:toyexist}. 
Following Proposition  \ref{prop:toyunique} we may as well substitute $S_0$ for $S$.

If $S_0$ is incompressible in $C_{1}$, then, according to Propositions \ref{prop:toyexist} and Corollary \ref{cor:naturality}, $h_{S_0}: (S^3, T) \to (S^3, T_g)$ defines both $h_{F(\calC_0)}$ and $h_{F(\calC_1)}$.  Thus $h_{F(\calC_0)} \sim h_{F(\calC_1)}$.  By inductive assumption $h_{F(\calC_1)} \sim h_{F(\calC_n)}$, completing the proof in this case.  

However, it is possible that the sphere $S_0$ is compressible in the chamber $C_{1}$.  (For example, the interior of a ball $W$ that $S_0$ bounds in $C_1$ could have contained a single handlebody chamber of $\calC_0$, one 
that disappears  in $\calC_{1}$ because it is decomposed by $\calD_0$ into a ball that is a goneball.)  If $S_0$  is compressible in $C_{1}$ then it bounds a non-trivial bubble $\frb$ in $(S^3, T)$, by \cite{Wa}, one that already appears as a bubble in the induced Heegaard splitting of $\calC_{1}$. (The bubble is non-trivial because the presence in $\inter(W)$ of other chambers means that $W$ is not a trivially split ball, per Proposition \ref{prop:Heegdisky2}.) In this case, we argue as follows:


For the rest of the decompositions
\[\calC_1 \xrightarrow{\calD_1} \calC_2 \xrightarrow{\calD_2} ... \xrightarrow{\calD_{n-1}}\calC_n\] 
in the sequence, choose $\frb$-wise alignments.  Then the bubble $\frb$ (bounded still by $S_0$) arrives intact in a single chamber of $\calC_n$.  In $\calC_n/\frb$, $S'$ can be taken to be disjoint from $*$,
simply by general position.  Then in $\calC_n$, $\frb$ is disjoint from $S'$.    Then Lemma \ref{lemma:preexistunique} shows that $h_{S_0}$ and $h_{S'}$ are eyeglass equivalent.  Hence $ h_{F(\calC_0)}  \sim h_{S_0} \sim h_{S'} \sim h_{F(\calC_n)}$ as required.  
\end{proof}

Having shown there exists a choice of alignment for which $\calC_0$ and $\calC_n$ cocertify, we next show that any choice of alignments will do.  The critical point is that we understand how different alignments of disk-sphere sets in a Heegaard split 3-manifold are related: By \cite{FS2} they differ by a sequence of eyeglass moves and passing bubbles through the disk-sphere set.  We begin with a lemma and proposition that are not specific to Heegaard splittings of $S^3$.  

\begin{lemma} \label{lemma:singlecertunique}  Suppose $\calC_0$ is a Heegaard split chamber complex in $M$ that supports the splitting $M = A \cup_T B$, and $$\calC_0 \xrightarrow{\calD_0} \calC_{1}$$ is a chamber complex decomposition.  Suppose $E \subset \calC_{1}$ is a disk-sphere set.  

Let $\calD^x_0$ and $\calD^y_0$ be possibly different alignments of the disk set $\calD_0$ in $\calC_0$, with resulting Heegaard split chamber complexes $\calC^x_1$ and $\calC^y_1$ respectively, and similarly let $E^x, E^y$ be alignments of the disk-sphere set $E$ in $\calC^x_1$ and $\calC^y_1$ respectively.  Then a sequence of bubble passes and eyeglass moves will change the alignment $\calD^x_0$ to $\calD^y_0$ and $E^x$ to $E^y$.
\end{lemma}

\begin{proof}  By \cite{FS2} there is a series of bubble passes and eyeglass moves that change the alignment $\calD^x_0$ to $\calD^y_0$.  The proof of the Lemma is by induction on $p$, the number of bubble passes required.  If $p = 0$ then, after an eyeglass move, we can take $\calD^x_0 = \calD^y_0$, so $\calC^x_1 = \calC^y_1$ and then apply  \cite{FS2} to the disk sets $E^x$ and $E^y$ in $\calC^x_1 = \calC^y_1$.

So suppose $p \geq 1$ and assume the Lemma is true whenever the number of bubble passes needed to change the alignment $\calD^x_0$ to $\calD^y_0$ is less than $p$.  Let $\frb$ be the 
first ball that is passed, through a disk $D \in \calD^x_0$, in changing the alignment $\calD^x_0$ to $\calD^y_0$.  Say $\frb$ is passed from a chamber $C \in \hat{\calC}^x_1$ to a chamber $C' \in  \hat{\calC}^x_1$, with possibly $C = C'$.  (Enthusiasts will note that if $C = C'$ then the bubble pass is also an eyeglass move, by Lemma \ref{lemma:nonseppass}, so $p$ can be reduced and we are done.)

Consider the subset $E^C$ of the aligned disk set $E^x$ that lies in $C$.   As aligned, $E^C$ may well intersect the interior of $\frb$.  However there is {\em some} alignment $E^{\frb}$ of $E^C$ in $C$ that is disjoint from the bubble $\frb$.  This follows from general position and \cite{Sc1} applied to $E^C$ in $ \hat{\calC}_1/\frb$.  
Then by \cite{FS2} the alignment $E^{\frb}$ can be obtained from $E^C$ by eyeglass moves and bubble passes through $E$ in $C$. Thus we may as well assume that $E^C = E^{\frb}$.  

By further bubble passes of $\frb$ through $E^C$ we can move $\frb$ adjacent to the disk $D \in \calD^x_0$ and complete the pass of $\frb$ through $D$.  Then only $p-1$ further bubble passes are needed to move $\calD^x_0$ to $\calD^y_0$, completing the inductive step.  
\end{proof}

This argument generalizes, though the statement and argument become more complex.   Let $\calC_0$ be a Heegaard split chamber complex in $M$ that supports the Heegaard splitting $M = A \cup_T B$.  Suppose
\[\vec{\calC}:\quad \calC_0 \xrightarrow{\calD_0} \calC_1 \xrightarrow{\calD_1} \calC_2 \xrightarrow{\calD_2} ... \xrightarrow{\calD_{n-1}}\calC_n\]
is a sequence of chamber complex decompositions and
$E$ is a disk-sphere set in $\calC_n$.  

Suppose that the disks $\calD_i$ are aligned at each successive stage of the decomposition,  that
\[\vec{\calC^x}:\quad \calC_0 \xrightarrow{\calD^x_0} \calC^x_1 \xrightarrow{\calD^x_1} \calC^x_2 \xrightarrow{\calD^x_2} ... \xrightarrow{\calD^x_{n-1}}\calC^x_n\] is the resulting aligned chamber complex decomposition sequence, and that  $E$ is aligned with the resulting Heegaard splitting of $\calC^x_n$.  Suppose further that for some $0 \leq i \leq n$ there is a bubble $\frb$ in one of the chambers $C$ of $\calC_i$ and the bubble $\frb$ remains disjoint from each set of disks $\calD_j, j \geq i$ and from $E \subset \calC_n$.  That is, the alignment of the disks in the remainder of the sequence is $\frb$-wise.    
Then say that $\frb$ is an {\em intact bubble} for the sequence of aligned chamber complex decompositions.  

Suppose next that the alignment of some disk set $\calD_i$ is altered simply by a bubble pass of $\frb$ to another chamber $C'$ of $\calC_{i+1}$, and each successive alignment of the $\calD_j, j \geq i+1$ and the alignment of $E$ all are $\frb$-wise, 
so that $\frb$ remains an intact bubble for the resulting aligned chamber complex decomposition
 \[\vec{\calC^y}:\quad \calC_0 \xrightarrow{\calD^y_0} \calC^y_1 \xrightarrow{\calD^y_1} \calC^y_2 \xrightarrow{\calD^y_2} ... \xrightarrow{\calD^y_{n-1}}\calC^y_n\] 
 and alignment of $E$. (By construction $\calC^x_k = \calC^y_k$ and $\calD^x_{k-1} = \calD^y_{k-1}$ for $1 \leq k \leq i$.)  Denote the two alignments of $E$ by $E^x$ and $E^y$ respectively.  

\begin{defin} \label{defin:ancbubblepass}  In the construction above, the aligned chamber complex decomposition sequence $\vec{\calC^y}$ and the alignment $E^y$ of $E$ are said to be obtained from  
$\vec{\calC^x}$ and $E^x$ by an {\em ancestral bubble pass}.  \index{Ancestral bubble pass}
This includes the degenerate case, in which $\vec{\calC^x} = \vec{\calC^y}$ but the alignments $E^x, E^y$ differ by a bubble pass in $\calC_n$.   
\end{defin}

\begin{prop}  \label{prop:seqbubblepass} Suppose aligned chamber complex decomposition sequences $\vec{\calC^x}$ and $\vec{\calC^y}$ and alignments $E^x, E^y$ are chosen for the chamber complex decomposition sequence $\vec{\calC}$ and the disk-sphere set $E$ in $\calC_n$.   Then there is a sequence of ancestral bubble passes and eyeglass moves that changes the pair $(\vec{\calC}^x, E^x)$ to $(\vec{\calC}^y, E^y)$.
\end{prop}

\begin{proof} The proof begins much as the proof of Lemma \ref{lemma:singlecertunique}.  By \cite{FS2} there is a series of bubble passes and eyeglass moves that change the alignment $\calD^x_0$ to $\calD^y_0$.  The proof is by induction on the pair $(n, p)$, lexicographically ordered, where $p$ is the minimal number of bubble passes in such a series. Lemma \ref{lemma:singlecertunique} covers the case $n = 1$, so assume $n \geq 2$.   In case $p = 0$, we can take $\calD^x_0 = \calD^y_0$, so $\calC^x_1 = \calC^y_1$ and then apply the inductive assumption to the shorter sequences 
\[\calC^x_1 \xrightarrow{\calD^x_1} \calC^x_2 \xrightarrow{\calD^x_2} ... \xrightarrow{\calD^x_{n-1}}\calC^x_n \supset E^x\] 
and 
\[ \calC^y_1 \xrightarrow{\calD^y_1} \calC^y_2 \xrightarrow{\calD^y_2} ... \xrightarrow{\calD^y_{n-1}}\calC^y_n \supset E^y.\]

So we may suppose $p \geq 1$ and, as in the proof of Lemma \ref{lemma:singlecertunique}, let $\frb$ be the first bubble that is passed, through a disk $D \in \calD^x_0$, from a chamber $C \in \hat{\calC}^x_1$ to a chamber $C' \in \hat{\calC}^x_1$.  Let $\calD^w_0$ be the disk set $\calD_0$ as realigned by this bubble pass.  Let
\[\calC^x_1 \xrightarrow{\calD^z_1} \calC^z_2 \xrightarrow{\calD^z_2} ... \xrightarrow{\calD^z_{n-1}}\calC^z_n \supset E^z\] 
and 
\[\calC^w_1 \xrightarrow{\calD^w_1} \calC^w_2 \xrightarrow{\calD^w_2} ... \xrightarrow{\calD^w_{n-1}}\calC^w_n\supset E^w\] be {\em $\frb$-wise} alignments of the aligned chamber complex sequences, with $\frb \subset C$ in $\calC^x_1$ (as is given) and $\frb \subset C'$ in $\calC^w_1$.  (Here the significant difference between the given $\calD^x_i$ and the new $\calD^z_i$, $i \geq 1$, is that each $\calD^z_i$ is $\frb$-wise aligned.)
By definition, the two sequences differ by an ancestral bubble pass and by inductive assumption (on $n$) the former sequence differs from \[\calC^x_1 \xrightarrow{\calD^x_1} \calC^x_2 \xrightarrow{\calD^x_2} ... \xrightarrow{\calD^x_{n-1}}\calC^x_n\supset E^x\] by a sequence of ancestral bubble passes and eyeglass moves.  Extending to the left we then have sequences 
\[\vec{\calC^x}:\quad \calC_0 \xrightarrow{\calD^x_0} \calC^x_1 \xrightarrow{\calD^x_1} \calC^x_2 \xrightarrow{\calD^x_2} ... \xrightarrow{\calD^x_{n-1}}\calC^x_n\supset E^x,\]
\[\vec{\calC^z}:\quad \calC_0 \xrightarrow{\calD^x_0} \calC^z_1 \xrightarrow{\calD^z_1} \calC^z_2 \xrightarrow{\calD^z_2} ... \xrightarrow{\calD^z_{n-1}}\calC^z_n\supset E^z,\]
\[\vec{\calC^w}:\quad \calC_0 \xrightarrow{\calD^w_0} \calC^w_1 \xrightarrow{\calD^w_1} \calC^w_2 \xrightarrow{\calD^w_2} ... \xrightarrow{\calD^w_{n-1}}\calC^w_n\supset E^w\]
\[\vec{\calC^y}:\quad \calC_0 \xrightarrow{\calD^y_0} \calC^y_1 \xrightarrow{\calD^y_1} \calC^y_2 \xrightarrow{\calD^y_2} ... \xrightarrow{\calD^y_{n-1}}\calC^y_n\supset E^y\]
and we have shown that any pair of the first three ($\vec{\calC^x}, \vec{\calC^z}$, and $\vec{\calC^w}$) differ by a sequence of ancestral bubble passes and eyeglass moves.  
Now notice that, by construction, the two aligned chamber complex decompositions  $\calC_0 \xrightarrow{\calD^w_0}  \calC^w_1$ and $\calC_0 \xrightarrow{\calD^y_0}  \calC^y_1$ differ by a sequence of eyeglass moves and at most $p-1$ bubble passes.  Hence, by inductive assumption (on $p$), the last two sequences (hence also the first and the last) differ by a sequence of ancestral bubble passes and eyeglass moves, as required. 
\end{proof}

We now turn to $S^3$.  Throughout the remaining part of this section, make the inductive Assumption \ref{ass:inductive}.  That is, the Goeritz group of $S^3$ is the eyeglass group on splittings of genus $<g$. 

Suppose a Heegaard split chamber complex $\calC_0$ supports the genus $g$ Heegaard splitting $S^3 = A \cup_T B$ and 
\[\vec{\calC}:\quad \calC_0 \xrightarrow{\calD_0} \calC_1 \xrightarrow{\calD_1} \calC_2 \xrightarrow{\calD_2} ... \xrightarrow{\calD_{n-1}}\calC_n\]
is a sequence of chamber complex decompositions.  Suppose further that $\calC_n$ contains an incompressible sphere $S$.

\begin{lemma} \label{lemma:passcocert} Suppose 
\[\vec{\calC^x}:\quad \calC_0 \xrightarrow{\calD^x_0} \calC^x_1 \xrightarrow{\calD^x_1} \calC^x_2 \xrightarrow{\calD^x_2} ... \xrightarrow{\calD^x_{n-1}}\calC^x_n \supset S^x\]  
and  
\[\vec{\calC^y}:\quad \calC_0 \xrightarrow{\calD^y_0} \calC^y_1 \xrightarrow{\calD^y_1} \calC^y_2 \xrightarrow{\calD^y_2} ... \xrightarrow{\calD^y_{n-1}}\calC^y_n \supset S^y\] 
are sequences of aligned chamber complex decompositions resulting from possibly different choices of alignment at each stage. Suppose further that the sequences differ by an ancestral bubble pass, of the bubble $\frb$.  (In particular $S$ is aligned with $\calC^x_n$ and $\calC^y_n$ and is disjoint from $\frb$.) Denote $S$ as so aligned $S^x$ and $S^y$ respectively.  
Then $\calC^x_n$ and $\calC^y_n$ cocertify.
\end{lemma}

\begin{proof} We sketch the proof; more detail appears in the earlier proof of Proposition \ref{prop:toyunique}. 

 Let $S_{\frb}$ be the sphere $\bdd \frb$, defining, after amalgamation, a reducing sphere for the original Heegaard splitting $A \cup_T B$ .  
 By Lemma \ref{lemma:preexistunique} $h_{S^x} \sim h_{S_\frb} \sim h_{S^y}: (S^3, T) \to (S^3, T_g)$ as required.
 \end{proof}

%

Now return to the setting of Proposition \ref{prop:seqcertifyexist}: Suppose $S^3 = A \cup_T B$ is a genus $g$ Heegaard splitting of $S^3$,  a Heegaard split chamber complex $\calC_0$ supports $T$ and  \[
\vec{\calC}:\quad: 
\calC_0 \xrightarrow{\calD_0} \calC_1 \xrightarrow{\calD_1} \calC_2 \xrightarrow{\calD_2} ... \xrightarrow{\calD_{n-1}}\calC_n\] is a sequence of chamber complex decompositions. 

\begin{prop} \label{prop:seqcertifyunique}
Suppose chambers $C_0 \in \calC_0$ and $C_n \in \calC_n$ each contain incompressible spheres. 
Then, for {\em any} alignment of each disk set $\calD_i$ in $\calC_i$, the Heegaard split chamber complexes
$\calC_0$ and $\calC_n$ cocertify.
\end{prop}

\begin{proof} By Proposition \ref{prop:seqcertifyexist} there is some alignment $\vec{\calC^x}$ for the sequence so that $\calC_0$ and $\calC^x_n$ cocertify.  By Proposition \ref{prop:seqbubblepass} any other alignment $\vec{\calC^y}$ can be obtained from $\vec{\calC^x}$ by a sequence of ancestral bubble passes and eyeglass moves.  Lemma \ref{lemma:passcocert} shows that then $\calC^x_n$ and $\calC^y_n$ cocertify.  Hence $\calC_0$ and $\calC^y_n$ cocertify.
\end{proof}

\begin{cor} \label{cor:seqcertify} Suppose $(S^3, T)$ is a genus $g$ Heegaard splitting of $S^3$ and \[
\calC_0 \xrightarrow{\calD_0} \calC_1 \xrightarrow{\calD_1} \calC_2 \xrightarrow{\calD_2} ... \xrightarrow{\calD_{n-1}}\calC_n\]  is a sequence of aligned chamber complex decompositions of Heegaard split chamber complexes supporting $T$.
If $S$ and $S'$ are each incompressible spheres in possibly different chambers of the sequence, then those Heegaard split chamber complexes cocertify.  \qed
\end{cor} 

\bigskip

Unfortunately, declaring every disky ball to be a goneball, as is done throughout this Section, erases too much information.  A critical example is this: consider a sequence of two aligned chamber complex decompositions of Heegaard split chamber complexes  
\[\calC_0 \xrightarrow{\calD_0} \calC_1 \xrightarrow{\calD_1} \calC_2 \]
and suppose that a chamber $C_1$ of $\calC_1$ is a handlebody that is not disky in the first decomposition.  We know from Proposition \ref{prop:Heegdisky2} that the Heegaard splitting of $C_1$ is non-trivial.  Now suppose that the disks of $\calD_1$ that are incident to the handlebody $C_1$ are a complete set of meridians for $C_1$.  The result is a ball chamber of $\calC_2$ that has non-trivial Heegaard splitting.  As shown in the proof of Proposition \ref{prop:toyexist}, this is (under the inductive Assumption \ref{ass:inductive}) enough information to determine an eyeglass equivalence class $(S^3, T) \to (S^3, T_g)$.  But if we take all disky balls to be goneballs, this ball would be absorbed because it is disky under the decomposition $\calC_1 \xrightarrow{\calD_1} \calC_2$.  Thus we lose the information it contains.

Motivated by this example, we will use a more complicated rule to declare that, under certain conditions, a disky ball should {\em not} be declared a goneball.  The rule will focus on how handlebodies are treated in a disk decomposition.  The critical step is to be more restrictive in how we allow disks to be aligned.

\section{Preferred alignment and flagged chamber complexes} \label{sect:prefdisk}

Return now to the general case, which we briefly review:  $M = A \cup_T B$ is a Heegaard splitting of a compact $3$-manifold; $\calC$ is a Heegaard split chamber complex in $M$ that supports $T$; $\calD$ is an aligned disk set in $\calC$; $\hat{\calC}_{\calD}$ is the chamber complex obtained by surgery on $\calD$; and $\calC_{\calD}$ is the Heegaard split chamber complex obtained by using Rule \ref{rule:diskyHeeg} to declare goneballs.  Thus \[\calC \xrightarrow{\calD} \calC_{\calD}\] is a Heegaard split chamber complex decomposition.

Because the disks are aligned, each chamber inherits a Heegaard splitting, as described before Definition \ref{defin:Heegdecomp}. We consider $3$-balls to be (genus $0$) handlebodies and introduce the terminology:

\begin{defin}  \label{defin:empty} A Heegaard split handlebody chamber in $\calC$ is {\em empty} \index{Empty chamber} if the Heegaard splitting is trivial.  If the Heegaard splitting is non-trivial, the handlebody chamber is called {\em occupied}.\index{Occupied chamber}  A {\em flagged} chamber complex is a Heegaard split chamber complex in which each handlebody chamber is labelled either empty or occupied. 

Two flagged chamber complexes are the same as flagged chamber complexes if they are the same as chamber complexes and the flagging of each handlebody chamber is the same.  The Heegaard splittings of any given chamber are not necessarily isotopic, nor even of the same genus.  \index{Flagged chamber complex}
\end{defin}

\begin{defin} \label{defin:prefdisk}  \index{Preferred alignment} The alignment of a disk set $\calD$ in $\calC$ is a {\em preferred alignment} if, in each chamber $C$ of $\calC$, it has these properties:
\begin{enumerate}
\item If any remnant of $C$ is {\em not} a disky handlebody, then each disky handlebody remnant is empty.
\item If every remnant of $C$ is a disky handlebody, then, per Proposition \ref{prop:remnant}, $C$ is a handlebody.  In this case:
\begin{enumerate}  
\item If $C$ is empty, so is every remnant
\item If $C$ is occupied, then exactly one remnant of $C$ is occupied.
\end{enumerate} 
\end{enumerate}
\end{defin}

\begin{cor} \label{cor:prefdisk}  Suppose $\calD$ is a disk set that is in preferred alignment in a flagged chamber complex $\calC$.  Suppose $\calC$ has no occupied handlebody chambers.  Then, under the decomposition 
\[\calC \xrightarrow{\calD} \calC_{\calD}\] each handlebody chamber in $\calC_{\calD}$ is empty if and only if it is disky. 
\end{cor}

\begin{proof}  If a handlebody chamber in $\calC_{\calD}$ is not disky, then it is occupied per Proposition \ref{prop:Heegdisky2}, so the interest is in the other direction.  Suppose $C'$ is a disky handlebody chamber in $\calC_{\calD}$, a remnant of a chamber $C$ in $\calC$.  By hypothesis, $C$ is not an occupied handlebody so, per Definition \ref{defin:prefdisk}(2a), if every remnant of $C$ is a disky handlebody then every remnant, including $C'$, is empty.  On the other hand, if not every remnant of $C$ is a disky handlebody then per \ref{defin:prefdisk}(1) $C'$ is empty.  Thus in every case $C'$ is empty, as required.  
\end{proof}

We will first show that any disk set has a preferred alignment.  A critical fact from Heegaard theory is that any Heegaard splitting of a handlebody is standard, that is it is a stabilization of the trivial splitting. (See, for example, the discussion at the end of Section \ref{sect:sequence} or, in more detail, the proof of Proposition \ref{prop:occupycertify} below.)  In particular, if $H = A_H \cup_{T_H} B_H$ is a Heegaard split handlebody with non-trivial splitting then there is a bubble $\frb$ for the splitting so that if the bubble is destabilized (for example by replacing $\frb$ with a ball containing just a properly embedded equatorial disk), the splitting becomes trivial.  That is, the splitting surface becomes parallel to $\bdd H$.  Such a bubble will be called a {\em maximal bubble} in $H$. \index{Maximal bubble}

\begin{lemma}  \label{lemma:doubbub} Suppose \[\calC \xrightarrow{\calD} \calC_{\calD}\] is a Heegaard split chamber complex decomposition and $H = A_H \cup_{T_H} B_H$ is an occupied disky handlebody remnant of a chamber $C \in \calC$, with $C = A_C \cup_{T_C} B_C$.  Then there is an isotopy of $T_H$ in $H$ so that afterwards
\begin{itemize}
\item There is a disk $E \subset T_H$ so that $T_H - E \subset T_C$ and  
\item  There is a maximal bubble $\frb$, disjoint from $E$, in the splitting of $H$ that is also a bubble in the splitting of $C$.
\end{itemize}
\end{lemma}

\begin{proof}  With no loss assume $C$ and hence $H$ are $A$-chambers.  We briefly set terminology and review the results of Heegaard split chamber complex decomposition from Subsection \ref{subsec:diskdecomp}.  Let $\Sigma_C$ be a spine for $B_C$ consisting of the union of $\bdd C$ and a graph $\gamma_C \subset C$ incident to $\bdd C$ only in its ends.  We may as well focus on those disks in $\calD$ that are incident to $C$ and retreat to denoting these as $\calD$.  Since $\calD$ is aligned, each disk $D \in \calD$ intersects the spine $\Sigma_C$ only in $\bdd D$; the interior of $D$ lies either in $C$ or in a chamber adjacent to $C$.  In particular, $\calD$ is disjoint from the graph $\gamma_C$.  Except for scars left on $\bdd H$ by $\calD$, $\bdd H \subset \bdd C$.  If $D \in \calD$ leaves an external scar on $\bdd H$ then $D$ lies in $C$; if it leaves an internal scar then $\inter(D)$ lies in an adjacent chamber.  There is a spine $\Sigma_H$ for $B_H$ which consists of the union of $\bdd H$ with two graphs in $H$ (whose union we denote $\gamma_H$): the graph $\gamma_C \cap H$ and a graph $\gamma_I \subset H$ determined by the internal scars in $\bdd H$ and goneballs in the interior of $H$.  Each edge in $\gamma_I$ is dual to a disk in $\calD$ whose interior lies outside $C$.  More specifically, each end of an edge of $\gamma_I$ lies either on an internal scar in $\bdd H$ or on a vertex in $\gamma_I \cap \inter(H)$ that corresponds to a  goneball.  Because the compression body $B_H$ is connected, each component of $\gamma_H$ has at least one end on $\bdd H$.

We consider increasingly complicated cases:
\medskip

{\em Case 1:} $H$ is a ball and has no internal scars or goneballs.  That is $\gamma_I = \emptyset$.

 \begin{figure}
\labellist
\small\hair 2pt
\pinlabel  $T_{C}$ at 50 130
\pinlabel  $B_{C}$ at 25 100
\pinlabel  $\gamma_H\subset\gamma_C$ at 135 240
\pinlabel  $\bdd C$ at 0 30
\pinlabel  scars at 125 0
\pinlabel  $E_{\bdd}$ at 125 18
\pinlabel  $E_A$ at 125 70
\pinlabel  $T_{H}$ at 330 130
\pinlabel  $E$ at 400 45
\pinlabel  $B_{H}$ at 350 52
\pinlabel  $\bdd \frb$ at 550 215
\endlabellist
    \centering
    \includegraphics[scale=0.6]{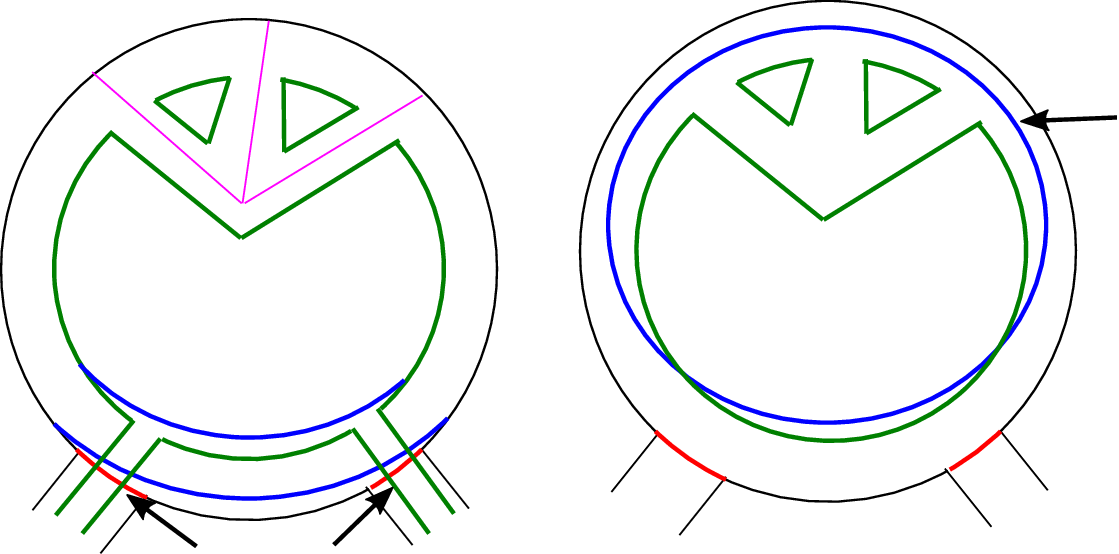}
     \caption{Case 1 of Lemma \ref{lemma:doubbub}} \label{fig:doubbub1}
    \end{figure}

Connect the (external) scars of $\calD$ on $\bdd H$ by a collection of arcs in $\bdd H$ which, by general position, can be taken to be disjoint from the ends of $\gamma_H$ in $\bdd H$. Viewing these arcs as edges and the scars as vertices of a graph in $\bdd H$, choose the arcs so that the resulting graph is a tree.  Let $E_{\bdd}\subset \bdd H$ be a (disk) regular neighborhood in $\bdd H$ of the union of the scars and arcs. Since that union is disjoint from the ends of $\gamma_H$ we can take $E_{\bdd}$ also to be disjoint from the ends of $\gamma_H$.  

The compression body $B_C$ is a regular neighborhood of the spine $\Sigma_C$.  Since $\calD$ is aligned and so disjoint from $\gamma_C$, $\calD$ intersects $B_C$ only in a collar of $\bdd B_C$.  Hence $B_C$  intersects each scar $\frs$ on $\bdd H$ in a collar of $\bdd \frs$ in $\bdd H$.  A regular neighborhood of $\Sigma_H$, to which we can isotope $B_H$, can then be constructed from $B_C$ in an obvious way: take the union of $B_C \cap H$ and a collar in $H$ of each scar on $\bdd H$.  See Figure \ref{fig:doubbub1}.  With this construction, 
$T_H$ lies in $T_C$ except for the collection of disks in $T_H$ that are parallel to the scars.  If we then take $E$ to be the disk in $T_H$ parallel to $E_{\bdd}$, we have $T_H - E \subset T_C$, verifying the first assertion of the lemma in this case.  

Since $E$ is disjoint from $\gamma_H$, the interior of $E$ can be pushed, rel $\bdd E$, into $A_H$.  Call the result $E_A$.  Since $H$ is a ball, the complement of $E_{\bdd}$ in $\bdd H$ is also a disk; let $E_B$ be a copy of the disk $\bdd H - E_{\bdd}$ pushed up into the compression body $B_H$ until its boundary lies on $\bdd E = \bdd E_A$ in $T_H$.  The union of $E_A$ and $E_B$ is a sphere bounding a ball in $H$ that contains all of $T_H - E$.  That ball is then a maximal bubble for $H$ that is also a bubble for the splitting of $C$, verifying the lemma in this case.  
\medskip

{\em Case 2:} $H$ is a ball that may contain internal scars and goneballs.  

Suppose $\frs$ is an internal scar on $\bdd H$, a scar  left by a disk in $\calD$ that lies in an adjacent chamber $C'$.  See Figure \ref{fig:doubbub2}. Since by assumption $H$ is a {\em disky} remnant, $\bdd \frs$ bounds a disk component $E_{\frs}$ of $\bdd C \cap \inter(H)$.  Since $H$ is irreducible, the disks $\frs$ and $E_{\frs}$ are isotopic rel $\bdd$ in $H$, and all ball components of $\hat{\calC}_{\calD}$ must be goneballs and therefore trivially split.  In particular the Heegaard splitting surface of the adjacent chamber $C'$ does not contribute any genus to $T_H$.   

Let $\bdd H_{\frs}$ be the copy of $\bdd H$ obtained by replacing each interior scar $\frs$ with the corresponding disk $E_{\frs} \subset \bdd C$.  Now apply the argument of Case 1 to the ball in $H$ bounded by $\bdd H_{\frs}$ instead of $\bdd H$, with the arcs in $\bdd H$ between external scars in that construction chosen to avoid also internal scars.  This completes the proof in this case.  (Note that with this construction, $T_H$ is disjoint from the ball in $H$ between the scar $\frs$ and the disk $E_{\frs}$.) 
\medskip

 \begin{figure}[ht!]
\labellist
\small\hair 2pt
\pinlabel  $E_{\frs}\subset\bdd C$ at 155 150
\pinlabel  $\bdd C$ at 125 28
\pinlabel  $C$ at 70 100
\pinlabel  $C'$ at 220 150
\pinlabel  $\frs$ at 245 160
\pinlabel  $T_{H}$ at 320 150
\endlabellist
    \centering
    \includegraphics[scale=0.5]{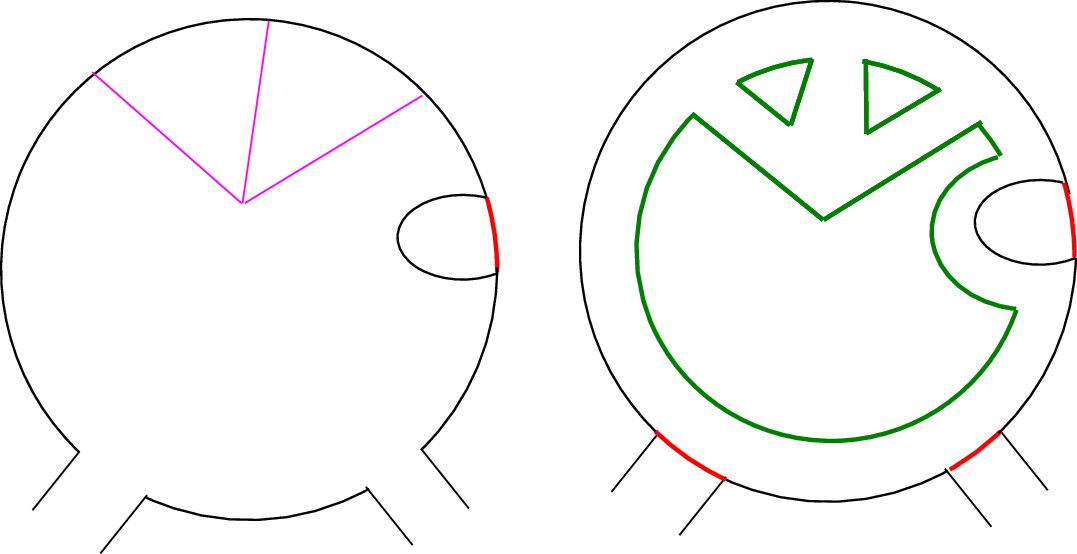}
     \caption{Case 2 of Lemma \ref{lemma:doubbub}} \label{fig:doubbub2}
    \end{figure}

{\em Case 3:} The general case.  

Define the disk $E \subset T_H$ parallel to a disk $E_{\bdd} \subset \bdd H$ as in Case 1, with arcs chosen in its construction to also be disjoint from internals scars.  Essentially the same argument as in Cases 1 and 2 shows that we can ensure that $T_H - E \subset T_C$.  So we proceed to the second assertion of the Lemma: there is a maximal bubble in $H$ that is disjoint from $E$ and is also a bubble for $C$.  

Let $\calD_+$ be a minimal complete collection of meridian disks in $H$.  That is, $\calD_+$ is a properly embedded collection of disks in $H$ so that the closed complement of a collar of $\calD_+$ in $H$ is a single ball $W$.  By general position, we may take $\bdd \calD_+ \subset \bdd H$ to be disjoint from $E$ and all internals scars, and per \cite{Sc1} we can take $\calD_+$ to be aligned with the Heegaard surface $T_H$. This means in particular that the graph $\gamma_H \subset \Sigma_H$ can be taken to be disjoint from the meridian disks $\calD_+$, so it lies entirely in the ball $W$.  

Let $E_{\bdd+}$ be the disk in $\bdd W$ obtained  by band summing $E_{\bdd}$ to each copy of $\calD_+$ in $\bdd W$. (For each disk $D \in \calD_+$ there are two copies $D_{\pm}$ in $\bdd W$.) See Figure \ref{fig:doubbub3}.  The complement in the sphere $\bdd W$ of $E_{\bdd+}$ is a disk that lies entirely in $\bdd H$.  In parallel fashion expand the disk $E_A$ (per Case 1 a copy of $E$ with interior pushed into $A$) to a disk $E_{A+}$ properly embedded in $A$ by band summing $E_A$ to each disk $D_{\pm} \cap A, D \in \calD_+$. Since $\bdd W - E_+$ is a disk in $\bdd H$, it follows as in Case 1 that $\bdd E_{A+} \subset T_H$ bounds a disk in $B_H$.  The union of the two disks is then a sphere parallel to $\bdd W$ and the ball $\frb$ it bounds in $W$ is disjoint from $E$ and contains all but the ends of $\gamma_H$.  So $\frb$ is a maximal bubble, as required.  
\end{proof}

 \begin{figure}[ht!]
\labellist
\small\hair 2pt
\pinlabel  $W$ at 100 110
\pinlabel  $\calD_+$ at 235 110
\pinlabel  $T_C$ at 40 110
\pinlabel  $T_H$ at 330 130
\pinlabel  $E_{A+}$ at 125 80
\pinlabel  $E_{\bdd+}$ at 120 30
\pinlabel  $E$ at 400 35
\pinlabel  $\bdd\frb$ at 430 90
\endlabellist
    \centering
    \includegraphics[scale=0.6]{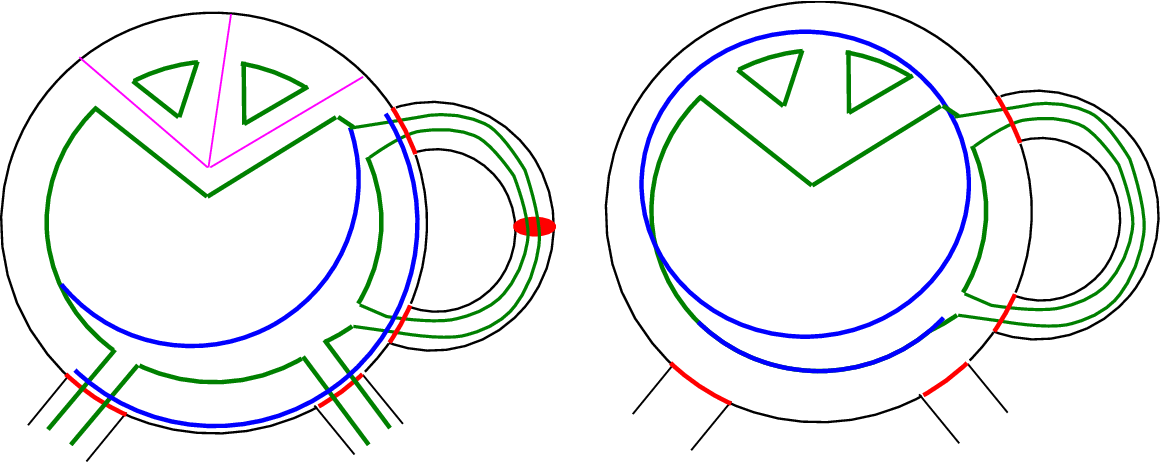}
     \caption{Case 3 of Lemma \ref{lemma:doubbub}} \label{fig:doubbub3}
    \end{figure}

\begin{prop} \label{prop:prefexist}
Suppose $\calC$ is a Heegaard split chamber complex in $M = A \cup_T B$ that supports $T$, and $\calD$ is a disk set in $\calC$.  Then $\calD$ has a preferred alignment.
\end{prop}

\begin{proof}  As discussed in Subsection \ref{subsec:diskdecomp}, $\calD$ can be aligned. The strategy will be to begin with any alignment and alter it to a preferred alignment.  

We first describe how to realign $\calD$ so as to achieve Definition \ref{defin:prefdisk}(1).  Suppose some remnant $C' \in \calC_{\calD}$ of a chamber $C \in \calC$ is not a disky handlebody, and let $H = \{H_1, ..., H_k\}$ be the collection of occupied disky handlebody remnants of $C$.  If $H = \emptyset$ there is nothing to prove.  Otherwise, let $\gamma = \{\gamma_1, ..., \gamma_k\}$ be a collection of arcs in $C$ with the following properties:
\begin{itemize}
\item Each $\gamma_i$ is a path in the Heegaard surface $T_C$ that is transverse to $\calD$ and runs from a point in $\inter(C')$ to a point in $\inter(H_i)$.
\item Among all such families of arcs, $|\gamma \cap \calD| \geq k$ is minimized.
\end{itemize}

Let $D \in \calD$ be the disk containing the closest intersection point on $\gamma_1$ to its endpoint in $H_1$.  Let $C''$ be the remnant that is on the other side of $D$.  Since $|\gamma \cap \calD|$ is minimized, $\gamma_1$ is otherwise disjoint from $D$ and $C'' \neq H_1$.  $H_1$ is assumed to be occupied; let $\frb$ be a maximal bubble for the splitting $H_1 = A_1 \cup_{T_1} B_1$ satisfying the conclusion of Lemma \ref{lemma:doubbub}.  That is, $\frb$ is also a bubble for the Heegaard splitting $C = A_C \cup_{T_C} B_C$.  

Take the endpoint of $\gamma_1$ in $H_1$ to be a point in the circle $c_{\frb} = \bdd \frb \cap T_C$, so the subpath $\gamma_H = \gamma_1 \cap H_1$ is an arc in $T_C$ from $D$ to the sphere $ \bdd \frb$.  Tube sum $D$ to $\bdd \frb$ along the path $\gamma_H$, creating a new disk $D'$.  Observe that $D'$ intersects $T_C$ in a single circle, namely the band sum along $\gamma_H$ of the circles $D \cap T_C$ and $\bdd \frb \cap T_C$.  Thus $D'$ is an aligned disk, and is isotopic to $D$ in $C$ because $\frb$ is a ball.  Moreover, the bubble $\frb$ lies on the same side of $D'$ as $C''$, not $H$.  Thus $D'$ is a realignment of $D$ that makes $H$ an empty handlebody.  If $C''$ is not a handlebody, or is a handlebody that was already occupied, then $k$ is reduced by 1.  If $C''$ was an empty handlebody then after the realignment $C''$ becomes an occupied handlebody, leaving $k$ unchanged.  But the path $\gamma_1$ now has its endpoint in $C''$ and no longer intersects the realigned $D'$, so $|\gamma_1 \cap \calD|$ (with $D$ realigned) is reduced by 1.  Continue the realignment of disks until $k = |\gamma \cap \calD| = 0$ as required.  

To establish Definition \ref{defin:prefdisk}(2a) note that if any remnant of $C$ is a disky occupied handlebody then Lemma \ref{lemma:doubbub} shows $C$ contains a bubble and so cannot have trivial splitting; it is occupied.  To establish \ref{defin:prefdisk}(2b), use the same argument as for \ref{defin:prefdisk}(1) above, but make an arbitrary choice of any remnant for $C'$. 
%
%
%
%
%
\end{proof}

\begin{defin} \label{defin:flagdecomp}
Suppose $\calC$ is a flagged chamber complex, $\calD$ is a disk set in $\calC$ with a preferred alignment, and $\hat{\calC}_{\calD}$ is the chamber complex obtained by surgery on $\calD$.  Let $\calC_{\calD}$ be the flagged chamber complex obtained by using Rule \ref{rule:diskyHeeg} to declare goneballs.  Then the Heegaard split chamber complex decomposition $$\calC \xrightarrow{\calD} \calC_{\calD}$$ is called a {\em flagged} chamber complex decomposition.  
\end{defin} 

For much of our argument, simply keeping track of the flagging of the chamber complexes (that is, whether a handlebody chamber is empty or occupied) will suffice. 

\begin{prop} \label{prop:flagprop} Suppose, for $\calD$ a disk set in $\calC$, $$\calC \xrightarrow{\calD} \calC_{\calD}$$ is a flagged chamber complex decomposition. 

\begin{enumerate}

\item Any ball chamber in $\calC_{ \calD}$ is occupied.  (That is, $\calC_{ \calD}$ is a Heegaard split chamber complex.)


\item If any remnant of a chamber $C$ in $\calC$ is an occupied disky handlebody then $C$ is an occupied handlebody and every remnant of $C$ in $\calC_{\calD}$ is a disky handlebody.

\item Suppose $\calD$ is given a different preferred alignment and  $\calC'_{\calD}$ is the resulting flagged chamber complex structure.  If $\calC'_{\calD}$ and $\calC_{\calD}$ differ as flagged chamber complexes, then $\calC$ contains an occupied handlebody chamber whose remnants in $\calC_{\calD}$ and $\calC'_{\calD}$ are all disky handlebodies.

%
%
\end{enumerate}
\end{prop}
 
\begin{proof}  

(1) If a ball chamber were not occupied, it would be empty so, by Definition \ref{defin:empty}, it would have trivial Heegaard splitting.  But then, by Rule \ref{rule:diskyHeeg}, it would have been a goneball.

(2) The proof is essentially the same as that of Corollary \ref{cor:prefdisk}:  If a remnant of $C$ is an occupied disky handlebody then from the contrapositive of Definition \ref{defin:prefdisk}(1) every remnant of $C$ is a disky handlebody.  Then the contrapositive of Definition \ref{defin:prefdisk}(2a) says that $C$ must be an occupied handlebody.

(3) In a flagged chamber complex decomposition, the flagging of the chambers after surgery, as described in Definition \ref{defin:prefdisk}, depends only on the disks $\calD$ and not on how they are aligned with the Heegaard surfaces in the chamber, except in satisfying Definition \ref{defin:prefdisk}(2b):  If $C$ is an occupied handlebody chamber of $\calC$ and each remnant of $C$ is a disky handlebody, then a choice of alignment is made to ensure that exactly one remnant (of possibly several) is occupied.  Hence if $\calC'_{\calD}$ and $\calC_{\calD}$ differ, so $\hat{\calC}_{\calD}$ and  $\hat{\calC'}_{\calD}$ differ, it must be because of a different such choice.  That is, in $\hat{\calC}_{\calD}$ one remnant is an occupied handlebody, and in $\hat{\calC'}_{\calD}$ a different remnant is.  Thus, per statement (2), $C$ is the required occupied handlebody chamber of $\calC$.
\end{proof}

Suppose, as described in the proof of Proposition \ref{prop:flagprop}(3), two flagged chamber complex decompositions $\calC \xrightarrow{\calD} \calC_{\calD}$ and $\calC \xrightarrow{\calD} \calC'_{\calD}$ differ only because, for some occupied handlebody chambers in $\calC$, different alignments of the disks $\calD$ in the chambers result in different remnants being occupied. 

\begin{defin} \label{defin:sibling}  The decompositions $\calC \xrightarrow{\calD} \calC_{\calD}$ and $\calC \xrightarrow{\calD} \calC'_{\calD}$ are  {\em sibling} decompositions, \index{Sibling decompositions} and the occupied handlebody chambers in $\calC$ are called {\em parent} chambers in the sibling decompositions.
\end{defin} 

Proposition \ref{prop:prefexist}. is an existence statement; we now move towards a parallel uniqueness statement.  

\begin{lemma} \label{lemma:bubinH}  Suppose $H = A \cup_T B$ is a Heegaard split handlebody.  Suppose $\frb^p$ and $\frb^q$ are not necessarily disjoint bubbles for $T$ of genus $p \leq q$ respectively.   Then
\begin{itemize}  
\item There is a genus $(q-p)$ bubble $\frb'$ for $T$ that is disjoint from $\frb^p$, and a homeomorphism $h: (H, T) \to (H, T)$ so that $h$ is the identity on $\bdd H$, and $\frb^q$ is the tube sum of $h(\frb^p)$ and $h(\frb')$.  
\item Under Assumption \ref{ass:inductive}, if $\genus(T) - \genus(H) \leq g-1$ then we may take $h$ to be an eyeglass move.  
\end{itemize} 
\end{lemma}

\begin{proof}  Pick an aligned disk set $\calD$ in $H$ so that $H - \eta(\calD)$ is a single ball.  This implies $|\calD| = \genus(H)$ and each disk is non-separating. Then surgery on $\calD$ gives a genus $g' = \genus(T) - \genus(H)$ Heegaard splitting of the ball which, by \cite{Wa}, is unique.   More explicitly, echoing the notation surrounding Figure \ref{fig:circlesinT}, there is a nested sequence of $g'$ bubbles $\frb_1 \subset ... \subset \frb_{g'}$ for $T$ so that each sphere $S_i = \bdd \frb_i$ intersects $T$ in a single circle $c_i$, and $T/\frb_{g'}$ is parallel to $\bdd H$. 

More specifically, align $\calD$ first with $T/\frb^q$ so that in the associated alignment of $\calD$ with $T$, $\frb^q$ is disjoint from $\calD$.  Then with no loss of generality we may assume that $\frb^q = \frb_q$.  For the first statement above, it then suffices to find a homeomorphism $h(H, T) \to (H, T)$ so that $h(\frb^p) = \frb_p$.  Denote this alignment of $\calD$ by $\calD^q$.  

Similarly, let $\calD^p$ be an alignment of $\calD$ found by first aligning with $T/\frb^p$ and then observing that the associated alignment of $\calD^p$ with $T$ is disjoint from $\frb^p$.  According to \cite{FS2} there is a sequence of bubble passes and eyeglass moves that will take $\calD^p$ to $\calD^q$ and, since each disk in $\calD$ is non-separating disk, it follows from Lemma \ref{lemma:nonseppass} that even the bubble passes are eyeglass moves. (To pass a bubble through a disk $D \in \calD$ choose the arc $\beta$ in the proof to be disjoint from all other disks in $\calD$.)  Thus there is an eyeglass move that takes $\calD^p$ to $\calD^q$.  Put another way, there is an eyeglass move $h_1: (H, T) \to (H, T)$ so that $h_1(\frb^p)$ lies in the ball $H - \eta(\calD^q)$.  By \cite{Wa} there is then a homeomorphism $h_2: (H, T) \to (H, T)$ so that $h_2(h_1(\frb^p)) = \frb_p$.  This completes the proof of the first statement, with $h = h_2 h_1$.

To prove the second statement, just note that $T$ induces a genus $g'$ splitting on $H - \eta(\calD^q)$, so if $g' \leq g-1$ then we may take $h_2$ to be an eyeglass move, so $h_2 h_1$ is also an eyeglass move.  
\end{proof}

\medskip

Throughout the remainder of this section $\calC$ is a flagged chamber complex in $M$ supporting a genus $g$ Heegaard splitting, with $\calD$ a disk set in $\calC$.  We further continue with Assumption \ref{ass:inductive}.  Let $\calD_{\pm}$ denote two preferred alignments of $\calD$ that result in the same flagged chamber complex decompositions $\calC \xrightarrow{\calD} \calC_{\calD_{\pm}}$.  That is, to repeat from Definition \ref{defin:empty}, $\calC_{\calD_{\pm}}$ are identical chamber complexes with identical flagging (handlebody chambers either empty or occupied), though the underlying Heegaard splittings of each chamber may differ.

\begin{prop} \label{prop:flagunique}  There is a sequence of preferred alignments of $\calD$ in $\calC$
\[\calD_- = \calD_0, \calD_1, \calD_2, \calD_3 ..., \calD_n = \calD_+\]
so that 
\begin{enumerate}[label=\alph*)]
\item for each $1 \leq i \leq n$, $\calD_i$ is obtained from $\calD_{i-1}$ by eyeglass moves and a single bubble pass, possibly through multiple disks of $\calD_{i-1}$.
\item The result $\calC_{\calD_i}$ of each flagged chamber complex decomposition $$\calC \xrightarrow{\calD_i} \calC_{\calD_{i}}$$ is the same flagged chamber complex as $\calD_{\pm}$.
\end{enumerate}
\end{prop}


%

\begin{proof}  The central theorem of \cite{FS2} is that there is such a sequence of alignments, but without the condition that each $\calD_i$ is in {\em preferred} alignment and that each resulting Heegaard split chamber complex $\calC_{\calD_i}$ has the same flagging.  We begin with such a sequence and alter it to achieve these conditions.  
\medskip

{\em Claim:} It suffices to find a sequence of (not necessarily preferred) alignments satisfying a) and b) so that each disky handlebody chamber of $\calC_{\calD}$ that is empty in $\calC_{\calD_-}$ is also empty in each $\calC_{\calD_i}$.  
\medskip

{\em Proof of claim:} We will {\em assume we have found such a sequence of alignments} and show how the proof of the proposition follows.

First observe that by Proposition \ref{prop:Heegdisky2} any handlebody chamber of $\calC_{\calD}$ that is not disky is occupied regardless of alignments, so the only difference in possible flaggings of the $\calC_{\calD_i}$ is in the flagging of disky handlebodies.   Since $\calC_{\calD_-}$ has a preferred alignment, it satisfies Definition \ref{defin:prefdisk}(1), namely that if any remnant of $C$ is not a disky handlebody, then each disky handlebody remnant is empty.  But by the assumption, each disky handlebody remnant is then empty in $\calC_{\calD_i}$, so each $\calC_{\calD_i}$ also satisfies Definition \ref{defin:prefdisk}(1).  

Similarly, if each remnant of a chamber $C$ is a disky handlebody (so $C$ is a handlebody, per Proposition \ref{prop:remnant}), then Definition \ref{defin:prefdisk}(2a) says that if $C$ is empty, every remnant is empty in $\calC_{\calD_-}$ and so, by assumption, also in each $\calC_{\calD_i}$.  Thus each $\calC_{\calD_i}$ satisfies Definition \ref{defin:prefdisk}(2a).  On the other hand, if $C$ is occupied, Definition \ref{defin:prefdisk}(2b) says that exactly one remnant, say $C' \in \calC_{\calD_-}$ is occupied, so the others are all empty.  By assumption, all remnants of $C$ in each $\calC_{\calD_i}$, except possibly $C'$, are then empty.  But since $C$ is occupied it has non-trivial splitting, that is the genus of its Heegaard splitting surface is greater than that of its boundary.  So decomposition cannot leave only trivially split handlebodies.  Thus some remnant of $C$ must be empty in each $\calC_{\calD_i}$, and $C'$ is the only possibility.  Thus exactly one remnant of $C$ (namely $C'$) in each $\calC_{\calD_i}$ is occupied, and so each $\calC_{\calD_i}$ also satisfies Definition \ref{defin:prefdisk}(2b), with the same flagging as $\calC_{\calD_-}$.  Thus each $\calD_i$ is in preferred alignment, as required.  
\medskip

Having established the claim, we proceed to show that in fact there is a sequence of alignments satisfying a) and b) so that each disky handlebody chamber of $\calC_{\calD}$ that is empty in $\calC_{\calD_-}$ is also empty in each $\calC_{\calD_i}$.  This follows from a min-max argument that we now describe.

Let $H$ be the collection of disky handlebodies in $\calC_{\calD}$ that are designated as empty in $\calC_{\calD_-}$ and so by assumption, also in $\calC_{\calD_+}$.  Since these are flagged chamber complexes, no chamber in $H$ is a ball.

Consider any sequence of (not necessarily preferred) alignments, given by \cite{FS2}, that begins with $\calD_-$ and ends with $\calD_+$ so that each alignment in the sequence differs from the previous alignment by eyeglass moves before and after a single bubble pass, possibly through multiple disks. For each $0 \leq i \leq n$, let $\sigma_i \geq 0$ be the difference between the sum of all the genera of the Heegaard surfaces for $H$ in $\calC_{\calD_i}$ and the sum of the genera of $\bdd H$.   If we can find a sequence so that each $\sigma_i = 0$, then each Heegaard splitting would be trivial, so each chamber in $H$ would continue to be empty in each $\calC_{\calD_i}$.  Following the claim above, this would conclude the proof of the lemma.  

In search of a sequence of alignments for which $\sigma_i = 0$, let \[\calD_- = \calD_0, \calD_1, \calD_2, \calD_3 ..., \calD_n = \calD_+\] be one that minimizes, in lexicographic order, the pair $(m, x)$, where $m$ is the maximum of $\sigma_i$ throughout the sequence, and $x$ is the number of integers $i$ for which $\sigma_i = m$. 
The aim is to show that $m = 0$.  Towards a proof by contradiction, suppose $m \geq 1$ and let $i$ be the smallest integer for which $\sigma_i = m$.  We know that $1 \leq i \leq n-1$ since $\sigma_0 = \sigma_n = 0$.  

So we have
\begin{enumerate}
\item $\sigma_{i-1} < m$
\item $\sigma_i = m$
\item $\sigma_{i+1} \leq m$.
\end{enumerate}

The first statement implies that the realignment of $\calD_{i-1}$ to $\calD_i$ moves a bubble from a remnant $R_1 \notin H$ to a remnant $R_2 \in H$. The third statement means that in the realignment of $\calD_i$ to $\calD_{i+1}$, which moves a bubble from  remnant $R_3$ to remnant $R_4$, we cannot have both $R_3 \notin H$ and $R_4 \in H$.  

Suppose that $R_2 \neq R_3$.  In this case, the bubble in $R_3$ could be passed to $R_4$ before the bubble passes from $R_1$ to $R_2$.  This does not affect $\sigma_{i-1}$ or $\sigma_{i+1}$ but may change $\sigma_i$, say to $\sigma'_i$.  Since we cannot have both $R_3 \notin H$ and $R_4 \in H$, $\sigma'_i \leq \sigma_{i-1}$. (1) then implies $\sigma'_i < m$.  Thus we have reduced $x$ (or $m$ if always $j > i \implies \sigma_j < m$).  

Next suppose $R_2 = R_3$, and just call it $R$.   We have already seen that this remnant must be in $H$, and so a disky handlebody. Since bubbles can only be passed between remnants of the same chamber, we deduce that $R_1, R, R_4$ are all remnants of the same chamber of $\calC$. In particular, it is possible to directly pass a bubble from $R_1$ to $R_4$.   Let $g_1, g_4$ be respectively the genera of the bubble $\frb_1$ that is moved into $R$ from $R_1$ and the bubble $\frb_4$ that is moved from $R$ into $R_4$.  There are two cases to consider:

{\em Case 1:} $g_1 \leq g_4$

Let $D_1 \subset \bdd R$ be the disk through which $\frb_1$ is passed into $R$ from $R_1$ as it lies in $\calD_{i-1}$.  We proceed as though $D_1$ is the only disk through which $\frb_1$ is passed, so it is passed from one remnant of $C$ to an adjacent one; if it is passed through multiple chambers the argument only needs to be altered by taking multiple parallel copies of $D_1$, corresponding to how the thin tube from $R$ to $\frb_1$ passes through other disks in $\calD_{i-1}$ in the bubble pass.  Similarly let $D_4 \subset \bdd R$ be the disk (or parallel copies of the disk) in $\calD_{i-1}$ through which $\frb_4$ is passed into $R_4$.  (Possibly $R_1 = R_4$ and $D_1 = D_4$.)

Let $(D_s, \bdd D_s)  \subset (R_1, \bdd D_1)$ be a disk parallel to $D_1$ obtained by tube-summing $\frb_1$ to $D_1$.  Similarly, let $H_+$ be the handlebody obtained from $R$ by the same tube-summing, replacing $D_1$ with $D_s$.  Then $H_+$ contains both $\frb_1$ and $\frb_4$.  According to Lemma \ref{lemma:bubinH} there is an eyeglass move $h: H_+ \to H_+$ and a bubble $\frb'$ of genus $g_4 - g_1$ so that $\frb_4$ is the tube sum of $\frb'$ and $h(\frb_1)$.  Replace by an eyeglass move $D_1 \in \calD_{i-1}$ with $D'_1 = h(D_1)$.  See the highly schematic Figure \ref{fig:flagunique}.

 \begin{figure}[ht!]
\labellist
\small\hair 2pt
\pinlabel  $R_1$ at 10 250
\pinlabel  $R$ at 60 250
\pinlabel  $R_4$ at 110 250
\pinlabel  $D_s$ at -10 280
\pinlabel  $D_1$ at 45 280
\pinlabel  $D_4$ at 100 280
\pinlabel  $\frb_1$ at 15 317
\pinlabel  $R$ at 60 250
\pinlabel  $\calD_{i-1}$ at 60 205
\pinlabel  $\sigma_{i-1}$ at 60 175
\pinlabel  $\calD_{i}$ at 220 215
\pinlabel  $\sigma_{i}$ at 220 195
\pinlabel  $\sigma_{i}-g_4$ at 220 155
\pinlabel  $\calD_{i+1}$ at 380 205
\pinlabel  $\sigma_{i+1}$ at 380 175
\pinlabel  $\frb_4$ at 220 340
\pinlabel  $h(\frb_1)$ at 220 240
\pinlabel  $\calD'_{i-1}$ at 60 -10
\pinlabel  $\calD'_{i}$ at 220 -10
\pinlabel  $\calD'_{i+1}=\calD_{i+1}$ at 380 -10
\pinlabel  $D'_1$ at 70 95

\endlabellist
    \centering
    \includegraphics[scale=0.75]{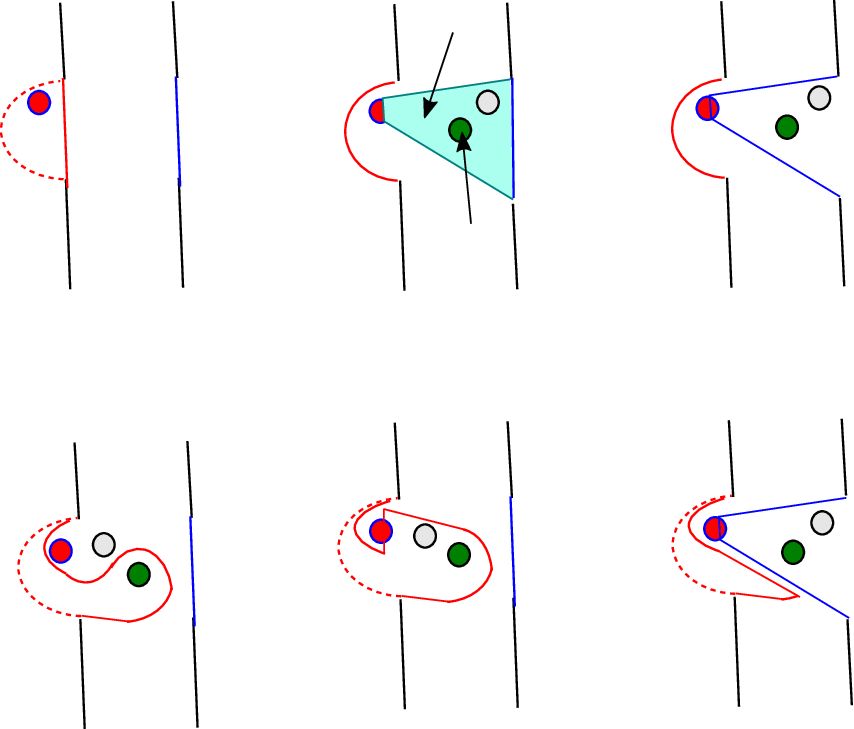}
     \caption{Case 1 
      when $R_2 = R_3 = R$} \label{fig:flagunique}
    \end{figure}

Do the transition from the alignment of $\calD_{i-1}$ to that of $\calD_{i+1}$ with these bubble passes instead of the given two:  
Tube sum $D'_1$ to $\frb'$, after which $\frb_4$ lies entirely in $R_1$.  This changes $\sigma$ to $\sigma_{i} - g_4$, since $R_1 \notin H$.
Then pass the bubble $\frb_4$ directly from $R_1$ into $R_4$.   This realignment leaves $D_4$ aligned so that $R_4$ contains $\frb_4$.  The upshot of this replacement of $D_1$ with $D'_1$ in $\calD_{i-1}$ is that the first occurrence of $\sigma_i = m$ in the sequence is replaced with $m - g_4$. 
Thus $x$ is reduced by one (or $m$ is reduced, if always $j > i \implies \sigma_j < m$).

{\em Case 2:} $g_4 \leq g_1$

The argument is much the same, using $H_+$ as defined above, with the following alteration:  According to Lemma \ref{lemma:bubinH} there is a homeomorphism $h: H_+ \to H_+$ and a bubble $\frb'$ of genus $g_1 - g_4$ so that $h(\frb_1)$ is the tube sum of $\frb'$ and $\frb_4$. 

Do the transition from the alignment of $\calD_{i}$ to that of $\calD_{i+1}$ with these bubble passes instead of the given two:  First replace $D_1$ with $h(D_1)$.  After this realignment 
$\frb_4 \subset h(\frb_1)$ lies entirely in $R_1$.  Next move $\frb_4$ directly to $R_4$.  This move does not raise $\sigma = \sigma_{i-1}$ and may lower it by $g_4$ if $R_4 \notin H$.  Next move $\frb'$ into $R$, returning the alignment to $\calD_{i+1}$.  This transition replaces the first occurrence of  $\sigma_i = m$ in the sequence with a term no larger than $\sigma_{i-1}$.   Thus again
$x$ is reduced by one (or $m$ is reduced, if always $j > i \implies \sigma_j < m$).

In any case we have contradicted the assumption that $(m, x)$ is minimal.  We conclude that $m = 0$, as required.
\end{proof}

\begin{defin} A flagged chamber complex is tiny if it is a tiny chamber complex, as in Definition \ref{defin:tiny1}, and, in the case that $F(\calC) \neq \emptyset$, each of the designated handlebodies is empty.
\end{defin}

In particular, if $F(\calC)$ is a single component dividing $M$ into two handlebodies, so either of them can play the role of designated handlebody, the flagged chamber complex is tiny unless {\em both} handlebodies are occupied.  

 \begin{prop}[Tinyness pulls back] \label{prop:tiny2} Suppose $$\calC \xrightarrow{\calD} \calC_{\calD}$$ is a flagged chamber complex decomposition.  If  $\calC_{ \calD}$ is tiny, so was $\calC$.
\end{prop}

\begin{proof} This is essentially a restatement of Proposition \ref{prop:tinyHeeg}.
\end{proof}

\begin{cor} \label{cor:tiny} Suppose 
 \[\calC_0 \xrightarrow{\calD_0} \calC_1 \xrightarrow{\calD_1} \calC_2 \xrightarrow{\calD_2} ... \xrightarrow{\calD_{n-1}}\calC_n\] 
  is a sequence of flagged chamber complex decompositions.  Then if $\calC_0$ is not tiny, neither is $\calC_n$.  In particular, $F(\calC_n) \neq \emptyset$.  
\end{cor}

\section{Occupied handlebodies and certification} \label{sect:certify}

In this section we demonstrate that the principal results from Section \ref{sect:sequence} continue to apply in the context of flagged chamber complexes, and with a broader version of certification.   Continue with Assumption \ref{ass:inductive}, that $G(S^3, T') = \calE$ whenever $\genus(T') \leq g-1$.

\begin{prop} \label{prop:occupycertify} Let $\calC$ be a flagged chamber complex for $S^3$ that supports the genus $g$ Heegaard splitting $(S^3, T)$.
Suppose $C = A_C \cup_{T_C} B_C$ is an occupied handlebody chamber of $\calC$.
\begin{enumerate}
\item  There is a reducing sphere $S$ for the Heegaard surface $T_C$ and a corresponding homeomorphism $h_C:(S^3; T) \to (S^3; T_g)$ so that $h_C(S) \in \{S_i, i = 1, ..., g-1\}$.  Moreover, the eyeglass equivalence class of $h_C$ is independent of $S$.
\item If $C'$ is another occupied handlebody chamber of $\calC$ then $h_C$ and $h_{C'}$ are eyeglass equivalent.
\item If $S$ is an incompressible sphere in a chamber of $\calC$ then $h_C$ is eyeglass equivalent to the homeomorphism $h_S$ given in Proposition \ref{prop:toyexist}.
\item Suppose $\calD$ is a 
 disk set in $\calC$ in preferred alignment, and $\calC \xrightarrow{\calD} \calC_{\calD}$ is the corresponding flagged chamber complex decomposition. 
Suppose all remnants of $C$ are disky handlebodies and $C'$ is the chamber  in $\calC_{\calD}$  that, per Definition \ref{defin:prefdisk}, is the unique remnant of $C$ that is occupied.  Then $h_C$ and $h_{C'}$ are eyeglass equivalent.
\item For any $\tau \in G(S^3, T)$ and $h_{\tau(C)}$ similarly defined, $h_{\tau(C)}\tau \sim h_C.$
\end{enumerate}
\end{prop}

\begin{proof}  We prove each statement in turn:
\begin{enumerate}
\item  Since $C$ is an occupied handlebody, $T_C$ is a non-trivial Heegaard splitting.  That is, $b = \genus(T_C) - \genus(\bdd C) \geq 1 $.  Following Lemma \ref{lemma:bubinH} there is a genus $b$ bubble $\frb$ in $T_C$ and, so long as $b < g$, $\frb$ is unique up to eyeglass moves.  

We have \[g = \genus(T) \geq \genus(T_C) = b + \genus(\bdd C)\] so indeed $b<g$ unless $C$ is a ball and $T_C$ is genus $g$.  But this is impossible, for  the complement of $C$ in $S^3$ would then be a trivially split ball, and so a goneball, so $\bdd C$ could not be a component of $F$.  We conclude that $\frb$ is unique up to eyeglass moves and take $\bdd \frb$ for the required reducing sphere.
\bigskip

\item This follows from Lemma \ref{lemma:preexistunique}.
\bigskip

\item This again follows from Lemma \ref{lemma:preexistunique}, since the handlebody chamber $C$ cannot be the same as the chamber with incompressible sphere.  
\bigskip

\item For the non-trivial Heegaard splitting of handlebody $C'$, let $\frb'$ be a bubble whose boundary defines $h_{C'}$, as we have just described.  According to Lemma \ref{lemma:doubbub}, $\frb'$ can be placed so that it is also a bubble for $T_C$.   According to Lemma \ref{lemma:bubinH} an eyeglass move will put $\frb'$ inside the bubble $\frb$ whose boundary defines $h_C$.  The result then again follows from Lemma \ref{lemma:preexistunique}.
%
\item This follows from the second statement in Proposition \ref{prop:toyexist}.  
\end{enumerate}
 \end{proof}

Motivated by Proposition \ref{prop:occupycertify}(1-3), we have

\begin{defin} \label{defin:flagcertify} A flagged chamber complex $\calC$ in $S^3$ {\em certifies} \index{Certify; certificate} if some chamber contains an incompressible sphere or some chamber is an occupied handlebody.  An incompressible sphere in such a chamber, or the boundary of a bubble in an occupied handlebody, is called {\em a certificate issued by $\calC$}.  
\end{defin}

\begin{cor} \label{cor:heegcertify} Let $\calC$ be a flagged chamber complex for $S^3$ that supports the genus $g$ Heegaard splitting $(S^3, T)$. If $\calC$ certifies then all homeomorphisms $(S^3, T) \to (S^3, T_g)$ determined by certificates in $\calC$ are eyeglass equivalent. 
\end{cor} 

Let $h_{\calC}: (S^3, T) \to (S^3, T_g)$ denote any homeomorphism given by such a certificate; by Corollary \ref{cor:heegcertify} $h_{\calC}$ is well defined up to eyeglass equivalence.  If $\calC'$ is another flagged chamber complex that certifies and the homeomorphisms $h_{\calC}$ and $h_{\calC'}$ are eyeglass equivalent, then we say that $\calC$ and $\calC'$ {\em cocertify} and write $\calC \sim \calC'$ . \index{cocertify}

\begin{cor}  \label{cor:calCcert} If $\calC$ certifies and $\tau \in G(S^3, T)$ then $h_{\tau(\calC)}\tau \sim h_{\calC}$. 
\end{cor}

\begin{proof} This follows from \ref{prop:occupycertify}(5).
\end{proof}

\bigskip

Suppose $S^3 = A \cup_T B$ is a genus $g$ Heegaard splitting of $S^3$, the Heegaard split chamber complex $\calC_0$ supports $T$, and we are given a sequence of 
flagged chamber complex decompositions
\[\vec{\calC}:\quad
\calC_0 \xrightarrow{\calD_0} \calC_1 \xrightarrow{\calD_1} \calC_2 \xrightarrow{\calD_2} ... \xrightarrow{\calD_{n-1}}\calC_n\]
Suppose further that both $\calC_0$ and $\calC_n$ certify.  

In analogy to Proposition \ref{prop:seqcertifyexist} we have.  

\begin{prop} \label{prop:flagcertifyexist}
For iteratively $0 \leq i \leq n$ there is a preferred alignment of each $\calD_i$ 
so that $\calC_0$ and $\calC_n$ cocertify.
\end{prop}

\begin{proof}  The proof is by induction on $n$.  
\medskip

{\em Initial Case: $n = 1$}

Following Proposition \ref{prop:prefexist}, take any preferred alignment of $\calD_0$ in $C_0$. 

One possibility is that in the flagged chamber complex decomposition $\calC_0 \xrightarrow{\calD_0} \calC_1$ there is an occupied handlebody chamber in $\calC_0$ for which every remnant in $\calC_1$ is a disky handlebody.  Then the result follows from Proposition \ref{prop:occupycertify}(4).  So assume that this is not the case.  Then by Proposition \ref{prop:flagprop}(2) a handlebody in $\calC_1$ is empty if and only if it is disky.  This implies that different choices of preferred alignment of $\calD_0$ 
have no effect on the flagging of $\calC_1$.  In particular, every disky ball in $\hat{\calC}_1$ is a goneball.  

Consider a chamber $C = A_C \cup_{T_C} B_C$ in $\calC_0$ that certifies.  There is a (reducing sphere) certificate $S \subset C$ for $T_C$ that is either incompressible in $C$ or, if $C$ is an occupied handlebody, bounds a bubble $\frb$ for $T_C$.  In either case there is a realignment of $\calD_0$ in $\calC$ that is disjoint from $S$, in the former case by the proof of Proposition \ref{prop:seqcertifyexist} and in the latter case by considering the Heegaard splitting $\calC_0/\frb$. 
We first need to show that this realignment of $\calD_0 \cap C$ can be done in a way that is still a preferred alignment and then show that the homeomorphisms $h_S$ and $h_{S'}$ determined by $S$ and some certificate $S'$ for $\calC_1$ are eyeglass equivalent.  

Consider the remnant $C'$ of $C$ in $\hat{\calC}_1$ that contains $S$. 
\medskip

{\em Subcase 1:} $C'$ is not a disky handlebody.  (In particular, $C'$ is not a goneball, and so $C'$ remains a chamber in $\calC_1$.)

For example, this is always the case if $S$ is incompressible in $C$ and so is either incompressible in $C'$ or bounds a ball containing components of $F$. 

The disk set $\calD_0 \cap C$ can be given a preferred alignment by emptying any remnant of $C$ that is a disky handlebody by passing bubbles to $C'$, see Definition \ref{defin:prefdisk}.  If $S$ is incompressible in $C'$ then it is a certificate for $\calC_1$ and we are done.   If there is a certificate $S'$ for $\calC_1$ in a chamber of $\calC_1$ other than $C'$, then $S$ and $S'$ are disjoint, so $h_S$ and $h_{S'}$ are eyeglass equivalent by Lemma \ref{lemma:preexistunique}.  This implies $h_{\calC_1} \sim h_{S'} \sim h_S \sim h_{\calC_0}$ and again we are done.

So we are reduced to the case where $C'$ is the only certifying chamber in $\calC_1$ and $S$ bounds a bubble $\frb$ in $C'$.  If $C'$ contains an incompressible sphere, then aligning it with the Heegaard surface for $C'$ in $\calC_1/\frb$ gives a certificate $S'$ for $\calC_1$ that is disjoint from $\frb$ and we are done as above.  If $C'$ does not contain an incompressible sphere then, since it contains a certificate for $\calC_1$, it must be an occupied handlebody in which $S'$ bounds a bubble.  By Lemma \ref{lemma:bubinH} the spheres $S$ and $S'$ can be made disjoint by an eyegelass move.  By Lemma \ref{lemma:preexistunique} this again implies $h_{\calC_1} \sim h_{S'} \sim h_S \sim h_{\calC_0}$ .  This concludes the argument in this subcase.
\medskip

{\em Subcase 2:}  $C'$ is a disky handlebody.

In this case, to obtain a preferred alignment of $\calD_0$ in $\calC$ we may need to empty $C'$, see Definition \ref{defin:prefdisk}.  Recall that this is done by finding a maximal bubble in $C'$, one that is also a bubble for $T_C$, and tube summing its boundary to a possible series of disks in $\calD_0$.  Lemma \ref{lemma:bubinH} shows that such a maximal bubble can be found that contains $\frb$ within it, so $\frb$ ends up in a remnant of $C$ that is not a disky handlebody.  Then Subcase 1 applies.  This concludes the argument for Subcase 2 and hence the Initial Case $n = 1$.  

\medskip

{\em The inductive step, $n \geq 2$:}

Inductively suppose the result is known for any sequence of length $\leq n-1$. Then we may as well assume that $\calC_i$ does not certify for $1 \leq i \leq n-1$, else we could break the sequence into two sequences of smaller length.  In particular, in none of the decompositions does an occupied handlebody have all remnants disky handlebodies, for when that occurs in a flagged chamber complex decomposition both the occupied handlebody and one of the remnants certify, see Definitions \ref{defin:flagdecomp} and \ref{defin:prefdisk}.  So, just as in the case $n = 1$, the alignment of each $\calD_i$ is preferred if and only if it has this property: each handlebody remnant is empty if and only if it is disky.  As a result, different choices of preferred alignment throughout the decomposition sequence have no effect on the flagging of the flagged chamber complexes.  

Since $\calC_1$ does not certify, no chamber  in $\calC_1$ contains an incompressible sphere. Also, no occupied handlebody in $\calC_0$ has only handlebody remnants in $\calC_1$, since at least one such remnant would certify.   There is a certificate for $\calC_0$ in some chamber $C$ of $\calC_0$.  As in the case $n = 1$ we can find such a certificate $S$ and realign $\calD_0 \cap C$ so that the sphere $S$ is disjoint from $\calD_0$.   Then, regardless of whether $S$ is an incompressible sphere or bounds a bubble in $C$, $S$ must compress in the chamber of $\calC_1$ that contains it, so it bounds a bubble $\frb$ in that chamber.

Following further the case $n = 1$, the disks $\calD_0 \cap C$ can be further realligned so that all disky handlebody remnants of $C$ are emptied, by bubble passes  to a remnant $C'$ of $C$ that is not a disky handlebody, and this can be done in a way that leaves the bubble $\frb$ intact.  The result is that the first decomposition again has a preferred alignment, but now with $\frb$ lying intact in the remnant $C' \in \calC_1$.  Finally note that $C'$ cannot be a handlebody at all: by construction it is not a disky handlebody and if it were not disky it would be occupied and so it would certify.  

Consider the flagged chamber complex $\calC_1/{\frb}$.  Since $\frb$ does not lie in a handlebody, there is no change in flagging.  Via Proposition 
 \ref{prop:prefexist} there is a preferred alignment for the disks $\calD_1$ so that  \[\calC_1/\frb \xrightarrow{\calD_1} \calC_2/\frb\]
 is a flagged chamber complex decomposition.  Since $C'$ is not a handlebody, not every remnant of $C'$ in $\calC_2$ is a disky handlebody and, unless $n = 2$, $\calC_2$ does not certify, so in fact not every remnant is a handlebody at all.  By isotoping $*$ in $\calC_1/\frb$ (or, equivalently, passing the bubble $\frb$ in $\calC_1$) we can ensure that $*$ lies in a remnant of $C'$ in $\calC_2/\frb$ that is not a handlebody.  Continue in this manner until $*$ reaches $\calC_n/\frb$ and lies in a chamber $C''$ that is not a disky handlebody (but may be an occupied handlebody).  The proof now concludes as it did for of $n = 1$, Subcase 1.
 \end{proof}
%

%
%

Just as Proposition \ref{prop:flagcertifyexist} for flagged decompositions parallels Proposition \ref{prop:seqcertifyexist}, we now work towards finding a parallel of Proposition \ref{prop:seqcertifyunique} for flagged decompositions.  The argument is mostly just a reprise of the previous argument.  We begin with an analogue of Lemma \ref{lemma:singlecertunique}.

\begin{lemma} \label{lemma:singleflagunique}  Suppose $\calC_0$ is a flagged chamber complex in $S^3$ that supports the splitting $S^3 = A \cup_T B$, and $$\calC_0 \xrightarrow{\calD_0} \calC_{1}$$ is a flagged chamber complex decomposition.  

Let $\calD^x_0$ and $\calD^y_0$ be possibly different preferred alignments of the disk set $\calD_0$ in $\calC_0$, with resulting flagged chamber complexes $\calC^x_1$ and $\calC^y_1$ respectively.
Then if $\calC_0$ and $\calC_1^x$ cocertify, so do $\calC_0$ and $\calC_1^y$.
\end{lemma}

\begin{proof}  If any occupied handlebody chamber $C$ in $\calC_0$ is decomposed entirely into disky handlebody chambers in $\calC_1$, then from Proposition \ref{prop:occupycertify}(4) $\calC_0$ and $\calC_1$ cocertify in any preferred alignment and we are done.  So we will assume there is no such decomposition of a chamber.  Then by Proposition \ref{prop:flagprop}(2) a handlebody in $\calC_1$ is empty if and only if it is disky.  This implies that $\calC_1^x$ and $\calC_1^y$ have the same flagging.

In the case that $\calC_1^x$ certifies because a chamber contains an incompressible sphere, the proof is the same as that of Lemma \ref{lemma:singlecertunique}, with this augmentation: Following Proposition \ref{prop:flagunique}, the various bubble passes that are needed can be done in such a way that at every stage the disk alignments are preferred and give the same flagging.  

In the case that $\calC_1^x$ certifies because a chamber is an occupied handlebody, the proof is similar.  From \cite{FS2} there is a sequence of bubble passes and eyeglass moves that change the alignment $\calD^x_0$ to $\calD^y_0$. From Proposition \ref{prop:flagunique} this can be done so that in the sequence the alignments are always preferred and the flagging never changes.  So it suffices to consider the case of a single bubble pass.  With no loss of generality then assume that $\calC_1^x$ and $\calC_1^y$ differ by a single bubble pass.  If the bubble pass does not involve an occupied handlebody, there is nothing to prove.  So assume that $\calC_1^y$ is obtained from $\calC_1^x$ by passing a bubble $\frb$ out of an occupied handlebody chamber $C'$ of $\calC_1$.  

Since the flagging remains the same after the bubble passes out of $C'$, $\frb$ is not a maximal bubble for $C'$ in $\calC_1^x$.  Following Lemma \ref{lemma:bubinH} there is a bubble $\frb^y$ in $C'$ that is disjoint from $\frb$ and the tube sum of the two bubbles $\frb^x$ is a maximal bubble for $C'$ in $\calC_1^x$.  In particular, $\frb^y$ is a maximal bubble for $C'$ in $\calC_1^y$.  Since the spheres $S^x = \bdd \frb^x$ and $S^y = \bdd \frb^y$ are disjoint, it follows from Lemma \ref{lemma:preexistunique}
 that the homeomorphisms $h_{S^x}, h_{S^y}: (S^3, T) \to (S^3, T_g)$ are eyeglass equivalent, so $\calC_1^x$ and $\calC_1^y$ cocertify, as required.
\end{proof}

Continue with the hypotheses of Proposition \ref{prop:flagcertifyexist}:

\begin{prop} \label{prop:flagcertifyunique} No matter which preferred alignment occurs at each step in $\vec{\calC}$, the flagged chamber complexes $\calC_0$ and $\calC_n$ cocertify.
\end{prop}

\begin{proof}  

The case $n = 1$ is Lemma \ref{lemma:singleflagunique}, so assume $n \geq 2$.  We can further inductively assume that none of the $\calC_i, 1 \leq i \leq n-1$ certify, so each handlebody chamber is empty.  In this case, the argument proceeds essentially as in Proposition \ref{prop:seqcertifyunique}: Proposition \ref{prop:seqbubblepass} and Lemma \ref{lemma:passcocert} remain true for flagged chamber complex decompositions, using Proposition \ref{prop:flagunique} to ensure that throughout the argument all decomposing disks are in preferred alignment. Then the proof of Proposition \ref{prop:seqcertifyunique} suffices if $\calC_n$ contains an incompressible sphere.  If instead $\calC_n$ certifies because it contains an occupied handlebody, the proof concludes as does the proof of Lemma \ref{lemma:singleflagunique}.
\end{proof}

Following Proposition \ref{prop:flagcertifyunique} it is possible to make the following definition:

\begin{defin} \label{defin:flagseqcertify} Suppose $(S^3, T)$ is a genus $g$ Heegaard splitting of $S^3$, and \[\vec{\calC}:\quad
\calC_0 \xrightarrow{\calD_0} \calC_1 \xrightarrow{\calD_1} \calC_2 \xrightarrow{\calD_2} ... \xrightarrow{\calD_{n-1}}\calC_n\] is a sequence of flagged
chamber complex decompositions supporting $T$.  If any chamber complex $\calC_i$ certifies we say that the sequence $\vec{\calC}$ certifies and let $h_{\vec{\calC}}: (S^3, T) \to (S^3, T_g)$ be (the eyeglass equivalence class of) any homeomorphism given by a certifying flagged chamber complex in the sequence.  

Two flagged chamber complex sequences $\vec{\calC}$ and $\vec{\calC}'$ cocertify (written $\vec{\calC} \sim \vec{\calC}'$) if both sequences certify and $h_{\vec{\calC}}  \sim h_{\vec{\calC}'}.$
\end{defin}

For example, suppose $\tau \in G(S^3, T)$ and $\vec{\calC}$ is a flagged chamber complex decomposition sequence supporting $T$ as above.  There is a natural way to define a similar flagged chamber complex decomposition sequence $\tau(\vec{\calC})$, namely
\[\tau(\vec{\calC}):\quad
\tau(\calC_0) \xrightarrow{\tau(\calD_0)} \tau(\calC_1) \xrightarrow{\tau(\calD_1)} \tau(\calC_2) \xrightarrow{\tau(\calD_2)} ... 
\xrightarrow{\tau(\calD_{n-1})} \tau(\calC_n)\] 

\begin{cor} \label{cor:vecCcert} If $\vec{\calC}$ certifies and $\tau \in G(S^3, T)$ then $h_{\tau(\vec{\calC})}\tau \sim h_{\vec{\calC}}$. 
\end{cor}

\begin{proof} This follows immediately from Corollary \ref{cor:calCcert}. 
\end{proof}

\section{Guiding spheres} \label{sect:guiding}

Suppose $\calC$ is a flagged chamber complex in $S^3$ with defining surface $F = F(\calC)$, and $S \subset S^3$ is an embedded sphere transverse to $F$.  In this and following sections we intend to associate to each such sphere $S$ a sequence $\overrightarrow{(\calC, S)}$ of flagged chamber complex decompositions so that:

\begin{enumerate}
\item For $\calC$ not tiny and $S_t, 0 \leq t \leq 1$ a sweep-out of $S^3$ by spheres, there is a non-empty interval $(a, b) \subset I$ so that, for each $t \in (a, b)$,  $\vec{\calC}_{S_t}$ certifies.
\item For any $t, t' \in (a, b)$ the sequences $\vec{\calC}_{S_t}$ and $\vec{\calC}_{S_{t'}}$ cocertify.
\item If, at any point $\calC_i \xrightarrow{\calD_i} \calC_{i+1}$ in the sequence $\overrightarrow{(\calC, S)}$, a disk $E$, disjoint from $S$, is added to $\calD_i$, the resulting sequence of flagged chamber complex decompositions and the sequence $\vec{\calC}_{S}$ cocertify.
\end{enumerate} 

\bigskip

Let $F = F(\calC)$ be the defining surface of $\calC$ and, in the standard way, associate to $S$ and $F$ a tree $Y$, in which each vertex corresponds to a component of $S - F$ and each edge corresponds to a (circle) component of $S \cap F$.  An edge associated to a circle $c \subset S - F$ connects the two vertices that correspond to the components of $S - F$ that are incident to $c$.  
Then the leaves of $Y$ correspond to the disk components of $S - F$, and the outermost edges in $Y$ correspond to innermost circles in $S$ of $F \cap S$.  See Figure \ref{fig:Tree1}.

\begin{figure}[ht!]
\labellist
\small\hair 2pt
\endlabellist
    \centering
    \includegraphics[scale=0.6]{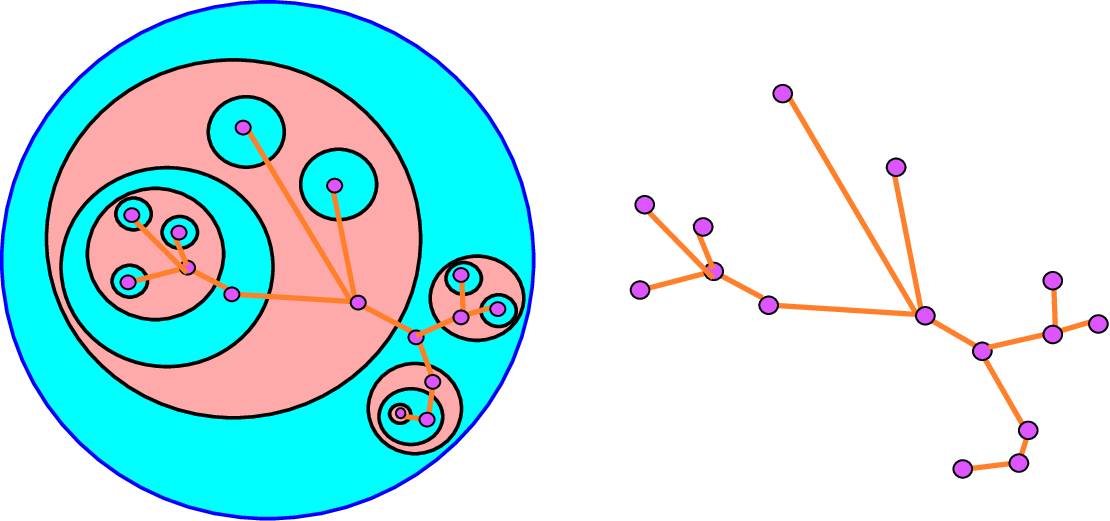}
     \caption{$S \cap F$ and its associated tree} \label{fig:Tree1}
    \end{figure}

\begin{defin} \label{defin:rhoe} For any edge $e \in Y$, with end vertices $v_{\pm}$, $Y - e$ consists of two trees $Y_{\pm}$, with $v_{\pm} \in Y_{\pm}$.  Let \[m_{\pm} = max\{d(v_{\pm}, v_\ell) | v_\ell \text{ a leaf in } Y \text{ lying in }Y_{\pm}\}\] and $\rho_Y(e) = min\{m_+, m_-\}$.  See Figure \ref{fig:Tree2}.
\end{defin}

\begin{figure}[ht!]
\labellist
\small\hair 2pt
\pinlabel  $e$ at 100 95
\pinlabel  $v_+$ at 130 70
\pinlabel  $v_-$ at 70 70
\pinlabel  $Y_-$ at 320 75
\pinlabel  $Y_+$ at 480 50
\endlabellist
    \centering
    \includegraphics[scale=0.6]{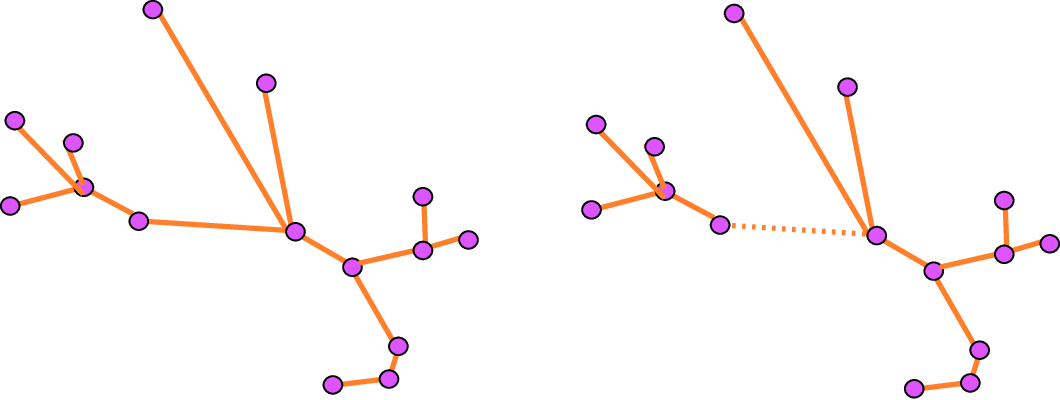}
     \caption{Edge $e$ and its trees $Y_{pm}$; $m_+ = 4, m=m_- = 2$ } \label{fig:Tree2}
    \end{figure}

For example, for $e$ an outermost edge, one of its incident vertices, $v_-$ say, is a leaf, so $Y_- = v_-$ and $\rho_Y(e) = m_- = 0$.   

\begin{lemma} \label{lemma:treetrim}  Suppose $e$ is an edge in a tree $Y$.
\begin{enumerate}
\item Suppose, with no loss of generality, that $\rho_Y(e) = m_- \leq m_+$ and $e'$ is an edge in $Y_-$.  Then $\rho_Y(e') < \rho_Y(e)$.  
\item Let $Y'$ be the tree obtained from $Y$ by trimming: that is, removing all leaves and outermost edges from $Y$.  For $e \in Y'$, $\rho_{Y'}(e) = \rho_Y(e) - 1$.
\item $0 \leq \rho_Y(e) \leq \lfloor \frac{\diamY - 1}{2} \rfloor$
\end{enumerate}
\end{lemma}

\begin{proof} We prove the items in turn:
\begin{enumerate}
\item Let $Y'_-$ be the component of $Y - e'$ that is contained in $Y_-$ and $v'_-$ the end of $e'$ in $Y'_-$.  For any leaf $v'_\ell$ of $Y'_-$ that lies in $Y$, the path from $v_-$ to $v'_\ell$ in $Y_-$ passes through $e'$, so \[\rho_Y(e) = m_- \geq d(v_-, v'_\ell) \geq d(v'_-, v'_\ell) + 1. \]
Since this is true for each leaf $v'_\ell$ of $Y'_-$, \[\rho_Y(e) \geq max\{d(v'_-, v'_\ell) | v'_\ell \text{ a leaf in } Y'_- \text{ lying in } Y\} + 1 \geq \rho_Y(e') +1 \]
\item Note that any leaf $v'_\ell$ in $Y'_+$ that lies in $Y'$ is necessarily incident to some outermost edges of $Y$, so $d(v_+, v'_\ell) \leq m_+ - 1$.  This implies that \[m'_+ = max\{d(v_+, v'_\ell) | v'_\ell \text{ a leaf in } Y'_+ \text{ lying in } Y'\} \leq m_+ - 1.\]  On the other hand, if $v_\ell$ is the most distant leaf from $v_+$ in $Y_+$ that lies in $Y$ then the outermost edge incident to $v_\ell$ is incident to a leaf $v'_\ell$ in $Y'_+$, so $m'_+ \geq m_+ -1$.  Together these imply $m'_+ = m_+ - 1$.  Similarly $m'_- = m_- - 1$.  Hence\[\rho_{Y'}(e) = min\{m'_+, m'_-\} = min\{m_+, m_- \} -1= \rho_Y(e) - 1,\] as required.

\item Repeat the process of trimming $\lfloor \frac{{\rm diam}(Y) - 1}{2} \rfloor$ times.  Each step reduces the diameter of the tree by 2 and, by conclusion (2), reduces the maximum of $\rho$ on any remaining edge by 1.  Thus it suffices to verify the case in which ${\rm diam}(Y) = 1$ or $2$, i. e. $Y$ is a single edge or a star graph, depending on the parity of ${\rm diam}(Y)$.  In both these cases, $\rho(e) = 0$ as required.  
\end{enumerate}
\end{proof}

Partition the collection of circles $S \cap F$ as follows: For $0 \leq i \leq \lfloor \frac{{\rm diam}(Y) - 1}{2} \rfloor$ let $\frc_i$ be the set of all circles for which the corresponding edge $e$ in $Y$ has $\rho_Y(e) = i$.  

\begin{lemma} \label{lemma:Didef}  Suppose $c \in \frc_i$.  Then $c$ bounds a disk $D$ in $S$ such that each circle $c' \subset \inter(D) \cap F$ lies in some $\frc_j, j<i$.
\end{lemma}

\begin{proof}  Let $Y$ be the tree corresponding to $S \cap F$ and apply Lemma \ref{lemma:treetrim}(1) to the edge in $Y$ corresponding to $c$.  The trees $Y_{\pm}$ are trees describing the circles $S \cap F$ lying in the complementary disks $D_{\pm}$ of $c$ in $S$, with $v_{\pm}$ corresponding to the components of $D_{\pm} - F$ that are adjacent to $c$.  With no loss of generality suppose $D_-$ corresponds to $Y_-$ with $i = \rho_Y(e) = m_-$.  Then  Lemma \ref{lemma:treetrim}(1) says that any component $c'$ of $\inter(D_-) \cap F$ has corresponding edge $e'$ in $Y'$ with $\rho_Y(e') < \rho_Y(e) = i$.  Setting $j =  \rho_Y(e')$ we have $c' \in \frc_j$.  
\end{proof}

Each circle in $\frc_0$ corresponds to a leaf in $Y$, so each $c \in \frc_0$ bounds a disk in $S - F$.  Let $\calD_0$ be the collection of disks, viewed as a disk set in $\calC$.  Choose a preferred alignment for $\calD_0$ and consider the flagged chamber complex decomposition $\calC \xrightarrow{\calD_0} \calC_1$.  Before the decomposition each circle in $\frc_1$ bounded a disk in $S$ that intersected $F$ only in components of $\frc_0$, by Lemma \ref{lemma:Didef}.  So after the disk decomposition along $\calD_0$, which surgers away all components of $\frc_0$, each circle in $\frc_1$ bounds a disk in $S$ that is disjoint from the defining surface $F_1 = F(\calC_1)$.  More correctly, this is true of each {\em remaining} circle in $\frc_1$, for some circles in $\frc_1$ may lie on the boundary of goneballs of the decomposition by $\calD_0$, and so not still appear in $F_1 \cap S$.  In any case, the collection of disks in $S - F_1$ bounded by (remaining) circles in $\frc_1$ will be denoted $\calD_1$ and a choice of preferred alignment defines a flagged chamber complex decomposition $\calC_1 \xrightarrow{\calD_1} \calC_2$.  Continue in this manner through $\calC_{n - 1} \xrightarrow{\calD_{n - 1}} \calC_n$, where $n-1 = \lfloor \frac{{\rm diam}(Y) - 1}{2} \rfloor$ or $n = \lfloor \frac{{\rm diam}(Y) + 1}{2} \rfloor$.  In the end, $F_n = F(\calC_n)$ is then disjoint from $S$.  

There are potentially two sources of ambiguity in this construction:

One ambiguity is this: in the last decomposition in the sequence, $\calC_{n - 1} \xrightarrow{\calD_{n - 1}} \calC_n$, it's possible that the defining surface $F_{n-1} = F(\calC_{n-1})$ intersects $S$ in a single circle, so $|\frc_{n-1}| = 1$.  This circle divides $S$ into two disks, one lying in the incident $A$-chamber of $\calC_{n - 1}$ and the other lying in the incident $B$-chamber.  The description above would allow either of these disks to be used as the disk set $\calD_{n-1}$ in the last decomposition.  Will it make a difference which one we use?  This issue will be addressed later - see Lemma \ref{lemma:ambig}.

The second potential source of ambiguity is that {\em prima facie} the construction above depends at each stage on the choice of preferred alignment of the disks $\calD_i$.  So the process is more accurately described as defining a {\em family} of decomposition sequences, which may differ at several points in the choice of preferred alignment.   We will now show that if any sequence in the family certifies, then they all cocertify.  

\begin{defin} \label{defin:Didef}  Given $\calC$ a flagged chamber complex in $S^3$, $S$ a sphere transverse to $F(\calC)$, and $Y$ the tree associated to the circles $F \cap S$ in $S$, a sequence of flagged chamber complex decompositions 
\[\vec{\calC}:\quad
\calC = \calC_0 \xrightarrow{\calD_0} \calC_1 \xrightarrow{\calD_1} \calC_2 \xrightarrow{\calD_2} ... \xrightarrow{\calD_{k-1}}\calC_k\]
as constructed above is said to be {\em guided by} the sphere $S$.  If $k = n = \lfloor  \frac{{\rm diam}(Y) + 1}{2} \rfloor$ the sequence is {\em complete}.  

The collection of all complete sequences will be denoted $\overrightarrow{(\calC, S)}$.
\end{defin}

Continue with Assumption \ref{ass:inductive}, that $G(S^3, T') = \calE$ whenever $\genus(T') \leq g-1$.

\begin{prop} \label{prop:guidecertify} If $\overrightarrow{(\calC, S)}$ contains more than one complete sequence of flagged chamber complex decompositions then each sequence in $\overrightarrow{(\calC, S)}$ certifies and all sequences cocertify.
\end{prop}

\begin{proof} Suppose  \[\vec{\calC}: \calC_0 \xrightarrow{\calD_0} \calC_1 \xrightarrow{\calD_1} ... \xrightarrow{\calD_{n-1}}\calC_n\] and
\[\vec{\calC'}: \calC_0 \xrightarrow{\calD_0} \calC'_1 \xrightarrow{\calD'_1} ... \xrightarrow{\calD'_{n-1}}\calC'_n\] are any two sequences in $\overrightarrow{(\calC, S)}$.
and $i \geq 1$ is the smallest value for which $\calC_i \neq \calC'_i$.   Then $\calC_{i-1} = \calC_{i-1}$ and $\calD_{i-1} = \calD'_{i-1}$ are the same as disk sets but are given different preferred alignments.  In particular, by Proposition \ref{prop:flagprop}(3),  the decompositions $\calC_{i-1} \xrightarrow{\calD_{i-1}} \calC_i$ and $\calC_{i-1} \xrightarrow{\calD'_{i-1}} \calC'_i$ contain sibling decompositions.  That is, there is at least one parent chamber $C$ in $\calC_{i-1}$ so that $C$ is an occupied handlebody and $\calD_{i-1}$ is given a different preferred alignment in $C$ for the two decompositions.  Then $C$ certifies for both $\vec{\calC}$ and $\vec{\calC'}$
\end{proof}

If $\vec{\calC} \in \overrightarrow{(\calC, S)}$ certifies, denote  a homeomorphism $h_{\vec{\calC}}: (S^3, T) \to (S^3, T_g)$ (see Definition \ref{defin:flagseqcertify}) by $h_{(\calC, S)}$.  By Proposition \ref{prop:guidecertify} $h_{(\calC, S)}$ is well-defined up to eyeglass equivalence; that is, it does not depend on the choice of decomposition sequence $\vec{\calC}\in \overrightarrow{(\calC, S)}$.  

\begin{cor} \label{cor:vScert} If $\vec{\calC} \in \overrightarrow{(\calC, S)}$ certifies and $\tau \in G(S^3, T)$ then $h_{(\tau(\calC), \tau(S))}\tau \sim h_{(\calC, S)}$. 
\end{cor}

\begin{proof} This follows immediately from Corollary \ref{cor:vecCcert}. 
\end{proof}

\section{Balanced and almost balanced spheres} \label{sect:balance}

\begin{defin} Suppose $S \subset S^3$ is a sphere transverse to $F = F(\calC)$ in a flagged chamber complex $\calC \subset S^3$, with $X$, $Y$ the complementary components of $S$.

Then $S$ is {\em balanced} \index{Balanced sphere} for $\calC$ if the compact surfaces $F \cap X$ and $F \cap Y$ are either both planar (planar balanced) or both non-planar (non-planar balanced). 

$S$ is {\em almost balanced} \index{Almost balanced sphere} if $\genus(F \cap X) = 0$ and $\genus(F \cap Y) = 1$ or vice versa.
\end{defin}

\begin{prop}  \label{prop:balance}. Suppose
\[\vec{\calC}: \quad \calC = \calC_0 \xrightarrow{\calD_0} \calC_1 \xrightarrow{\calD_1} ... \xrightarrow{\calD_{n-1}}\calC_n\]
is a complete flagged chamber complex decomposition sequence guided by $S \subset S^3$.  That is, $\vec{\calC} \in \overrightarrow{(\calC, S)}$.   Then
\begin{itemize}
\item If $S$ is planar (resp non-planar) balanced in any chamber complex in the sequence, then it is in every chamber complex in the sequence.  In this case, call $S$ planar (resp non-planar) balanced for the sequence.
\item If $S$ is balanced for the sequence and $\calC$ is not tiny,  then $\calC_n$ certifies, so $\vec{\calC}$ certifies.  
\end{itemize}
\end{prop}

\begin{proof}  For the first statement, note that the effect on $F \cap X$, say, of decomposing along $\calD_i$ is 2-fold:
\begin{enumerate}
\item Cap off some of boundary components of $F \cap X$ with disks.
\item Delete some spheres from $F$ (the boundaries of the goneballs). 
\end{enumerate}
Neither of these steps affects the genus of $F \cap X$ or $F \cap Y$.  

For the second statement, note that since $\calC$ is not tiny, $\calC_n$ is not empty (Corollary \ref{cor:tiny}).  Let $F = F(\calC_n) \neq \emptyset$ be the defining surface for $\calC_n$. When the sequence is non-planar balanced, there are (non-planar) closed components of $F$ in both $X$ and $Y$.  This implies that $S$ is a reducing sphere for the chamber of $C_n$ in which it lies, so $S$ itself is a certificate.  If the sequence is planar balanced, then each component of $F$ is a closed planar surface, i. e. a sphere.  Since there are no empty balls in a flagged chamber complex, an innermost sphere in $F$ bounds an occupied ball, which certifies (see Definition \ref{defin:flagcertify}).
\end{proof}

Similarly we have:

\begin{prop} \label{prop:almostbalance} Suppose
\[\vec{\calC}: \quad \calC = \calC_0 \xrightarrow{\calD_0} \calC_1 \xrightarrow{\calD_1} ... \xrightarrow{\calD_{n-1}}\calC_n\]
is a complete flagged chamber complex decomposition sequence guided by $S \subset S^3$.  That is, $\vec{\calC} \in \overrightarrow{(\calC, S)}$.   Then
\begin{itemize}
\item If $S$ is almost balanced in any chamber complex in the sequence, then it is in every chamber complex in the sequence.  In this case, call $S$ almost balanced for the sequence.
\item If $S$ is almost balanced for the sequence and $\calC$ is not tiny,  then $\calC_n$ certifies, so $\vec{\calC}$ certifies.  
\end{itemize}
\end{prop}

\begin{proof} The proof of the first statement is unchanged.  

For the proof of the second, observe that since the sequence is complete, $F_n = F(\calC_n)$ is disjoint from $S$, so both $F_n \cap X$ and $F_n \cap Y$ are closed surfaces.  Since $F_n$ is almost balanced, $F_n \cap X$ has genus $0$ and $F_n \cap Y$ has genus $1$, or vice versa.  In any case $F_n$ is the union of a torus and some spheres.  If there are spheres, at least one must bound an occupied ball and so certifies.  If $F$ is just a torus it bounds a solid torus $Y$ in $S^3$, or perhaps complementary solid tori $Y_{\pm}$.  Since $\calC$ is not tiny, $\calC_n$ is not tiny (Corollary \ref{cor:tiny}), so $Y$ cannot be empty nor, in the second case, can both solid tori $Y_{\pm}$ be empty.  Hence one of these is an occupied solid torus and so is a certificate for $\calC_n$.  
\end{proof}

We now briefly return to what was described, preceding Definition \ref{defin:Didef}, as the first source of ambiguity in the construction of a sequence of flagged chamber complex decompositions guided by $S$.  Suppose $\calC$ is a flagged chamber complex supporting the genus $g$ splitting $S^3 = A \cup_T B$ and $S$ is a sphere in $S^3$ intersecting $\calC$ in a single circle $c$ and dividing $S^3$ into two $3$-balls, $X$ and $Y$.  The circle $c$ divides $S$ into two disks $D_A$ lying in an $A$-chamber $C_A$ of $\calC$ and $D_B$ lying in a $B$-chamber $C_B$.  

\begin{lemma} \label{lemma:ambig} Let $\calC_A$ and $\calC_B$ be the flagged chamber complexes that result from the decompositions $\calC \xrightarrow{D_A} \calC_A$ and $\calC \xrightarrow{D_B} \calC_B$ respectively.  Suppose $S$ is balanced or almost balanced for $\calC$ so both $\calC_A$ and $\calC_B$ certify.  Then $\calC_A$ and $\calC_B$ cocertify.
\end{lemma}

\begin{proof} 
Let $c$ be the circle described before the statement of the lemma.  

{\em Case 1:}  $c$ is essential in $F = F(\calC)$.

Then $c$ divides $F$ into two non-planar components and $F$ is non-planar balanced.  We show that our rule for certifying a non-planar balanced chamber complex cocertifies $\calC_A$ and $\calC_B$:

As usual, align $\calD_A$ in $\calC_A$ and $\calD_B$ in $\calC_B$, via an appropriate choice of spine for the compression bodies $B_{\calC_A}$ in $\calC_A$ and $A_{\calC_B}$ in $\calC_B$.  Since $c$ is essential in $F$ it is essential in the Heegaard surface $T$ of $S^3$ that the chamber complex supports.  Hence $S$ is a reducing sphere for $T$.  In particular there is a homeomorphism $h: (S^3, T) \to (S^3, T_g)$ as described before Definition \ref{defin:eyeequiv} in which $h(S) = S_i$ for some $1 \leq i \leq g-1$.  But this coincides with the definition of the homeomorphism $h_S: (S^3, T) \to (S^3, T_g)$ in Proposition \ref{prop:toyexist}, and that is the homeomorphism which certifies $S$ in both $\calC_A$ and $\calC_B$.  So $\calC_A$ and $\calC_B$ cocertify in this case.
\medskip

{\em Case 2:}  $c$ bounds a disk $D_F \subset F \cap X$ but does not bound a disk in $F \cap Y$ (or vice versa).  

Let $S_A$ be the sphere $D_A \cup D_F$, a component of $F(\hat{\calC}_A)$ and, symmetrically, $S_B$ be the sphere $D_B \cup D_F$, a component of $F(\hat{\calC}_B)$.  One possibility is that $S_A$ is essential in $\calC_A \cap X$, perhaps because the ball $B_{X \cap A}$ in $X$ bounded by $S_A$ contains other components of $F(\calC_A)$ or because it is occupied, and so not a goneball in $\hat{\calC}_A$.  Then $S_A$, pushed slightly out of $B_{X \cap A}$, is also essential in $\calC_A$ because it separates $B_{X \cap A}$ from the (non-spherical) component of $F(\calC_A)$ that contains the circle $c$.  Thus it certifies for $\calC_A$.  See Figure \ref{fig:ambig}.

\begin{figure}[ht!]
\labellist
\small\hair 2pt
\pinlabel  $F$ at 130 15
\pinlabel  $A$ at 60 15
\pinlabel  $B$ at 200 15
\pinlabel  $D_B\subset S$ at 240 120
\pinlabel  $D_A\subset S$ at 5 120
\pinlabel  $c$ at 145 120
\pinlabel  $D_F\subset F\cap X$ at 240 200
\pinlabel  $S_A$ at -5 205
\pinlabel  $X$ at 185 85
\pinlabel  $Y$ at 210 70
\endlabellist
    \centering
    \includegraphics[scale=0.7]{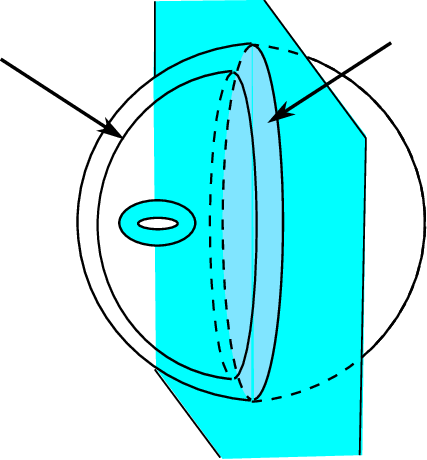}
     \caption{Case 2: $c$ bounds a disk in $F \cap X$ but not in $F \cap Y$} \label{fig:ambig}
    \end{figure}

 If $S_B$ similarly certifies for $\calC_B$ then $S_A$ and $S_B$ define disjoint reducing spheres for $T$ and the result follows from Lemma \ref{lemma:preexistunique}.  
On the other hand, if $S_B$ does not certify for $\calC_B$ then $B_{X \cap B}$ is a goneball.  In this case, decomposing by $D_B$ has no effect, that is $\calC_B$ and $\calC$ are isotopic, so $\calC$ also certifies.  By Proposition \ref{prop:flagcertifyunique}
, $\calC_A \sim \calC \sim \calC_B$ as required.  

Finally, if $S_A$ is inessential in $\calC_A \cap X$ and $S_B$ is inessential in $\calC_B \cap X$, so $B_{X \cap B}$ and $B_{X \cap A}$ are goneballs, then $\calC_A$ and $\calC_B$ are isotopic to $\calC$, so $\calC_A$ and $\calC_B$ again cocertify.  
\medskip

{\em Case 3:}  $c$ lies on a sphere component $S_F$ of $F$, dividing it into disks $D_X \subset X$ and $D_Y \subset Y$.  

Since $S_F$ does not bound a goneball in $\calC$, it either divides $S^3$ into two occupied balls or is parallel to an incompressible sphere in an adjoining chamber of $\calC$.  Thus $S_F$ is a certificate for $\calC$.  Then, per Proposition \ref{prop:flagcertifyunique} 
, $\calC_A \sim \calC \sim \calC_B$ as required.  
\end{proof}

This dispenses with the ``first ambiguity'' preceding Definition \ref{defin:Didef}.
\bigskip

Suppose $\calC \subset S^3$ is a flagged chamber complex, with $F = F(\calC)$ in general position with respect to the height projection $p: S^3 \to [-1, 1]$.  Recall that $S^3 - \{poles\}$ is swept out by the family of level spheres $S_s$, where $S_s = p^{-1}(s), s \in (-1, 1)$.   We will consider these level spheres as potential guiding spheres for decomposition of $\calC$.  

\begin{defin} \label{defin:gaandgb} Let $F = F(\calC)$, $\mathfrak{g} = \genus(F)$.  

For each $s \in (-1, 1)$ that is a regular value of $p|F$, let $F_{a(bove)}(s)$ be the part of $F$ lying above $S_s$, that is $F_a(s) =  F \cap p^{-1}([s, 1])$. Similarly, let $F_{b(elow)}(s)$ be the part of $F$ lying below $S_s$, that is $F_b(s) = F \cap p^{-1}([-1, s])$.  Finally, let  $g_a(s) = \genus(F_a(s))$ and $g_b(s) = \genus(F_b(s))$.

\end{defin}

\begin{lemma}  \label{lemma:gaandgb}
The functions $g_a, g_b$ have these properties:
\begin{enumerate}
\item As $s$ ascends from $-1$ to $1$, $g_b$ ascends from $0$ to $\mathfrak{g}$ and $g_a$ descends from $\mathfrak{g}$ to $0$.
\item As $s$ ascends through a critical point of $p|F: F \to [-1, 1]$, $g_b(s)$ may increase by $1$ but it will not  decrease.  Symmetrically, $g_a(s)$ may decrease by $1$ but it will not increase.
\item As $S_s$ ascends through a critical point of $p|F$, at most one of $g_a(s), g_b(s)$ will change.
\end{enumerate}
\end{lemma}

\begin{proof}  The first statement follows from the definition.

For the second and third statements, observe that passing through a critical point on $F$ of index 0 (a minimum) or 2 (a maximum) simply adds or subtracts a disk from $F_a$ and $F_b$, which has no effect on genus.  So the interest is in index 1 critical points, that is saddle points of tangency. 

Suppose, with little loss in generality, that as $s$ ascends through a saddle tangency, two circles in $F \cap S_s$ fuse into one. See Figure \ref{fig:saddle}. (If instead one circle is divided in two, just turn the following argument upside down.)  Thus a band is added to $F_b$ with its ends on two different boundary circles.  If the two circles lie on different components of $F_b$, the genus of $F_b$ does not change.  If the two lie on the same component, $g_b$ ascends by $1$.  This proves the second statement.  For the third statement, consider $F_a$: the effect on $F_a$ is to cut out a band between two boundary components, effectively replacing a pair of pants in $F_a$ with an annulus.  This has no effect on $g_a$.  This proves the third statement.
\end{proof}

\begin{figure}[ht!]
\labellist
\small\hair 2pt
\pinlabel  $g_a=2$ at 200 250
\pinlabel  $F_a(s)$ at 40 250
\pinlabel  $F_b(s)$ at 40 50
\pinlabel  $F_a(s+\epsilon)$ at 480 250
\pinlabel  $F_b(s+\epsilon)$ at 480 50
\pinlabel  $g_a=2$ at 300 250
\pinlabel  $S_s$ at 200 160
\pinlabel  $S_{s+\epsilon}$ at 300 190
\pinlabel  $g_b=1$ at 200 50
\pinlabel  $g_b=2$ at 300 50
\endlabellist
    \centering
    \includegraphics[scale=0.6]{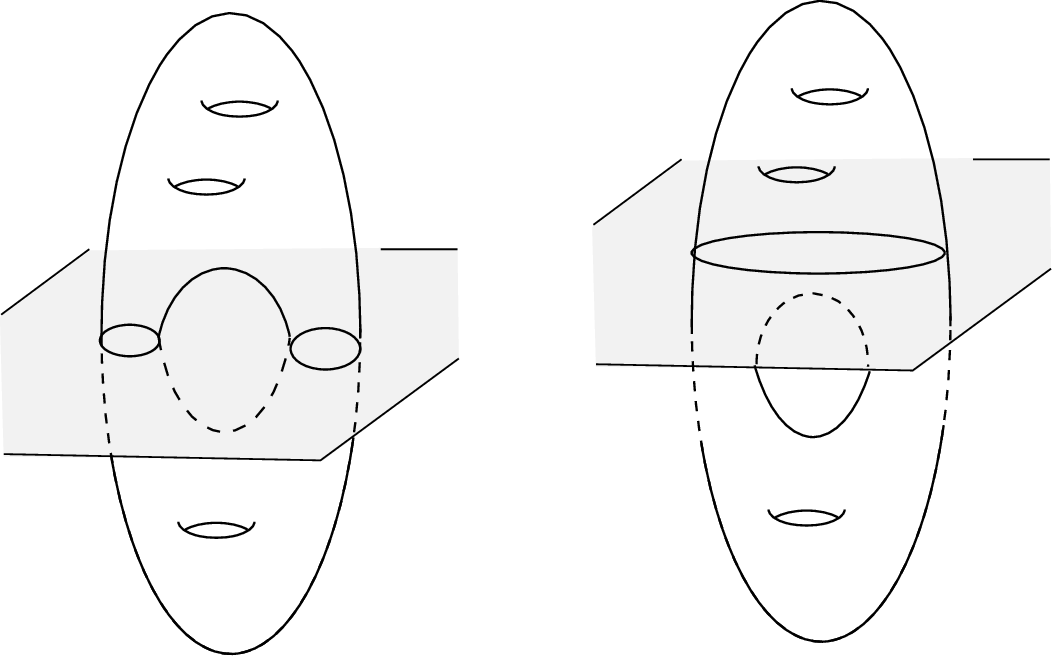}
     \caption{$s$ ascends through a saddle and two circles fuse} \label{fig:saddle}
    \end{figure}

\begin{cor} \label{cor:balinterval} Let
\[s_b = \sup\{s\in[0,1] | g_b(s) = 0\}\hspace{1cm} s_a = \inf\{s\in[0,1] | g_a(s) = 0\}.\]  Then
\begin{enumerate}
\item $s_a \neq s_b$.
\item If $s_a < s_b$ then for all $s_a < s < s_b$, $S_s$ is planar balanced.
\item If $s_b < s_a$ then for all $s_b < s < s_a$, $S_s$ is non-planar balanced.
\item If $s_a < s_b$ then for $s$ just below $s_a$ and $s$ just above $s_b$, $S_s$ is almost balanced.
\item If $s_b < s_a$ and $p|F$ has no index 1 critical points for  $s \in (s_b, s_a)$  then for $s$ just below $s_b$ and $s$ just above $s_a$, $S_s$ is almost balanced.
\end{enumerate}
\end{cor}

\begin{proof} All but the last two statements follow directly from the definitions and Lemma \ref{lemma:gaandgb}; the last two require little more:

It follows from Lemma \ref{lemma:gaandgb}(2) that for $s$ just above $s_b$, $g_b(s) = 1$ and, if $s_a < s_b$, $s \in \{s|g_s(a)=0\}$ so also $g_a(s) = 0$. So $S_s$ is almost balanced.  A symmetric argument shows that it is also almost balanced for $s$ just below $s_a$.  This proves (4).  

Suppose $s_b < s_a$.  Then for $s$ just above $s_b$, $g_b(s) = 1$ and $g_a(s) \geq 1$; similarly for $s$ just below $s_a$, $g_a(s) = 1$ and $g_b(s) \geq 1$.  Since $p|F$ has no index one critical points in $(s_b, s_a)$, neither $g_a$ or $g_b$ changes in that interval.  Hence for any $s \in (s_b, s_a)$, $g_a(s) = 1 = g_b(s)$.  It follows from Lemma \ref{lemma:gaandgb}(3) that for $s$ just above $s_a$, $g_a(s) = 0$ and $g_b(s) = 1$ so $S_s$ is almost balanced.  The symmetric argument shows that for $s$ just below $s_b$, $S_s$ is almost balanced.  This proves (5).
\end{proof}

Suppose that the chamber complex $\calC$ is not tiny.  Then for any $s$ between $s_a$ and $s_b$ it follows from Corollary \ref{cor:balinterval} and Propositions \ref{prop:balance} and  \ref{prop:guidecertify} 
that the homeomorphism \[h_s = h_{(\calC, S_s)}:(S^3, T) \to (S^3, T_g)\] is well-defined up to eyeglass equivalence.  What remains unclear is whether the eyeglass equivalence class of $h_s$ depends on $s$.  The central issue is whether the eyeglass equivalence class can change as $s$ passes through a critical point of $p|F$.
We will show (Theorem \ref{thm:balanceisotopy}) that in fact the eyeglass equivalence class does not change.  The proof is technically complex; it occupies the next five sections.  

\section{Deflation and bullseyes}  \label{sect:deflate}

Suppose $\calC \subset S^3$ is a flagged chamber complex supporting the genus $g$ Heegaard splitting $S^3 = A \cup_T B$.   We continue under the inductive Assumption \ref{ass:inductive}.

Let $H$ be an occupied handlebody chamber of $\calC$ that is not a ball.  In the following description we assume that $H$ is an $A$-chamber, but everything can also be taken symmetrically, see Definition \ref{defin:hbydeflate}.  
Let $C$ be the $B$-chamber adjacent to $H$ and $C_A$ be an $A$-chamber adjacent to $C$ so that $C_A \neq H$.  Since the associated Heegaard splitting of $C = A_C \cup_{T_C} B_C$ is pure, there is a spine for the compression body $A_C$ in which an edge $e$ is incident to $\bdd H$ on one end and $\bdd C_A$ on the other.  


Let $(D_e, \bdd D_e) \subset (A_C, T_C)$ be the disk dual to $e$ in $A_C$, that is a meridian of  the tube in $A_C$ that corresponds to a regular neighborhood of the edge $e$, so the tube 
runs from $H$ to $C_A$.  Alter $D_e$ by tube-summing with a maximal bubble $\frb$ in $H$ and let $\calC_d$ be the resulting flagged chamber complex. (More precisely, after amalgamation of $\calC$ to the original splitting, $D_e$ becomes a disk in $A$ and $\frb$ a bubble for $T$; it is these that are tube-summed to create the disk $D'_e$ that replaces $D_e$ in $\calC_d$.)  In effect, the bubble $\frb$ has been moved from $H$ to the chamber $C_A$.  See Figure \ref{fig:deflate1}.   
Since $\frb$ is maximal in $H$, $H$ changes from an {\em occupied} handlebody in $\calC$ to an {\em empty} handlebody in $\calC_d$.   It is natural to call this process a {\em deflation} of $H$.  That is, the flagged chamber complex $\calC_d$ is obtained from $\calC$ by deflation (of $H$ to $C_A$) \index{Deflation}.  Note that the process is well-defined up to eyeglass equivalence by a choice of $H$ and $C_A$, since a bubble move is an eyeglass move. 

\begin{figure}[ht!]
\labellist
\small\hair 2pt
\pinlabel  $e$ at 90 390
\pinlabel  $H$ at 190 400
\pinlabel  $\frb$ at 140 385
\pinlabel  $C$ at 90 340
\pinlabel  $C$ at 90 170
\pinlabel  $\calC$ at 90 270
\pinlabel  $\calC_d$ at 90 220
\pinlabel  $C_A$ at 20 390
\pinlabel  $D_e$ at 420 380
\pinlabel  $D'_e$ at 450 98
\endlabellist
    \centering
    \includegraphics[scale=0.5]{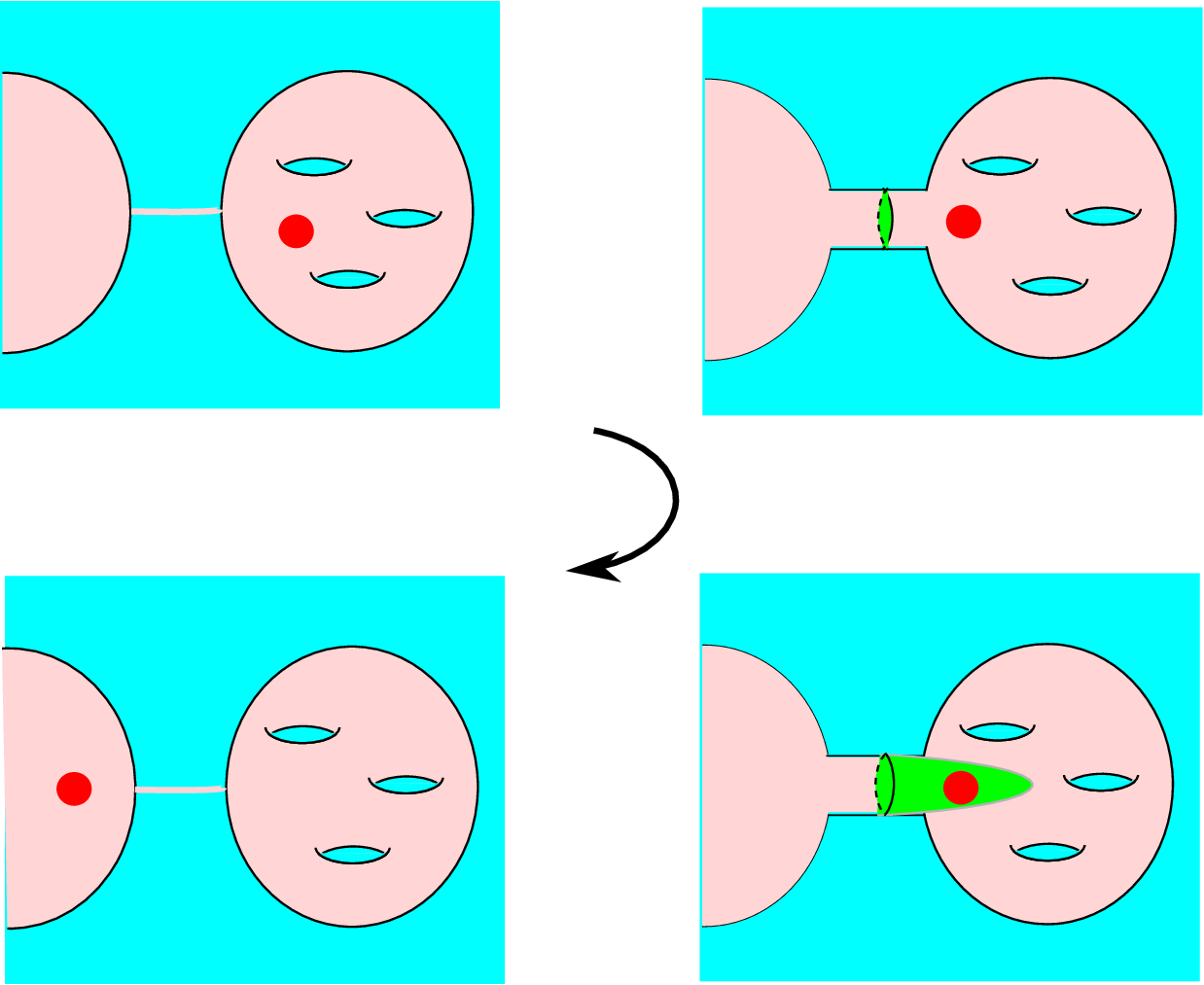}
     \caption{Simple deflation} \label{fig:deflate1}
    \end{figure}

More generally, the process can be applied to a collection of edges in a spine of $A_C$, each of which is incident to $\bdd H$ and an $A$-chamber adjacent to $C$, but not necessarily the same $A$-chamber for all the edges.  That is, redefine the meridian $A$-disks dual in $A_C$ to the collection of edges by tube-summing them with disjoint non-trivial bubbles in $H$, bubbles which together constitute a maximal bubble in $H$.  Generalizing further, bubbles from $H$ may be sent to further $A$-chambers $\{C_{A_i}\}$ throughout $\calC$, not just $A$-chambers adjacent to $C$, by further tube-summing to $A$-disks in $B$-chambers. See Figure \ref{fig:deflate2}. The only requirement is that in the resulting flagged chamber complex $\calC_d$, the chamber $H$ is empty.  In this general form we have:

\begin{figure}[ht!]
\labellist
\small\hair 2pt
\pinlabel  $H$ at 190 220
\pinlabel  $C_{A_1}$ at 490 380
\pinlabel  $C_{A_3}$ at 490 25
\pinlabel  $C_{A_2}$ at 340 200
\pinlabel  $\calC$ at 260 320
\pinlabel  $\calC_d$ at 320 320
\endlabellist
    \centering
    \includegraphics[scale=0.5]{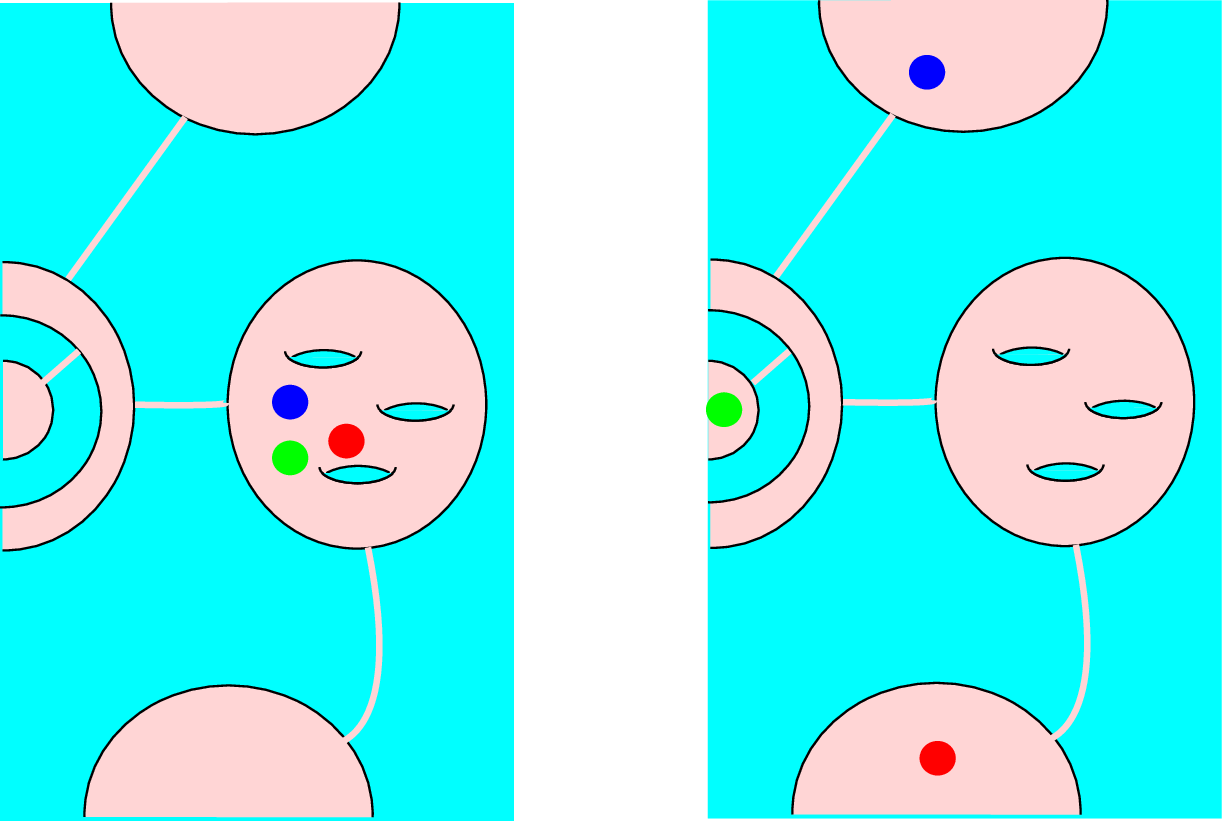}
     \caption{General deflation} \label{fig:deflate2}
    \end{figure}

\begin{defin} \label{defin:hbydeflate}  The flagged chamber complex $\calC_d$ is obtained from $\calC$ by {\em handlebody deflation} (of $H$ to the collection of chambers $\{C_{A_i}\}$).  Deflation of an occupied handlebody $B$-chamber of $\calC$ that is not a ball is defined symmetrically.
\end{defin}

As is true for $C_A$ a single chamber, the change of flagged chamber complex from $\calC$ to $\calC_d$ is well-defined up to eyeglass equivalence by a choice of $H$, the collection $\{C_{A_i}\}$, and the genus of the bubbles moved into each.  The effect of deflation on the flagging of $\calC$ is to change $H$ to an empty handlebody chamber in $\calC_d$ (hence the term deflation) and, for those chambers in $\{C_{A_i}\}$ that are empty handlebodies, to change the flagging to occupied.  

\begin{prop} \label{prop:hbydeflate} Suppose the flagged chamber complex $\calC_d$ is obtained from $\calC$ by deflating an occupied non-ball handlebody $H$.  If $\calC_d$ certifies then $\calC_d$ and $\calC$ cocertify.
\end{prop}

\begin{proof}  The occupied handlebody $H$ is a certificate for $\calC$ so $\calC$ certifies. With no loss assume, as above, that $H$ is an $A$-chamber deflated into the set of $A$-chambers $\{C_{A_i}\}$.  Suppose $\calC_d$ certifies and let $C_d$ be a certifying chamber of $\calC_d$. Since $H$ is empty in $\calC_d$, $C_d \neq H$.  The only other chambers whose flagging is affected by the deflation of $H$ are $\{C_{A_i}\}$ so, unless $C_d \in \{C_{A_i}\}$, the chamber $C_d$ also certifies for $\calC$, as required.  So suppose $C_d \in \{C_{A_i}\}$, i. e. after the bubble pass one or more of the chambers in $\{C_{A_i}\}$, call it $C_a$, certifies for $\calC_d$.  

The only change in $C_a$ from $\calC$ to $\calC_d$ is that a bubble $\frb$ from $H$ has been passed into it.  Since the chamber $C_a$ certifies in $\calC_d$, it is either an occupied handlebody or it contains an incompressible sphere.  If $\calC_a$ is an occupied handlebody then, per Proposition \ref{prop:occupycertify}(1), $\bdd \frb \subset \calC_a$ is a certificate for $\calC_d$ as well as $\calC$ completing the proof.  Suppose, on the other hand, that the chamber $C_a$ contains an incompressible sphere $S$.  By \cite{Sc1} $S$ can be aligned with the Heegaard surface $T_a$ for $C_a$ in $\calC$, so $S$ becomes a certificate for $\calC$.  In constructing $\calC_d$ pass the bubble $\frb$ into the chamber $C_a$ so that $\frb$ remains disjoint from the aligned $S$.  Then the sphere $S$ certifies for both $\calC$ and $\calC_d$, as required.
%
\end{proof}

The same construction can be done when $H$ is an occupied ball, but an important difference is that once the occupied ball is deflated, so $H$ has a genus 0 splitting, $H$ must be declared a goneball, else $\calC_d$ is not a flagged chamber complex.  That is, the sphere $\bdd H$ is removed from the defining surface $F$ and $H$ is absorbed into its neighboring chamber (the chamber $C$ in the construction above).  This difference requires a change in strategy to arrive at the equivalent of Proposition \ref{prop:hbydeflate}.  

A set of disjoint spheres in a ball $B^3$ is {\em nested} if only one complementary component is a ball.

\begin{defin} Suppose $F_{\frb}$ is a nested set of spheres in a 3-ball $B^3$, with $\bdd B^3$ the outermost sphere. Call the collar between any pair of spheres a {\em shell}.   A {\em bullseye} \index{Bullseye} is the flagged chamber complex in $B^3$ whose defining surface is $F_{\frb}$, each shell has a genus $0$ (pure) Heegaard splitting, and the ball $B_\frb$ bounded by the innermost sphere is occupied with maximal bubble $\frb$.  

A bullseye in a chamber complex $\calC$ is {\em maximal} if it is not contained in any larger bullseye.    
\end{defin}

We now extend the notion of deflation to include bullseyes, not just occupied non-ball handlebodies.  Suppose $F_{\frb}$ is a bullseye in a chamber complex $\calC$ and, with no loss of generality, the chamber adjacent to $F_{\frb}$ and outside it is a $B$-chamber C.  Let $\{C_{A_i}\}$ (resp. $\{C_{B_i}\}$) be a collection of $A$-chambers (resp. $B$-chambers) in $\calC$, other than those contained in $F_{\frb}$.  Use the same process as described above for deflating  a handlebody to deflate $\frb$ either into the chambers $\{C_{A_i}\}$, if the ball chamber of $F_{\frb}$ is an $A$-chamber, or into $\{C_{B_i}\}$, if the ball chamber in $F_{\frb}$ is a $B$-chamber.  (In the former case, the number of nested spheres in $F_{\frb}$ is odd; in the latter it is even.)  
After the deflation, each sphere in $F_{\frb}$ bounds a goneball, so the spheres in $F_{\frb}$ disappear in the resulting flagged chamber complex.

\begin{defin} \label{defin:bulldeflate}  The flagged chamber complex $\calC_d$ is obtained from $\calC$ by {\em bullseye deflation} (of $F_{\frb}$ to $\{C_{A_i}\}$ or $\{C_{B_i}\}$) or, informally, deleting a bullseye.
Similarly, we informally say that $\calC$ is obtained from $\calC_d$ by 
inserting a bullseye into $C$.
\end{defin}

\begin{prop} \label{prop:bulldeflate} Suppose the flagged chamber complex $\calC_d$ is obtained from $\calC$ by deflating a bullseye $F_{\frb}$.  If $\calC_d$ certifies then $\calC_d$ and $\calC$ cocertify.
\end{prop}

\begin{proof}  The proof proceeds as for Proposition \ref{prop:hbydeflate}, with one difference.  In the case of bullseye deflation, the chamber $C$ changes, because all the spheres in $F_{\frb}$ disappear.  That is, the bullseye becomes a nested set of goneballs that are absorbed into $C$ to become a chamber we denote $C_d$.  See Figure \ref{fig:deflate3}. The proof then proceeds much as in Proposition \ref{prop:hbydeflate} except, unlike the situation in Proposition \ref{prop:hbydeflate}, it is possible that $C$ turns into the lone certifying chamber $C_d$ for $\calC_d$.  We examine this possibility:

\begin{figure}[ht!]
\labellist
\small\hair 2pt
\pinlabel  $F_{\frb}$ at 190 245
\pinlabel  $\calC$ at 260 320
\pinlabel  $\calC_d$ at 320 320
\pinlabel  $C$ at 200 320
\pinlabel  $C_d$ at 380 320
\endlabellist
    \centering
    \includegraphics[scale=0.5]{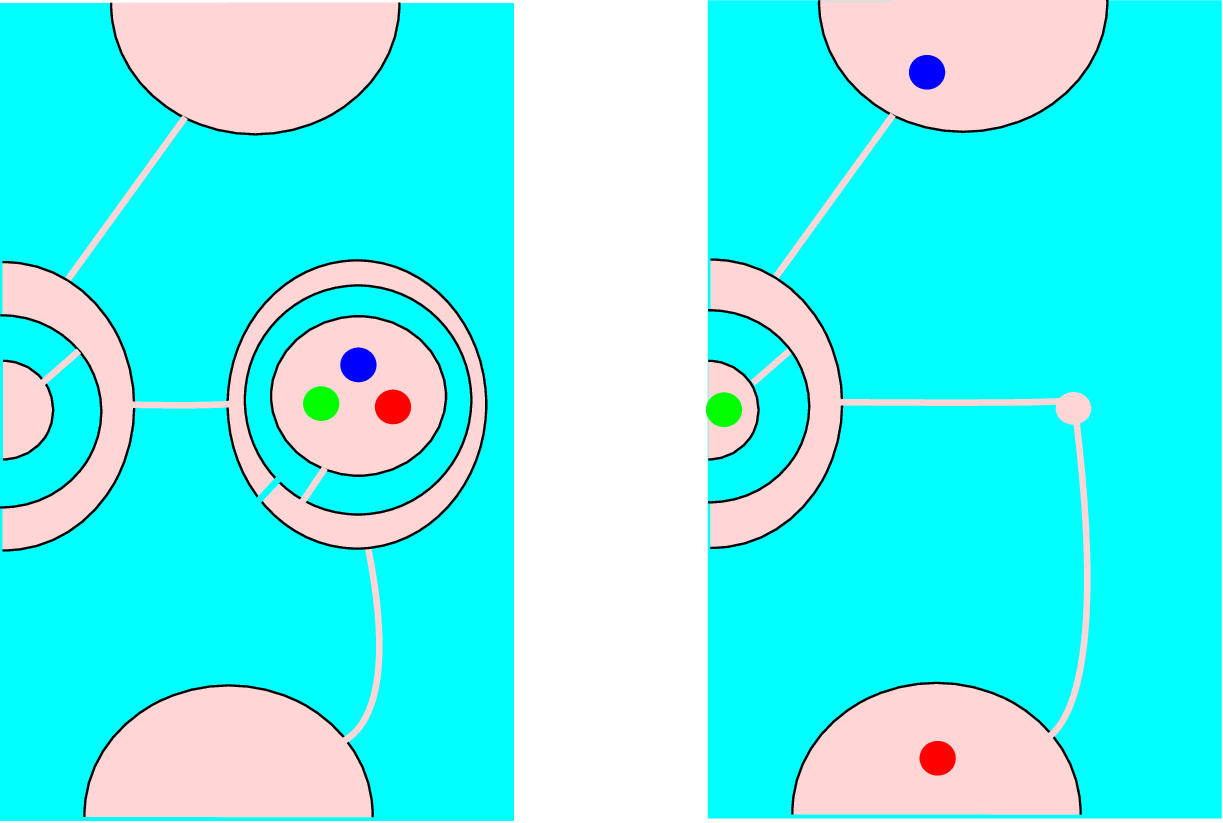}
     \caption{A nested set of goneballs absorbed into $C$} \label{fig:deflate3}
    \end{figure}


If $C_d$ certifies because it contains an incompressible sphere, let $S \subset C_d$ be such an incompressible sphere.  By general position (with a point in the center of the ball chamber of $F_{\frb}$) $S$ can be isotoped to be disjoint from $F_{\frb}$ and so lie in $C$.  Align $S$ with the Heegaard splitting of $C$.  Then it is aligned also with the Heegaard splitting of $C_d$ regardless of whether $C \in \{C_{B_i}\}$.
The sphere $S$ then certifies for both $\calC$ and $\calC_d$ as required.

%

Finally, suppose $C_d$ certifies because it is an occupied handlebody.  There are two possibilities: $C \in \{C_{B_i}\}$ and $C \notin \{C_{B_i}\}$.  If $C \in \{C_{B_i}\}$ then a sub-bubble $\frb'$ of $\frb$ becomes a bubble in $C_d$.  $\bdd \frb'$ is disjoint from $\bdd \frb$, which certifies for $\calC$, so by Lemma \ref{lemma:preexistunique} $\bdd \frb'$ also certifies for $\calC$.  After deflation, the sphere $\bdd \frb' \subset C_d$, may be taken to be disjoint from the boundary of the maximal bubble of $C_d$, by Lemma \ref{lemma:bubinH}, so $\bdd \frb'$ also certifies for $\calC_d$.  Hence $\calC$ and $\calC_d$ cocertify, as required.  

On the other hand, if $C \notin \{C_{B_i}\}$, so the entire bubble $\frb$ is dispersed away from $C_d$ and still $C_d$ is occupied, the maximal bubble $\frb_d$ of $C_d$ must have come from the chamber $C$ in $\calC$.  Since, in $\calC$, $\frb_d$ lies in $C$, it is disjoint in $C$ from $\frb$.   Hence, by Lemma \ref{lemma:preexistunique}, $\bdd \frb_d$ and $\bdd \frb$ cocertify, as required.  
\end{proof}

%
%
%

Suppose $\calC$ is a flagged chamber complex in $S^3$ and $\calC_d$ is a flagged chamber complex obtained by (non-ball) handlebody deflation.  
Suppose $\calD$ is a disk set in $\calC$ in preferred alignment and \[\calC \xrightarrow{\calD} \calC_{\calD}\] is the resulting flagged chamber complex decomposition.  Similarly construct the flagged chamber complex decomposition \[\calC_d  \xrightarrow{\calD} (\calC_d)_{\calD}\]

\begin{prop} \label{prop:hbydeflate2}   If $\calC_d$ does not certify, then either $(\calC_d)_{\calD} = \calC_{\calD}$ or $(\calC_d)_{\calD}$ is obtained from $\calC_{\calD}$ by handlebody deflation or by bullseye deflation. 
\end{prop}

\begin{proof}  With no loss of generality assume the deflated non-ball handlebody $H$ is an $A$-chamber and the chambers $C$ and $\{C_{A_i}\}$ are as described preceding Definition \ref{defin:hbydeflate}.  The chambers $\{C_{A_i}\}$ cannot contain a handlebody because when $H$ is deflated into $\{C_{A_i}\}$ the result would be an occupied handlebody in $\calC_d$, contradicting the assumption that $\calC_d$ does not certify.  So the only difference between $\calC$ and $\calC_d$ as flagged chamber complexes is that $H$ is occupied in the former but is empty in the latter.  Thus the difference between $\calC_{\calD}$ and $(\calC_d)_{\calD}$, if any, is in the flagging of the remnants of $H$, perhaps including the disappearance of remnants as goneballs.  So we examine these remnants:

Unless each remnant of $H$ is a disky handlebody, the flagging of the remnants in both $\calC_{\calD}$ and $(\calC_d)_{\calD}$ is determined by Definition \ref{defin:prefdisk}: Each handlebody remnant of $H$ is empty if and only if it is disky.  Hence in this case $\calC_{\calD} = (\calC_d)_{\calD}$ as flagged chamber complexes.  On the other hand, if each remnant of $H$ is a disky handlebody then, per Definition \ref{defin:prefdisk},  in $\calC_{\calD}$ there is exactly one handlebody remnant $H'$ (possibly a ball) that is occupied, while in $(\calC_d)_{\calD}$ each remnant is empty.  We examine this situation more closely:

Since $\{C_{A_i}\}$ contains no handlebodies, some remnant in $\calC_{\calD}$ of each $C_{A_i}$ is not a disky handlebody, by Proposition \ref{prop:remnant}.  Let $R$ be such a remnant of some $C_{A_i}$.  Consider the flagged chamber complex obtained from $\calC_{\calD}$ by deflating $H'$ into $R$.  Since $R$ is not a disky handlebody it is either not a handlebody or it is an occupied handlebody; in either case moving a bubble into $R$ will not change the flagging of $R$.  Thus this deflation changes $\calC_{\calD}$ to the same flagged chamber complex as $(\calC_d)_{\calD}$, as required.  

If $H'$ is non-ball handlebody then the deflation is a handlebody deflation.  If $H'$ is a ball then its deflation turns it into a goneball, so $(\calC_d)_{\calD}$ is a bullseye deflation of $\calC_{\calD}$, where the bullseye is centered on the ball $H'$.  (In fact, it is not hard to see that the bullseye consists only of $H'$.)  
\end{proof}

There is a similar result for bullseye deflation, but it is a bit more difficult to state and prove.  Suppose $\calC$ is a flagged chamber complex in $S^3$ and $\calC_d$ is the flagged chamber complex obtained by deflating a bullseye $F_{\frb}$ in $\calC$.  Suppose $\calD$ is a disk set in $\calC$ in preferred alignment and \[\calC \xrightarrow{\calD} \calC_{\calD}\] is the resulting flagged chamber complex decomposition.  Similarly construct the flagged chamber complex decomposition \[\calC_d  \xrightarrow{\calD'} (\calC_d)_{\calD'},\] where $\calD' \subset \calD$ is the subcollection of disks not incident to any sphere in $F_{\frb}$.  (Since the spheres in $F_{\frb}$ disappear in $\calC_d$ it would not make sense to include disks incident to them among the decomposing disk set in $\calC_d$.)   

\begin{prop} \label{prop:bulldeflate2}   Continue with the notation just given. 
If $\calC_d$ does not certify, then either $(\calC_d)_{\calD'} = \calC_{\calD}$ or $(\calC_d)_{\calD'}$ is obtained from $\calC_{\calD}$ 
by bullseye deflation. 
\end{prop}

\begin{proof}  As usual, without loss of generality, assume the chamber in $\calC$ adjacent to $F_{\frb}$ is a $B$-chamber $C$ which, after deflation and absorption of the nested set of goneballs into $C$, becomes a chamber we denote $C_d \in \calC_d$.   Let $\calD_{\frb} = \calD - \calD'$, the subcollection of disks that are incident to spheres in $F_{\frb}$.  The boundary $\bdd D$ of any disk $D \in \calD_{\frb}$ divides the sphere $S \subset F_{\frb}$ in which it lies into two disks $E_{\pm}$.  If $D$ lies in the occupied ball or a shell of $F_{\frb}$ then it is parallel to one or both of $E_{\pm}$ and so is inessential in $C$.  If $(D, \bdd D) \subset (C, \bdd F_{\frb})$ and neither sphere $D \cup E_{\pm}$ bounds a ball in $C$, then the two spheres (which become parallel spheres in $C_d$, where $F_{\frb}$ is gone) would be essential in $C_d$,  and so would certify for $\calC_d$, contradicting our assumption.  We conclude that each disk in $\calD_{\frb}$ is inessential in the chamber of $\calC$ in which it lies.  
\medskip

{\em Claim:} The remnants of $F_{\frb}$ in $\calC_{\calD}$ consist of a bullseye $F'_{\frb}$ that intersects $C$ only in a collar of $\bdd F_{\frb}$.  The occupied ball component of  $F'_{\frb}$ still has $\frb$ as a maximal bubble.

The proof of the claim is straightforward but a bit tedious - here is a sketch of the major steps:  Since each disk of $\calD$ that is incident to $\bdd F_{\frb}$ and lies in $C$ is inessential, there is a collar of $\bdd F_{\frb}$ that contains them all.  Then each remnant of $F_{\frb}$ in $\hat{\calC}_{\calD}$ lies in a ball $F_{\frb}^+$, the union of $F_{\frb}$ and the collar, and is bounded by spheres, so it is a punctured ball.  Among the remnants of each shell in $F_{\frb}$ is one containing an incompressible sphere, so that remnant is not a handlebody; similarly the adjacent remnant of the chamber $C$ is not a handlebody.  Hence, per Definition \ref{defin:prefdisk}, each ball remnant in $\calC_{\calD}$ of each shell and of $C$ is a goneball.  This implies that each remnant of the ball $B_\frb$ in $\calC_{\calD}$ is a ball and, per Definition \ref{defin:prefdisk}, exactly one is occupied, and so must contain all of $\frb$.  

Each component of $F_{\calD}$ that  lies in $F_{\frb}^+$ is a sphere; the ball it bounds in $F_{\frb}^+$ is a goneball in $\calC_{\calD}$ unless it contains $\frb$.  Hence the components of $F(\calC_{\calD})$ that lie in $F_{\frb}^+$ consist of a collection of nested spheres centered on a ball containing $\frb$; in other words, they define a bullseye $F'_\frb$, as claimed.  See Figure \ref{fig:bull2bull}.
\medskip

\begin{figure}[ht!]
\labellist
\small\hair 2pt
\pinlabel  $\calD_{\frb}\subset F_{\frb}^+\subset \calC$ at 90 -15
\pinlabel  $\hat{\calC}_{\calD}$ at 300 -15
\pinlabel  $\calC_{\calD}$ at 500 -15
\endlabellist
    \centering
    \includegraphics[scale=0.6]{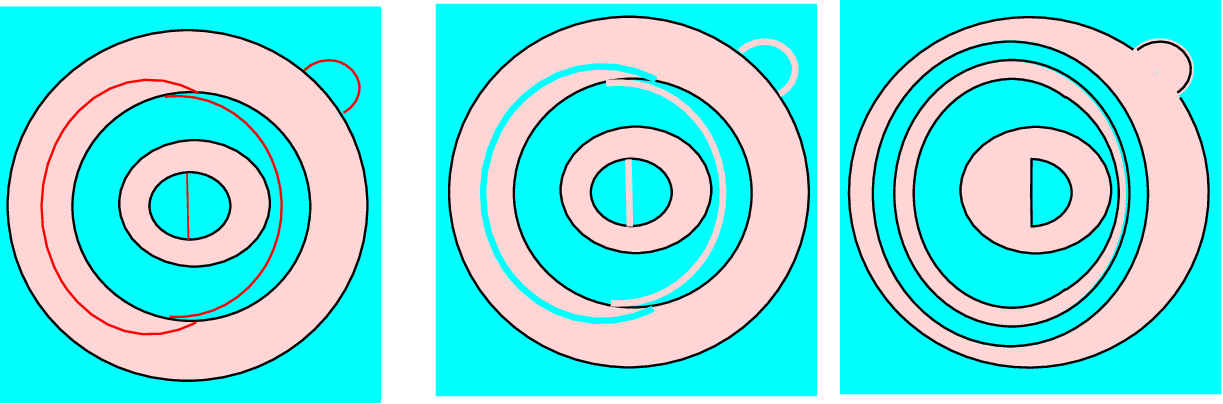}
     \caption{From bullseye to bullseye} \label{fig:bull2bull}
    \end{figure}


Since each ball remnant of $C$ in $F'_{\frb}$ is a goneball, the Heegaard surface $T_C$ for $C$ (as well as the splittings of other chambers in $\calC$) is unaffected by decomposition along $\calD_{\frb}$.  So the original preferred alignment of $\calD$ in $\calC$ remains an alignment of $\calD'$ in $\calC$. That alignment is easily made into a preferred alignment of $\calD'$ in $\calC_d$ as follows: $\calD'$ is already in preferred alignment in chambers not $C_d$ nor among the chambers $\{C_{A_i}\}$ (say) into which $\frb$ is deflated, since these are the only chambers changed by the deflation.  Whichever chamber, in $\{C_{A_i}\}$ or $C_d$, that $\frb$ is deflated into cannot be a handlebody, since $\calC_d$ does not certify, so there is some remnant of that chamber that is not a disky handlebody; choose the alignment of $\calD'$ in that chamber so that $\frb$ appears in a remnant that is not a disky handlebody.  This does not change the flagging and thereby ensures that $\calD'$ is in preferred alignment with $\calC_d$.

Now consider the effect of $\calD$ surgery {\em away} from the bullseye $F_{\frb}$, that is surgery via $\calD'$ on the surface $F - F_{\frb}$.  The resulting surface $\hat{F}$ can be viewed in two ways: Since $F - F_{\frb}$ is also the defining surface for $\calC_d$, $\hat{F}$ can be described as the first (surgery) step in the construction of the decomposition $\calC_d  \xrightarrow{\calD'} (\calC_d)_{\calD'}$.  And $\hat{F}$ is also that part of the surface $F_{\calD}$ that lies outside the new bullseye $F'_{\frb}$. 

It remains only to understand the effect of the second step in disk decomposition - eliminating goneballs after surgery.  We have already determined that eliminating goneballs from surgery on $F_{\frb}$ results in the bullseye $F'_{\frb} \subset \calC_{\calD}$.  So we need only consider sphere components of $\hat{F}$.
 Any such sphere bounds a ball (in fact two balls) in $S^3$.  If such a ball is a goneball in $\calC_{\calD}$ then that ball is unaffected by a deflation of $F'_{\frb}$ (since $F'_{\frb}$ is necessarily outside the goneball) so it remains a goneball in $(\calC_d)_{\calD'}$.  

Consider the symmetric question: if a sphere component of $\hat{F}$  bounds a goneball in $(\calC_d)_{\calD'}$, does it bound a goneball in $\calC_{\calD}$? The goneball in $(\calC_d)_{\calD'}$ may or may not contain the deflated $F'_{\frb}$.  If the goneball does not contain $F'_{\frb}$ then the goneball lies also in $\calC_{\calD}$ and so the sphere does not appear in either $\calC_{\calD}$ or $(\calC_d)_{\calD'}$.  Suppose, on the other hand, that  there is a sphere component of $\hat{F}$ that bounds a goneball in $(\calC_d)_{\calD'}$ and that goneball does contain the deflated $F'_{\frb}$.  Let $S$ be an innermost one.  Since the ball it bounds contains $F'_{\frb}$ it is not a goneball in $\calC_{\calD}$, so $S$ persists as a sphere in the defining surface of $\calC_{\calD}$.  On the other hand, observe this: since a goneball must have genus 0 splitting, the region between $S$ and $\bdd F'_{\frb}$ in $\calC_{\calD}$, a collar of $\bdd F'_{\frb}$, must have genus 0 splitting.  Hence the region is a genus 0 shell.  We can then regard $S$ as yet another nested sphere component in $F'_{\frb}$ and continue the argument.  In the end, the only components of $\hat{F}$ that appear in $\calC_{\calD}$ but not in $(\calC_d)_{\calD'}$ are spheres that just add shells to $F'_{\frb}$.  In particular, they are all eliminated by a bullseye deflation, in this case of of a bullseye that properly contains $F'_{\frb}$.  So $(\calC_d)_{\calD'}$ is obtained from $\calC_{\calD}$ by bullseye deflation, as required. 
\end{proof}

Digressing momentarily, Proposition \ref{prop:bulldeflate2} gives an easy and early example of a {\em guiding set of disks}.  These will be further discussed in Section \ref{sect:guidingisk}.

\begin{defin} \label{defin:guidedisk}  Suppose $\calC$ is a flagged chamber complex in $S^3$ and $\ocalD$ is a finite set of disjoint disks in $S^3$ transverse to $F = F(\calC)$ so that $F \cap \ocalD = \bdd \calD$, for some $\calD \subset \ocalD$.  Then $\ocalD$ is a {\em guiding set of disks} for the flagged chamber complex decomposition $\calC \xrightarrow{\calD} \calC_{\calD}$.  Slightly abusing notation we can then write $\calC \xrightarrow{\ocalD} \calC_{\calD}$. \index{Guiding set of disks}
\end{defin}

For example, in Proposition \ref{prop:bulldeflate2}, $\calD$ is a disk set for the decomposition $\calC \xrightarrow{\calD} \calC_{\calD}$ and $\calD$ is also a {\em guiding disk set} for the decomposition $\calC_d \xrightarrow{\calD'} (\calC_d)_{\calD}$.  So, with no loss of meaning, we could also write the latter $\calC_d \xrightarrow{\calD} (\calC_d)_{\calD}$.
\medskip

In general we will write $\calC' \dashrightarrow  \calC$ to denote that $\calC'$ is obtained from $\calC$ by either (non-ball) handlebody deflation or bullseye deflation. (The direction of the arrow is meant to indicate that the defining surface $F(\calC')$ embeds in $F(\calC)$.) \index{$\dashrightarrow$} Thus, in both Propositions  \ref{prop:hbydeflate2} and \ref{prop:bulldeflate2} above, $\calC_d \dashrightarrow  \calC$.  Applying Definition \ref{defin:guidedisk} these two  propositions, together with Propositions \ref{prop:hbydeflate} and \ref{prop:bulldeflate}, can then be summarized as follows:


\begin{cor} \label{cor:diagrdeflate}  Suppose $\calC, \calC'$ are flagged chamber complexes for which  $\calC' \dashrightarrow \calC$ and suppose $\calD$ is a disk set in $\calC$. 

If $\calC'$ certifies, then $\calC'$ and $\calC$ cocertify. 

 If $\calC'$ does not certify, then the following square of flagged chamber complexes commutes:
 \[\xymatrix {\calC'  \ar[r]^{\calD} \ar @{-->} [d] & \calC'_{\calD} \ar @{-->} [d]\\ 
\calC \ar[r]^{\calD}   & \calC_{\calD}   }\]
\end{cor}

Broaden now to sequences of decompositions:

\begin{defin}. Suppose 
\[\vec{\calC}: \quad \calC_0 \xrightarrow{\calD_0} \calC_1 \xrightarrow{\calD_1} ... \xrightarrow{\calD_{n-1}}\calC_n\] and
\[\vec{\calC'}: \quad \calC'_0 \xrightarrow{\calD_0} \calC'_1 \xrightarrow{\calD_1} ... \xrightarrow{\calD_{n-1}}\calC'_n\]
are flagged chamber complex decomposition sequences 
so that for each $i \geq 0$, $\calC'_i \dashrightarrow \calC_i$. 
Then the decomposition sequence $\vec{\calC}'$ is a {\em deflation} of the decomposition sequence $\vec{\calC}$. \index{Deflation of decomposition sequence}
\end{defin}

Suppose 
\[\vec{\calC}: \quad \calC_0 \xrightarrow{\calD_0} \calC_1 \xrightarrow{\calD_1} ... \xrightarrow{\calD_{n-1}}\calC_n\] is a flagged chamber complex decomposition sequence 
 and $\calC'_0$ is a flagged chamber complex such that $\calC'_0 \dashrightarrow \calC_0$.  Iteratively define a flagged chamber complex decomposition sequence 
\[\vec{\calC}'^{m(ax)}: \quad \calC'_0 \xrightarrow{\calD_0} \calC'_1 \xrightarrow{\calD_1} ... \xrightarrow{\calD_{m-1}}\calC'_m\]
by the following process: If $\calC'_i \dashrightarrow \calC_i$
take $\calC'_{i+1}$ to be the result of flagged chamber complex decomposition of $\calC'_i$ by (the guiding set of disks) $ \calD_i$.   Continue until either $m = n$ or $\calC'_{m+1}$ is not a deflation of $\calC_{m+1}$.  The sequence $\vec{\calC}'^{m}$ is a deflation of the sequence
\[\calC_0 \xrightarrow{\calD_0} \calC_1 \xrightarrow{\calD_1} ... \xrightarrow{\calD_{m-1}}\calC_m\] so the following definition is natural:

\begin{defin} $\vec{\calC}'^{m}$ is the maximal deflationary subsequence of $\vec{\calC}'$. \index{Maximal deflationary subsequence}
\end{defin}

\begin{cor} \label{cor:longdeflate}  
Suppose \[\vec{\calC}: \quad \calC_0 \xrightarrow{\calD_0} \calC_1 \xrightarrow{\calD_1} ... \xrightarrow{\calD_{n-1}}\calC_n\] is a flagged chamber complex decomposition sequence supporting $S^3 = A \cup_T B$ 
 and $\calC'_0$ is a flagged chamber complex such that $\calC'_0 \dashrightarrow \calC_0$. Let $\vec{\calC}'^{m}$ be the maximal deflationary subsequence of $\vec{\calC}'$. 
 \begin{itemize}
 \item If $\vec{\calC}'^{m}$ does not certify then $m = n$.
 \item If $\vec{\calC}'^{m}$ does certify then  
 $\vec{\calC'}$ and $\vec{\calC}$ cocertify.
 \end{itemize}
\end{cor}

\begin{proof}   Iteratively apply Corollary \ref{cor:diagrdeflate}.
\end{proof}

\section{Reordering disks} \label{sect:addisk}

\subsection{Parallel vs sequential flagged chamber complex decompositions} \label{subsect:parallelorseq}  
Suppose $\calC$ is a flagged chamber complex supporting $S^3 = A \cup_T B$, with $F = F(\calC)$ its defining surface.  Suppose $\calD_x, \calD_y$ are disjoint disk sets in $\calC$.  In this section we compare the results of flagged chamber complex decomposition using these disk sets in series vs in parallel.  That is, we compare the two flagged chamber complexes that appear in the bottom row of this diagram:

\[\xymatrix {\calC  \ar[d]_{\calD_x \cup \calD_y} \ar[r]^{\calD_x} & \calC_x \ar[d]^{\calD_y} \\ 
\calC_{x+y} &  \calC_{xy} } \]

In general the answer is complex: for starters, notice that some disks in $\calD_y$ may be incident to the boundary of goneballs in $\calC_x$, so the disks $\calD_y$ become not a disk set for the decomposition $\calC_x \xrightarrow{\calD_y} \calC_{xy}$ but a guiding disk set.  In light of such difficulties, eventually we will restrict our interest to the case in which $\calD_x$ or $\calD_y$ consists of a single disk.  See Subsection \ref{subsect:single}.

We first identify circumstances in which $\calC_{x+y} = \calC_{xy}$.  To that end:

\begin{defin} \label{defin:coherent}
Suppose $H$ is a handlebody in $S^3$ such that $ \bdd H \subset F_{\calD_x \cup \calD_y}$, the surface obtained from $F$ by surgery on $\calD_x \cup \calD_y$.  We say that $H$ is {\em coherent} for the decompositions by $\calD_x \cup \calD_y$ in the diagram above when 
\begin{itemize}  
\item $H$ is a chamber in $\calC_{x+y}$ if and only if it is a chamber in $\calC_{xy}$ and
\item if $H$ is a chamber in both $\calC_{x+y}$ and $\calC_{xy}$ then it has the same flagging.
\end{itemize}
\end{defin}

Here is an easy example:

\begin{lemma} \label{lemma:nonplaninH} Suppose there is a component of $F_{\calD_x \cup \calD_y}$ that lies in $\inter(H)$ and is not a sphere.  Then $H$ is not a chamber in either $\calC_{x+y}$ or $\calC_{xy}$.  Hence $H$ is coherent.
\end{lemma}

\begin{proof} Let $F_{x+y}$ be a component of $F_{\calD_x \cup \calD_y}$ that lies in $\inter(H)$ and is not a sphere.  Since $F_{x+y}$ is not a sphere it is not contained in a goneball of $\calC_{x+y}$ (cf Lemma \ref{lemma:disky}(2)) so it is a component of the defining surface $F(\calC_{x+y})$ that lies inside $H$.  Hence $H$ is not a handlebody chamber in $\calC_{x+y}$.

The complement of the scars in $F_{x+y}$ is a non-planar connected surface contained in a closed component $F_0 \subset F$.  Similarly the complement of just the scars of $\calD_y$ in $F_{x+y}$ is contained in a non-planar closed surface $F_x$ of $F_{\calD_x}$.  Since both $F_0$ and $F_x$ contain a non-planar subsurface they are not spheres.  Since $F_x$ is not a sphere it is not contained in a goneball of $\calC_x$, again by Lemma \ref{lemma:disky}(2).  In particular it is a component of the defining surface $F(\calC_x)$.  $F_{x + y}$ is a component of the surface that results from surgery by $\calD_y$ on $F_x$.  Since it is not a sphere it also is not contained in a goneball of $\calC_{xy}$.  Hence $F_{x+y}$ is a closed component of the defining surface for $\calC_{xy}$ lying in the interior of $H$.  Thus $H$ is not a handlebody chamber in $\calC_{xy}$.
\end{proof}

Returning to the definition of coherent, note that $H$ could fail to be a chamber in $\calC_{x+y}$ for two reasons:  $H$ contains a component of the defining surface $F(\calC_{x+y})$ in its interior, as in Lemma \ref{lemma:nonplaninH}, or $H$ could be a goneball.  Similarly, $H$ could fail to be a chamber in $\calC_{xy}$ because it has a component of the defining surface $F(\calC_{xy})$ in its interior, or  because $H$ is a goneball under the 
sequence of decompositions 
\[\calC \xrightarrow{\calD_x} \calC_x \xrightarrow{\calD_y} \calC_{xy}.\] 
Suppose the latter.  $H$ could become a goneball under the second decomposition, so the complement $\bdd H_{-}$ of the scars of $\calD_y$ in $\bdd H$ is part of a surface in $F(\calC_x)$.  On the other hand, $\bdd H_{-}$ may not be part of  a surface in $F(\calC_x)$ if $\bdd H_{-}$ is part of a sphere $G \subset F = F(\calC)$ that bounds a goneball in $\calC_x$.  In the last case, if $H$ also has no other components of $\calC_{xy}$ in its interior, it is natural to call $H$ a goneball 
of the decomposition sequence $\calC \xrightarrow{\calD_x} \calC_x \xrightarrow{\calD_y} \calC_{xy}$.  As we will see, such a ball $H$ need not be a goneball in $\calC_{x+y}$.  

The following lemma explains why coherence is useful: 

\begin{lemma}  \label{lemma:cohertoequal} Suppose every handlebody in $S^3$ whose boundary is in $F_{\calD_x \cup \calD_y}$ is coherent.  Then $\calC_{x+y} = \calC_{xy}$ as flagged chamber complexes.
\end{lemma}

\begin{proof}  Let $F_0$ be a component of $F_{\calD_x \cup \calD_y}$.  If $F_0$ is not a sphere, then, as in the proof of Lemma \ref{lemma:nonplaninH}, $F_0$ is a component of both defining surfaces $F(\calC_{x+y})$ and $F(\calC_{xy})$.   If $F_0$ is a sphere, then it divides $S^3$ into two 3-balls, $H_{\pm}$.  By assumption both $H_{\pm}$ are coherent, so each is a goneball in $\calC_{x+y}$ if and only if it is a goneball in $\calC_{xy}$.  In particular, if either $H_{\pm}$ is a goneball in one of $\calC_{x+y}$ or $ \calC_{xy}$ then it is a goneball in both and $F_0$ is not in either $F(\calC_{x+y})$ and $F(\calC_{xy})$. On the other hand, neither of $H_{\pm}$ is a goneball in $\calC_{x+y}$ if and only if neither is a goneball in $ \calC_{xy}$ and, in this case, $F_0$ does lie in both $F(\calC_{x+y})$ and $F(\calC_{xy})$.  Thus, ignoring flagging, $\calC_{x+y} = \calC_{xy}$ as chamber complexes.

Now consider flagging of a handlebody chamber $H$ in one of $\calC_{x+y}$ or $\calC_{xy}$ and so in both.  Suppose $H$ is a handlebody chamber in one of $\calC_{x+y}$ or $ \calC_{xy}$.  Since $H$ is coherent, it is then a handlebody chamber in both  $\calC_{x+y}$ and $\calC_{xy}$, and the flagging of $H$ is the same.  Thus $\calC_{x+y} = \calC_{xy}$ as flagged chamber complexes.
\end{proof}

Here is a helpful example of incoherence in a handlebody:  
Let $C$ be a chamber in $\calC$ that a disk $D$, properly embedded in $C$, divides into two components: $U$, not a handlebody, and $V$, the complement in a handlebody $H$ of a regular neighborhood $\eta$ of a properly embedded arc $\alpha \subset H$.  Let $E$ be a meridian disk in $\eta$, so $E$ is properly embedded in $\eta$ and meets $\alpha$ in a single point.   Note that $F \cap \inter(H)$ is the annulus $\bdd \eta$.  See Figure \ref{fig:paraldisk}.    

\begin{figure}[ht!]
\labellist
\small\hair 2pt
\pinlabel  $\calC$ at 90 300
\pinlabel  $D$ at 90 380
\pinlabel  $E$ at 140 360
\pinlabel  $\calC_D$ at 90 530
\pinlabel  $U$ at 45 520
\pinlabel  $V$ at 135 520
\pinlabel  $\calC_E$ at 90 50
\pinlabel  $\calC_{DE}$ at 390 530
\pinlabel  $\calC_{D+E}$ at 390 300
\pinlabel  $\calC_{ED}$ at 390 50
\pinlabel  empty at 450 100
\pinlabel  occupied at 450 350
\pinlabel  occupied at 450 580
\endlabellist
    \centering
    \includegraphics[scale=0.5]{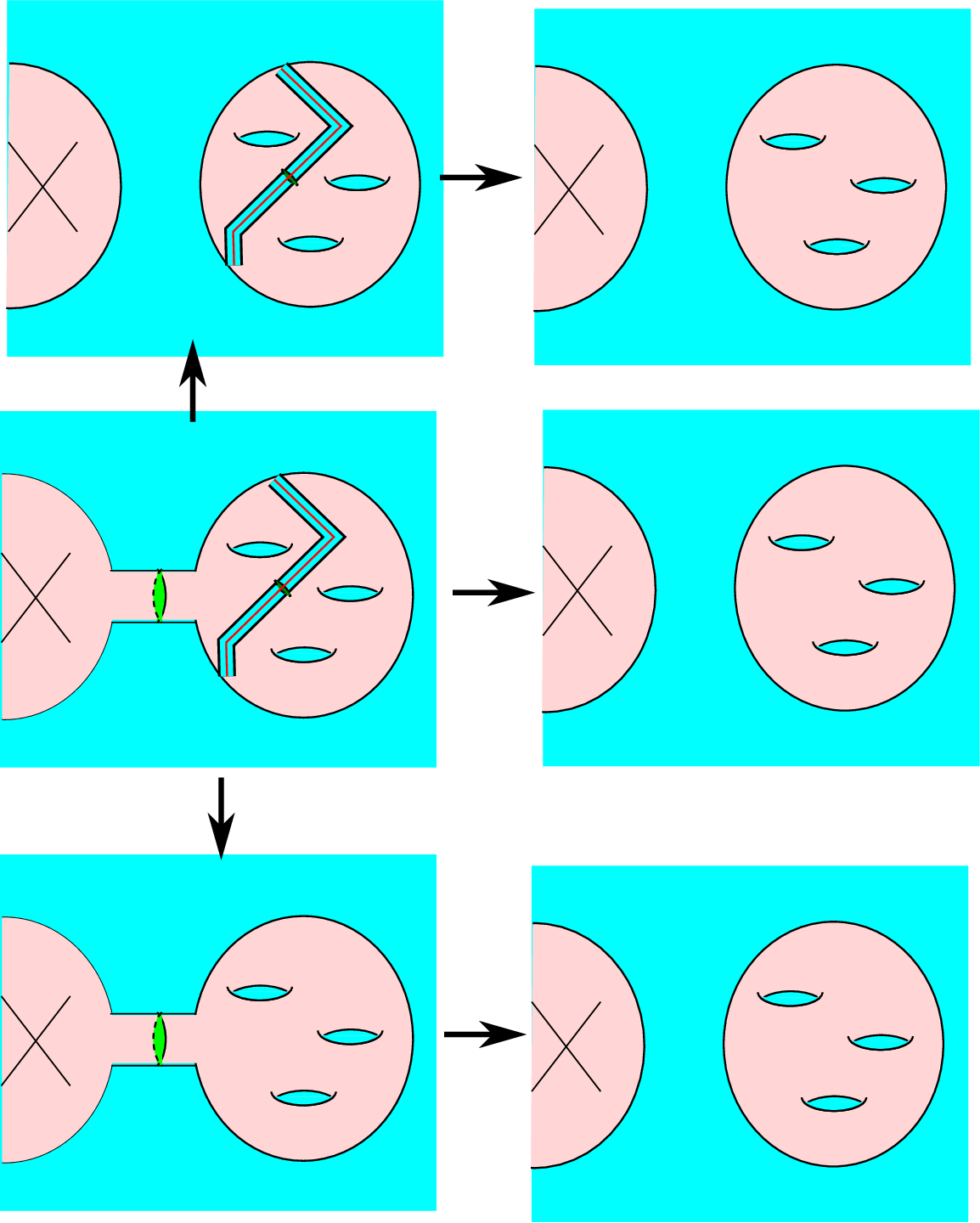}
   \caption{$\calC_{DE} = \calC_{D+E} \neq \calC_{ED}$} \label{fig:paraldisk}
    \end{figure}

Since $F \cap \inter(H)$ is an annulus and not a collection of disks, $H$ is occupied after the decomposition $\calC \xrightarrow{D \cup E} \calC_{D + E}$ and after the sequence of decompositions \[\calC \xrightarrow{D} \calC_D \xrightarrow{E} \calC_{DE}.\]  This means that $H$ is a coherent handlebody for $\calC_{D + E}$ and $\calC_{DE}$.  On the other hand, after the decomposition $\calC \xrightarrow{E} \calC_E$ the disk $D$ divides $\calC_E$ into $U$ and the handlebody $H$.  For the next decomposition $\calC_E \xrightarrow{D} \calC_{ED}$, Definition \ref{defin:prefdisk} requires that $E$ be aligned so that $H$ is empty in $\calC_{ED}$.  Thus $H$ is not a coherent handlebody for $\calC_{D + E}$ and $\calC_{ED}$. 

Replacing $D$ in this example of incoherence by a disk family consisting of $D$ and a disk $D'$ parallel to $E$ has no effect on the above argument: For $\calD = {D \cup D'}$,  $H$ is coherent for $\calC_{\calD + E}$ and $\calC_{\calD E}$ but incoherent for $\calC_{\calD + E}$ and $\calC_{E\calD}$.  This leads to the following observation, using analogous notation:

\begin{lemma} \label{lemma:paraldisk}  Let $\calD \cup E$ be a disk family in a flagged chamber complex $\calC$ in $S^3$.  Suppose the disk $E$ is parallel to a disk $D \in \calD$ within a chamber $C$ of $\calC$.  Then
\begin{enumerate}

\item Unless the chamber $C_D$ of $\calC_\calD$ containing $E$ is an occupied handlebody, $\calC_\calD = \calC_{\calD E}$.  If $C_D$ is an occupied handlebody then for one of the two resulting siblings in $\calC_{\calD E}$, $\calC_{\calD} = \calC_{\calD E}$. 

\item Unless the chamber $C$ is an occupied handlebody with only disky handlebody remnants in $\calC_\calD$, the chamber complex $\calC_\calD = \calC_{\calD + E}$. The same is almost true in the remaining case: If $C$ is an occupied handlebody with only disky handlebody remnants in $\calC_\calD$ then $C$ has sibling remnants in $\calC_{\calD + E}$ and deleting a single such sibling gives the remnants of $C$ in $\calC_\calD$.

\item either $\calC_{\calD + E} = \calC_{\calD E}$ or $C_D$ is an occupied handlebody chamber in $\calC_{\calD}$ and, in one or both of two resulting siblings in $\calC_{\calD E}$, $\calC_{\calD + E} = \calC_{\calD E}$.  
\end{enumerate}   
\end{lemma}

\begin{proof}   Since $D$ and $E$ are parallel, after surgery on $\calD \cup E$ there is a sphere $G$ in $F_{\calD \cup E}$ bounding a disky ball $B_G$ with one scar of each of the disks $D$ and $E$ on $G$. Moreover the other remnant $C_- = C_D - B_G$ of $C_D$ in $\calC_{\calD E}$ is homeomorphic to $C_D$, so it is a handlebody if and only if $C_D$ is.  If $C_D$ is not a handlebody, so also $C_-$ is not a handlebody, or if $C_D$ is an empty handlebody, then per Definition \ref{defin:prefdisk} the ball $B_G$ is empty and so a goneball in  $\calC_{\calD E}$.  Then $\calC_{\calD} = \calC_{\calD E}$ (with $C_-$ replacing $C_D$).  If $C_D$ is occupied then the same statement is true for one of the two siblings of $C_D$.  This proves 1).

The remnants of $C$ in $\calC_{\calD + E}$ are the same as in $\calC_{\calD}$, with the addition of the disky ball $B_G$.  If $B_G$ is a goneball, we are done as in the proof of 1).  The only way in which $B_G$ could not be a goneball,  per Definition \ref{defin:prefdisk}, is if $C$ is an occupied handlebody with only disky remnants.  In the latter case, only the sibling in which $B_G$ is occupied does not also occur among the sibling remnants of $C$ in $\calC_{\calD}$. This proves 2).   

The proof of 3) follows almost immediately.  Following 1), we may as well assume that $\calC_\calD = \calC_{\calD E}$.  The result follows from 2), unless $C$ is an occupied handlebody with only disky handlebody remnants in $\calC_\calD$.  In this case, $C_-$ is also disky, so $C$ also has only disky handlebody remnants in $\calC_{\calD + E}$.  Then per Definition \ref{defin:prefdisk} the remnants of $C$ in both $\calC_{\calD + E}$ and $\calC_{\calD E}$ have the same description of siblings, one corresponding to exactly one of the remnants of $\calC_{\calD +E}$ being occupied, with remnant $B_G$ a goneball except when it is the occupied remnant.  
\end{proof}

Note: Switching the order of $\calD$ and $E$ makes a difference.  The example preceding Lemma \ref{lemma:paraldisk} shows that there may be a handlebody in $S^3$ whose boundary is in $F_{\calD \cup E}$ and which is not even coherent for $\calC_{\calD + E}$ and $\calC_{E\calD}$.
\bigskip

Return now to the discussion of decomposition by $\calD_x$ and $\calD_y$.

\begin{defin} \label{defin:incdisk}  Suppose $H$ is a handlebody in $S^3$ so that $\bdd H \subset F_{\calD_x \cup \calD_y}$.  Then a disk $D \in \calD_x \cup \calD_y$ {\em is incident to $H$} if $D \subset \inter(H)$ or $D$ leaves a scar on $\bdd H$. 
The handlebody $H$ is an {\em $x$-handlebody} (resp. $y$-handlebody) if $H$ is incident only to disks in $\calD_x$ (resp. $\calD_y$).  
An {\em $x$-handlebody} (resp. $y$-handlebody) $H$ is an {\em $x$-handlebody} (resp. $y$-handlebody) {\em chamber} in $\calC_{xy}$ if it is a chamber there.  That is, $\bdd H$ is in the defining surface $F(\calC_{xy})$ and no other chamber lies in $\inter(H)$. \index{x-handlebody} \index{y-handlebody}
\end{defin}

\begin{lemma} \label{lemma:incdisk} If $H$ is an $x$-handlebody chamber in $\calC_{xy}$ 
then $H$ is also a chamber of $\calC_x$.  
\end{lemma}

\begin{proof}  Let $C$ be the chamber of $\calC_x$ which has $H$ as a remnant.  If any disks in $\calD_y$ are incident to $C$ then each remnant of $C$ in $\hat{\calC}_{xy}$ has a scar from $\calD_y$, so the same is true after getting rid of goneballs.  Since $H$ has no scars from $\calD_y$, no disks from $\calD_y$ are incident to $C$, and  $C$ is unaffected by the decomposition by $\calD_y$.  That is, $C = H$ as required.
\end{proof}

Suppose $\calC$ contains no occupied handlebody chambers.  Then by Corollary \ref{cor:prefdisk} a handlebody chamber in $\calC_x$ or a handlebody chamber in $\calC_{x+y}$ is empty if and only if it is disky.  

 \begin{lemma} \label{lemma:FH1} Suppose $\calC$ contains no occupied handlebody chambers. Then: 
 \begin{itemize}
 \item Each $x$-handlebody chamber of $\calC_{xy}$ is empty if and only if it is a disky chamber under the decomposition $\calC \xrightarrow{\calD_x} \calC_{x}$.  
 \item If $\calC_x$ contains no occupied handlebody chambers then each $y$-handlebody chamber of $\calC_{xy}$ is empty if and only if it is disky under the decomposition $\calC_x \xrightarrow{\calD_y} \calC_{xy}$.  
 \item If $\calC_x$ does contain occupied handlebody chambers, there is a preferred alignment of $\calD_y$ in $\calC_x$ so that each $y$-handlebody chamber of $\calC_{xy}$ is empty if and only if it is disky under the decomposition $\calC_x \xrightarrow{\calD_y} \calC_{xy}$.  Any other preferred alignment of $\calD_y$ gives a sibling decomposition $\calC_x \xrightarrow{\calD_y} \calC_{xy}$. 
 \end{itemize}
\end{lemma}
%

\begin{proof}  The first statement follows from Corollary \ref{cor:prefdisk} and Lemma \ref{lemma:incdisk}.  The second statement follows from Corollary \ref{cor:prefdisk}. 

The third statement is more difficult to prove.  Observe first that by the argument in Corollary \ref{cor:prefdisk}, any handlebody remnant of a chamber in $\calC_x$ that is not an occupied handlebody is empty if and only if it is disky.  Moreover the argument extends, via Definition \ref{defin:prefdisk}, even to occupied handlebody chambers in $\calC_x$ so long as at least one remnant in $\calC_{xy}$ is not a disky handlebody.  So we need only describe how to align those disks in $\calD_y$ that are incident to each chamber $C_x$ in $\calC_x$ of the following type: $C_x$ is an occupied handlebody and every remnant of $C_x$ in $\calC_{xy}$ is a disky handlebody.  
%
%
%
%

This is how it is done:  
 Since $C_x$ is occupied, by Corollary \ref{cor:prefdisk} $C_x$ is not disky.  So there is some component $F_x$ of $F \cap \inter(C_x)$ that is not a disk. $F_x$ must be incident to some disk in $\calD_x$ else it would be a subsurface of $F(\calC_x)$ and not lie in the interior of $C_x$.
 
  The subcollection of disks in $\calD_y$ that are incident to $\bdd C_x$ is disjoint from $F_x$, since $\inter(\calD_y)$ is disjoint from $F$.  Hence 
 the surface $F_x$ will lie in the interior of some remnant  $C_{xy}$ of $C_x$ in $\calC_{xy}$, so that remnant too is incident to $\calD_x$.  In particular, $C_{xy}$ is not a $y$-handlebody.  
Following Definition \ref{defin:prefdisk}, choose a preferred alignment where every (disky) handlebody remnant of $C_x$ other than $C_{xy}$ is empty.  Then in particular each $y$-handlebody remnant is both disky and empty, as required. 
\end{proof}

 \begin{lemma} \label{lemma:FH2} Suppose $\calC$ does not certify 
and $H$ is a handlebody in $S^3$ so that $\bdd H \subset F_{\calD_x \cup \calD_y}$. If $H$ is an $x$-handlebody it is coherent.  If $H$ is a $y$-handlebody then it is coherent in one of the sibling decompositions of $\calC_x \xrightarrow{\calD_y} \calC_{xy}$.  
\end{lemma}

\begin{proof}  If appropriate (i. e. if $\calC_x$ contains occupied handlebody chambers) choose the alignment for $\calD_y$ in $\calC_x$ given in Lemma \ref{lemma:FH1} to ensure that each $y$-handlebody chamber of $\calC_{xy}$ is empty if and only if it is disky under the decomposition $\calC_x \xrightarrow{\calD_y} \calC_{xy}$.  Also from Lemma \ref{lemma:FH1} we have that each $x$-handlebody chamber of $\calC_{xy}$ is empty if and only if it is disky under the decomposition $\calC \xrightarrow{\calD_x} \calC_x$.

{\em Case 1: $x$-handlebodies}

 The proof will be by contradiction.  Suppose there is an $x$-handlebody that is not coherent and let $H$ be an innermost one that is not coherent.  Suppose first that there are no components of $F_{\calD_x \cup \calD_y}$ contained in the interior of $H$.  Then $\calC \xrightarrow{\calD_x} \calC_x$ is disky (for $F \cap \inter(H)$ has no components at all) and hence empty (possibly a goneball) in $\calC_{xy}$.  Similarly $\calC \xrightarrow{\calD_x \cup \calD_y} \calC_{x+y}$ is disky and hence $H$ is empty in $\calC_{x+y}$.  Since $H$ is empty in both $\calC_{xy}$ and $\calC_{x+y}$ it is coherent, a contradiction.  
 
We deduce then that there must be components of $F_{\calD_x \cup \calD_y}$ in the interior of $H$. 
Suppose a subsurface of $F_{\calD_x \cup \calD_y}$ in the interior of $H$ becomes in $\calC_x$ a component $F_x$ of $F(\calC_x)$.  Then $H$ is not a chamber in $\calC_x$ (since it contains a component of $F(\calC_x)$ in its interior) and, since $H$ is incident to no disks in $\calD_y$, $H$ remains not a chamber in $\calC_{xy}$.  Since no disk in $\calD_y$ is incident to $F_x$, $F_x$ is also a component of $\calC_{x+y}$ unless it is a goneball there.  But if it were a goneball there, and not in $\calC_{xy}$ it would be a further in handlebody that is not coherent, contradicting our choice of $H$.  We conclude that $F_x$ remains as a surface in $\calC_{x+y}$ so $H$ is not a chamber there either.  But this implies $H$ is coherent, a contradiction.  

We deduce then that $F_{\calD_x \cup \calD_y} \cap \inter (H)$ must become a collection of spheres bounding goneballs in $\calC_x$, hence in $\calC_{xy}$.  Then, by our consistency hypothesis, the spheres bound goneballs also in $\calC_{x+y}$.  Hence in each of $\calC_{xy}$ and $\calC_{x+y}$, $H$ is either a chamber or is itself a goneball.  Whether $H$ is occupied or empty (so perhaps a goneball) is determined in $\calC_{x+y}$ by whether the components of $F \cap  \inter(H)$ consist entirely of disks, by Corollary \ref{cor:prefdisk}.  By the first statement of Lemma \ref{lemma:FH1} the same is true in $\calC_{xy}$.  Hence $H$ is coherent a final contradiction.

{\em Case 2: $y$-handlebodies}

 
The proof will be by contradiction, similar to Case 1.  Suppose there is a $y$-handlebody that is not coherent and let $H$ be an innermost one that is not coherent.  Suppose there is a component $F_0$ of $F_{\calD_x \cup \calD_y}$ that lies in $\inter(H)$ and remains in $F(\calC_{xy})$, so $H$ is not a chamber in $\calC_{xy}$.  By Lemma \ref{lemma:nonplaninH} $F_0$ is a sphere, so it bounds a ball in $H$. The ball it bounds is not a goneball in $\calC_{xy}$ so, by the consistency hypothesis, $F_0$ does not bound a goneball in $\calC_{x+y}$.  This implies that $H$ is not a chamber in either $\calC_{xy}$ or $\calC_{x+y}$, contradicting our assumption that $H$ is not coherent. 

We deduce then that $\inter(H)$ is disjoint from $F(\calC_{xy})$ so $H$ is a chamber in $\calC_{xy}$.  A symmetric argument shows that $H$ is also a chamber in $\calC_{x+y}$.  Suppose $F \cap \inter(H)$ has a component $F_0$ that is not a disk, so $H$ is an occupied handlebody in $\calC_{x+y}$.
  
{\em Subcase 2a:} $F_0 \subset \inter(H)$ is a closed surface.

Since $H$ is a $y$-handlebody, $F_0$ is unaffected by surgery on $\calD_x$ so it persists as a component of $F_{\calD_x}$. 
Since $F_0$ does not bound a goneball in $\calC$, it does not bound one in $\calC_s$, so 
$F_0$ persists as a component of $F(\calC_{x})$.  In particular, $H$ is not disky under the decomposition $\calC_{x} \xrightarrow{\calD_y} \calC_{xy}$.  This implies that $H$ is an occupied handlebody in $\calC_{xy}$, contradicting the hypothesis that $H$ is not coherent.

{\em Subcase 2b:} $\bdd F_0 \neq \emptyset$.

Let $\bdd_- H = \bdd H \cap F$, that is $\bdd H$ with all scars removed.  Note that since $F_0$ is not a disk it is either non-planar or has more than one boundary component; in either case $F_0 \cup \bdd_- H$ is non-planar.  In particular, the subsurface of $F_{\calD_x}$ that contains it cannot be a sphere, so $F_0 \cup \bdd_- H$ persists as a subsurface of $F(\calC_x)$. In particular, $F_0$ persists, and, just as in Subcase 2a this implies that $H$ is an occupied handlebody in $\calC_{xy}$, contradicting the hypothesis that $H$ is not coherent.

We are thus reduced to the case that $F \cap \inter(H)$ consists entirely of disks, so $H$ is disky in $\calC_{x+y}$.  Since we are assuming $\calC$ does not certify, and therefore has no occupied handlebody chambers, this implies that $H$ is empty (and so a goneball if $H$ is a ball) in $\calC_{x+y}$.  It's possible that in $\calC_{x}$ the surface $\bdd_- H$ is part of a sphere $G \subset F$ that bounds a goneball in $\calC_{x}$.  This would imply that $\bdd H$ is planar, so $H$ is a ball, and would make $H$ a goneball of the decomposition sequence $\calC \xrightarrow{\calD_x} \calC_x \xrightarrow{\calD_y} \calC_{xy}$, as described before Lemma \ref{lemma:cohertoequal}.  

On the other hand, if $F \cap \bdd H$ does remain in $F(\calC_{D_x})$, what we can conclude is that $H$ is disky in the decomposition $\calC_x \xrightarrow{\calD_y} \calC_{xy}$ and so, per our choice of preferred alignment of $\calD_y$, $H$ is empty.  (In particular, it is again a goneball if $H$ is a ball.)  In either case, $H$ is then coherent for $\calC_{xy}$ and $\calC_{x+y}$, a final contradiction.  
\end{proof}

\subsection{When a disk set is a singleton} \label{subsect:single} We take the discussion above one step further:

\begin{defin} \label{defin:excdisk}  
Suppose $H$ is a handlebody in $S^3$ so that $\bdd H \subset F_{\calD_x \cup \calD_y}$.  $H$ is an $xe$-handlebody (resp. $ey$-handlebody) \index{xe-handlebody} \index{ey-handlebody} if $H$ is incident only to disks in $\calD_x$ (resp. $\calD_y$) except for a single disk $E(xeptional)$ which is also incident to $H$ but lies in $\calD_y$ (resp. $\calD_x$).  
\end{defin}

We proceed to two lemmas that are parallel in spirit and whose proofs will use the same figures.  To keep their use disentangled, in Figures  \ref{fig:addiskv2_2}, \ref{fig:addiskv2_3} and \ref{fig:addisk2v_6} the subscripts $xe, x+e$ will refer to the case in which $\calD_y = E$ and the subscripts $ey, e+y$ will refer to the case in which $\calD_x = E$.  

\begin{lemma} \label{lemma:ey}  Suppose $\calC$ does not certify, $H$ is an $ey$-handlebody and, if $\calC_x$ contains an occupied handlebody, $\calD_y$ is given the preferred alignment in $\calC_x$ of Lemma \ref{lemma:FH1}.  Then either $H$ is coherent or 
\begin{itemize}
\item $E$ lies in $\inter(H)$ and
\item either $\bdd E$ is non-separating in $F - \bdd \calD_y$ or $\bdd E$ is non-separating in $F$ and $E$ is parallel to a disk in $\calD_y$ and
\item $H$ lies in a chamber\footnote{The phrase `$H$ lies in a chamber' means that either $H$ is itself a chamber, or $H$ is a goneball in a chamber.} of $\calC_{xy}$ and $H$ is either 
\begin{itemize} 
\item empty (and so a goneball if $H$ is a ball), the deflation of a handlebody chamber in $\calC_{x+y}$ or 
\item the deflation of a single bullseye of $\calC_{x+y}$ in $H$. (Possibly $H$ is part of the bullseye.)
\end{itemize}
\end{itemize}
\end{lemma}

\begin{proof}  As usual, let $F_{E \cup \calD_y}$ be the surface obtained from $F = F(\calC)$ by surgery on $E \cup \calD_y$.  Following Lemma \ref{lemma:nonplaninH} we will assume that each component of $F_{E \cup \calD_y}$ that lies in $\inter(H)$ is a sphere and so bounds a ball in $H$.  Let $\mathfrak{S}$ be this set of spheres
\medskip

{\em Special case:}  Each ball in $H$ bounded by a sphere in $\mathfrak{S}$ is coherent.   

{\em Claim:} In this special case, $H$ lies entirely in a chamber of $\calC_{x+y}$ if and only if it lies entirely in a chamber of $\calC_{xy}$.

{\em Proof of claim:} For both $\calC_{x+y}$ and $\calC_{xy}$, $H$ is contained in a chamber if and only if each ball in $H$ bounded by a sphere in $\mathfrak{S}$ is a goneball, for this determines whether there is any chamber contained in $\inter(H)$.  But in this special case, each ball bounded by a sphere in $\mathfrak{S}$ is coherent, so the ball it bounds is a goneball in $\calC_{x+y}$ if and only if it is a goneball in $\calC_{xy}$.  This proves the claim.
\medskip

Following the claim, but still in the special case, we will assume that $H$ lies entirely in a chamber (possibly it is itself an entire chamber) of both $\calC_{x+y}$ and $\calC_{xy}$ and need to determine under what circumstances $H$ is not coherent.  That is, 
\begin{itemize}
\item  If $H$ is a ball, when is it a goneball in one of $\calC_{x+y}$ or $\calC_{xy}$ but not the other?
\item  If $H$ is not a ball, so it is itself a chamber, when is it empty in one of $\calC_{x+y}$ or $\calC_{xy}$ but not the other?
\end{itemize}

We examine possible subcases.  For all but the last, $H$ will turn out to be coherent:
\medskip

{\em Subcase 1a:} There is a component $F_0$ of $F \cap \inter(H)$ that is not a disk and not incident to $E$.  

Since $F_0$ is not a disk, $H$ is not a disky handlebody in the decomposition $\calC \to \calC_{x+y}$.  Since $F_0$ is not incident to $E$, $F_0$ is not affected by the decomposition $\calC \xrightarrow{E} \calC_x$.  This implies that $H$ is also not a disky handlebody in the decomposition $\calC_x \to \calC_{xy}$.  $H$ is then an occupied handlebody in both $\calC_{x+y}$ or $\calC_{xy}$, so $H$ is coherent.  

{\em Subcase 1b:}  A component $F_e$ of $F \cap \inter(H)$ is incident to $E$ and surgery on $E$ (in the decomposition $\calC \xrightarrow{E} \calC_x$) does not turn $F_e$ into the union of disks and spheres.

Essentially the same argument applies: $F_e$ can't be a disk, so $H$ is occupied in $\calC_{x+y}$; some non-disk component of the surgered $F_e$ remains as a component of $F(\calC_x)$ so $H$ is occupied in $\calC_{xy}$.  
\medskip

{\em Subcase 2:} $E$ does not lie in $\inter(H)$, so it leaves either one or two external scars on $\bdd H$. 

Following Subcase 1a, we can assume every component of $F \cap \inter(H)$ is a disk, hence $H$ is disky and so, by Corollary \ref{cor:prefdisk} $H$ is an empty chamber in $\calC_{x+y}$.  $H$ is also a disky handlebody remnant of the decomposition $\calC_x \to \calC_{xy}$ 
but this does not immediately imply that $H$ is empty in $\calC_{xy}$: per Definition \ref{defin:prefdisk} $H$ might be the remnant of an occupied handlebody $C_x$ in $\calC_x$ all of whose other remnants are also disky handlebodies.    We explore this possibility further:

Suppose this were the case, and $C$ is the chamber in $\calC$ of which $C_x$ is the occupied handlebody remnant.  If the disk $E$ lies outside $C$ then the scars of $E$ in $\bdd C_x$ would be internal scars.  By the hypothesis of this subcase we may therefore assume that $E$ lies in the interior of $C$.  The one or two remnants of $C$ in $\calC_x$ are then subsets of $C$, so $F$ is disjoint from these remnants and the remnant $C_x$ in particular is disky.  Per Definition \ref{defin:prefdisk} the only way that $\calC_x$ can be a disky remnant of $C$ and still be occupied is if $C$ is an occupied handlebody with all remnants disky.  But a hypothesis of the Lemma is that $\calC$ does not certify, so $C$ is not an occupied handlebody.  From the contradiction we conclude that $H$ is empty in $\calC_{xy}$ and so $H$ is coherent.

\medskip
Following Subcase 2, we henceforth assume that $E \subset \inter(H)$.

{\em Subcase 3:}  $F \cap \inter(H)$ consists of disks, so $H$ is either a goneball or an empty chamber in $\calC_{x+y}$.

Let $F_e$ be the component of $F \cap \inter(H)$ on which $\bdd E$ lies.   Since $H$ is irreducible, $E$ is parallel in $H$ to a subdisk $F_0$ of $F_e$.  Then the decomposition $\calC \xrightarrow{E} \calC_x$ has essentially no effect on $\calC$: after removing the goneball bounded by $E \cup F_0$, all that has changed is a proper isotopy of $F_e$ that replaces $F_0$ with $E$ (and the removal of any disks in $\calD_y$ that were incident to $F_0$, but these only give rise to goneballs in $\calC_{x+y}$ and $\calC_{xy}$). Then the decomposition $\calC_x \to \calC_{xy}$ also leaves $H$ either a goneball or an empty chamber in $\calC_{xy}$, so $H$ is coherent.  
\medskip

{\em Subcase 4:} All that remains: $F \cap \inter(H)$ has a single non-disk component $F_e$,  $F_e$ is incident to $E$, and surgery on $E$ turns $F_e$ into the union of disks and spheres. 

$F_e$ could be a torus bounding an empty solid torus, a once-punctured torus, or an annulus, and in each case $E$ is a meridian disk. (See the left column in Figures \ref{fig:addiskv2_2} and \ref{fig:addiskv2_3}.) It follows that $\bdd E$ is non-separating in $F$ and, unless $\calD_y$ also contains a disk whose boundary is parallel to $\bdd E$, $\bdd E$ is also non-separating in $F - \bdd \calD_y$.  If $\calD_y$ does contain a disk $D$ whose boundary is parallel to $\bdd E$ then $D$ and $E$ are parallel, since otherwise between them would lie a closed component of $F \cap \inter(H)$, contradicting the hypothesis of this subcase.  So, in the end, we are in the situation described in the statement of the Lemma: {\bf $\bdd E$ is non-separating in $F - \bdd D_y$ or $\bdd E$ is non-separating in $F$ and $E$ is parallel to a disk in $\calD_y$}.   In this case $H$ is an occupied handlebody in $\calC_{x+y}$, since $F_e$ is not a disk, but may be empty in $\calC_{xy}$, since $\calC_x \xrightarrow{\calD_y} \calC_{xy}$ is disky.  Since $H$ could be incoherent, we explore this situation further:

Consider the chamber $C_x$ in $\calC_x$ which has $H$ as a remnant.  If every other remnant of $C_x$ is also a disky handlebody, then $C_x$ is a handlebody and it must be occupied because $F_e$ is not a disk.  Since $E \subset \inter(H)$, each other remnant of $C_x$ is a disky $y$-handlebody, so per Lemma \ref{lemma:FH1}, each is empty.  This implies that $H$ must be occupied in $\calC_{xy}$, and $H$ is coherent.  Finally, if any remnant of $C_x$ in $\calC_{xy}$ 
is not a disky handlebody then Definition \ref{defin:prefdisk} says that a preferred alignment of $\calD_y$ leaves $H$ deflated to such a remnant, an outcome allowed in Lemma \ref{lemma:ey}.  This concludes the proof in the special case.  
\medskip

%

{\em The general case:}  

Following the special case, the remaining possibility is that there is a sphere in $\mathfrak{S}$ bounding a ball in $H$ that is not coherent.  Let $G$ be an innermost such sphere.  By the special case, $E$ lies in the ball $B_G$ that $G$ bounds in $H$ and $B_G$ is occupied in $\calC_{x+y}$ but deflates into another chamber in $\calC_{xy}$.  Since $B_G$ is occupied in $\calC_{x+y}$, it is a chamber, so $H$ does not lie inside a chamber in $\calC_{x+y}$.  If any sphere in $\mathfrak{S}$ does not bound a goneball in $\calC_{xy}$ then $H$ is also not a chamber in $\calC_{xy}$ and so is coherent.  See Figure \ref{fig:eybull}.

\begin{figure}[ht!]
 \labellist
\small\hair 2pt
\pinlabel  $\calC$ at 15 350
\pinlabel  $\calC_x$ at 230 350
\pinlabel  $\frb$ at 300 265
\pinlabel  $\frb$ at 75 75
\pinlabel  $\calC_{ey}$ at 240 150
\pinlabel  goneballs at 300 80
\pinlabel  $\calD_y$ at 80 390
\pinlabel  $E$ at 80 270
\pinlabel  $\calC_{e+y}$ at 20 150
\endlabellist
    \centering
    \includegraphics[scale=0.7]{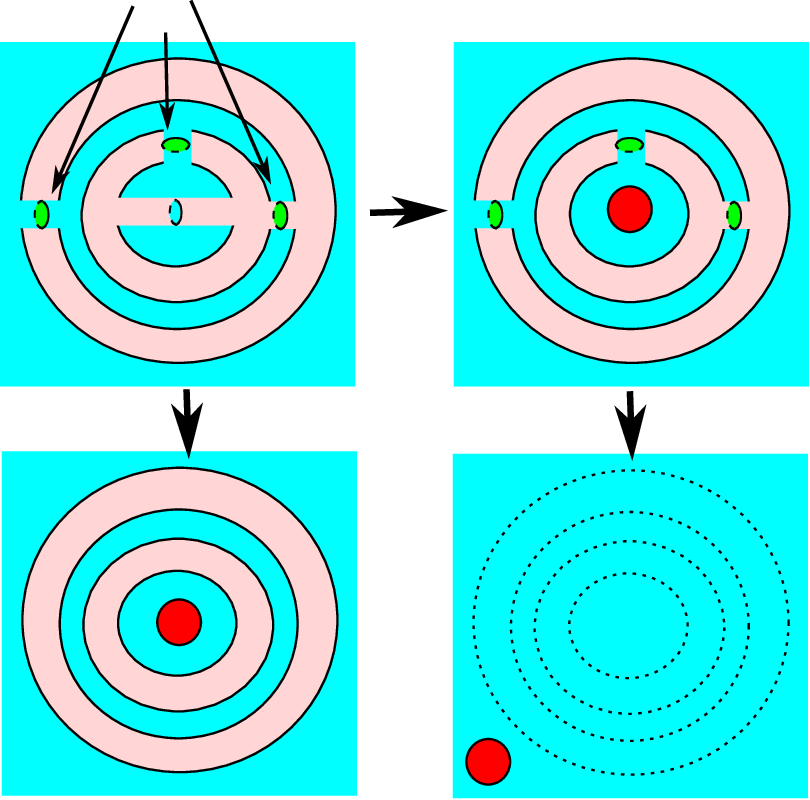}
    \caption{Bullseye deflation $\calC_{ey} \dashrightarrow \calC_{e+y}$}  \label{fig:eybull}
    \end{figure}

We are left with the case that every sphere in $\mathfrak{S}$ bounds a goneball in $\calC_{xy}$.  Any such sphere not containing $G$ in the interior of the ball it bounds will be a $y$-handlebody, since $E \subset B_G$, so it follows from Lemma \ref{lemma:FH2} that the ball will also be a goneball in $\calC_{x+y}$.  So all spheres in $\mathfrak{S}$ that don't bound goneballs in $\calC_{x+y}$ are nested around $G$.  Let $G'$ be an outermost sphere in this nested collection of spheres $\mathfrak{S}_e \subset \mathfrak{S}$ bounding goneballs in $\calC_{xy}$ but not in $\calC_{x+y}$, or $\bdd H$ itself if $H$ is a ball.  Since $G'$ bounds a goneball in $\calC_{xy}$, the ball has trivial Heegaard splitting.  It follows that if $\mathfrak{S}_e$ has more than one sphere, the collar between any two has trivial splitting.  Thus $\mathfrak{S}_e$ is a bullseye in $\calC_{x+y}$, the final possibility allowed by Lemma \ref{lemma:ey}.
\end{proof}

 \begin{figure}[ht!]
 \labellist
\small\hair 2pt
\pinlabel  $\calC$ at 220 560
\pinlabel  $\calC_x$ at 220 320
\pinlabel  $E$ at 380 320
\pinlabel  $\calC_{ey}$ at 230 80
\pinlabel  goneball at 120 80
\pinlabel  $\calD_y$ at 110 200
\pinlabel  $\calD_x$ at 250 500
\pinlabel  $\calC_{x}$ at 450 500
\pinlabel  $E$ at 100 430
\pinlabel  $E$ at 130 480
\pinlabel  $\calC_{xe}=\calC_{x+e}=\calC_{e+y}$ at 540 210
\endlabellist
    \centering
    \includegraphics[scale=0.6]{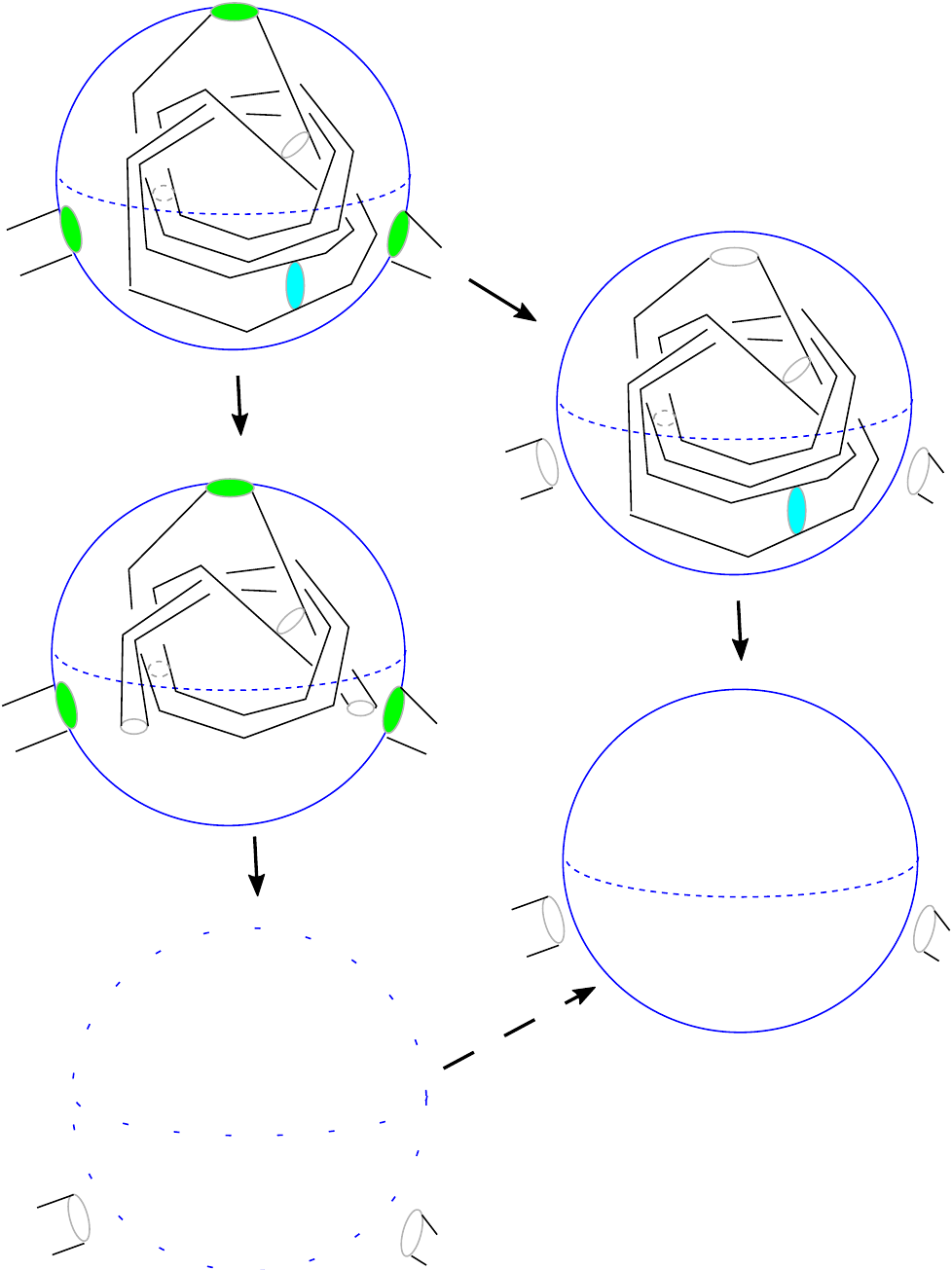}
    \caption{$H$ a ball, $F_e$ a punctured torus}  \label{fig:addiskv2_2}
    \end{figure}

\begin{figure}[ht!]
 \labellist
\small\hair 2pt
\pinlabel  $\calC$ at 220 560
\pinlabel  $\calC_x$ at 220 320
\pinlabel  $E$ at 380 320
\pinlabel  $\calC_{ey}$ at 240 90
\pinlabel  goneball at 120 80
\pinlabel  $\calD_y$ at 110 200
\pinlabel  $\calD_x$ at 250 500
\pinlabel  $\calC_x$ at 450 500
\pinlabel  $E$ at 100 430
\pinlabel  $E$ at 130 480
\pinlabel  $\calC_{xe}=\calC_{x+e}=\calC_{e+y}$ at 550 200
\endlabellist
    \centering
    \includegraphics[scale=0.5]{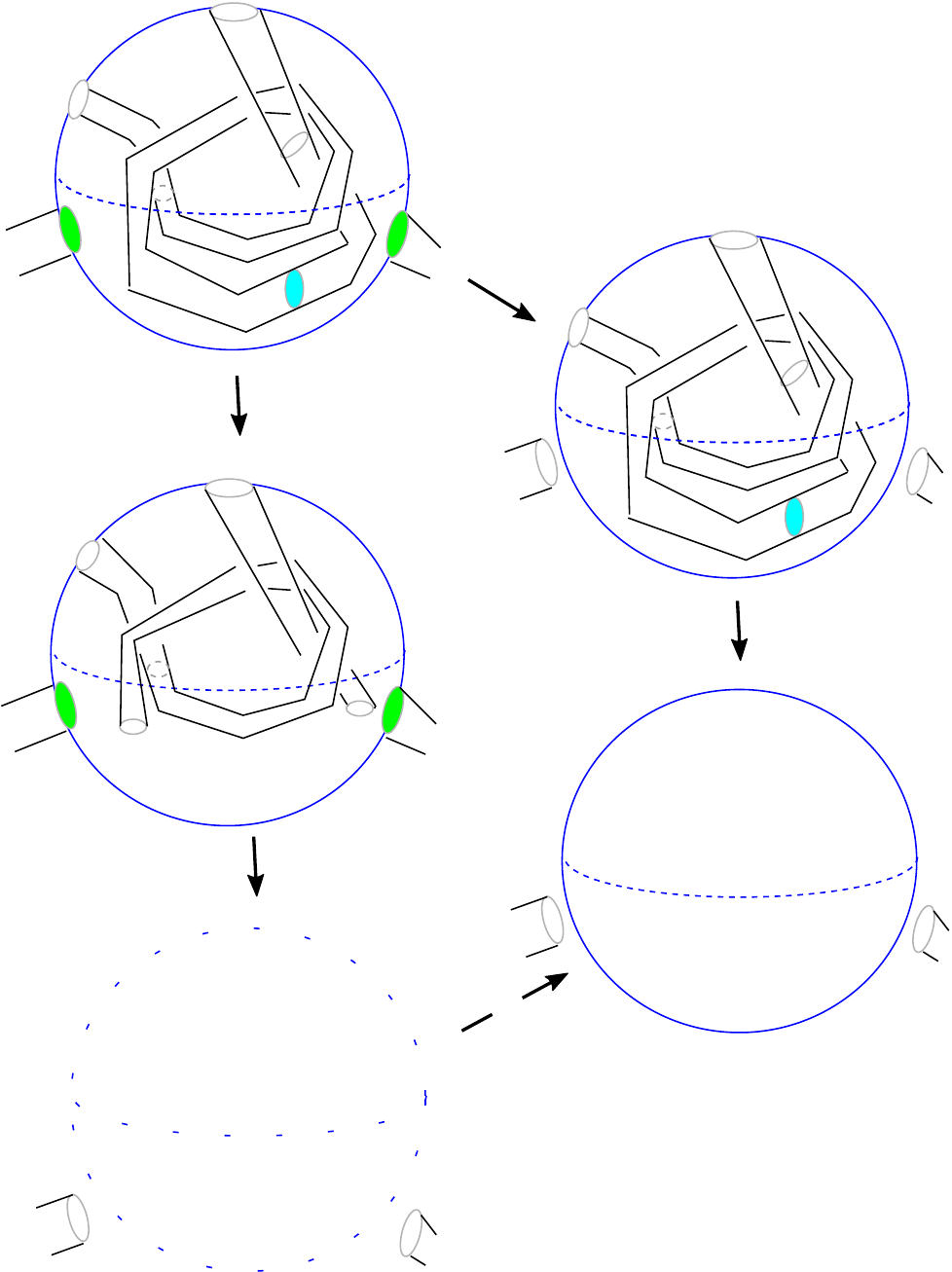}
    \caption{$H$ a ball, $F_e$ an annulus}  \label{fig:addiskv2_3}
    \end{figure}

The next lemma is more complicated to state, mostly because it has to accomodate the possibility that even if $\calD_y = E$, there may be two $xe$-handlebody chambers in $\calC_{x+y}$ and $\calC_{xy}$.  Namely if $E$ divides a handlebody chamber $C_x \subset \calC_x$ into two handlebodies, then they each are $xe$-handlebody chambers, since $E$ leaves an external scar on both of them.  According to Definition \ref{defin:prefdisk} if $C_x$ is an occupied handlebody, there are two preferred alignments of $E$: $H$ is empty and $H'$ is occupied, or vice versa.  The different alignments produce sibling flaggings of $\calC_{xy}$. 
One sibling is obtained from the other by deflation of $H'$ into $H$ (which is a bullseye deflation if $H'$ is then a goneball) or vice versa. In this case it is natural to call replacing the flagging of the chambers $H$ and $H'$ in $\calC_{x+y}$ with either of the flaggings in $\calC_{xy}$ a {\em sibling deflation}. 

 
\begin{lemma} \label{lemma:xe}  Suppose $\calC$ does not certify and $H$ is an $xe$-handlebody. 
Then either 
\begin{itemize}
\item $H$ is coherent, or 
\item $\calC_x$ has an occupied handlebody chamber containing $E$ and, in one of the two sibling deflations, $H$ is coherent, or
\item $H$  is a chamber in $\calC_{xy}$, either the deflation of an occupied chamber in $\calC_{x+y}$ or the result of the deflation of a single bullseye of $\calC_{x+y}$ in $H$.

\end{itemize}
Furthermore, in the third instance, when neither of the first two outcomes applies, 
\begin{itemize}
\item $E$ leaves exactly one scar on $\bdd H$,
\item $\bdd E$ is separating in a component of $F - \bdd \calD_x$, and
\item $E$ is not parallel to any disk in $\calD_x$.
\end{itemize}
\end{lemma}

\begin{proof}  If $E$ is parallel to a disk in $\calD_x$, Lemma \ref{lemma:paraldisk} says one of the first two outcomes occurs, and we are done.  So henceforth we assume that $E$ is not parallel to a disk in $\calD_x$.

Here are some preliminary observations:  
\begin{itemize}
\item If $H$ (or $H'$ in the case in which sibling deflation may arise) is coherent we are done.  So we assume $H$ (and $H'$, if relevant) is not coherent.
\item Following Lemma \ref{lemma:nonplaninH}, unless every component of $F_{\calD_x \cup E}$ lying in $\inter(H)$ is a sphere, $H$ is not a chamber in either $\calC_{xy}$ nor $\calC_{x+y}$ and so is coherent.  
So we henceforth assume that every component of $F_{\calD_x \cup E}$ lying in $\inter(H)$ is a sphere.  Let $\mathfrak{S}$ denote the collection of all such spheres.  
\item  Sibling deflation may arise when $E$ is a separating disk in an occupied handlebody chamber of $\calC_x$.  If that chamber has $H$ as one of the two remnants, each ball bounded by a sphere in $\mathfrak{S}$ is an $x$-handlebody and so, per Lemma \ref{lemma:FH2} the sphere is coherent. 
\end{itemize}

On the last point, sibling deflation, there is more to say:  
\medskip

{\em The Sibling Claim:} Suppose that $E$ is a separating disk in an occupied handlebody chamber $C_x$ of $\calC_x$.  In at least one of the two sibling deflations of $C_x$ in $\calC_{xy}$, either 
\begin{itemize}
\item $H$ is coherent, or 
\item $H$ is the deflation of an occupied chamber in $\calC_{x+y}$ or
\item the only remnant of $C_x$ in the interior of $H$ is a coherent ball.  
\end{itemize}. 

{\em Proof of the Sibling Claim:}  Since $\calC$ has no occupied handlebodies, the handlebody $\calC_x$ is occupied if and only if $F \cap \inter(C_x)$ contains a non-disk component, by Corollary \ref{cor:prefdisk}. Hence one or both of the handlebody chamber remnants $C_x - E$ are occupied in $\calC_{x+y}$.  It follows that either chamber (but not necessarily both) can be made coherent by a choice of sibling deflation in $\calC_{xy}$.  This proves the Sibling Claim if either of these chambers is $H$.  Another possibility is that both handlebody chambers lie in the interior of $H$, and so are necessarily balls.  We have seen then that at least one ball is occupied in $\calC_{x+y}$ and exactly one in $\calC_{xy}$.  Thus $H$ is not a chamber in either $\calC_{x+y}$ or $\calC_{xy}$ so $H$ is coherent.  

The final possibility is that one of the handlebody chamber remnants is a ball $B_H$ in the interior of $H$ and the other lies outside $H$ (in fact it is then the complement $S^3 - H$ since its boundary is connected).  $B_H$ is either empty or occupied in $\calC_{x+y}$.  By choosing the right sibling deflation we can match that status for $B_H$ in $\calC_{xy}$.  This makes $B_H$ coherent and so proves the Sibling Claim.

\medskip

It is possible that a sphere in $\mathfrak{S}$ bounds a ball that is not coherent - that is, it is a goneball in one of $\calC_{xy}$ or $\calC_{x+y}$ but not the other (regardless of the choice of sibling deflation, when $E$ lies in an occupied handlebody of $\calC_x$).  If this occurs, let $B_G$ be an innermost such ball, 
$B_G$ can't be an $x$-handlebody, by Lemma \ref{lemma:FH2} so $E$ is incident to $B_G$.  By the Sibling Claim, we may  assume that if $E$ lies in  an occupied handlebody in $\calC_x$, $B_G$ is not one of its remnants in $\calC_{xy}$.


In view of these possibilities, in the discussion below we will use the term {\em target manifold} \index{Target manifold} to refer either to $B_G$, if there is a $B_G$ as above, or to $H$ itself if there is no $B_G$.  In either case, any handlebody contained in the interior of the target manifold is a coherent ball.

Let $F_x = F(\calC_x)$ denote the defining surface of $\calC_x$. 
We consider possible scars that $E$ might leave on the boundary of the target manifold in $F_{\calD_x \cup E}$.  Only in the last case, when there is exactly one external scar on the boundary of the target manifold, does the possibility of sibling deflations arise:
\medskip

{\em Claim 1:}  There is a scar of $E$ on the boundary of the target manifold (either $G = \bdd B_G$ or $\bdd H$.)

The argument is essentially the same for either target manifold, so we assume here the target manifold is $B_G$.  If there is no scar on $G$ then $G$ is among the components of $F_{\calD_x}$, the surface obtained by $\calD_x$ surgery on $F$.  Notice that $B_G$ must be a chamber in both $\calC_{x+y}$ and $\calC_{xy}$, for if there are occupied balls in the interior of one, they would be occupied in the other (because $B_G$ is innermost among the non-coherent), making $B_G$ not a chamber and therefore coherent, contradicting hypothesis.   

If $F \cap \inter(B_G)$ consists of disks, then $B_G$ is disky and hence a goneball in both $\calC_{x+y}$ and $\calC_x$, by Corollary \ref{cor:prefdisk}.  Hence $B_G$ is also a goneball in $\calC_{xy}$, again contradicting that  $B_G$ is not coherent.  On the other hand, if $F \cap \inter(B_G)$ is not all disks.  Then $B_G$ is an occupied ball in both $\calC_{x+y}$ and $\calC_x$.  The latter implies that $B_G$ is also an occupied ball in $\calC_{xy}$, since $E$ is not incident to $B_G$; $B_G$ is its own and only remnant in $\calC_{xy}$.  This would make $B_G$ coherent, a final contradiction that proves Claim 1.

%
    
    \medskip

{\em Claim 2:} $E$ does not leave two external scars on the target manifold.  

Again the argument for target manifold $B_G$ applies also to $H$, so we can suppose that $B_G$ is the target manifold and consider what happens if $E$ leaves two external scars on $G$.    Then the chamber $C_x$ in $\calC_x$ of which $B_G$ is a remnant is obtained from $G$ by attaching a $1$-handle dual to $E$, and so is a handlebody in $\calC_x$.  $B_G$ is the sole remnant of $C_x$ under the decomposition $\calC_x \to \calC_{xy}$ and its interior is disjoint from $F_x$, so $B_G$ is occupied in $\calC_{xy}$ if and only if $C_x$ is occupied in $\calC_x$.  Since $\calC$ has no occupied handlebodies, the handlebody $\calC_x$ is occupied if and only if $F \cap \inter(C_x)$ contains a non-disk component, by Corollary \ref{cor:prefdisk}.  But $F \cap \inter(C_x)$ contains a non-disk component if and only if $F \cap \inter(B_G)$ contains a non-disk component and the latter is equivalent to $B_G$ being occupied in $\calC_{x+y}$.  Hence $B_G$ is occupied in $\calC_{xy}$ if and only if it is occupied in $\calC_{x+y}$.  Hence $B_G$ is coherent, a contradiction that proves the claim.

%
    
    \medskip

{\em Claim 3:} $E$ does not leave two internal scars on the target manifold.

The case when the target manifold is $B_G$ is representative.  If $E$ leaves two internal scars on $G$ then the belt annulus is a non-disk component of both $F \cap \inter(B_G)$ and $F_x \cap \inter(B_G)$, so $B_G$ is occupied in both $\calC_{x+y}$ and $\calC_{xy}$.  Hence $B_G$ is coherent, a contradiction that proves the claim.



    \medskip

{\em Claim 4:} $E$ does not leave one internal scar on the the target manifold

This is the most difficult claim to verify, for the argument involves some unusual computation.  To avoid the danger of over-simplifying we will assume that the more complicated $H$ is the target manifold; the argument when $B_G$ is the target manifold is then just a special case.  Suppose that $E$ leaves one internal scar on $\bdd H$.  Then the other scar must be on a component of $F_{\calD_x \cup E}$ that lies in the interior of $H$ and so is a sphere $S \in \mathfrak{S}$.  Put another way, in $F_{\calD_x}$ the belt annulus of $E$ connects the scar  of $E$ on $\bdd H$ to the disk complement in $S$ of the scar of $E$ on $S$.  So the scar of $E$ on $\bdd H$ is parallel in $H$ to a disk component $F_0$ of $F_x \cap \inter(H)$ (the union of the belt annulus of $E$ and a subdisk of $S$).  Since $E$ is the only disk in the decomposition $\calC_x \to \calC_{xy}$, the disk $F_0$ is the only component of $F_x \cap \inter(H)$.  

By assumption the ball $B_S$ that $S$ bounds in $H$ is coherent: unless it is a goneball in both, it persists in both $\calC_{x+y}$ and $\calC_{xy}$, so $H$ would not be a chamber in either and therefore is coherent, a contradiction.  So we henceforth assume that $B_S$ is a goneball in both $\calC_{x+y}$ and $\calC_{xy}$.  The former means that $F \cap \inter(B_S)$ consists of disks, and the latter means that the chamber $C_x$ of $\calC_x$ of which $H$ is a remnant is just the complement in $H$ of the ball (parallel to $B_S$) between the scar of $E$ in $\bdd H$ and $F_0$. That is, $H$ is parallel to $C_x$ via this ball.  In particular, $H$ is the only remnant of $C_x$ so, via Definition \ref{defin:prefdisk}, $H$ is occupied in $\calC_{xy}$ if and only if $C_x$ is occupied in $\calC_x$.

We wish now to examine $F \cap \inter(H)$.  $F_0$ is a disk containing scars of $\calD_x$; the complement of the scars is a subset of $F$.  In particular, if $F \cap \inter(H)$ contains any closed components they would be disjoint from $F_0$.  A closed component of $F$ can't lie in $B_S$, since $B_S$ is empty in $\calC_{x+y}$, so it would have to lie in the chamber $C_x$ of $\calC_x$, making $C_x$, hence $H$, occupied in $\calC_{xy}$ as well as $\calC_{x+y}$.  Then $H$ would be coherent, a contradiction.  So we need only consider the case in which $F \cap \inter(H)$ contains no closed components.  

%
%

Here is a useful observation about compact orientable surfaces with no closed components:  For any closed connected orientable surface $Q$, $\chi(Q) \leq 2$ with equality if and only if $Q$ is a sphere.  It follows that for any compact connected surface $Q$, $\chi(Q) + |\bdd Q| \leq 2$, with equality if and only if $Q$ is planar, and so, when $\bdd Q \neq \emptyset$, $\chi(Q) - |\bdd Q| \leq 0$ with equality if and only if $Q$ is planar and has a single boundary component, i. e. is a disk.  More generally, this is then true for any compact orientable surface, so long as there are no closed (in fact no sphere) components.  We have thus verified this claim:

{\em Claim:}  Suppose $Q$ is a compact orientable surface with no closed components.  Then $\chi(Q) - |\bdd Q| \leq 0$ with equality if and only if $Q$ consists entirely of disks.  

\begin{figure}[ht!]
 \labellist
\small\hair 2pt
\pinlabel  $C_x$ at 180 100
\pinlabel  $H$ at 450 120
\pinlabel  $B_S$ at 120 75
\pinlabel  $F\cap \inter(B_S)$ at 455 90
\pinlabel  $F\cap \inter(C_x)$ at 310 125
\pinlabel  $E$ at 140 30
\pinlabel  $F_0$ at 155 60
\pinlabel  $F\cap F_0$ at 435 60
\endlabellist
    \centering
    \includegraphics[scale=0.65]{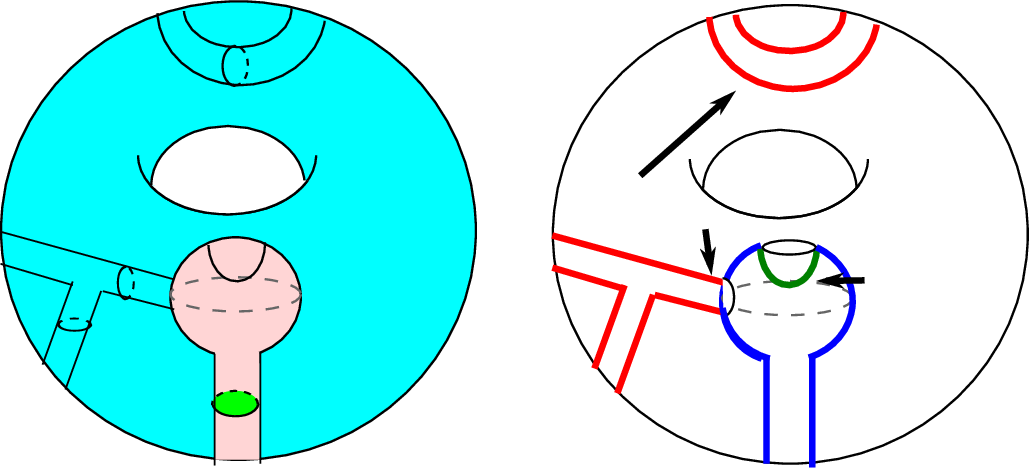}
    \caption{$F \cap \inter(H) = (F \cap \inter(C_x)) \cup (F \cap F_0) \cup (F \cap \inter(B_S))$}  \label{fig:xe}
    \end{figure}

To apply the Claim, begin with the observation 
\[F \cap \inter(H) = (F \cap \inter(C_x)) \cup (F \cap F_0) \cup (F \cap \inter(B_S)). \tag{1} \] See Figure \ref{fig:xe}.
For the purpose of this computation, include all boundary circles on each surface, so each surface is compact. We then have 
\[\chi(F \cap \inter(H)) - \chi(F \cap \inter(C_x)) = \chi(F \cap F_0) + \chi(F \cap \inter(B_S)). \tag{2}\] 
Note that $F \cap F_0$ is the disk $F_0$ with all the scars of $\calD_x$ that lie in $F_0$ removed.  Hence $ \chi(F \cap F_0) = 1 - |\text{scars on } F_0|$.  These scars on $F_0$ are of two types: those that come from components of $F$ lying in $\inter(B_S)$ (all of them disks, since $B_S$ is a goneball) and those coming from components of $F$ lying in $\inter(C_x)$.  Each disk in $F \cap \inter(B_S)$ contributes 1 to $\chi(F \cap \inter(B_S))$ and -1 to $\chi(F \cap F_0)$ so 
\[\chi(F \cap F_0) + \chi(F \cap \inter(B_S)) = 1 - |\text{scars on } F_0 \text{ left by }F \cap \inter(C_x)| \tag{3}.\]

Another way to characterize the number of scars left on $F_0$ by $F \cap \inter(C_x)$ is that it is the difference of the number of scars left on $\bdd C_x$ and those left on $\bdd H$.  Furthermore, the total number of scars left on $\bdd H$ is one more than the number left by $\bdd C_x$, since it includes $\bdd E$.  We then obtain:
\[1 - |\text{scars on } F_0 \text{ left by }F \cap \inter(C_x)| = |\bdd(F \cap \inter(H))| - |\bdd (F \cap \inter(C_x))| \tag{4}.\]

Combining (2), (3) and (4) we have 
\[\chi(F \cap \inter(H)) - \chi(F \cap \inter(C_x)) = |\bdd(F \cap \inter(H))| - |\bdd (F \cap \inter(C_x))| \tag{5}\]
which implies 
\[\chi(F \cap \inter(H)) - |\bdd(F \cap \inter(H))| =  \chi(F \cap \inter(C_x))- |\bdd (F \cap \inter(C_x))| \tag{6}.\]

Now apply the Claim: The left side of (6) is zero if and only if $H$ is disky and hence empty; the right side of (6) is zero if and only if $C_x$ is disky and hence empty.  Thus $H$ is empty in $\calC_{x+y}$ if and only if it is empty in $\calC_{xy}$.  That is, $H$ is coherent, a contradiction that proves Claim 4.

\medskip

{\em Claim 5:}   If the target manifold contains exactly one external scar of $E$ then $H$ satisfies one of the three outcomes listed in Lemma \ref{lemma:xe}.

The argument for this claim depends on whether $H$ or $B_G$ is the target manifold.  

{\em Case 1:} The target manifold is $B_G$.   

Let $C_x$ be the chamber in $\calC_x$ of which $B_G$ is a remnant.  Since $E$ leaves exactly one external scar on $G$, $E$ separates $C_x$ into $B_G$ and another remnant $R_x$, so 
$B_G$ is disky under the decomposition $\calC_x \to \calC_{xy}$.  
\medskip


{\em Subcase 1a:} $C_x$ is a handlebody.  

We first show that $C_x$ cannot be an empty handlebody.   Since $\calC$ contains no occupied handlebodies, $C_x$ can only be empty if $F \cap \inter(C_x)$ consists entirely of disks, by Corollary \ref{cor:prefdisk}.  Since $\inter(B_G) \subset \inter(C_x)$ this implies that $F \cap \inter(B_G)$ consists entirely of disks, so $B_G$ is empty in $\calC_{x+y}$.  Since $C_x$ is empty in $\calC_x$, Definition  \ref{defin:prefdisk} says $B_G$ is also empty in $\calC_{xy}$.  This contradicts the requirement that $B_G$ is not coherent.  

We deduce that $C_x$ is an occupied handlebody.  In this case, apply the Sibling Claim: the first two outcomes of that claim imply the first two outcomes of the Lemma; the third only arises when the target manifold is $H$, and then it contradicts Claim 4, completing the proof of the Lemma in this subcase.
%
%
%
\medskip

{\em Subcase 1b:} $C_x$ is not a handlebody.

In this case $R_x$ cannot be a handlebody and in particular cannot be a ball, and so it can't be a chamber in $\inter(H)$.  Then $R_x \subset S^3 - H$.  Moreover,
Definition  \ref{defin:prefdisk} says that in this case the preferred alignment of $E$ leaves $B_G$ empty in $\calC_{x+y}$ and so a goneball.  Our assumption is that $B_G$ is not coherent, so $B_G$ must be occupied in $\calC_{x+y}$.  The difference then between $H$ in $\calC_{x+y}$ and $\calC_{xy}$ is that the latter is obtained from the former by deflation of $B_G$.  This can be viewed as a bullseye deflation, the third outcome allowed.  
%
%
 See the right side of Figure \ref{fig:addisk2v_6}.
 \medskip
 
{\em Case 2:}  The target manifold is $H$.
 
 The proof is highly analogous to Case 1, so we merely sketch: 
 As in Case 1, the component $C_x$ of $\calC_x$ of which $H$ is a remnant is divided by $E$ into two chambers, $H$ and $R_x$.  If $C_x$ is an empty handlebody then, just as in Case 1, $H$ is empty in both $\calC_{x+y}$ and $\calC_{xy}$ and so is coherent, completing the proof.  If $\calC_x$ is an occupied handlebody then the Sibling Claim applies; the first two outcomes of that Claim imply the first two claims of the Lemma, and the third outcome contradicts Claim 4.   Finally, if $C_x$ (so also $R_x$) is not a handlebody, then the preferred alignment given by Definition  \ref{defin:prefdisk} leaves $H$ empty in $\calC_{xy}$.  Thus if $H$ is empty in $\calC_{x+y}$ it is coherent; if it is occupied then $\calC_{xy}$ is obtained from $\calC_{x+y}$ by deflation into $R_x$, implying the third outcome of the Lemma. 
 
 \begin{figure}[ht!]
 \labellist
\small\hair 2pt
\pinlabel  $\calC$ at 220 560
\pinlabel  $\calC_E$ at 220 320
\pinlabel  $E$ at 380 320
\pinlabel  $\calC_{ey}=\calC_{e+y}=\calC_{x+e}$ at 290 80
\pinlabel  goneball at 350 180
\pinlabel  $\calD_y$ at 110 200
\pinlabel  $\calD_x$ at 250 500
\pinlabel  $\calC_x$ at 450 500
\pinlabel  $E$ at 100 430
\pinlabel  $E$ at 30 520
\pinlabel  $\calC_{xe}$ at 480 250
\endlabellist
    \centering
    \centering
    \includegraphics[scale=0.6]{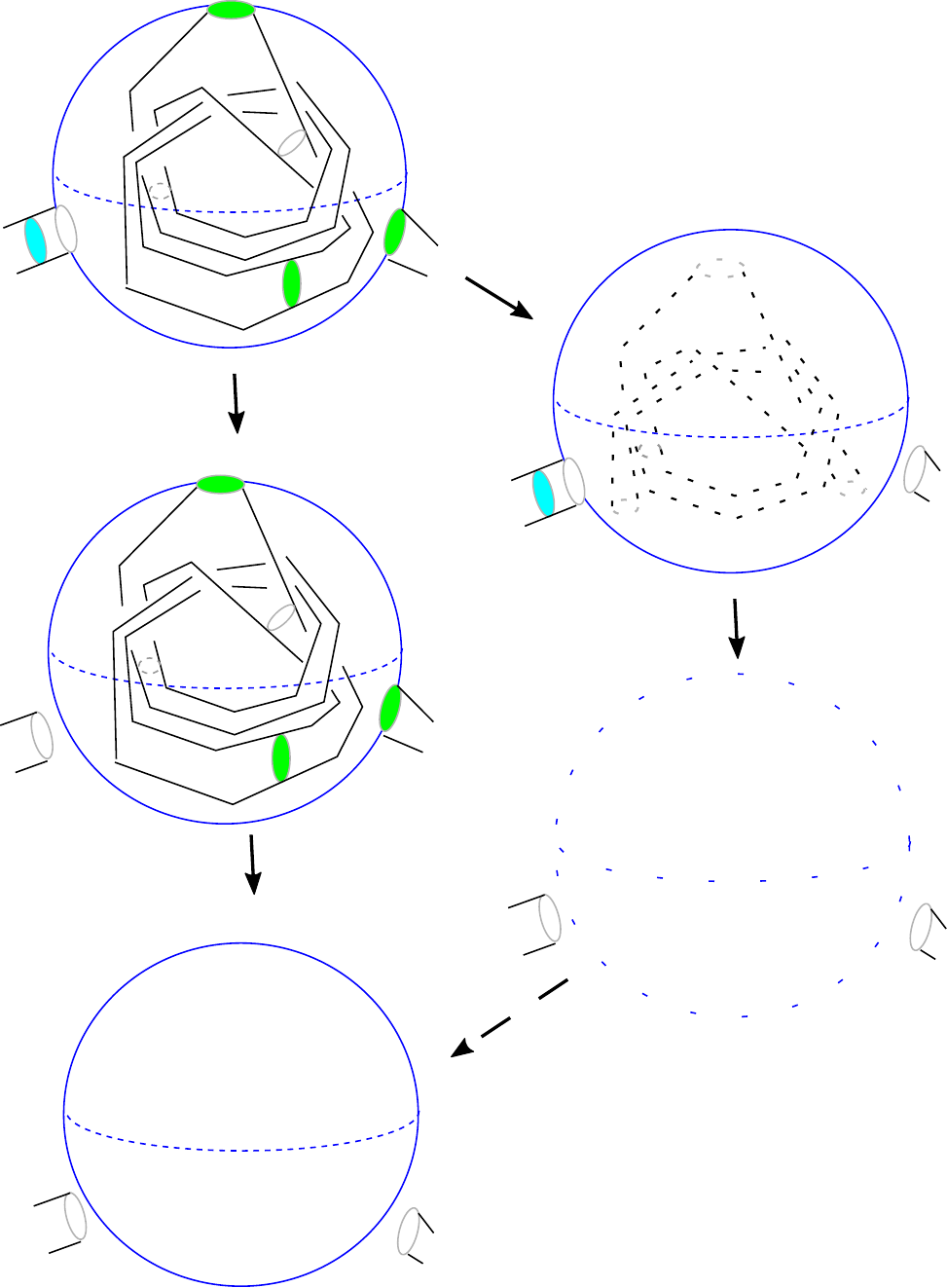}
    \caption{$E$ leaves one external scar}  \label{fig:addisk2v_6}
    \end{figure}
    
    \bigskip

\medskip

Having established the claims, we now prove the Lemma:  Suppose a ball $B_G \subset \inter(H)$ presents itself, as described before Claim 1.  Claim 1 says that $E$ leaves scars on $G$, claims 2 and 3 show that $E$ leaves exactly one scar on $G$, and Claim 4 says that it cannot be an internal scar.  Thus $E$ must leave one external scar on $G$, in which case Claim 5 shows that one of the three outcomes of Lemma \ref{lemma:xe} occurs.

On the other hand, if no $B_G$ satisfies the criteria described before Claim 1 then Claims 1, 2, 3 and 4 show that, unless $\bdd H$ contains a single external scar of $E$, $H$ is coherent.  Claim 5 shows that when $\bdd H$ does contain a single external scar of $E$, one of the three outcomes of Lemma \ref{lemma:xe} occurs. 
\medskip

We now examine when neither of the first two outcomes occurs, to verify the final three claims of Lemma \ref{lemma:xe}.  We have already concluded that $E$ is not parallel to any disk in $\calD_x$.  The proof so far shows also that the target manifold must have a single external scar, and then Claim 5 shows that in this case $\bdd H$ contains exactly one scar of $E$, an internal scar when the target manifold is $B_G$ and an external scar when the target manifold is $H$. 
All that remains is to show that $\bdd E$ is separating in $F - \bdd \calD_x$:

The definition of $F_x$ implies that the complement of the scars of $\calD_x$ in $F_x$ is a collection of components of $F - \bdd \calD_x$.  Not removing the scars of the disks in $\calD_x$, that is, adding disks back, does not affect the number of components in the surface, so to show that $\bdd E$ is separating in a component of $F - \bdd D_x$ it suffices to show that $\bdd E$ is separating in a component of $F_x$.  But that follows immediately from the fact that $E$ leaves only one scar on $\bdd H$: $\bdd H$ is a component of $F_x - \bdd E$ that is distinct from the component on which the other scar of $E$ lies.  
\end{proof}

\begin{prop} \label{prop:x+yvsxy}  Suppose $\calC$ is a flagged chamber complex supporting $S^3 = A \cup_T B$, $\calC$ does not certify and $\calD_x, \calD_y$ are disjoint disk sets in $\calC$. 
If either $|\calD_x|$ or $|\calD_y| = 1$, then, after an appropriate choice of sibling if $\calC_x$ contains occupied handlebody chambers, either $\calD_{xy} = \calD_{x+y}$ or $\calD_{xy} \dashrightarrow \calD_{x+y}$.  

In more detail:  Suppose $\calD \cup E$ is a disk set in $\calC$, where $E$ is a single disk.  Let \begin{itemize}
\item $\calC_{DE}$ be the flagged chamber complex given by the composition 
$\calC \xrightarrow{\calD} \calC_{\calD} \xrightarrow{E} \calC_{DE}$ 
\item $\calC_{ED}$ be the flagged chamber complex given by the composition 
$\calC \xrightarrow{E} \calC_{E} \xrightarrow{\calD} \calC_{ED}$
and $\calD$  the preferred alignment described in Lemma \ref{lemma:FH2}
\item $\calC_{D+E}$ be the flagged chamber complex given by the decomposition 
$\calC \xrightarrow{\calD \cup E} \calC_{D+E}$.
\end{itemize}
(Note that here $D$ in the subscripts refers not to a single disk but to the family $\calD$ of disks.)

Then either
\begin{itemize}
\item $\calC_{DE} = \calC_{D+E} = \calC_{ED}$ or
\item $\calC_{DE} = \calC_{D+E} \dashleftarrow \calC_{ED}$ or
\item $\calC_{DE} \dashrightarrow \calC_{D+E} = \calC_{ED}$ or
\item $E$ lies in an occupied handlebody in $\calC_{\calD}$ and its sibling $\calC'_{DE}$ satisfies 
$\calC'_{DE} = \calC_{D+E} = \calC_{ED}$
\end{itemize}
If $\calC_{D+E} \neq \calC_{ED}$ then either $\bdd E$ is non-separating in $F - \bdd \calD$ or $\bdd E$ is non-separating in $F$ and $E$ is parallel to a disk in $\calD$.
\end{prop}  

\begin{proof}  
The proof breaks naturally into two cases:

 {\em Case 1:}  $\calC_{ED} \neq \calC_{D+E}$.  
 
By Lemma \ref{lemma:cohertoequal} there is a handlebody $H$ in $S^3$ whose boundary is in $F_{\calD \cup E}$ and which is not coherent in $\calC_{ED}$ and $\calC_{D+E}$.  Then applying Lemma \ref{lemma:ey} to $H$, either $\bdd E$ is non-separating in $F - \bdd \calD$ or $\bdd E$ is non-separating in $F$ and $E$ is parallel to a disk in $\calD$.  Then by Lemma \ref{lemma:xe} every handlebody in $S^3$ whose boundary is in $F_{\calD \cup E}$ is coherent for $\calC_{DE}$ and $\calC_{D+E}$. So, by Lemma \ref{lemma:cohertoequal}, $\calC_{DE}  = \calC_{D+E}$. Thus what is left to show in this case is that $\calC_{ED} \dashrightarrow \calC_{D+E}$.  

There cannot be two disjoint handlebodies in $S^3$ whose boundaries are in $F_{\calD \cup E}$ and which are not coherent for $\calC_{ED}$ and $\calC_{D+E}$ since Lemma \ref{lemma:ey} says that the disk $E$ lies in the interior of each such handlebody.  Then there is a maximal such handlebody - that is, one that contains any other - so without loss take $H$ to be the maximal one.  By Lemma \ref{lemma:ey}  $H$ is contained in a chamber of $\calC_{ED}$ and is obtained from a chamber of $\calC_{D+E}$ by deflation of a chamber or of a bullseye in $\calC_{D+E}$.  All chambers of $\calC_{D+E}$ outside of $H$ are the same in $\calC_{D+E}$ and $\calC_{ED}$ by the same proof as that of Lemma \ref{lemma:cohertoequal}.  Thus $\calC_{ED} \dashrightarrow \calC_{D+E}$ as required.

 {\em Case 2:}  $\calC_{DE} \neq \calC_{D+E}$
 
The proof is analogous to that of Case 1, but a bit more complicated.   By Lemma \ref{lemma:cohertoequal} there is a handlebody $H$ in $S^3$ whose boundary is in $F_{\calD \cup E}$ and which is not coherent in $\calC_{DE}$ and $\calC_{D+E}$.  Then according to Lemma \ref{lemma:xe} $E$ is not parallel to a disk in $\calD$ and $\bdd E$ is separating in $F - \bdd \calD$.  Then by Lemma \ref{lemma:ey} every handlebody in $S^3$ whose boundary is in $F_{\calD \cup E}$ is coherent for $\calC_{ED}$ and $\calC_{D+E}$ so by Lemma \ref{lemma:cohertoequal} $\calC_{ED}  = \calC_{D+E}$. Thus what is left to show in this case is that $\calC_{DE} \dashrightarrow \calC_{D+E}$.  

According to Lemma \ref{lemma:xe} $E$ leaves exactly one scar on $\bdd H$.  If it is an internal scar, then per Claim 5 in the proof of Lemma \ref{lemma:xe}, $E$ lies in the interior of a chamber $C_x$ of $\calC_x$ dividing $C_x$ into a ball in $\inter(H)$ and a chamber $R_x$ adjacent to and outside of $H$, as in Figure \ref{fig:xe}.  Unless $R_x$ is also a handlebody, the choice of preferred alignment of $E$ has no effect on the flagging of chambers outside of $H$. In particular, there can be no handlebody chamber incoherent for $\calC_{DE}$ and $\calC_{D+E}$ that is disjoint from $H$.    In this case, the argument concludes as for Case 1, using Lemma \ref{lemma:xe} instead of Lemma \ref{lemma:ey}.  On the other hand, if $R_x$ is a handlebody (and so, as noted in Lemma \ref{lemma:xe} Claim 5, Subcase 1a, the only chamber of $\calC_{x+y}$ that lies outside of $H$) then the choice of alignment of $E$ in $C_x$ does affect the flagging of $R_x$.  The argument in this case is much like the argument in the case that $E$ leaves an external scar on $H$, so we turn to that case, noting that the argument applies here as well.  

 If $E$ leaves an external scar on $H$ then it is possible that there is a disjoint handlebody $H'$ in $S^3$ just like $H$: $\bdd H'$ is in $F_{\calD \cup E}$, $H'$ is not coherent for $\calC_{DE}$ and $\calC_{D+E}$, and $E$ also leaves an external scar on $\bdd H'$. Then, just as is true for $H$, $H'$ lies in a handlebody chamber in $\calC_{DE}$.  Thus there is a handlebody chamber $C_x \in \calC_{\calD}$ containing $E$, with $E$ dividing $C_x$ into $H$ and $H'$. 
 
 We now reprise part of the argument of Claim 4a in Lemma \ref{lemma:ey}:  $C_x$ must be an occupied handlebody, for if $C_x$ is empty then $H$ and $H'$ would both be empty in $\calC_{D+E}$  and $\calC_{DE}$, contradicting the assumption that $H$ is not coherent.  Then, per Definition \ref {defin:prefdisk}, $C_x$ is the parent of two siblings in $\calC_{DE}$: one in which $H$ is occupied and $H'$ is empty, the other in which $H'$ is occupied and $H$ is empty.  
 
Since $C_x$ is occupied, $F \cap \inter(C_x)$ has non-disk components by Corollary \ref{cor:prefdisk}. There are three cases to consider, depending on where these non-disk components lie:
\begin{itemize}
\item If some non-disk components lie in $H$ and some in $H'$, then $H$ and $H'$ are both occupied in $\calC_{D+E}$.  Since $H$ is not coherent, it must then be empty in $\calC_{DE}$.  In this case the sibling for $C_{DE}$ is the one in which $H$ is empty and $H'$ is occupied, that is the one that results from an alignment for $E$ which deflates $H$ into $H'$, leaving $H'$ occupied in $\calC_{DE}$ as well.  Thus $\calC_{DE} \dashrightarrow \calC_{D+E}$, an allowed outcome.  
\item If each component of $F \cap \inter(H)$ is a disk, then $H$ is empty in $\calC_{D+E}$.  Since $H$ is not coherent it  is then occupied in $\calC_{DE}$.  But in this case the sibling $C'_{DE}$ of $C_{DE}$ is one in which $H$ is empty and $H'$ is occupied.  Then in this sibling   $\calC'_{DE} = \calC_{D+E}$,  an allowed outcome.
\item If each component of $F \cap \inter(H')$ is a disk then $H'$ is empty and $H$ is occupied in $\calC_{D+E}$.  Since $H$ is not coherent, it is empty in $\calC_{DE}$.   In this case the sibling $C'_{DE}$ of $C_{DE}$ has $H$ occupied and $H'$ empty, just as is true for $\calC_{D+E}$.  Thus $\calC'_{DE} = \calC_{D+E}$, an allowed outcome.
\end{itemize}
\end{proof}

\begin{cor} \label{cor:delaydisk}
Suppose $\calC$ is a flagged chamber complex supporting $S^3 = A \cup_T B$, $\calC$ does not certify and  $\calD \cup E$ is a disk set in $\calC$, where $E$ is a single disk.  In the notation of Proposition \ref{prop:x+yvsxy}, one of these is true (for some choice of sibling, if siblings exist):
\begin{itemize}
\item  $\calC_{DE} = \calC_{ED}$,
\item $\calC_{DE} \dashrightarrow \calC_{ED}$ or
\item $\calC_{DE} \dashleftarrow \calC_{ED}$
\end{itemize}
\end{cor}

We denote the three possibilities above by $\calC_{DE} \dashleftrightarrow \calC_{ED}$. \index{$\dashleftrightarrow$} Then one way of expressing the conclusion of Corollary \ref{cor:delaydisk} is to say that this diagram commutes:   

\[\xymatrix {\calC_0  \ar[dd]_{E} \ar[r]^{\calD} & \calC_{\calD} \ar[d]^{E} \\ 
&  \calC_{DE} \ar@{<-->}[d] \\
\calC_E  \ar[r]^{\calD}&\calC_{ED}    } \]

\section{Guiding disk sets and disk addition} \label{sect:guidingisk}

Recall from Section \ref{sect:deflate} the notion of guiding disks (Definition  \ref{defin:guidedisk}) for flagged chamber decompositions.  The example given there involved decomposition of two chamber complexes that are related by bullseye deflation.   Here is another example:  Suppose $\calD \cup E$ is a disk set in a flagged chamber complex $\calC$.  Let $\calC'$ be the result of the flagged chamber complex decomposition $\calC \xrightarrow{E} \calC'$ and let $\calD'$ be the set of disks in $\calD$ whose boundaries lie on $\calC'$.  The decomposing set $\calD'$ is contained in $\calD$, but may not be all of $\calD$, because it does include those disks in $\calD$ that lie on goneballs of the decomposition $\calC \xrightarrow{E} \calC'$.  Nevertheless $\calD$ remains as a guiding set of disks for the decomposition $\calC' \xrightarrow{\calD'} \calC'_{\calD'}$. 

We have already encountered another example of guiding disk sets, in the context of guiding spheres in Section \ref{sect:guiding}.  Suppose $\calC$ is a flagged chamber complex in $S^3$ with defining surface $F = F(\calC)$, and $S \subset S^3$ is an embedded sphere transverse to $F$.  Consider a sequence $\vec{\calC}$ of flagged chamber complex decompositions  guided by $S$, as given in Definition \ref{defin:Didef} and described after the proof of Lemma \ref{lemma:Didef}:  The circles $F \cap S$ are partitioned into sets $\mathfrak{c}_i, 0 \leq i \leq n$ and to each $\frc_i$ a collection of disjoint disks $\ocalD_i \subset S$ is assigned.      
The associated sequence $\vec{\calC}$ of flagged chamber complex decompositions guided by $S$ has, as each disk set $\calD_i$ those disks in $\ocalD_i$ that remain incident to $F(\calC_i)$ after all goneballs in earlier decompositions have been removed.  Thus each set of disks $\ocalD_i$ is a guiding set of disks for the decomposition $\calC_i \xrightarrow{\calD_i} \calC_{i+1}$ not the disk set itself.  Consistent with the notation of Definition \ref{defin:guidedisk} we could then write the sequence in Definition \ref{defin:Didef} as 
\[
\calC = \calC_0 \xrightarrow{\ocalD_0} \calC_1 \xrightarrow{\ocalD_1} \calC_2 \xrightarrow{\ocalD_2} ... \xrightarrow{\ocalD_{k-1}}\calC_k\]


Joining this example with the previous one, suppose $E$ is a properly embedded disk in $\calC$ that is disjoint from the guiding sphere $S$ and, as above, let $\calC'$ be the result of the flagged chamber complex decomposition $\calC \xrightarrow{E} \calC'$.  Then it is natural to construct a sequence $\vec{\calC'}$ of flagged chamber complex decompositions
\[\vec{\calC'}:\quad
\calC' = \calC'_0 \xrightarrow{\calD'_0} \calC'_1 \xrightarrow{\calD'_1} \calC'_2 \xrightarrow{\calD'_2} ... \xrightarrow{\calD'_{n-1}}\calC'_n \tag{\ref{sect:guidingisk}.1}\] by setting each $\calD'_i$ to be the collection of disks in $\calD_i$ (and so in $\ocalD_i$) whose boundaries lie on $\calC'_i$.  Then, just as for $\vec{\calC}$, each decomposition $\calC'_i \xrightarrow{\calD'_i} \calC'_{i+1}$ has guiding disk set $\ocalD_i$.   So the sequence could also be written
\[\vec{\calC'}:\quad
\calC' = \calC'_0 \xrightarrow{\ocalD_0} \calC'_1 \xrightarrow{\ocalD_1} \calC'_2 \xrightarrow{\ocalD_2} ... \xrightarrow{\ocalD_{n-1}}\calC'_n.\]
%
%

One way to describe the relation between the sequences above is that the sequence $\vec{\calC}'$ is what the sequence $\vec{\calC}$ would be if we had inserted decomposition with the disk $E$ before beginning the sequence.    Generalizing, we could decompose by $E$ at a later stage of the decomposition sequence.  That is, for each $0 \leq k \leq n$, define a flagged chamber complex decomposition sequence 
\[\vec{\calC}_E^k: \calC_0 \xrightarrow{\calD_0} \calC_1 \xrightarrow{\calD_{1}} ... \xrightarrow{\calD_{k-1}} \calC_k \xrightarrow{E} \calC^k_k 
\xrightarrow{\calD^k_k} \calC^k_{k+1} \xrightarrow{\calD^k_{k+1}} ... \xrightarrow{\calD^k_{n-1}}\calC^k_n \tag{\ref{sect:guidingisk}.2}\] where the disks $\calD^k_i \subset \ocalD_i, k \leq i \leq n$ are those disks whose boundaries lie on $F(\calC^k_i)$.  In particular, $\vec{\calC'}$ of Sequence \ref{sect:guidingisk}.1 is $\vec{\calC}_E^0$ with the first decomposition $\calC \xrightarrow{E} \calC'$ dropped.  As usual, a simplified way of writing each sequence is
\[\vec{\calC}_E^k: \calC_0 \xrightarrow{\ocalD_0} \calC_1 \xrightarrow{\ocalD_{1}} ... \xrightarrow{\ocalD_{k-1}} \calC_k \xrightarrow{E} \calC^k_k 
\xrightarrow{\ocalD_k} \calC^k_{k+1} \xrightarrow{\ocalD_{k+1}} ... \xrightarrow{\ocalD_{n-1}}\calC^k_n \]

\begin{lemma} \label{lemma:coguide}  Suppose $\calC$ and $\calC'$ are flagged chamber complexes and their defining surfaces have the property that $F(\calC') \subset F(\calC)$.  Suppose $\calC \xrightarrow{\calD} \calC_{\calD}$ and $\calC' \xrightarrow{\calD'} \calC'_{\calD'}$ have a guiding disk set $\ocalD$ in common.  Then $\calD$ is a guiding disk set for $\calD'$.
\end{lemma}

\begin{proof} Each $D' \in \calD'$ is a disk in $\ocalD$ whose boundary lies on $F(\calC')$.  Since $F(\calC') \subset F(\calC)$, $\bdd D'$ also lies on $F(\calC)$, so $D' \in \calD$.  Thus $D'$ is a disk in $\calD$ whose boundary lies in $F(\calC')$.  Conversely, suppose $D \in \calD$ has its boundary on $F(\calC')$.  Since $D \in \calD$, $D \in \ocalD$ and, since $\ocalD$ is a guiding disk set for $\calD'$ and $\bdd D \subset F(\calC')$, $D \in \calD'$.  Thus $\calD'$ consists exactly of those disks in $\calD$ whose boundaries lie on $F(\calC')$, as required.
\end{proof}

\begin{prop} \label{prop:Eaddcocert}  Let $\calC$ be a flagged chamber complex that supports $S^3 = A \cup_T B$ and \[\vec{\calC}: \calC = \calC_0 \xrightarrow{\calD_0} \calC_1 \xrightarrow{\calD_1} ... \xrightarrow{\calD_{n-1}}\calC_n\] be a flagged chamber complex decomposition sequence.  Suppose
\begin{itemize}
\item there is a guiding disk set $\ocalD_i$ for each decomposition $\calC_i \xrightarrow{\calD_i} \calC_{i+1}$.  
\item $E$ is a proper disk in a chamber of $\calC$, and $E$ is disjoint from all disks in all $\ocalD_i$
\item $\vec{\calC}'$ is the sequence defined in Formula \ref{sect:guidingisk}.1 above, so each $\ocalD_i$ is a guiding disk set for $\calC'_i \xrightarrow{\calD'_i} \calC'_{i+1}$
\item for each $0 \leq k \leq n$ and decomposition sequence $\vec{\calC}_E^k$ defined in \ref{sect:guidingisk}.2 above each $\ocalD_i$ is also a guiding disk set for the corresponding decomposition in $\vec{\calC}_E^k$. 
\end{itemize} 
If $\calC_n$, $\calC'_n$ and all $\calC_n^k, 0 \leq k \leq n$
certify, then the sequences $\vec{\calC}$ and $\vec{\calC}'$ cocertify.  
\end{prop}

\begin{proof}  
{\em Case 1:}  $\calC$ certifies.

Then the sequence $\vec{\calC}$ and the sequence 
\[\vec{\calC}_E^0: \calC \xrightarrow{E} \calC' = \calC'_0 \xrightarrow{\calD'_0} \calC'_1 \xrightarrow{\calD'_1} ... \xrightarrow{\calD'_{n-1}}\calC'_n\] cocertify because the chamber complex $\calC$ certifies for both. (The same argument shows that all sequences $\vec{\calC}_E^k$ cocertify.)  Similarly the sequences $\vec{\calC}_E^0$ and $\vec{\calC}'$ cocertify because, by hypothesis, $\calC'_n$ certifies and, as noted above, the latter sequence is contained in the former sequence (by dropping $\calC \xrightarrow{E} \calC'$).  Thus we have  $\vec{\calC}\sim \vec{\calC}_E^0 \sim \vec{\calC}'$ as required.
\bigskip

{\em Case 2:} $\calC$ does not certify, but $\calC_1$ does.   

Since $\calC$ does not certify then per Corollary \ref{cor:delaydisk}, with some choice of siblings if there are siblings, the following diagram commutes:
\[\xymatrix {\calC  \ar[dd]_{E} \ar[r]^{\calD_0} & \calC_{1} \ar[d]^{E}  \ar[r]^{\calD_1} & ... & \ar[r]^{\calD_{n-1}} & \calC_n & = \vec{\calC}\\ 
&  \calC^1_1 \ar@{<-->}[d]  \ar[r]^{\calD^1_1} & ... & \ar[r]^{\calD^1_{n-1}}& \calC^1_n & = \vec{\calC}_E^1 \\
\calC'  \ar[r]^{\calD'_0}&\calC'_1  \ar[r]^{\calD'_1} & ... &  \ar[r]^{\calD'_{n-1}}& \calC'_n & = \vec{\calC}'  } \]
We henceforth ignore the issue of choosing siblings, since different choices of sibling give cocertifying sequences, via the certifying parent chamber.  (See the proof of Proposition \ref{prop:guidecertify}.)

We would like to apply Corollary \ref{cor:longdeflate} to the sequences $\vec{\calC'}$ and $\vec{\calC}_E^1$.  The Corollary requires that if $\calC'_1$, say, deflates to $\calC^1_1$ (i. e. $\calC^1_1 \dashrightarrow \calC'_1)$  that the disks in $\calD_1^1$ are a subset of the disks in $\calD'_1$, namely those that have their boundaries on the defining surface for $\calC_1^1$.  In our language, $\calD'_1$ needs to be a guiding disk set for $\calD^1_1$. Moreover, it is required that the same true for each $\calD'_i$ and $\calD^1_i$ throughout the maximal deflationary sequence. But since, by hypothesis, $\ocalD_i$ is a guiding set of disks for both $\calD'_i$ and $\calD^1_i$, this is true, via Lemma \ref{lemma:coguide}.  

So we can apply Corollary \ref{cor:longdeflate}: Since $\vec{\calC}^1_E$ certifies by the hypothesis that $\calC_n^1$ certifies, $\vec{\calC}^1_E$ and $\vec{\calC}'$ cocertify.  A symmetric argument leads to the same conclusion in the case that $\calC^1_1$ deflates to $\calC'_1$ or, most simply, when $\calC^1_1 = \calC'_1$. 
Since, in this case, we are assuming that $\calC_1$ certifies, it certifies for the sequences that contain it, namely $\vec{C}$, and $\vec{\calC}_E^1$.  So we have $\vec{\calC} \sim \vec{\calC}_E^1 \sim \vec{\calC}'$ , as required.  
\medskip

We are left with the possibility that $\calC_1$ does not certify.  The argument just concluded suggests an inductive approach, which we now pursue:   

{\em Inductive Step:} Suppose the sequence $\vec{\calC'}$ and, for $0 \leq k \leq \ell$, all $\vec{\calC}_E^k$ cocertify.  (The proof of Cases 1 and 2 show that this is true for $\ell = 1$.) Then either $\vec{\calC} \sim \vec{\calC}'$ and we are done, or the same statement is true with $\ell$ replaced by $\ell +1$.  

If $\calC_\ell$ certifies, then it certifies for sequences that contain it, including $\vec{\calC}$, and $\vec{\calC}_E^\ell$.  Thus $\vec{\calC} \sim \vec{\calC}_E^\ell$.  By inductive hypothesis,  $\vec{\calC}_E^\ell \sim \vec{\calC}'$.  Hence $\vec{\calC} \sim \vec{\calC}'$, as required.

If $\calC_\ell$ does not certify, Corollary \ref{cor:delaydisk} says the following diagram commutes:

\[\xymatrix {\calC \ar[r]^{\calD_0} &\calC_1 \ar[r]^{\calD_1}&...\ar[r]^{\calD_{\ell-1}}&\calC_\ell  \ar[dd]_{E} \ar[r]^{\calD_\ell} & \calC_{\ell+1} \ar[d]^{E}  \ar[r]^{\calD_{\ell+1}} & ... & = \vec{\calC}\\ 
&&&&  \calC^{\ell+1}_{\ell+1} \ar@{<-->}[d]  \ar[r]^{\calD^{\ell+1}_{\ell+1}} & ... & = \vec{\calC}_E^{\ell+1} &\\
&&&\calC^\ell_\ell  \ar[r]^{\calD^\ell_\ell}&\calC^\ell_{\ell+1}  \ar[r]^{\calD^\ell_{\ell+1}} & ... & = \vec{\calC}_E^{\ell}  & } \]

Applying Corollary \ref{cor:longdeflate} much as in Case 2, $\vec{\calC}_E^{\ell}$ and $\vec{\calC}_E^{\ell+1}$ cocertify.  Hence $\vec{C}'$ and, for $0 \leq k \leq \ell+1$, all $\vec{\calC}_E^k$ cocertify, completing the inductive step.
\medskip

Continue to iterate the inductive step, showing that $\vec{\calC'}$ and each $\vec{\calC}_E^{k}$ cocertify, until a $k$ for which $\calC_k$ certifies is reached.  We know that we will eventually reach such a $k$, since $\vec{\calC}$ certifies.  At that point,  $\vec{\calC}_E^{k}$ cocertifies with both $\vec{\calC}$ and $\vec{\calC'}$, so $\vec{\calC}$ and $\vec{\calC}'$ cocertify, as required.
\end{proof}

\begin{thm}   \label{thm:addisk} Suppose $\calC$ is a flagged chamber complex in $S^3$ that is not tiny, and $S$ is a balanced or almost balanced sphere transverse to $F(\calC)$.  Let $E$ be a disk disjoint from $S$ properly embedded and essential in a chamber of $\calC$ that is not an empty torus.  Let $\calC'$ be the flagged chamber complex obtained from $\calC$ by decomposition along $E$.  Suppose also that $S$ is balanced or almost balanced for $\calC'$.  Then any sequence in $\overrightarrow{(\calC, S)}$ and any sequence in $\overrightarrow{(\calC', S)}$ cocertify.


\end{thm}

Note: the requirements that $E$ be essential in its chamber and that $C$ not be an empty torus will be removed later, see Corollary \ref{cor:addisk}.

\begin{proof}  Let $F, F'$ denote the defining surfaces $F(\calC), F(\calC')$ respectively.  The circles $F \cap S$ and $F' \cap S$ determine  decomposition sequences $\vec{\calC}$ and $\vec{\calC'}$ as described in Section \ref{sect:guiding}.  
Since $E$ is essential and not the meridian of an empty torus, surgery on $E$ creates no new goneballs, so in fact $F \cap S = F' \cap S$.  Thus the guiding disk sets $\ocalD_i$ for the two decomposition sequences $\vec{\calC}$ and $\vec{\calC'}$ are the same. Import the notation for $\vec{\calC}$ and $\vec{\calC'}$ from the paragraph that precedes Proposition \ref{prop:Eaddcocert} to these decomposition sequences $\vec{\calC}$ and $\vec{\calC'}$.

We are given that the chamber complexes $\calC$ and $\calC'$, related by the decomposition $\calC \xrightarrow{E} \calC'$, are each balanced or almost balanced.  
As noted in the proof of Proposition \ref{prop:balance} during the decomposition sequence the genus of the defining surface above $S$ and below $S$ does not change. Surgery on $E$ may or may not lower by one the genus of the defining surface on the side of $S$ in which $E$ lies, but we are given that $S$ is balanced or almost balanced for $\calC'$ as well as for $\calC$.  It follows that for any $0 \leq k \leq m$ the chamber complex $\calC^k_k$ defined by the decomposition $\calC_k \xrightarrow{E} \calC^k_k$ in $\vec{\calC}_E^k$ is also balanced or almost balanced.  

Then Propositions \ref{prop:balance} and \ref{prop:almostbalance} imply that the last term in each of the sequences $\vec{\calC},  \vec{\calC}$ and each $\vec{\calC}_E^k$ certifies.  Proposition \ref{prop:Eaddcocert} then says that $\vec{\calC}$ and $\vec{\calC'}$ cocertify, as claimed. 
\end{proof}

\section{Delaying a disk} \label{sect:diskdelay}

In this section we consider what happens in a sequence of two flagged chamber complex decompositions if, when this is possible, a decomposition disk in the first decomposing disk set is transferred to the second decomposing disk set.

Suppose $\calC_0$ is a flagged chamber complex supporting $S^3 = A \cup_T B$, and \[\calC_0 \xrightarrow{\calD_0} \calC_1 \xrightarrow{\calD_1} \calC_2\] is a sequence of flagged chamber complex decompositions.  Suppose $E$ is a properly embedded disk in  a chamber of $\calC_0$ and $E$ is disjoint from both $\calD_0$ and $\calD_1$.  There are a variety of decomposition sequences that can be defined in this situation.  Here is the notation we will use for them: 
\begin{itemize}
\item $\calC_0 \xrightarrow{\calD_0} \calC_1 \xrightarrow{E} \calC_{0E}$
\item $\calC_1 \xrightarrow{E} \calC_E \xrightarrow{\calD'_1} \calC_{E1}$
\item $\calC_0 \xrightarrow{\calD_0 \cup E}  \calC_{0+E}$
\item $\calC_1 \xrightarrow{E \cup \calD_1}  \calC_{E+1}$
\end{itemize}

In the second of these sequences $\calD'_1$ has the usual meaning: those disks in $\calD_1$ whose boundaries lie on the defining surface of $\calC_E$ (and so do not lie on the boundary of goneballs after decomposition by $E$).  Note that if $\bdd E$ lies on the boundary of a goneball in $\hat{\calC}_1$, then it does not lie on the defining surface for $\calC_1$, so decomposing on a disk set containing $E$ is the same as decomposing on the disk set without $E$.  So in this situation $\calC_{E1} = \calC_{E+1} = \calC_2$.  

\begin{lemma}  \label{lemma:delayE1}  Suppose neither $\calC_0$ nor $\calC_1$ certify.  Then, up to a choice of siblings, if siblings exist, either
\begin{enumerate}
\item $\calC_{0E} = \calC_{0+E}$ and  $\calC_{E1} = \calC_{E+1}$ or
\item $\calC_{0E} \dashrightarrow \calC_{0+E}$ and  $\calC_{E1} = \calC_{E+1}$ or
\item $\calC_{0E} = \calC_{0+E}$ and  $\calC_{E1} \dashrightarrow \calC_{E+1}$.
\end{enumerate}
\end{lemma}

\begin{proof}  Let $F_0, F_1$ be the defining surfaces for respectively $\calC_0, \calC_1$.  Suppose first that $\bdd E$ is separating in a component of $F_0 - \bdd \calD_0$.  Then it is separating in a component of $F_1$, because $F_1$ is obtained from $F_0 - \bdd \calD_1$ simply by adding (scar) disks to $\bdd \calD_1$ and perhaps deleting some spheres.  Then $\bdd E$ (if it is not on the boundary of a goneball) is also separating in a component of $F_1 - \bdd D_1$.  It follows from the last sentence of Proposition \ref{prop:x+yvsxy} 
that $\calC_{E1} = \calC_{E+1}$ and, further applying Proposition \ref{prop:x+yvsxy}, that outcome (1) or (2) of Lemma \ref{lemma:delayE1} occurs.  

On the other hand, if $\bdd E$ is non-separating in $F_0 - \bdd \calD_0$ then by Lemma \ref{lemma:xe} $\calC_{0E} = \calC_{0+E}$. Then by Proposition \ref{prop:x+yvsxy}, outcome (1) or (3) occurs. 
\end{proof}

We now put this in the context of decomposition sequences.  Suppose $\ocalD_1, ..., \ocalD_n$ are a sequence of guiding disk sets in $S^3$ so that $\ocalD_1$ determines the decomposition $\calC_1 \xrightarrow{E \cup \calD_1}  \calC_{E+1}$ above.  Then the guiding disk sets $\ocalD_i$ determine extensions of the three sequences that we label as follows:

\[\vec{\calC}^{0+E}: \calC_0 \xrightarrow{\calD_0 \cup E} \calC_{0+E} \xrightarrow{\ocalD_1} \calC_2^{0+E} \xrightarrow{\ocalD_2} \calC^{0+E}_3 \xrightarrow{\ocalD_3}... \xrightarrow{\ocalD_{n-1}} \calC^{0+E}_n\]
\[\vec{\calC}^{E1}: \calC_0 \xrightarrow{\calD_0} \calC_{1} \xrightarrow{E} \calC_{0E} \xrightarrow{\ocalD_1} \calC_2^{0E} \xrightarrow{\ocalD_2} \calC^{0E}_3 \xrightarrow{\ocalD_3}... \xrightarrow{\ocalD_{n-1}} \calC^{0E}_n\]
\[\vec{\calC}^{E+1}: \calC_0 \xrightarrow{\calD_0} \calC_1 \xrightarrow{E \cup \ocalD_1} \calC_{E+1} \xrightarrow{\ocalD_2} \calC^{E+1}_3 \xrightarrow{\ocalD_3}... \xrightarrow{\ocalD_{n-1}} \calC^{E+1}_n
\]

\begin{lemma} \label{lemma:3diagrams}  Suppose all three decomposition sequences certify.  Then all three decomposition sequences cocertify.
\end{lemma}

\begin{proof}  If $\calC_0$ certifies, this is obvious:  $\calC_0$ is a term in all three sequences, so if it certifies then all three sequences cocertify.

If $\calC_1$ certifies, then $\vec{\calC}^{E+1}$ and $\vec{\calC}^{E1}$ cocertify for the same reason: $\calC_1$ is in each sequence. If further $\calC_0$ does not certify, we can apply Proposition \ref{prop:x+yvsxy} to conclude that either $\calC_{0E} = \calC_{0+E}$ or $\calC_{0E} \dashrightarrow \calC_{0+E}$.  In the former case the sequences
$\vec{\calC}^{E1}$ and $\vec{\calC}^{0+E}$ also cocertify because, starting with $\calC_{0E} = \calC_{0+E}$, they are the same sequence and before that in $\vec{\calC}^{0+E}$ is only $\calC_0$, which does not certify.  In the latter case, when $\calC_{0E} \dashrightarrow \calC_{0+E}$, they cocertify by Corollary \ref{cor:longdeflate}. 

Suppose neither $\calC_0$ nor $\calC_1$ certify.  Then following Lemma \ref{lemma:delayE1} the other decompositions can be arranged towards the left in exactly one of the following three diagrams, which differ only in which of two possible deflations (shown as vertical dashed arrows) are in fact equalities:

\[\xymatrix {\calC_0  \ar[ddr]_{\calD_0 \cup E} \ar[r]^{\calD_0} &\calC_1 \ar[d]^{E}   \ar[r]^{\ocalD_1}&\calC_{E+1} \ar[r]^{\ocalD_2} & \calC^{E+1}_3  \ar[r]^{\ocalD _3} & ... & = \vec{\calC}^{E+1}\\ 
&  \calC_{0E} \ar@{-}[d]^{= }  \ar[r]^{\ocalD_1} &\calC_{E1} \ar[r]^{\ocalD_2} \ar@{-}[u]_{= }&... && = \vec{\calC}^{E1} &\\
&\calC_{0+E}  \ar[r]^{\ocalD_1}&\calC^{0+E}_1  \ar[r]^{\ocalD_2} & ... && = \vec{\calC}^{0+E}  &  } \tag{\ref{lemma:delayE1}.1}\]

\[\xymatrix {\calC_0  \ar[ddr]_{\calD_0 \cup E} \ar[r]^{\calD_0} &\calC_1 \ar[d]^{E}   \ar[r]^{\ocalD_1}&\calC_{E+1} \ar[r]^{\ocalD_2} & \calC^{E+1}_3  \ar[r]^{\ocalD _3} & ... & = \vec{\calC}^{E+1}\\ 
&  \calC_{0E}  \ar@{-->}[d] \ar[r]^{\ocalD_1} &\calC_{E1} \ar[r]^{\ocalD_2} \ar@{-}[u]_{= }&... && = \vec{\calC}^{E1} &\\
&\calC_{0+E}  \ar[r]^{\ocalD_1}&\calC^{0+E}_1  \ar[r]^{\ocalD_2} & ... && = \vec{\calC}^{0+E}  &  } \tag{\ref{lemma:delayE1}.1}\]

\[\xymatrix {\calC_0  \ar[ddr]_{\calD_0 \cup E} \ar[r]^{\calD_0} &\calC_1 \ar[d]^{E}   \ar[r]^{\ocalD_1}&\calC_{E+1} \ar[r]^{\ocalD_2} & \calC^{E+1}_3  \ar[r]^{\ocalD _3} & ... & = \vec{\calC}^{E+1}\\ 
&  \calC_{0E} \ar@{-}[d]^{= }  \ar[r]^{\ocalD_1} &\calC_{E1} \ar[r]^{\ocalD_2} \ar@{-->}[u]&... && = \vec{\calC}^{E1} &\\
&\calC_{0+E}  \ar[r]^{\ocalD_1}&\calC^{0+E}_1  \ar[r]^{\ocalD_2} & ... && = \vec{\calC}^{0+E}  &  } \tag{\ref{lemma:delayE1}.1}\]

In the first diagram it is obvious that all three sequences cocertify: because of the equalities, all three sequences are in fact the same sequence after the terms $\calC_{E+1}, \calC_{E1}$ and $\calC_1^{0+E}$.  Similarly the sequences $\vec{\calC}^{E+1}$ and $\vec{\calC}^{E1}$ become the same sequence after those terms in the second diagram and so cocertify; the same is true (even one step earlier) for $\vec{\calC}^{E1}$ and $\vec{\calC}^{0+E}$ in the third diagram.  

Finally, by Corollary \ref{cor:longdeflate}, both $\vec{\calC}^{E1}$ and $\vec{\calC}^{0+E}$ cocertify in the second diagram and $\vec{\calC}^{E+1}$ and $\vec{\calC}^{E1}$ cocertify in the third.
\end{proof}

\section{Ghost circles and timing of disks} \label{sect:ghostcircles}

Let $\calC$ be a flagged chamber complex in $S^3$ supporting $A \cup_T B$, and $S \subset S^3$ be a sphere transverse to $\calC$.  It was shown in Section \ref{sect:guiding} how $S$ defines a flagged chamber complex decomposition sequence for $\calC$. We briefly recall the process, filling in further description, assisted by augmented notation: For $F$ the defining surface for $\calC$, the innermost in $S$ circles $\frc_0$ of $F \cap S$ bound disks in $S$ which we regard as a disk set $\calD_0$ for $\calC$. 
The set of circles $\ofrc_1$ in $F \cap S$ (denoted simply $\frc_1$ in Section \ref{sect:guiding}) consists of those circles which bound disks in $S$ whose interiors may contain circles in $\frc_0$ but otherwise are disjoint from $F$.  Call this set of disks $\ocalD_1$. Proceeding in this manner partition the entire collection of circles $F \cap S$ into sets $\ofrc_0 = \frc_0, \ofrc_1 ..., \ofrc_n$, where $n = \lfloor  \frac{{\rm diam}(Y) + 1}{2} \rfloor$, so that each circle in $\ofrc_k$ bounds a disk $D$ in $S$, and any circle of $F \cap S$ that lies in the interior of $D$ lies in some $\ofrc_i, i < k$.  Denote the collection of such disks in $S$ bounded by the circles in $\ofrc_k$ by $\ocalD_k$.  

The decomposition sequence \[\vec{\calC}:\quad
\calC = \calC_0 \xrightarrow{\calD_0} \calC_1 \xrightarrow{\calD_1} \calC_2 \xrightarrow{\calD_2} ... \xrightarrow{\calD_{n-1}}\calC_n\] is then iteratively defined, the first step being the decomposition $\calC = \calC_0 \xrightarrow{\calD_0} \calC_1$.  After this decomposition two important things happen to the sets of circles $\ofrc_i, 1 \leq i \leq n$: 
\begin{itemize}
\item The circles $\frc_0$ have been surgered away, so the interior of each disk in $\ocalD_1$ is disjoint from the defining surface $F_1$ for $\calC_1$.   
\item Some of the circles in $\ofrc_1$ may lie on the goneballs of the decomposition, and so no longer lie in $S \cap F_1$.  
\end{itemize} 
Then let $\calD_1 \subset \ocalD_1$ be the subset of disks that are left, i. e. the subset of disks whose boundaries lie on $F_1$.  

Since, after the decomposition by $\calD_0$, the interiors of the disks $\calD_1$ are disjoint from $F_1$, they are a disk set in $\calC_1$.  Use these disks to define the decomposition $\calC_1 \xrightarrow{\calD_1} \calC_2$.  Continue in this manner to get the entire decomposition sequence, noting that each $\calD_i$ in the construction consists of those disks in $\ocalD_i$ whose boundaries lie on the defining surface $F_i = F(\calC_i)$, and not on the boundaries of goneballs from earlier decompositions in the sequence.  In other words:

\begin{lemma} For each $1 \leq i \leq n-1$, the disk set $\ocalD_i$ is a guiding disk set for the decomposition $\calC_i \xrightarrow{\calD_i} \calC_{i+1}$.  \qed
\end{lemma} 

\begin{defin} \label{defin:ghostcircle} For each $1 \leq i \leq n-1$, the circles $\ofrc_i - \frc_i$ (that is $\bdd \ocalD_i - \bdd \calD_i$) are called the {\em ghost circles} \index{Ghost circles} of such a decomposition sequence in $\overrightarrow{(\calC, S)}$.  The circles $\frc_i = F_i \cap S$, that is the circles that bound the disks $\calD_i \subset S$, will be called the $F$-circles of the decomposition sequence.  \index{F-circles}
\end{defin}

The disk set $\calD_i$ may well contain ghost circles in its interior; 
these indicate where $\calD_i$ intersects spheres bounding goneballs in earlier decompositions, and  so can be ignored.  Nevertheless, the ghost circles do play a role in the construction of the decomposition sequence, namely they are part of what determines, at the beginning of the decomposition sequence, before they are ghost circles, where in the sequence of decompositions a particular sub-disk of $S$ appears.   

Following up on this idea, it is reasonable to consider how introducing ghost circles in $S$ at the beginning of the decomposition process (when the sets of circles $\frc_i$ for the decomposition sequence are being defined) might affect the decomposition sequence.  
For example, if a ghost circle $c'$ were to be added, at the beginning of the decomposition sequence, parallel to a circle $c \in \frc_{i-1}$ and just inside the disk $D \in \calD_{i-1}$ that $c$ bounds, one effect would be to redefine $D$ as a disk in $\calD_{i}$, because $c$ is no longer $i$-th innermost in $F \cap S$ but $(i+1)$-st innermost in the collection of circles $\frc \cup c'$.  There could well be a similar delay to the use of other disks in the following $\calD_k, i < k < n$ in a manner that may be a bit difficult to describe - see Figure \ref{fig:ghostcircle}.   Our next goal is to understand to some extent how adding or removing ghost circles (and thereby changing the location of the decomposing disks in the sequence of disk decompositions) affects the certification of a decomposition sequence.  The news is good: eventually we will show that adding or removing ghost circles may alter the decomposition sequence, as we have seen, but the new decomposition sequence cocertifies with the old.  
 \begin{figure}[ht!]
 \labellist
\small\hair 2pt
\pinlabel  \text{ghost circle} at 30 260
\pinlabel  \text{center edge} at 350 320
\pinlabel  \text{center vertex} at 350 50
\pinlabel  $\in\frc_2$ at 130 360
\pinlabel  $\in\frc_3$ at 140 80
\pinlabel  0 at 315 410
\pinlabel  0 at 350 410
\pinlabel  0 at 325 375
\pinlabel  1 at 355 380
\pinlabel  2 at 400 365
\pinlabel  3 at 465 380
\pinlabel  2 at 495 340
\pinlabel  1 at 355 100
\pinlabel  2 at 385 120
\pinlabel  3 at 410 115
\pinlabel  3 at 465 100
\endlabellist
    \centering
    \centering
    \includegraphics[scale=0.5]{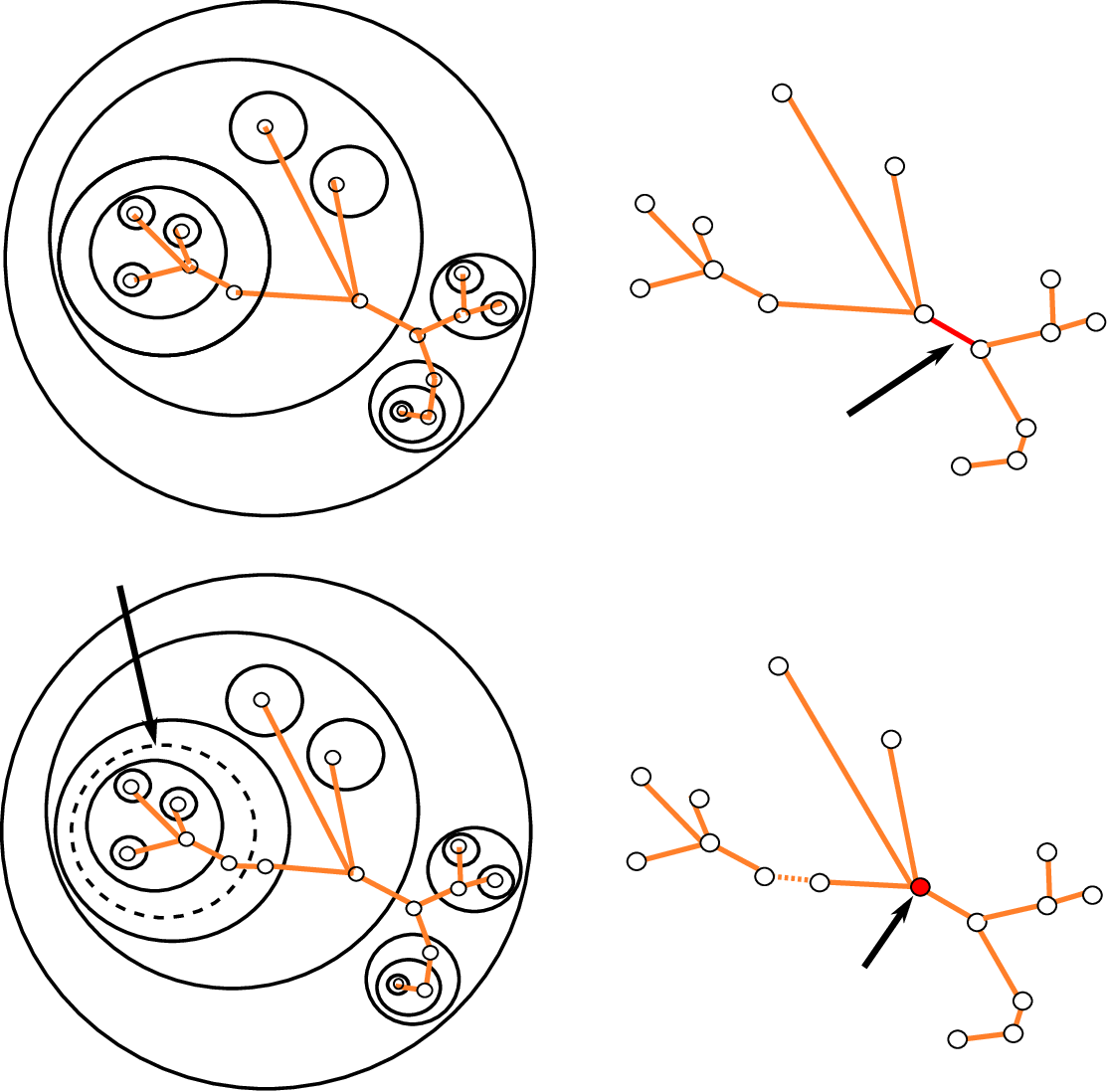}
    \caption{Adding a ghost circle, $i=3$}  \label{fig:ghostcircle}
    \end{figure}

To this end we generalize the discussion of Section \ref{sect:guiding} to spheres in $S$ transverse to the defining surface $F$ of a flagged chamber complex $\calC$ in $S^3$ and consider a finite collection of circles $\ofrc$ in $S$ containing all of the circles $\frc = F \cap S$ (which we have distinguished as {\em $F$-circles}) but augmented by other possible circles $\overline{\frc} - F$ in $S$ called ghost circles. The collection of circles $\overline{\frc}$ determines a sequence of flagged chamber complex decompositions guided by disk sets $\ocalD_i$, where each circle in $\ofrc$ is the boundary of a disk in some $\ocalD_i$, as described above and in Section \ref{sect:guiding}.  (This setting naturally arises in the middle of a flagged chamber complex decomposition sequence, as we have seen.)

As in Section \ref{sect:guiding}, let $Y$ be the tree associated to the circles $\ofrc \subset S$, and, as in Definition \ref{defin:Didef}, let 
\[\vec{\calC}:\quad
\calC = \calC_0 \xrightarrow{\calD_0} \calC_1 \xrightarrow{\calD_1} \calC_2 \xrightarrow{\calD_2} ... \xrightarrow{\calD_{n-1}}\calC_n\]
be the complete sequence of flagged chamber complex decompositions guided by the sphere $S$, where 
$n = \lfloor  \frac{{\rm diam}(Y) + 1}{2} \rfloor$.  It's perfectly possible that for some values of $i$, $\ofrc_i$ consists entirely of ghost circles, so none of the disks $\ocalD_i$ have their boundaries on $F$ and so $\calD_i = \emptyset$.  In this case the decomposition $ \calC_i \xrightarrow{\calD_i} \calC_{i+1}$ is just the identity: $\calC_i = \calC_{i+1}$.  

Recalling Definition \ref{defin:rhoe}, each edge $e$ in $Y$ (corresponding to a circle $c \in \ofrc$) is assigned a value $0 \leq \rho_Y(e) \leq  \lfloor \frac{{\rm diam}(Y) - 1}{2} \rfloor$.  For each vertex $v \in Y$ we can similarly assign $\rho_Y(v) = max\{\rho_Y(e)\; |\;  e \text{ an edge incident to } v\}$.  
 
Here are some elementary observations.  For $e$ an edge in $Y$, let $m^e_{\pm}$ be the numbers defined for $e$ in Definition \ref{defin:rhoe}.    

\begin{lemma}  \label{lemma:Yodd} There is an edge $f$ in $Y$ such that $m^f_+ = m^f_-$ if and only if ${\rm diam}(Y)$ is odd.  In this case $\rho_Y(f) = \frac{{\rm diam}(Y) - 1}{2}$ and $f$ is unique: it is the mid-edge of each path in $Y$ whose length is ${\rm diam}(Y)$.  Call $f$ the {\em center edge} \index{center edge of an odd diameter tree} of the tree $Y$.  See Figure \ref{fig:ghostcircle}.
\end{lemma}

\begin{proof} Let $e$ be any edge in $Y$, with endpoints $v_+$ and $v_-$ belonging to the trees $Y_+$ and $Y_-$ respectively in $Y - e$, as discussed in Definition \ref{defin:rhoe}.  If there is a leaf $v_{\ell}$ of $Y$ that belongs to $Y_+$ and is a distance $m > \frac{{\rm diam}(Y) - 1}{2}$ from $v_+$, then $v_{\ell}$ is a distance $m+1 > \frac{{\rm diam}(Y) +1}{2}$ from $v_-$.  It follows from the definition of diameter that no leaf of $Y$ that belongs to $Y_-$ can be a distance more than \[{\rm diam}(Y) - m - 1 = \frac{{\rm diam}(Y) - 1}{2} + \frac{{\rm diam}(Y) + 1}{2} - m -1 < \frac{{\rm diam}(Y) - 1}{2}\] from $v_-$.  Hence $m_- \leq  \frac{{\rm diam}(Y) - 1}{2} <  m \leq m_+$ and $e$ is not as $f$ is described in the Lemma.

Thus an edge $f$ as described in the Lemma has the property that any leaf of $Y$ in $Y_{\pm}$ has distance at most $\frac{{\rm diam}(Y) - 1}{2}$ from respectively the vertices $v_{\pm}$.  This implies that any path in $Y_{\pm}$ can have length at most ${\rm diam}(Y) - 1$, so in particular no path $\gamma$ in $Y$ of length ${\rm diam}(Y)$ can lie entirely in either of $Y_{\pm}$.  This implies that $f$ must be in $\gamma$ and in fact be the mid-edge of $\gamma$, defining $f$ uniquely and establishing that $\rho_Y(f) = m^f_{\pm} = \frac{{\rm diam}(Y) - 1}{2}$.

Conversely, suppose ${\rm diam}(Y)$ is odd and $f$ is the mid-edge of a path $\gamma$ in $Y$ of length ${\rm diam}(Y)$.  If any leaf in $Y$ lying in $Y_+$ (resp $Y_-$) is a distance greater than $\frac{{\rm diam}(Y) - 1}{2}$ from $v_+$ (resp $v_-$) then conjoining a path from that leaf to $v_+$ with the part of $\gamma - v_+$ containing $v_-$ would give a path of length greater than ${\rm diam}(Y)$, a contradiction.  Hence both $m^f_{\pm} = \frac{{\rm diam}(Y) - 1}{2}$ as required.
%
%
\end{proof}

Quite similarly we have:

\begin{lemma} \label{lemma:Yeven} There is a vertex $w$ in $Y$ that is incident to more than one edge $e$ with $\rho_Y(e) = \rho_Y(w)$ if and only if ${\rm diam}(Y)$ is even.  In this case $w$ is unique and $\rho_Y(w) = \frac{{\rm diam}(Y)}{2} - 1$.  Call $w$ the {\em center vertex} \index{center vertex of an even diameter tree} of the tree $Y$. See Figure \ref{fig:ghostcircle}.
\end{lemma}
\begin{proof}  Consider any vertex $v$ in $Y$ and suppose there is a path $\gamma$ in $Y$ from $v$ to a leaf of $Y$ and that the length of $\gamma$ is $m > \frac{{\rm diam}(Y)}{2}$.  Let $e$ be the edge of $\gamma$ that is adjacent to $v$ and observe that, by definition of diameter, any path from $v$ to a leaf of $Y$ that is longer than ${\rm diam}(Y) - m < m$ must pass through $e$.  Definition \ref{defin:rhoe} then implies that $\rho_Y(e)$ is the length of the longest path in $Y$ {\em not passing through $e$} from $v$ to a leaf of $Y$, and 
any edge $e'$ incident to $v$ other than $e$ has $\rho_Y(e') \leq \rho_Y(e) - 1$.  In particular, $e$ is the only edge incident to $v$ with $\rho_Y(e) = \rho_Y(v)$, so $e$ is not an edge as described in the lemma.

It follows that if $w$ is as described in the lemma, then the length of any path from $w$ to a leaf in $Y$ has length $\leq \frac{{\rm diam}(Y)}{2}$.  So, by definition of diameter, any path in $Y$ whose length is ${\rm diam}(Y)$ must pass through $w$ and have its length bisected by $w$.  Thus ${\rm diam}(Y)$ is even and the distance from $w$ to a most distant leaf will be $\frac{{\rm diam}(Y)}{2}$.  In particular, any edge $e$ incident to $w$ will find the vertex at its other end a distance of at most $\frac{{\rm diam}(Y)}{2} - 1$ from a leaf.  Thus $\rho_Y(e) \leq \frac{{\rm diam}(Y)}{2} -1$.  Since this is true for all edges incident to $w$, $\rho_Y(w) \leq \frac{{\rm diam}(Y)}{2} -1$.  Furthermore, since any path in $Y$ whose length is ${\rm diam}(Y)$ must pass through $w$, there are at least two edges $e_{\pm}$ incident to $w$ lying on such a path, and then $\rho_Y(e_{\pm}) = \frac{{\rm diam}(Y)}{2} -1$.  It follows that for both edges $\rho_Y(e_{\pm}) = \frac{{\rm diam}(Y)}{2} -1 = \rho_Y(w)$ as required.

Conversely, if ${\rm diam}(Y)$ is even, $\gamma$ is a path in $Y$ of length ${\rm diam}(Y)$, and $w$ is the vertex that bisects $\gamma$, the same argument shows that for $e$ either of the two edges in $\gamma$ incident to $w$, $\rho_Y(e) = \rho_Y(w)$.   
%
%
\end{proof}

\begin{lemma} \label{lemma:PDk} Suppose $P \subset S - \ofrc$ is the planar surface corresponding to a vertex $v \in Y$.  If $P \subset D \in \ocalD_k$, then $\rho_Y(v) \leq k$.
\end{lemma}

\begin{proof}  For $D \in \ocalD_k$, $\bdd D \in \ofrc_k$ and from Lemma \ref{lemma:treetrim} each circle in $\ofrc \cap \inter(D)$ lies in some $\ofrc_j, j < k$.  Hence if $P \subset D$, each circle $c \in \bdd P$ lies in some $\ofrc_j, j \leq k$.  Corresponding to this, we have that each edge $e$ in $Y$ that is incident to $v$ has $\rho_Y(e) \leq k$.  Hence by definition $\rho_Y(v) \leq k$ as required.
%
\end{proof}
 \begin{lemma}  \label{lemma:addghostvert} Suppose a ghost circle is added to $S$ in a planar surface $P \subset S - \ofrc$ corresponding to a vertex $v \in Y$ and \[\vec{\calC'}:\quad
\calC = \calC_0 \xrightarrow{\calD'_0} \calC'_1 \xrightarrow{\calD'_1} \calC'_2 \xrightarrow{\calD'_2} ... \]
is the resulting sequence of flagged chamber complex decompositions.  Suppose $\rho_Y(v) \geq 1$.  Then for each $1 \leq i \leq \rho_Y(v)$, $\calD'_{i - 1} = \calD_{i-1}$ and $\calC'_i = \calC_i$.  
\end{lemma}

\begin{proof} If $i \leq \rho_Y(v)$ then $\rho_Y(v) \nleq i-1$.  It follows from Lemma \ref{lemma:PDk} that any disk $D \in \ocalD_{i-1}$ does not contain $P$ and so is unaffected by the addition of the ghost circle.  Thus, for each $1 \leq i \leq \rho_Y(v)$, $\ocalD'_{i-1} = \ocalD_{i-1}$.  Now argue iteratively: Since $C'_0 = C_0$ and $\ocalD'_0 = \ocalD_0$ it follows that $\calD'_0 = \calD_0$. Hence $\calC'_1 = \calC_1$, so if $\ocalD'_1 = \ocalD_1$ then $\calD'_1 = \calD_1$.  The argument continues iteratively until $i = \rho_Y(v)$: For each $1 \leq i \leq \rho_Y(v)$, $\calC'_{i-1} = \calC_{i-1}$ and $\ocalD'_{i-1} = \ocalD_{i-1}$ imply $\calD'_{i-1} = \calD_{i-1}$, so $\calC'_i = \calC_i$, as required.  
\end{proof}

\begin{cor} \label{cor:addghosteven} Suppose $\diamY$ is even and a ghost circle is added to $S$ in the planar surface $P \subset S - \ofrc$ corresponding to the center vertex $v \in Y$.  Then $\vec{\calC'} = \vec{\calC}$.  
\end{cor}

\begin{proof}  Since ${\rm diam}(Y)$ is even, the complete sequence $\vec{\calC}$ has last term $\calC_n$ with $n = {\rm diam}(Y)/2 = \rho_Y(v)$. If $\rho_Y(v) = 0$ then $Y$ is a star graph, each circle in $\ofrc$ bounds a unique disk in $\ocalD_0$ and both $\vec{\calC'}$ and $\vec{\calC}$ are simply the sequence $\calC_0 \xrightarrow{\calD_0} \calC_1$.  Indeed, after this decomposition the only circle remaining in $S$ is the added ghost circle, so $S \cap F(\calC_1) = \emptyset$ and there is no further decomposition in the sequence.    If, on the other hand, $\rho_Y(v) \geq 1$ then by Lemma \ref{lemma:addghostvert},  for all $0 \leq i \leq n$, $\calC'_i = \calC_i$ and after these decompositions, the only circle that might remain is again the added ghost circle. 
\end{proof}

\begin{cor} \label{cor:addghostodd} Suppose $\diamY$ is odd and a ghost circle is added to $S$ in a planar surface $P \subset S - \ofrc$ corresponding to a vertex at the end of the center edge $e \in Y$.  Then if $\vec{\calC}$ certifies, $\vec{\calC'}$ and $\vec{\calC}$ cocertify.  
\end{cor}

\begin{proof} One possibility is that $\diamY = 1$, that is $Y$ is the single edge $e$.  In this case $\ofrc$ is a single circle $c$ (corresponding to $e$) dividing $S$ into two disks $D_A$ and $D_B$, as in Lemma \ref{lemma:ambig}.  Say the ghost circle is added to $D_A = P$.  If $c$ is also a ghost circle, then $F$ is disjoint from $S$ and there is no further decomposition. In particular, $\vec{\calC'} = \vec{\calC}$. If, on the other hand, $c$ is an $F$-circle, so $c = F \cap S$ then it is shown in Lemma \ref{lemma:ambig} that if $\vec{\calC}$ certifies, the choice of whether to decompose $\calC$ along $D_A$ or $D_B$ makes no difference: the two results cocertify.  Adding a ghost circle to $D_A$ has the sole effect of declaring that in $\calC'$ the decomposition will be along $D_B$ (since $D_A$ is no longer a disk component of $S - \ofrc$).  Hence $\vec{\calC'}$ and $\vec{\calC}$ cocertify.

If, on the other hand, $\diamY \geq 3$ then, per Lemma \ref{lemma:Yodd} $\rho_Y(v) = \rho_Y(e) = \frac{\diamY - 1}{2} \geq 1$.  Then according to Lemma \ref{lemma:addghostvert} the decomposition sequences $\vec{\calC}$ and $\vec{\calC'}$ are identical through $\calC'_{\frac{\diamY - 1}{2}} = \calC_{\frac{\diamY - 1}{2}}$.  But for $\calC_{\frac{\diamY - 1}{2}}$, $\ofrc$ has shrunk to a single circle, and $Y$ is a single edge.  Thus we can apply the argument just given for this case and conclude that $\vec{\calC'} \sim \vec{\calC}$ as required.
\end{proof}

\begin{prop} \label{prop:addghost}  Suppose $\calC$ is a flagged chamber complex in $S^3$ that is not tiny.  Let $S \subset S^3$ be a sphere, transverse to the defining surface $F = F(\calC)$, that is either balanced or almost balanced.  Let $\ofrc$ be a collection of circles in $S$ containing the set $\frc = F \cap S$ and let $\ofrc'$ be the union of $\ofrc$ and a disjoint (ghost) circle $c'$ in $S$.  Let $\vec{\calC}$ and $\vec{\calC'}$ be the sequences of flagged chamber complex decompositions determined by $\ofrc$ and $\ofrc'$ respectively.  Then $\vec{\calC}$ and $\vec{\calC'}$ cocertify.
\end{prop}

\begin{proof}  
Let $Y$ be the tree associated to $\ofrc$, $P \subset S - \ofrc$ be the planar surface component to which the ghost circle $c'$ is added, and $v$ be the vertex in $Y$ corresponding to $P$.  The proof is by induction on the distance between $v$ and the center vertex or edge of $Y$.  Corollaries \ref{cor:addghosteven} and \ref{cor:addghostodd} show that the proposition is true when that distance is $0$.  

Following Lemma \ref{lemma:Yodd}, as used in the proof of Corollary \ref{cor:addghosteven}, the decomposition sequences are identical until the tree has been shrunk to the point that $v$ is a leaf.  So we may as well assume that $v$ is a leaf of $Y$, so $P$ is an innermost disk of $S - \ofrc$.  Let $\ofrc''$ be the collection of circles $\ofrc \cup c''$, where $c''$ is a circle parallel to $\bdd P$ but just {\em outside} $P$, that is in the component $P''$ of $S - \ofrc$ that is adjacent to $P$.  See Figure \ref{fig:moveghost}.  Then the vertex $v''$ corresponding to $P''$ is closer to the center edge or vertex of $Y$ so, by inductive assumption, the sequence of flagged chamber complex decompositions $\vec{\calC''}$ determined by $\ofrc''$ and $\vec{\calC}$ cocertify.  So all that remains is to show that $\vec{\calC''}$ and $\vec{\calC'}$ cocertify.

 \begin{figure}[ht!]
 \labellist
\small\hair 2pt
\pinlabel  $c''$ at 60 260
\pinlabel  $c'$ at 50 520
\pinlabel  $\vec{\calC}'$ at 350 320
\pinlabel  $\vec{\calC}''$ at 350 50
\pinlabel  $P$ at 110 410
\pinlabel  $P''$ at 160 80
\pinlabel  $\bdd P$ at 123 95
\pinlabel  $v$ at 400 385
\pinlabel  $\bdd P$ at 410 355
\pinlabel  $\bdd P$ at 400 115
\pinlabel  $v''$ at 465 110
\endlabellist
    \centering
    \centering
    \includegraphics[scale=0.5]{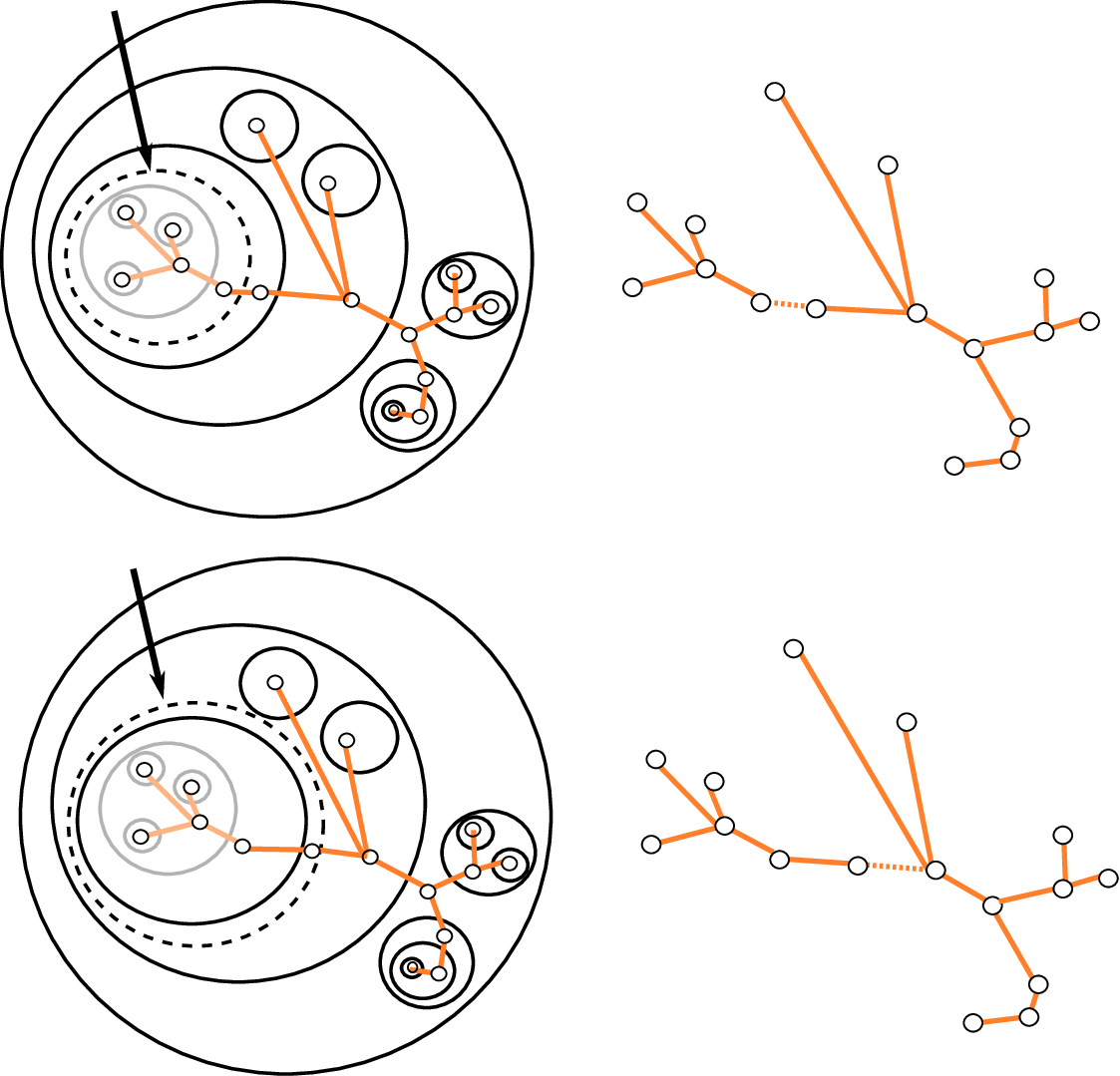}
    \caption{Moving a ghost circle, so $\bdd P \in \frc'_3 \to \bdd P \in \frc''_2$}  \label{fig:moveghost}
    \end{figure}

If $\bdd P$ is a ghost circle, then $\ofrc'' = \ofrc'$ so in fact $\vec{\calC''} = \vec{\calC'}$ and we are done.  So assume that $\bdd P$ is an $F$-circle in $\ofrc$ and note that the only difference between $\ofrc''$ and $\ofrc'$ is the side on which a ghost circle is added.  This switch in sides has no effect on the tree and so no effect on the disks in the decomposition sequence  $\vec{\calC'}$ except this:  In $\vec{\calC}''$ the disk $P$ is in $\calD_0$, since $P$ is an innermost disk, whereas in $\calC'$ the disk $P$ is in $\calD_1$ because the disk bounded by $\bdd P$ contains the single added ghost circle $c'$.  

Apply Lemma \ref{lemma:3diagrams} to this situation, where here the disk $P$ plays the role of $E$ in that lemma, $\vec{\calC}''$ here corresponds to $\vec{\calC}^{0+E}$ there and $\vec{\calC}'$ here corresponds to $\vec{\calC}^{E+1}$ there.  Since $S$ is balanced or almost balanced, both sequences $\vec{\calC''}$ and $\vec{\calC'}$ certify, as does the sequence $\vec{\calC}^{P1}$, defined for $P$ as $\vec{\calC}^{E1}$ is defined for $E$ in the preamble to Lemma \ref{lemma:3diagrams}.  
The conclusion of Lemma \ref{lemma:3diagrams} then says that $\vec{\calC}''$ and $\vec{\calC'}$ cocertify, so $\vec{\calC} \sim \vec{\calC}'' \sim \vec{\calC'}$, as required.
\end{proof}

\begin{cor}  \label{cor:addghost} Suppose $\calC$ is a flagged chamber complex in $S^3$ that is not tiny and $S \subset S^3$ is a sphere transverse to $F= F(\calC)$ which is either balanced or almost balanced for $\calC$.  Let $\ofrc$ and $\ofrc'$ be collections of circles in $S$, each containing the set $\frc = F \cap S$.   Let $\vec{\calC}$ and $\vec{\calC'}$ be the sequences of flagged chamber complex decompositions determined by $\ofrc$ and $\ofrc'$ respectively.  Then $\vec{\calC}$ and $\vec{\calC'}$ cocertify.
\end{cor}

\begin{proof}  It suffices to prove the Corollary in the case that $\ofrc = \frc$.  In this case $\ofrc'$ is obtained from $\ofrc$ by adding some number $m$ of ghost circles.  But this case follows by applying Proposition \ref{prop:addghost} $m$ times.
\end{proof}

\begin{cor}  \label{cor:addisk} Theorem \ref{thm:addisk} remains true even when the chamber $C$ containing $E$ is an empty torus, or  when $E$ is inessential.
\end{cor}

\begin{proof} Referring to the proof of Theorem \ref{thm:addisk}, the requirement that $E$ be essential and $C$ not be an empty torus is to ensure that the original decomposition $C = C_0 \xrightarrow{E} C'$ creates no goneballs.  The boundaries of such goneballs would intersect $F = F(\calC)$ in a family of circles $\ofrc_G$ in $F \cap S$ that continue to determine which circles in $F \cap S$ bound disks in each $\calD_i$ during the subsequent decomposition \[\vec{\calC'}:\quad
\calC' = \calC'_0 \xrightarrow{\calD'_0} \calC'_1 \xrightarrow{\calD'_1} \calC'_2 \xrightarrow{\calD'_2} ... \xrightarrow{\calD'_{n-1}}\calC'_n \tag{\ref{sect:guidingisk}.1}\] used in the proof of Theorem \ref{thm:addisk}. 

In contrast, it is the collection of circles $F(\calC') \cap S$ that determines a decomposition sequence in $\overrightarrow{(\calC', S)}$ and the goneballs of the decomposition $C = C_0 \xrightarrow{E} C'$ do not appear in the definition of $F' = F(\calC')$.  In other words, $F \cap S = (F' \cap S) \cup \ofrc_G$, so $\ofrc_G$ acts as a collection of ghost circles in the decomposition sequence $\vec{\calC'}$.  Corollary \ref{cor:addghost} says that the addition of such ghost circles makes no difference in certification.  That is, we may as well assume that $F \cap S = F' \cap S$ and proceed with the proof of Theorem \ref{thm:addisk} as written.  
\end{proof}

\section{Isotopies of guiding spheres} \label{sect:sphereisotope}

\begin{defin} \label{defin:simseq}
Suppose $\calC$ is a flagged chamber complex in $S^3$ and $S$, $S'$ are spheres transverse to the defining surface $F = F(\calC)$.  If any $\vec{\calC} \in \overrightarrow{(\calC, S)}$ and $\vec{\calC'} \in \overrightarrow{(\calC, S')}$ cocertify, then we say $\overrightarrow{(\calC, S)}$ and  $\overrightarrow{(\calC, S')}$ cocertify and write $ \overrightarrow{(\calC, S)} \sim  \overrightarrow{(\calC, S')}$.  
\end{defin}

Our goal in this section is to prove:

\begin{thm} \label{thm:balanceisotopy} Suppose $\calC$ is a flagged chamber complex in $S^3$ that is not tiny and supports the Heegaard splitting $S^3 = A \cup_T B$.  Suppose further that $S_0$ and $S_1$ are balanced or almost balanced spheres for $\calC$ that are isotopic in $S^3$ through balanced or almost balanced spheres $S_t, 0 \leq t \leq 1$.  Then $ \overrightarrow{(\calC, S_0)} \sim  \overrightarrow{(\calC, S_1)}$. 
\end{thm}

As usual, let $F$ denote the defining surface $F = F(\calC)$ and put the isotopy in general position so that generically $S_t$ is transverse to $F$, but for certain discrete values of $t$, $S_t$ is tangent to $F$ at a single point, a non-degenerate critical point of index zero ($F$ has a local max at $v$ in a collar of $S$), two (a local min) or one (a saddle tangency).  

\begin{lemma} \label{lemma:maxormin}
Theorem \ref{thm:balanceisotopy} is true if the isotopy $S_t$ passes through only a single critical point, of index zero or two (a max or a min).
\end{lemma}

\begin{proof}   The cases of index zero or index two are symmetric.  Let $S_{m(iddle)}$ be the sphere on which the critical point $v$ occurs, and, in a collar of $S_m$, we may suppose $F$ has a minimum.  Call the bottom of the collar $S_{b(elow)}$ and the top $S_{a(bove)}$.  We can assume that other than the local minimum, $F$ intersects the collar in vertical cylinders.  Then $F \cap S_a$ is the union of $F \cap S_b$ and a single circle $c$, where $c$ bounds a disk $D_F$ in $F - S_a$ (the disk in $F$ containing $v$) and a parallel disk $D_S$ in $S_a - F$.  Let $\vec{\calC}_b$ (resp $\vec{\calC}_a$) be the sequence of flagged chamber complex decompositions determined by $\calC$ and $S_b$ (resp. $S_a$).  Thus $\vec{\calC}_b \in \overrightarrow{(\calC, S_1)}$ (say) and $\vec{\calC}_a \in \overrightarrow{(\calC, S_0)}$.  See Figure \ref{fig:SaSb0}.
Augment the collection of circles $F \cap S_b$ with the single (ghost) circle $c'$ parallel in $S_b$ to $c \subset S_a$ and call the resulting sphere with ghost circle $S_c$.  Let $\vec{\calC}_c$ be a sequence of flagged chamber complex decompositions determined by $\calC$ and $S_c$.  Then by Corollary \ref{cor:addghost} $\vec{\calC}_b \sim \vec{\calC}_c$.

 \begin{figure}[ht!]
 \labellist
\small\hair 2pt
\pinlabel  $(\calC_a)_0$ at 0 350
\pinlabel  $(\hat{\calC}_a)_1$ at 600 370
\pinlabel  $(\calC_a)_1$ at 600 160
\pinlabel  $(\calC_c)_0$ at 0 280
\pinlabel  $(\calC_c)_1$ at 0 120
\pinlabel  $S_a$ at 280 380
\pinlabel  $S_m$ at 280 320
\pinlabel  $S_b,S_c$ at 280 270
\pinlabel  $S_c$ at 280 50
\pinlabel  {\huge $\sim$} at 300 120
\pinlabel  $v$ at 160 320
\pinlabel  $c$ at 190 400
\pinlabel  $D_S$ at 150 385
\pinlabel  $D_F$ at 210 350
\pinlabel  $c'$ at 170 280
\pinlabel  $B_G$ at 450 350
\endlabellist
    \centering
    \centering
    \includegraphics[scale=0.55]{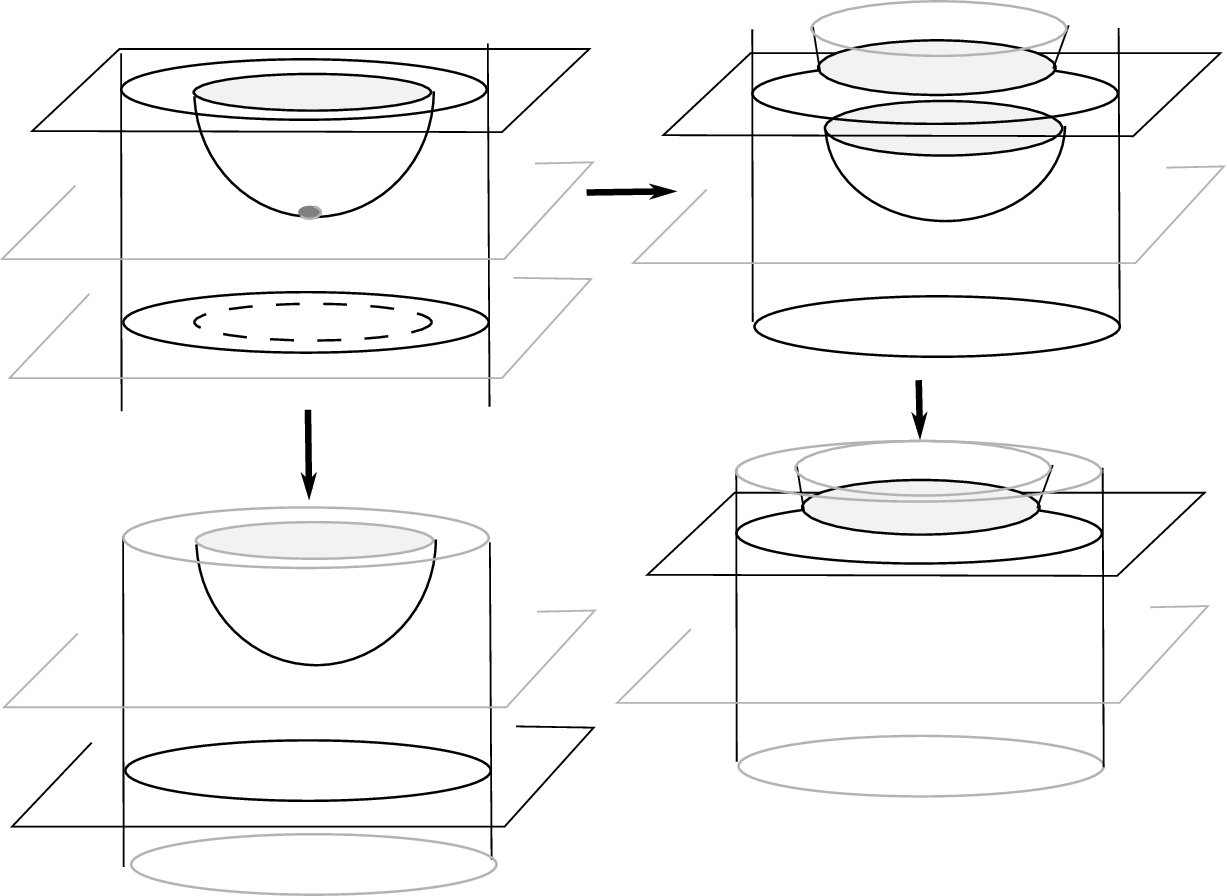}
    \caption{Passing $S$ through a minimum}  \label{fig:SaSb0}
    \end{figure}

Now compare $\vec{\calC}_a$ with $\vec{\calC}_c$. The pattern of circles in $F \cap S_a$ is the same as in $F \cap S_c $ except that the $F$-circle $c$ in $S_a$ appears as a ghost circle $c'$ in $S_c$.  Thus the  decomposition $\calC_0 \xrightarrow{\calD_0} \calC_{1}$ is the same in each sequence, except in $\vec{\calC}_c$ the ghost circle $c'$ is just removed, whereas in $\vec{\calC}_a$, $F$ is surgered by $D_S$, creating an extra ball chamber $B_G$ bounded by the sphere $D_S \cup D_F$.
But our convention for flagging declares the handlebody $B_G$ to be empty and therefore a goneball, so the result of the flagged chamber complex decomposition $\calC_{1}$ is the same in both cases, as is the remaining pattern of circles $S_a \cap F(\calC_{1}) = S_c \cap F(\calC_{1})$.  It follows that the entire ensuing flagged chamber complex decomposition sequence is the same, so $\vec{\calC}_a \sim \vec{\calC}_c$.  Combining, we have $\vec{\calC}_a \sim \vec{\calC}_c \sim \vec{\calC}_b$ so $ \overrightarrow{(\calC, S_0)} \sim  \overrightarrow{(\calC, S_1)}$ as required.  
\end{proof}

%

Next suppose the isotopy $S_t$ has a single critical point that is a saddle tangency with $F$ at a point $v \in F$.
As in the proof of Lemma \ref{lemma:maxormin} we consider how $F$ intersects a collar of $S_m$, the sphere of the tangency at $v$, as shown in Figure \ref{fig:SaSb1}.  A neighborhood in $F$ of  $S_m \cap F$ (a figure eight) is a pair of pants with three disjoint circles: $c_b$, whose parallel in $S_m$ is incident to $v$ in two points, and $c_1$, $c_2$ each of whose parallels in $S_m$ is incident to $v$ in a single point.  With $S_a$ and $S_b$ as in Lemma \ref{lemma:maxormin}, we can assume that $c_1$ and $c_2$ both lie in $S_a$, and $c$ lies in $S_b$. 

 \begin{figure}[ht!]
 \labellist
\small\hair 2pt
\pinlabel  $c_b$ at 260 310
\pinlabel  $c_1$ at 210 310
\pinlabel  $c_2$ at 110 310

\pinlabel  $c_b$ at 210 40
\pinlabel  $c_1$ at 210 170
\pinlabel  $c_2$ at 110 170

\pinlabel  $c_2$ at 500 175
\pinlabel  $c_1$ at 410 175
\pinlabel  $c_b$ at 370 40

\pinlabel  $c_b$ at 395 275
\pinlabel  $c_1$ at 450 275
\pinlabel  $c_2$ at 340 275
\pinlabel  $S_{a}$ at 295 180
\pinlabel  $S_{m}$ at 295 110
\pinlabel  $S_{b}$ at 295 40
\endlabellist
    \centering
    \centering
    \includegraphics[scale=0.6]{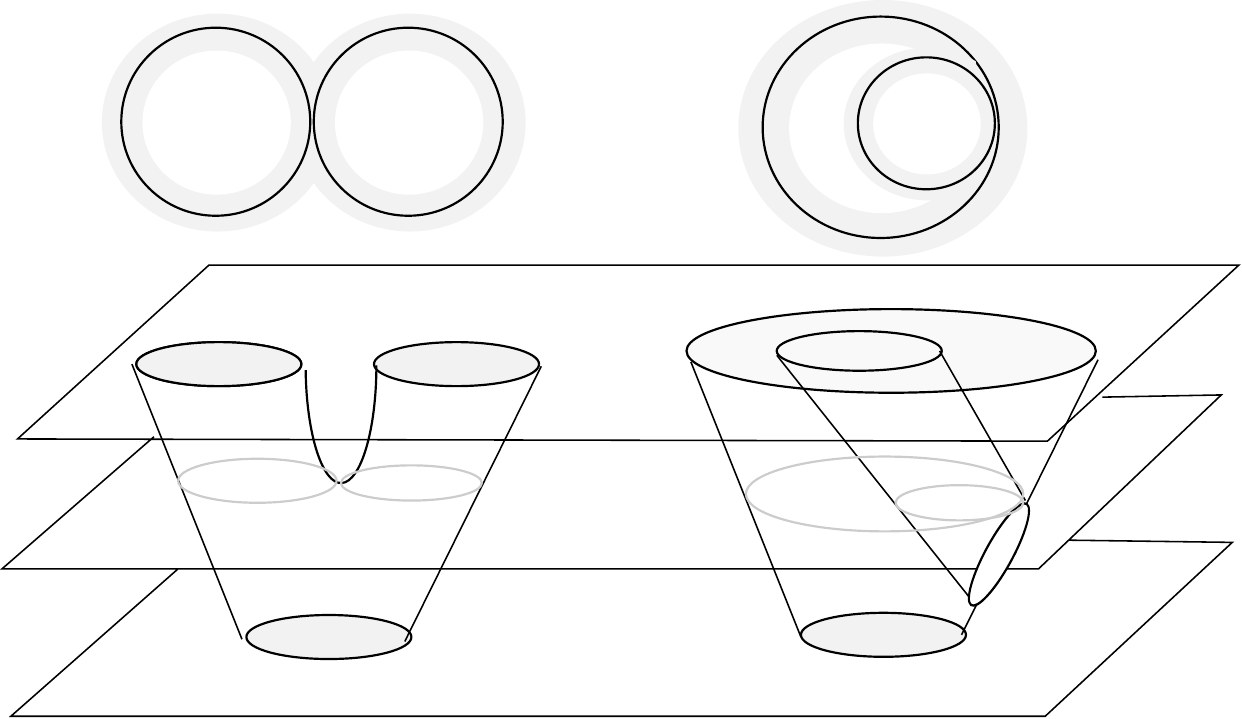}
    \caption{A neighborhood of critical point $v$: two viewpoints}  \label{fig:SaSb1}
    \end{figure}

As in Section \ref{sect:guiding} let $Y_a$ (resp $Y_b$) be the tree defined by the collection of circles $F \cap S_a$ (resp $F \cap S_b$), $e_1, e_2$ be the edges in $Y_a$ corresponding to $c_1, c_2$ and $e_b$ be the edge in $Y_b$ corresponding to $c_b$.  As defined in \ref{defin:rhoe}, let $\rho_i = \rho_{Y_a}(e_i), i = 1, 2$ and $\rho_b = \rho_{Y_b}(e_b)$.  Put another way, in the decomposition sequences of $\calC$ guided by $S_a$ and $S_b$, each of the two circles $c_i, i = 1, 2$ bounds a disk $D_i \in \calD_{\rho_i}$ in the decomposition $\calC_{\rho_i} \xrightarrow{\calD_{\rho_i}} \calC_{\rho_i +1}$ and $c_b$ bounds a disk $D_b \in \calD_{\rho_b}$ in the decomposition $\calC_{\rho_b} \xrightarrow{\calD_{\rho_b}} \calC_{\rho_b +1}$.  With no loss of generality, assume $\rho_1 \leq \rho_2$ (i. e. $\calD_{\rho_1}$ does not appear after $\calD_{\rho_2}$ in the decomposition sequence).  There are two possibilities for the disks $D_1, D_2 \subset S_a$,  and the two viewpoints in Figure \ref{fig:SaSb1} are chosen so that in either case each $D_i$ is on the bounded side of $c_i$ in $S_a$: either $D_1$ and $D_2$ are disjoint, as shown in the left side of the figure, or $D_1 \subset D_2$ as shown in the right side.  That is, the two viewpoints in the figure are chosen so that in each the $D_i$ appear as bounded regions.  Similarly, $D_b$ can either lie on the side of $c_b$ that contains $v$ (the bounded side in the left of the figure and unbounded side in the right) or the other side.  We have:

\begin{lemma} \label{lemma:vside}  If $D_b$ lies on the side of $c_b$ that contains $v$, then $\rho_1 \leq \rho_b$.
\end{lemma}
\begin{proof} Recall that, aside from the component that contains $v$, $F$ intersects the collar of $S$ only in vertical annuli.  Hence the interior of the bounded side of $c_b$ on the left of Figure \ref{fig:SaSb1} (or the the interior of the unbounded side of $c_b$ on the right) contains a parallel copy of each circle of $F \cap \inter(D_1)$.  
\end{proof}

\begin{lemma} \label{lemma:saddle}
Theorem \ref{thm:balanceisotopy} is true if the isotopy $S_t$ passes through only a single critical point, of index one (a saddle tangency).
\end{lemma}

\begin{proof}  We make two claims:
\medskip

{\em Claim 1:} With no loss of generality, we can assume that $\rho_1 \leq \rho_b$.

{\em Proof of Claim 1:} By Lemma \ref{lemma:vside} this is true when $D_b$ lies on the side of $c_b$ that contains $v$, so suppose $D_b$ lies on the other side.  That is, suppose $D_b$ lies on the outside (unbounded side) of $c_b$ in the left of Figure \ref{fig:SaSb1} or on the inside (bounded side) of $c_b$ on the right.  Suppose $\rho_b < \rho_1$ and consider what happens if we add $\rho_1 - \rho_b$ ghost circles to $D_b$, parallel to $\bdd D_b$.  This raises $\rho_b$ so that $\rho_b = \rho_1$.  Since there is a copy of the pattern of circles $F \cap D_1$ on the other side of $c_b$ in $S_b$, that is in $S_b - D_b$, (see left panel in Figure \ref{fig:SaSb4}), it remains true that the maximal distance in $Y_b$ from $e_b$ to a leaf of $Y_{\pm}$ is still minimized in $D_b$, so we do not change the side of $c_b$ on which $D_b$ lies.  Let $S_b^+$ denote $S_b$ with these ghost circles added, and denote by $\vec{\calC}_b^+$ be the sequence of flagged chamber complex decompositions on $\calC$ determined by $S_b^+$.  By Corollary \ref{cor:addghost} it suffices to show that $\vec{\calC}_b^+$ and $\vec{\calC}_a$ cocertify.  This proves Claim 1.
\medskip

 \begin{figure}[ht!]
 \labellist
\small\hair 2pt
\pinlabel  $D_1$ at 260 120
\pinlabel  $D_b$ at 105 15
\pinlabel  $F'$ at 450 50
\pinlabel  $S_{a}$ at 180 140
\pinlabel  $S_{b}$ at 200 30
\endlabellist
    \centering
    \centering
    \includegraphics[scale=0.6]{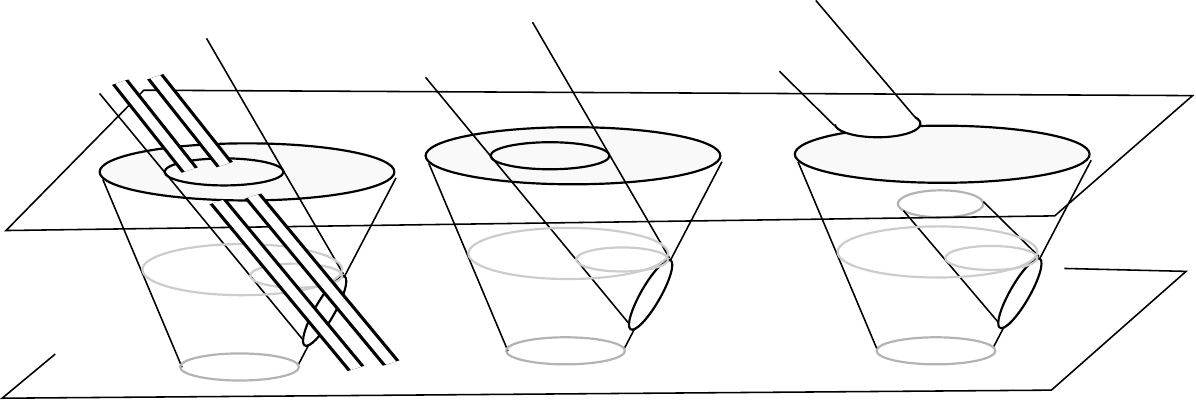}
    \caption{When $D_b$ lies on the side of $c_b$ not containing $v$}  \label{fig:SaSb4}
    \end{figure}

Since, per the claim, $\rho_1 \leq \rho_b$ and by definition $\rho_1 \leq \rho_2$ there is no difference between the decomposition sequences for $\vec{\calC}_a$ and $\vec{\calC}_b$ until we reach the decomposition $\calC_{\rho_1} \xrightarrow{\calD_{\rho_1}} \calC_{\rho_1 +1}$.  $F_{\rho_1} = F(\calC_{\rho_1})$ still appears as in Figure \ref{fig:SaSb1}, but with the added information that $\inter(D_1)$ contains neither ghost circles nor $F$-circles.  In particular, it is possible to define a new chamber complex $\calC'$ obtained from $\calC_{\rho_1}$ by decomposing along $D_1$ alone.  That is \[\calC_{\rho_1} \xrightarrow{D_1} \calC'.\] 
\medskip

{\em Claim 2:} Both $S_a$ and $S_b$ are still balanced or almost balanced for $\calC'$.

{\em Proof of Claim 2:} Let $F'$ be the defining surface $F(\calC')$.  See Figure \ref{fig:SaSb4}.  We name 4 regions and 6 genera:
\begin{itemize}
\item $X$ is the region below $S_b$ and $g_X$ is the genus of $F_{\rho_1} \cap X = F' \cap X$.
\item $X^+ \supset X$ is the region below $S_a$, $g_X^+$ is the genus of $F_{\rho_1} \cap X^+$, and $g'^+_X$ is the genus of $F' \cap X^+$.  
\item $Y$ is the region above $S_a$, $g_Y$ is the genus of $F_{\rho_1} \cap Y$, and $g'_Y$ is the genus of $F' \cap Y$
\item $Y^+ \supset Y$ is the region above $S_b$, 
and $g'^+_Y$ is the genus of $F' \cap Y^+$
\end{itemize}

Surgery on $D_1 \subset S_a$ will not affect the genus of the part of $F_{\rho_1}$ lying below or above $S_a$,  so
\medskip

(0)  $g'^+_X = g^+_X$, $g'_Y = g_Y$ and $S_a$ is balanced or almost balanced for $\calC'$.  
\medskip

To prove the claim, we then need only focus on $S_b$.  Here are two further observations:
\begin{enumerate}
\item Since $F_{\rho_1} \cap X^+$ is obtained from $F_{\rho_1} \cap X$ by attaching a (genus 0) pair of pants on a single boundary component, $g^+_X = g_X$.  
\item Since $F' \cap Y^+$ is obtained from $F' \cap Y$ by attaching a (genus 0) annulus on a single boundary component, $g'^+_Y = g'_Y$.  
\end{enumerate}

 By definition $S_b$ is balanced or almost balanced for $\calC'$ if and only if the pair of numbers $\{g_X, g'^+_Y\} \neq \{0, k\}, k \geq 2$.  Indeed, by definition, the pair is $\{0, 0\}$ if and only if $S_b$ is planar balanced, it is $\{0, 1\}$ if and only if $S_b$ is almost balanced, and the pair is $\{j, k\}, j, k \geq 1$ if and only if $S_b$ is non-planar balanced.  Similarly, since $S_a$ is balanced or almost balanced for $\calC_{\rho_1}$, we have $\{g^+_X, g_Y\} \neq \{0, k\}, k \geq 2$.  From (1) we have $g_X = g^+_X$ and combining (0) and (2) gives $g'^+_Y = g_Y$.  In particular, $\{g_X, g'^+_Y\} = \{g^+_X, g_Y\}\neq \{0, k\}, k \geq 2$, completing the proof of Claim 2.
%
\medskip

 \begin{figure}[ht!]
 \labellist
\small\hair 2pt
\pinlabel  $D_1$ at 230 185
\pinlabel  $F$ at 200 105
\pinlabel  $F'$ at 450 105
\pinlabel  $S_{a}$ at 100 220
\pinlabel  $S_{b}$ at 100 40
\endlabellist
    \centering
    \centering
    \includegraphics[scale=0.6]{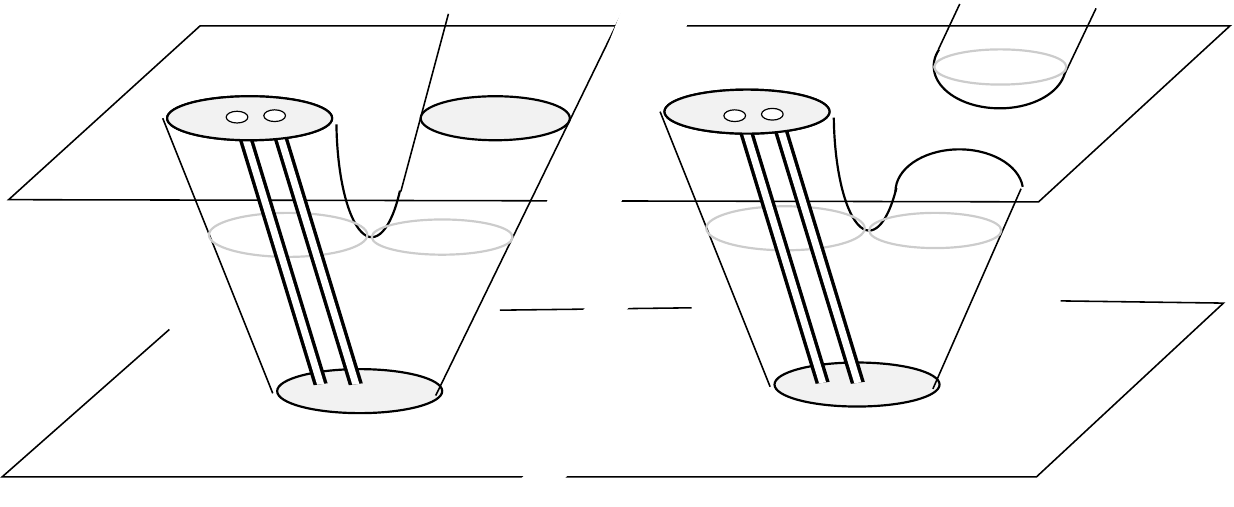}
    \caption{Only annuli between the planes}  \label{fig:SaSb2}
    \end{figure}

$F'$ intersects the collar between $S_a$ and $S_b$ entirely in spanning annuli (see Figure \ref{fig:SaSb2} and right panel of Figure \ref{fig:SaSb4}), so the decomposition sequence on $\calC'$ determed by $S_a$ and $S_b$ is the same; we denote it $\vec{\calC'}$. As in the proof of Lemma \ref{lemma:maxormin} let $\vec{\calC}_b$ (resp $\vec{\calC}_a$) be the sequence of flagged chamber complex decompositions determined by $\calC$ and $S_b$ (resp. $S_a$).  Thus $\vec{\calC}_b \in \overrightarrow{(\calC, S_1)}$ (say) and $\vec{\calC}_a \in \overrightarrow{(\calC, S_0)}$. Now apply, for both $S_a$ and $S_b$, Theorem \ref{thm:addisk} as augmented by Corollary \ref{cor:addisk} to conclude
\[\vec{\calC}_a \sim \vec{\calC'} \sim  \vec{\calC}_b \quad \text{so} \quad \overrightarrow{(\calC, S_0)} \sim \overrightarrow{(\calC, S_1)} \]
\end{proof}

\begin{proof}[Proof of Theorem \ref{thm:balanceisotopy}]  As noted before Lemma \ref{lemma:maxormin} a generic isotopy $S_t$ from $S_0$ to $S_1$ consists of a sequence of isotopies, each containing a single critical point and hence one to which either Lemma \ref{lemma:maxormin} or Lemma \ref{lemma:saddle} applies.
\end{proof}
    
 \section{All balanced or almost balanced spheres cocertify} \label{sect:balisotopy}

In this section we show that the requirement in Theorem \ref{thm:balanceisotopy} that $S_0$ and $S_1$ are isotopic through balanced or almost balanced spheres is superfluous.  In particular we show:

\begin{prop} \label{prop:graphicisotopy} Suppose $\calC$ is a flagged chamber complex in $S^3$.  Then any pair  of balanced or almost balanced spheres for $\calC$ are isotopic in $S^3$ through balanced or almost balanced spheres for $\calC$. 
\end{prop}

We defer the proof in order to quickly observe some consequences:  Following Theorem \ref{thm:balanceisotopy} this immediately implies:

\begin{cor} \label{cor:graphicisotopy} Suppose $\calC$ is a flagged chamber complex in $S^3$ that supports the genus $g$ Heegaard splitting $(S^3, T)$ and is not tiny.  Let $S^x$ and $S^y$ be balanced or almost balanced spheres for $\calC$.  Then $ \overrightarrow{(\calC, S^x)} \sim  \overrightarrow{(\calC, S^y)}$.  \qed
\end{cor}

In particular, any such flagged chamber complex $\calC$ in $S^3$ determines a unique certificate.  That is, the homeomorphism $h_{(\calC, S)}: (S^3, T) \to (S^3, T_g)$, for $S$ a balanced or almost balanced sphere (provided, say, by Corollary \ref{cor:balinterval}), does not in fact depend on the choice of $S$, up to eyeglass equivalence.  We will denote (a choice of) such a homeomorphism simply $h_{\calC}: (S^3, T) \to (S^3, T_g)$.  This expands the use of the notation, first introduced following Corollary \ref{cor:heegcertify}, beyond just those flagged chamber complexes that themselves certify, to all flagged chamber complexes that are not tiny.  Following Proposition \ref{prop:weaknottiny}, it applies to flagged chamber complexes obtained from a Heegaard surface by weak reduction.  It also allows the following natural definition and notation: 

\begin{defin} \label{defin:sim}   Suppose $\calC, \calC'$ are flagged chamber complexes in $S^3$ that support the genus $g$ Heegaard splitting $(S^3, T)$ and are not tiny.  Then $\calC$ and $\calC'$ cocertify (written $\calC \sim \calC'$) if 
the homeomorphisms $h_{\calC}, h_{\calC'}: (S^3, T) \to (S^3, T_g)$ (as just defined above) are eyeglass equivalent.
\end{defin}

\begin{cor}  \label{cor:calCcert2} With this expanded definition of $h_{\calC}$, Corollary \ref{cor:calCcert} remains true.  That is, if $\calC$ certifies and $\tau \in G(S^3, T)$ then $h_{\tau(\calC)}\tau \sim h_{\calC}$. 
\end{cor}

\begin{proof} This follows immediately from Corollary \ref{cor:vScert}.
\end{proof}

The proof of Proposition \ref{prop:graphicisotopy} is brief (it appears just before Proposition \ref{prop:singlesweep}) but the background needed is extensive.  Since it makes use of elementary Morse and Cerf Theory [Mi], [Ce], we will work in the smooth category.  That is, we will take the defining surface $F = F(\calC)$ to be a smooth submanifold of $S^3$, and the  pair of spheres $S^x$ and $S^y$, to be smooth submanifolds, smoothly transverse to $F$.  Let $p: S^3 \to [-1, 1]$ be the standard height function, denote by $S^2$ the standard equator $p^{-1}(0)$, and call the points $p^{-1}(1), p^{-1}(-1)$ the north and south poles of $S^3$.  Denote by $B_{\epsilon}$ a pair of small ball neighborhoods of the poles, that is $B_{\epsilon} = p^{-1}([-1, -1 + \epsilon] \cup [1-\epsilon, 1])$.  

Choose a pair of points $q_n, q_s$ (to correspond to north and south poles) disjoint from $F \cup S^x \cup S^y$ so that $q_n$ and $q_s$ lie on opposite sides of each of the spheres $S^x$ and $S^y$.  This is easily done: If $S^x$ and $S^y$ are disjoint, choose a generic point in the interior of each of the two ball components of $S^3 - (S^x \cup S^y)$.  If, on the other hand, $S^x$ and of $S^y$ intersect, observe that a regular neighborhood of a curve $c$ of intersection intersects $S^x \cup S^y$ in a copy of $\mathsf{X} \times c$.   (Here $\mathsf{X}$ means two crossed lines, e. g. the $x$- and $y$- axes in $\R^2$.)  Then set $q_n, q_s$ to be generic points in opposite quadrants of some $\mathsf{X} \times \{point\}$.  

By the Schoenflies Theorem there are diffeomorphisms $\phi^x, \phi^y:(S^3; q_n, q_s) \to (S^3; p^{-1}(1), p^{-1}(-1))$ so that $\phi^x(S^x) = \phi^y(S^y) = S^2 \subset S^3$.  Choose a small neighborhood $B_\epsilon$ of the poles in $S^3$ that is disjoint from $\phi^x(F) \cup \phi^y(F)$ and adjust $\phi^x, \phi^y$ near $q_n, q_s$ so that $p_x$ and $p_y$ coincide on a neighborhood $U = (p^x)^{-1}(B_\epsilon) = (p^y)^{-1}(B_\epsilon)$ of $q_n, q_s$.  It is elementary in this case (or use \cite{Lau} and the fact that the space of orientation preserving diffeomorphisms of the ball is connected, indeed contractible) that there is an isotopy $\theta: (S^3, U) \times I \to (S^3, B_\epsilon)$ between $\phi^x$ and $\phi^y$, fixed on $U$.  That is, for $\theta_t: (S^3, U) \to (S^3, B_\epsilon)$ defined by $\theta_t(z) = \theta(z, t)$, we have $\theta_0 = \phi^x$ and $\theta_1 = \phi^y$ and for all $t$, and each $u \in U$, $\theta_t(u) = \phi^x(u) = \phi^y(u)$.  

Now consider the function $p^x: F \to [-1, 1]$ defined by $p^x = p\phi^x|F$ and similarly $p^y = p\phi^y|F$.  The functions $p^x, p^y$ are homotopic via the homotopy $p_t = p\theta_t|F$.   Morse and Cerf theory tell us that for a generic copy of $F \subset S^3$ arbitrarily near the original $F$ (a copy which we henceforth take as $F$), the functions $p^x, p^y$ are Morse functions and the homotopy $p_t$ between them is Cerf.  Here are the salient points of what is meant by this:
\medskip

{\em $p^x$ (and $p^y$) are Morse:}  The function $p^x: F \to [-1, 1]$ has only a finite number of critical values $\{s_i\} \subset (-1, 1)$ and each critical value is the image of a single non-degenerate critical point in $F$.  Non-degenerate means that the Hessian of the function is non-singular (see \cite{Mi}).  Put another way, each sphere $S^x_s = (p\phi^x)^{-1}(s)$ is transverse to $F$ for each regular value of $s$,  and, for each critical value $s_i$, $S^x_{s_i}$ is transverse to $F$ except at a single point of tangency, where $F$ intersects a neighborhood of the tangency point in $S^x_{s_i}$ as either a maximum, a minimum, or a saddle point.  
\medskip

{\em The homotopy $p_t$ is Cerf:}  For generic $t$ the function $p_t: F \to [-1, 1]$ is Morse.  At a finite set $\tau = \{t_j\} \subset I$ the function $p_t$ is not Morse, solely because either
\begin{itemize}
\item near a single exceptional critical point on $z \in F$, $p_t$ has a "birth-death" singularity given by the local model $p_t(z_1, z_2) = p_t(z) + z_1^3 \pm tz_1 \pm z_2^2$ or
\item there are two non-degenerate critical points of $p_t$ on $F$ with the same critical value
\end{itemize}
Beyond this, and perhaps clarifying it, is a description of the ``Cerf graphic" in the square $(t, s) \in I \times [-1, 1]$ (see, for example, introductory remarks in \cite{GK}, from which Figure \ref{fig:graphic1} is taken):  There are a finite collection $\Gamma$ of curves in $I \times [-1, 1]$ so that each $\gamma \in \Gamma$ is the graph of a smooth function to $[-1, 1]$ on a closed interval in $I$ whose end points are among the set $\tau \cup \bdd I$. See green in Figure \ref{fig:graphic1}.  Each generic vertical arc $t \times [-1, 1]$ is transverse to the graphs $\Gamma$ and intersects each curve $\gamma \in \Gamma$ in at most one point.   Each vertical arc $t_i \times [-1, 1], t_i \in \tau$ has the same property, with a single exceptional point $s \in [-1, 1]$ where either
\begin{itemize}
\item $(t_i, s)$ is a left end point of exactly two such graphs, a birth point. (So $t_i \times [-1, 1] \in $ red vertical lines in Figure \ref{fig:graphic1})
\item $(t_i, s)$ is a right end point of exactly two such graphs, a death point. (So $t_i \times [-1, 1] \in $ red vertical lines in Figure \ref{fig:graphic1})
\item At $(t_i, s)$ two graphs cross transversally.  (So $t_i \times [-1, 1] \in $ blue vertical lines in Figure \ref{fig:graphic1})
\end{itemize}
The collection $\Gamma$ of graphs constitute all points where the functions $p_t$ have critical points.  That is, $(t, s) \in I \times [-1, 1]$ lies in $\Gamma$ if and only if $s$ is a critical value for $p_t$.  In particular, each complementary component $R$ of $\Gamma$ in $I \times [-1, 1]$ (called a {\em region} of the graphic) has the property that for any $(t, s) \in R$, $s$ is a regular value for $p_t$.  

 \begin{figure}[ht!]
 \labellist
\small\hair 2pt
\pinlabel  $s$ at -20 150
\pinlabel  $t\to$ at 200 -10
\pinlabel  $t_i$ at 285 -5
\pinlabel  $0$ at 5 -7
\pinlabel  $1$ at 420 -7
\pinlabel  $-1$ at -12 10
\pinlabel  $1$ at -10 280

\endlabellist
    \centering
    \centering
    \includegraphics[scale=0.6]{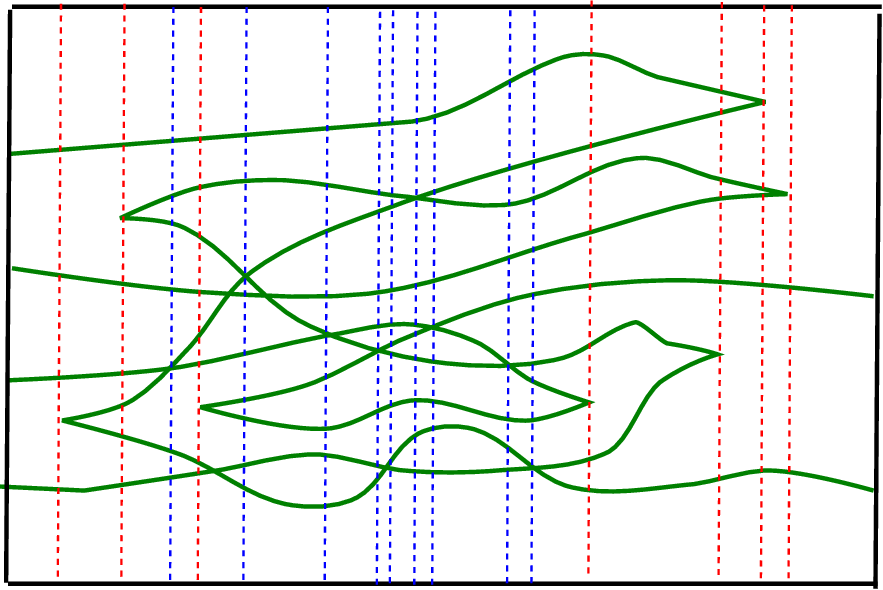}
    \caption{Graphic (green) with $t_i \in \tau$}  \label{fig:graphic1}
    \end{figure}

Given this as background, observe that the natural sweep-out of $S^3 - \{poles\}$ by the $2$-spheres $S_s = p^{-1}(s), s \in (-1, 1)$, in which $S_0$ is the equator $S^2$ induces a $t$-parameterized smooth family of sweep-outs $S^t_s = \theta_t^{-1}(S_s)$ beginning with a sweep-out $(\phi^x)^{-1}(S_s)$ that contains $S^x = (\phi^x)^{-1}(S^2)$ and ending with a sweep-out $(\phi^y)^{-1}(S_s)$ that contains $S^y = (\phi^y)^{-1}(S^2)$.  Taking this sweep-out point of view, we can now incorporate the ideas of Section \ref{sect:balance}, in particular Definition \ref{defin:gaandgb} and Lemma \ref{lemma:gaandgb}, as we now describe. 

Suppose $(t, s) \in R$, where $R$ is a region of the graphic.  Since $s$ is a regular value of $p_t$, the sphere $S^t_s$ is transverse to $F$, and, since $R$ is connected, any other sphere $S^{t'}_{s'}, (t', s') \in R$ is isotopic to $S^t_s$ through spheres transverse to $F$.  Following Definition \ref{defin:gaandgb} we can then assign a label $\frac{a_R}{b_R}$ to $R$, where the non-negative integer $a_R$ is the genus of the part of $F$ lying {\em a}bove $S^t_s$ and similarly $b_R$ is the genus of the part of $F$ lying {\em b}elow $S^t_s$.  Here are some elementary observations:

\begin{enumerate}
\item $S^t_s$ is balanced or almost balanced for $\calC$ if one of these is true:
\begin{itemize}
\item $0 \leq a_R, b_R \leq 1$
\item Both $a_R, b_R \geq 1$
\end{itemize}
In this case we say that the region itself is balanced or almost balanced.
\item Hence $S^t_s$ is {\em not} balanced or almost balanced for $\calC$ (say $S^t_s$ is {\em akilter}) \index{Akilter sphere for $F$} if one of these is true:
\begin{itemize}
\item $a_R = 0$ and $b_R \geq 2$ (say then $+$akilter) or 
\item $b_R = 0$ and $a_R \geq 2$ (say then $-$akilter).
\end{itemize}
In this case we say that the region itself is respectively $+$akilter or $-$akilter.

\item Suppose $R$ and $R'$ are adjacent regions in the graphic, both incident to a birth singularity $(t_0, s_0)$.  That is, the regions are separated near $(t_0, s_0)$ by the pair of curves in $\Gamma$ that end at $(t_0, s_0)$.  Then a generic sweep-out to the right of $(t_0, s_0)$ (that is via $S^{(t_0 +\epsilon)}_s, s \in (-1, 1)$), passing briefly through $R'$, say,  differs from a sweep-out just to the left of $(t_0, s_0)$, passing through $R$, by the brief introduction of a minimum (or maximum), followed by a cancelling saddle tangency.  The consequent  brief transfer of a disk in $F$ from one side of $S^t_s$ to the other has no effect on the genus of $F$ on each side, so $\frac{a_R}{b_R} = \frac{a_{R'}}{b_{R'}}$.  
\item The same is true for $R$ and $R'$ adjacent regions in the graphic, both incident to a death singularity $(t_0, s_0)$. 
\end{enumerate}

We now interpret Lemma \ref{lemma:gaandgb} in this context.  

\begin{lemma} \label{lemma:graphic}  Suppose, in the setting described above, $\genus(F) \geq 2$ and $t_0$ is a generic value of $t$, that is $t_0 \in I - \tau$.  The vertical arc $\mathfrak{v} = t_0 \times [-1, 1]$ has the following properties:
\begin{enumerate}
\item There is a unique curve $\gamma_- \subset \Gamma$, intersecting $\mathfrak{v}$ in a point $(t_0, s_-)$, so that the region just below $\gamma_-$ is $-$akilter and the region just above $\gamma_-$ is balanced or almost balanced.
\item There is a unique curve $\gamma_+ \subset \Gamma$ intersecting $\mathfrak{v}$ in a point $(t_0, s_+)$, so that the region just below $\gamma_+$ is balanced or almost balanced and the region just above $\gamma_+$ is $+$akilter.
\item $s_- < s_+$ and the subinterval $t_0 \times [s_-, s_+]$ of $\mathfrak{v}$ intersects at least three different regions of the graphic.  
\end{enumerate}
\end{lemma}

\begin{proof}  We will apply Lemma \ref{lemma:gaandgb} to the sweepout of $F$ by the spheres $S^{t_0}_s$.  (The lemma is applicable because it is equivalent to considering the sweep-out of $\theta_{t_0}(F)$ by the spheres $S_s$.) Note first that the bottom regions of the graphic (those incident to $I \times \{-1\}$) are all labeled $\frac{genus(F)}{0}$ and the top regions are all labeled $\frac{0}{genus(F)}$.  Since $\genus(F) \geq 2$ the regions at the bottom are $-$akilter and those at the top are $+$akilter.  

As noted in Lemma \ref{lemma:gaandgb} as $s$ rises in $\mathfrak{v}$, passing from region to region, $a_R$ cannot rise and $b_R$ cannot fall.  To progress under these rules from the label $\frac{a}{0}, a \geq 2$ to the label $\frac{0}{b}, b \geq 2$ there is a first transition from a region $R$ that is $-$akilter to a region $R'$ that is balanced or almost balanced, either because $\frac{a_R}{b_R} = \frac20$ and $\frac{a_{R'}}{b_{R'}} = \frac10$ or because $\frac{a_R}{b_R} = \frac{a}{0}, a \geq 2$ and $\frac{a_{R'}}{b_{R'}} = \frac{a}{1}$.  Moreover, once that transition is made, there is no way to transition back to $-$akilter without either lowering the denominator in $\frac{a}{1}$ or raising the numerator in $\frac10$, neither of which is allowed.  So this transition defines the curve $\gamma_-$ and so $s_-$.  The symmetric argument identifies the curve $\gamma_+$ and shows that it lies above $\gamma_-$, that is $s_+ > s_-$.  See Figure \ref{fig:graphic2}.  

 \begin{figure}[ht!]
 \labellist
\small\hair 2pt
\pinlabel  $s_-$ at 310 35
\pinlabel  $s_+$ at 310 160
\pinlabel  $\frac{genus(F)}{0}\implies -$akilter at 200 20
\pinlabel  $\frac{0}{genus(F)}\implies +$akilter at 200 250
\pinlabel  $t_0$ at 295 -5
\pinlabel  $\mathfrak{v}$ at 305 190
\pinlabel  $\gamma_-$ at 370 70
\pinlabel  $\gamma_+$ at 370 160
\pinlabel  $f_-$ at 50 60
\pinlabel  $f_+$ at 50 120
\endlabellist
    \centering
    \centering
    \includegraphics[scale=0.6]{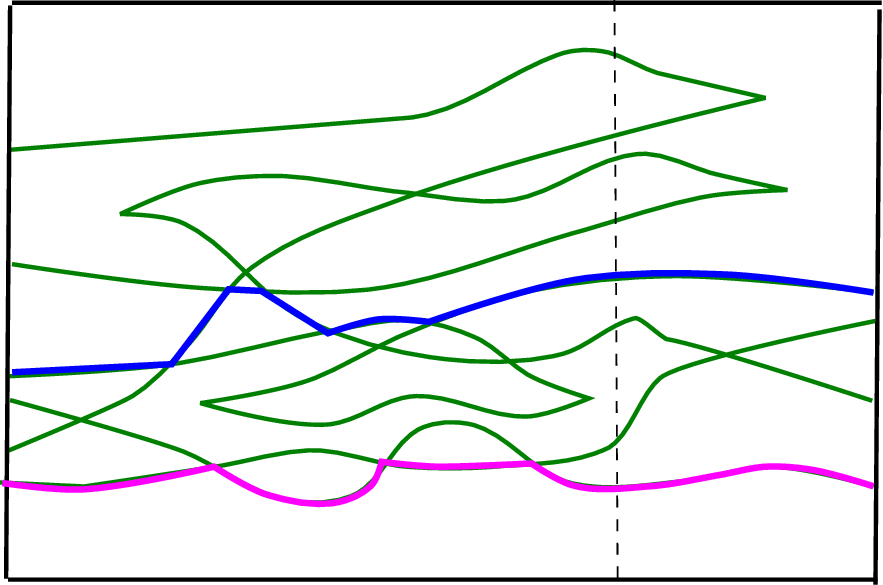}
    \caption{Spanning arcs $\gamma_+$, $\gamma_-$ and the gap between}  \label{fig:graphic2}
    \end{figure}

More subtly, notice that according to Lemma \ref{lemma:gaandgb} as $s$ passes upward through a curve $\gamma$ from region $R$ to region $R'$ the difference $a_R - b_R \in \zed$ does not increase (for this would require either $a_R$ to rise or $b_R$ to fall) and can decrease by at most one (because at most one of these occurs: $a_R$ drops by 1 or $b_R$ rises by 1).  That is, $(a_{R'} - b_{R'}) \leq (a_R - b_R) \leq (a_{R'} - b_{R'}) + 1$.  So it takes at least 4 transitions to go from a $-$akilter region, where $a_R - b_R \geq 2$, to a $+$akilter region, where $a_R - b_R \leq -2$. So, not counting the end points, there are at least two graphs in $\Gamma$ that cross the interior of $t_0 \times [s_-, s_+]$ and so 3 distinct regions through which $t_0 \times [s_-, s_+]$ passes.
\end{proof}

\begin{lemma} \label{lemma:2graphs} Suppose, in the setting described above, $\genus(F) \geq 2$. Then there are continuous (indeed, piecewise smooth) functions $f_-, f_+: I \to [-1, 1]$ so that
\begin{itemize}
\item For all $t$, $f_-(t) < f_+(t)$.
\item The graphs of both $f_{\pm}$ lie in $\Gamma$.
\item Regions below the graph of $f_-$ are all $-$akilter.
\item Regions above the graph of $f_+$ are all $+$akilter.
\item Regions between the two graphs are all balanced or almost balanced.
\item $f_-(0) < 0 < f_+(0)$ and $f_-(1) < 0 < f_+(1)$.
\end{itemize}
\end{lemma}

\begin{proof}   For $t_0 \notin \tau$ define $f_{\pm}(t_0) = s_{\pm}$ as given in Lemma \ref{lemma:graphic}.  The functions $f_{\pm}$ are smooth near $t_0$ because each curve $\gamma \in \Gamma$ is smooth.  The same argument applies for $t_0 \in \tau$ with some $(t_0, s)$ a birth-death point since, as we have seen above, $\frac{a_R}{b_R}$ does not change as $s$ passes through such a point.  

The situation is only a bit more complicated for $t_0 \in \tau$ with two curves in $\Gamma$ crossing at (exactly one) point $(t_0, s_0)$ in $t_0 \times [-1, 1]$.  In this case, the region below $(t_0, s_0)$ could be $-$akilter and the region above $(t_0, s_0)$ either balanced or almost balanced, and then it is natural to define $f_-(t_0) = s_0$.  By examining generic nearby values of $t$ we see that the interval $t_0 \times (s_0, s_+)$ must still pass through at least two distinct regions, so $s_0 < s_+$ and $f_-(t_0) < f_+(t_0)$ as required.  The graph of $f_-$ may not be smooth at $t_0$ since values of $f_-(t)$ for $t$ just less than $t_0$ may lie on one of the two crossing curves in $\Gamma$ and  values of $f_-(t)$ for $t$ just more than $t_0$ lie on the other.  Still the graph of $f_-$ is piecewise smooth at $t_0$.  Symmetric comments apply when the region above $(t_0, s_0)$ is $+$akilter and the region below $(t_0, s_0)$ either balanced or almost balanced.  

Finally, our initial definition was that the balanced or almost balanced sphere $S^x = S^0_0$ so $f_-(0) < 0 < f_+(0)$ and the balanced or almost balanced sphere $S^y= S^1_0$ so $f_-(1) < 0 < f_+(1)$.
\end{proof}

\begin{proof}[Proof of Proposition \ref{prop:graphicisotopy}]  
If $\genus(F) \leq 1$ let $f:[0, 1] \to [-1, 1]$ be any smooth function whose graph is transverse to the curves $\Gamma$ in the graphic and for which $f(0) = f(1) = 0$.  If  $\genus(F) \geq 2$ pick such a function so that, in addition, 
$f_- < f < f_+$
Then the 1-parameter family of spheres $S^t_{f(t)}, 0 \leq t \leq 1$ is an isotopy of $S^x$ to $S^y$.  Moreover each sphere $S^t_{f(t)}$ is balanced or almost balanced for $\calC$, since the graph of $f$ passes only through balanced or almost balanced regions of the graphic.
\end{proof}

There is a related proposition about single sweep-outs, whose proof is similar to 
that of Lemma \ref{lemma:graphic}:

\begin{prop} \label{prop:singlesweep}  Suppose $\calC$ is a flagged chamber complex in $S^3$ that is not tiny and $S_s, s \in [-1, 1]$ is a generic sweep-out of $S^3$ by spheres.  Suppose, $E$ is a properly embedded disk in a chamber of $\calC$ so that, for some generic $-1 < e < 1 $, $\bdd E$ lies entirely in $S_e$.  That is, $\bdd E$ is a circle component of $S_{e} \cap F(\calC)$.  Put $E$ in preferred alignment and let $\calC'$ be the flagged chamber complex obtained from $\calC$ by decomposition along $E$.  Then $\calC' \sim \calC$.
\end{prop}
%

\begin{proof}  Since $S_s$ is a generic sweep-out, there is a finite set $\sigma \subset [-1, 1]$ so that for $s \notin \sigma$, $S_s$ is transverse to $F = F(\calC)$, and for each $s \in \sigma$, $F$ and $S_s$ have a single non-degenerate point of tangency.  Apply the notational convention for regions introduced before Lemma \ref{lemma:graphic} to the (interval) components of $[-1, 1] - \sigma$.  That is, for each such interval $\mathfrak{i} \subset [-1, 1] - \sigma$ pick $s \in \mathfrak{i}$ and let $a_{\fri} \geq 0$ be the genus of the part of $F$ lying above $S_s$ and $b_{\fri}$ be the genus of the part of $F$ lying below $S_s$,  and then assign to $\fri$ the symbol $\frac{a_{\fri}}{b_{\fri}}$.  Following Lemma \ref{lemma:gaandgb}, and similar to the argument used for a generic value of $t$ in the proof of Lemma \ref{lemma:graphic}, we have:
\begin{itemize}
\item The lowest interval in $[-1, 1] - \sigma$ is assigned $\frac{genus(F)}{0}$
\item The highest interval $[-1, 1] - \sigma$ is assigned $\frac{0}{genus(F)}$
\item As $s$ rises through each critical level $s \in \sigma$ the integer $a_{\fri} - b_{\fri}$ does not increase and can decrease by at most 1.
\end{itemize}
Since, as we ascend through the interval components of $[-1, 1] - \sigma$, the difference $a_{\fri} - b_{\fri}$ begins at $\genus(F)$ and ends at $-genus(F)$ it follows that there is an interval $\fri_0$ for which the difference is zero, that is $a_{\fri_0} = b_{\fri_0}$.  In particular, for $s \in \fri_0$, $S_s$ is balanced for $\calC$, either planar balanced (if 
$a_{\fri_0} = b_{\fri_0} = 0$) or non-planar balanced (if $a_{\fri_0} = b_{\fri_0} \geq 1$).  Pick a generic $s_0 \in \fri_0$ ($s_0 \neq e$) so that $S_{s_0}$ is transverse to $E$.

{\em Claim:} The sphere $S_{s_0}$ is balanced or almost balanced for $\calC'$ as well as for $\calC$. 

{\em Proof of Claim:}. Define the pair of whole numbers $a', b'$ in direct analogy to the pair $a_{\fri_0}, b_{\fri_0}$: Namely, $a'$ is the genus of the part of $F'$ lying above $S_{s_0}$ and $b'$ is the genus of the part of $F'$ lying below $S_{s_0}$
We will now show by construction that there is a close relation between the pairs $a', b'$ and $a_{\fri}, b_{\fri}$.

 We can assume, with no loss, that $s_0 > e$. (If $s_0 < e$ reverse the rolls of $a$ and $b$ in the following argument.). Let $\frc \subset S_{s_0}$ be the collection of circles $S_{s_0} \cap E$.  Decomposing $\calC$ along $E$ to obtain $F' = F(\calC')$ can be viewed as a 3-stage process: 
\begin{itemize}
\item first remove a small annular neighborhood $\eta(\bdd E)$ of $\bdd E$ from $F$, then 
\item attach a pushed-off copy of $E$ to each boundary component of $F - \eta(\bdd E)$; call this added pair of disks $2E$.  The resulting surface $F_{\eta}$ intersects the level sphere $S_{s_0}$ in the curves $(F \cap S_{s_0}) \cup 2 \frc$, 
where $2 \frc$ 
denotes two parallel copies of $\frc$, one copy for each copy of $E$ in $2E$. 
\item Finally, remove any resulting sphere components of $F_{\eta}$ that bound goneballs.
\end{itemize}


Observe the behavior of the genus through each of these three stages of the construction.  We consider the stages in reverse order:
Removing spheres from $F_{\eta}$ does not change the genus of the part of the surface lying above $S_{s_0}$ or below, since each component removed is a subsurface of spheres and is therefore planar.  Similarly, adding $2E$ does not affect the genus of the surface above or below, since each component added is a subsurface of $2E$ and therefore planar, and the ends of the annulus $\eta(\bdd E)$ are attached to different planar subsurfaces of $2E$.  Finally, removing  $\eta(\bdd E)$ from $F$ does not affect the genus of the part of $F$ above  $S_{s_0}$, and, though it may lower the genus of the part of $F$ below $S_{s_0}$ (if it is non-separating in the component in which it lies), it lowers it by at most one.


We conclude that $a' = a_{\fri_0}$ and $b_{\fri_0} - 1 \leq b'  \leq b_{\fri_0}$.
In particular
\begin{itemize}
\item If $\frac{a_{\fri_0}}{b_{\fri_0}} = \frac00$ then $\frac{a'}{b'} = \frac00$.
\item If $\frac{a_{\fri_0}}{b_{\fri_0}} = \frac11$ then $\frac{a'}{b'}  = \frac11$ or $\frac10$.
\item If $\frac{a_{\fri_0}}{b_{\fri_0}} = \frac{m}{m}, m \geq 2$ then $\frac{a'}{b'} = \frac{m}{m}$ or $\frac{m}{m-1}$.
\end{itemize}
In any case, we see that the sphere $S_{s_0}$ is balanced or almost balanced for $\calC'$, proving the Claim.  
\medskip

The proof now proceeds by induction on $|\frc|$, the number of curves of intersection of $E$ with $S_{s_0}$.  If $E$ is disjoint from $S_{s_0}$, so $|\frc| = 0$ the result follows from Theorem \ref{thm:addisk}, as augmented by Corollary \ref{cor:addisk}.  

Assume then that $|\frc| \geq 1$ and let
\[\overrightarrow{(\calC, S_{s_0})}: \quad \calC \xrightarrow{\calD_0} \calC_1 \xrightarrow{\calD_1} ... \xrightarrow{\calD_{n-1}}\calC_n\] and
\[\overrightarrow{(\calC', S_{s_0})}: \quad \calC' \xrightarrow{\calD'_0} \calC'_1 \xrightarrow{\calD'_1} ... \xrightarrow{\calD'_{n-1}}\calC'_n\]
be complete flagged chamber complex decomposition sequences guided by $S_{s_0}$ for $\calC$ and $\calC'$ respectively.  (First add the circles $\frc$ as ghost circles for the decomposition sequence for $\calC$, so that both sequences use the same set of guiding disks throughout, and so run in parallel. By Corollary \ref{cor:addghost} the addition of these ghost circles does not change how $\calC$ certifies.)  Focus attention on the first decomposition $ \calC'_i \xrightarrow{\calD'_i} \calC'_{i+1}$ in which a circle in $\frc$ appears as the boundary of a disk in $\calD'_i$.  Let $\frc_0 \in \frc$ be such a circle, chosen among all candidates to be innermost in $2E$; let $D_0 \in \calD'_i$ be the disk that $\frc_0$ bounds in $S_{s_0}$; and let $E_0 \subset 2E$ be the disk that $\frc_0$ bounds in $2E$.  

Although $E_0$ may intersect $S_{s_0}$ in other components of $\frc$, by construction none bound disks in any $\calD'_j, j \leq i$, so $E_0$ remains intact as part of the defining surface  $F'_{i} = F( \calC'_{i})$.  Consider whether $D_0$ is an essential disk in the chamber of  $ \calC'_{i} $ in which it lies: 

If $D_0$ is an essential disk, then the ball $B_0 \subset S^3$ that the sphere $S = D_0 \cup E_0$ bounds (on the side in $S^3$ not containing $\bdd E$) contains components of $F'_{i}$.  By construction, $B_0$ is disjoint from $2E$ so these components of $F'_{i}$ must also be components of the defining surface $F_{i} = F(\calC_i)$.  Thus $S$ is an incompressible sphere in a chamber of both $\calC'_i$ and $\calC_i$.  This implies that $\calC'_i \sim \calC_i$, hence $\overrightarrow{(\calC, S_{s_0})} \sim \overrightarrow{(\calC', S_{s_0})}$ and so, following Definition \ref{defin:sim},  $\calC \sim \calC'$ as required.

On the other hand, if $D_0$ is an inessential disk, then $\inter(B_0)$ is disjoint from $F'_{i}$ so $B_0$ is disky in  $\hat{\calC}'_{i+1}$ and so (perhaps with an appropriate choice of sibling) becomes a goneball in $\calC'_{i+1}$.  In other words, after decomposition by $\calD'_i$ (and the similar elimination of the parallel copy of $E_0$ in $2E$ at the next stage of the decomposition sequence) we obtain the same chamber complexes as if $E_0$ had been replaced by $D_0$ in the disk $E$.  But this replacement lowers $|\frc|$ by at least 1.  By our inductive assumption on $|\frc|$, $\calC \sim \calC'$, as required.

To provide more detail in this last case, when $D_0$ is inessential, let $E'$ be the disk that results from replacing the disk $E_0 \subset E$ by $D_0$ as just described.  Recall the following diagram for the inductive step in the proof of Proposition \ref{prop:Eaddcocert}, which in turn enables the proof of Theorem \ref{thm:addisk}:

\[\xymatrix {...\ar[r]^{\calD_{i-1}}&\calC_i  \ar[dd]_{E} \ar[r]^{\calD_i} & \calC_{i+1} \ar[d]^{E}  \ar[r]^{\calD_{i+1}} & ... & = \vec{\calC}\\ 
&&  \calC^{i+1}_{i+1} \ar@{<-->}[d]  \ar[r]^{\calD^{i+1}_{i+1}} & ... & = \vec{\calC}_E^{i+1} &\\
&\calC^i_i  \ar[r]^{\calD^i_i}&\calC^i_{i+1}  \ar[r]^{\calD^i_{i+1}} & ... & = \vec{\calC}_E^{i}  & } \]

As the argument there would ultimately be applied here, the decomposing disks are all guided by $S_{s_0}$.  (One consequence is that the notation for the decomposing disks in the diagram can all be simplified to those for the guiding disks in the top row.)  The problem with further use of that inductive step here is that for $j > i$ the decomposing disks $\calD_j$ might pass through $E$.  This means that their interiors may intersect $F(\calC'_j)$, making later inductive steps potentially nonsensical.  However, since $B_0$ is a goneball in the decomposition $\calC'_i \xrightarrow{\calD_i} \calC'_{i+1}$ nothing is lost by replacing $E$ with $E'$ to get the simplified diagram:

\[\xymatrix {...\ar[r]^{\calD_{i-1}}&\calC_i  \ar[d]_{E'} \ar[r]^{\calD_i} & \calC_{i+1} \ar[r]^{\calD_{i+1}} & ... & = \vec{\calC}\\ 
&\calC^i_i  \ar[r]^{\calD^i_i}&\calC^i_{i+1}  \ar[r]^{\calD^i_{i+1}} & ... & = \vec{\calC}_E^{i}  & } \]
Then observe that since $|E' \cap S_{s_0}|< |E \cap S_{s_0}| = |\frc|$ the inductive hypothesis on $|\frc|$ implies that $\overrightarrow{(\calC_i, S_{s_0})} \sim \overrightarrow{(\calC^i_i, S_{s_0})}$.  This eliminates the need for the further inductive steps used in the proof of Proposition \ref{prop:Eaddcocert}.  
\end{proof}

Because Proposition \ref{prop:singlesweep} requires that $\calC$ not be tiny, it cannot be directly applied to the flagged chamber complex defined by a Heegaard surface $T$, since both chambers are empty handlebodies.  But it does tell us something crucial about flagged chamber complexes that are obtained from $T$ by weak reduction:


\begin{prop} \label{prop:heegaddisk}   Suppose $S^3 = A \cup_T B$ is a Heegaard splitting and $S_s, s \in [-1, 1]$ is a generic sweep-out of $S^3$ by spheres. Suppose $\calD$ is a weakly reducing collection of disks for $T$ and $E$ is a properly embedded disk disjoint from $\calD$, lying in either $A$ or $B$.  Suppose further that, for some generic $-1 < e < 1 $, $\bdd E$ lies entirely in $S_e$. Let $\calC$ be the flagged chamber complex obtained from $T$ by decomposition along $\calD$ and $\calC'$ be the flagged chamber complex obtained from $T$ by decomposition along $\calD \cup E$. Then $\calC' \sim \calC$.
\end{prop}

%

\begin{proof}  Let $\calC_{DE}$ denote the flagged chamber complex obtained from $\calC$ by decomposition along $E$.  Proposition \ref{prop:singlesweep} does apply to this decomposition since, by Proposition \ref{prop:weaknottiny}, $\calC$ is not tiny.  Thus $\calC \sim \calC_{DE}$.  It remains to prove a similar relation between $\calC'$ and $\calC_{DE}$.

We are in a position to apply Proposition \ref{prop:x+yvsxy}, where the Heegaard splitting here plays the role of $\calC$ in Proposition \ref{prop:x+yvsxy}; 
$\calC'$ here corresponds to $\calC_{D+E}$ there; and the notation $\calC_{DE}$ here was chosen to be consistent with the same notation there.  Then, according to Proposition \ref{prop:x+yvsxy}, either $\calC_{DE} = \calC'$ or $\calC_{DE} \dashrightarrow \calC'$.  In the former case we are done, so we consider the case $\calC_{DE} \dashrightarrow \calC'$, and turn to Corollary \ref{cor:longdeflate}.  

Following Corollary \ref{cor:balinterval} let $S_{s_0}$ be a sphere in the sweep-out that is balanced for $\calC'$ and let 
\[\vec{\calC'}: \quad \calC' \xrightarrow{\calD_0} \calC'_1 \xrightarrow{\calD'_1} ... \xrightarrow{\calD'_{n-1}}\calC'_n\] 
be a complete flagged chamber complex decomposition sequence guided by $S_{s_0}$.  Similarly, let $\vec{\calC}_{DE}$ be a complete flagged chamber complex decomposition sequence for $\calC_{DE}$ guided by $S_{s_0}$.  Since $\calC_{DE} \dashrightarrow \calC'$ their defining surfaces $F_{DE}$ and $F'$ differ by at most the insertion of a bullseye.  Since the components of the bullseye are all spheres, the insertion does not affect the genus of the parts of the surface above and below $S_{s_0}$, so $S_{s_0}$ is balanced for $\calC_{DE}$ as well.  In particular, by Proposition \ref{prop:balance}, both $\vec{\calC'}$ and $\vec{\calC}_{DE}$ certify.

Apply Corollary \ref{cor:longdeflate} to $\calC_{DE} \dashrightarrow \calC'$: Let $\vec{\calC}^m_{DE}$ be the maximal deflationary sequence of $\calC_{DE}$.  If $\vec{\calC}^m_{DE} = \vec{\calC}_{DE}$ so the entire sequence is a deflationary sequence, then by the second outcome of Corollary \ref{cor:longdeflate}, $\vec{\calC'}$ and $\vec{\calC}_{DE}$ cocertify.  If, on the other hand, $\vec{\calC}^m_{DE}$ is not the entire sequence $\vec{\calC}_{DE}$ then by the first outcome of Corollary \ref{cor:longdeflate}, $\vec{\calC}^m_{DE}$ certifies, so by the second outcome, again $\vec{\calC'}$ and $\vec{\calC}_{DE}$ cocertify.  In other words, $ \overrightarrow{(\calC_{DE}, S_{s_0})} \sim  \overrightarrow{(\calC', S_{s_0})}$ or, following Definition \ref{defin:sim}, $\calC_{DE} \sim \calC'$ as required.
\end{proof}

\section{The Goeritz group is the eyeglass group} \label{sect:finale}

Suppose $(S^3, T_g)$ is the standard genus $g \geq 2$ Heegaard splitting of Section \ref{sect:motiv} and $\tau$ is an element of the Goeritz group $G(S^3, T_g)$, as described in \cite{JM}.  Let $T_{\theta} \subset S^3, 0 \leq \theta \leq 2\pi$ be a representative of $\tau$ in $\pi_1(\Img(S^3, T))$, with $S^3 = A_{\theta} \cup_{T_{\theta}} B_{\theta}$ and $T_0 = T_{2\pi} = T_g$.   

Put another way, $\tau$ is represented by an isotopy $\Theta_\theta: S^3 \to S^3, 0 \leq \theta \leq 2\pi$ with $\Theta_0$ the identity and $\Theta_{2\pi}(T_g) = T_g$. $T_\theta$ then denotes the surface $\Theta_\theta(T_g) \subset S^3$ and $\Theta_{2\pi}$ represents $\tau:(S^3, T_g) \to (S^3, T_g)$.  Similarly $A_\theta, B_\theta$ denote respectively $\Theta_\theta(A), \Theta_\theta(B)$.  Finally, for any $\theta, \theta' \in [0, 2\pi]$, denote the composition $\Theta_\theta\Theta_{\theta'}^{-1}$ by $\Theta^{\theta'}_{\theta}: (S^3, T_{\theta'}) \to (S^3, T_{\theta})$.  This definition implies that, for any $\theta'' \in [0, 2\pi]$, $\Theta^{\theta''}_{\theta} \Theta^{\theta'}_{\theta''} = \Theta^{\theta'}_{\theta}$.

We briefly review some of the results of \cite{FS1}, where $S^3$ is swept out by level spheres $S_s$ of the standard height function $p: S^3 \to [-1, 1]$.   

There are values $0 < \theta_1 < \theta_2 < ... <\theta_n < 2\pi$ so that for each $\theta \notin \{\theta_i, 1 \leq i \leq n\}$ there is a pair of weakly reducing disks $(a_{\theta}, b_{\theta})$ associated to $T_\theta$ so that:

\begin{itemize}
\item $a_{\theta} \subset A_{\theta}$ and $b_{\theta} \subset B_{\theta}$
\item The isotopy class of the pair $(a_{\theta}, b_{\theta})$ does not change throughout each interval in $[0, 2\pi] - \{\theta_i, 1 \leq i \leq n\}$.  By this we mean, if $\theta, \theta'$ are in the same interval of $[0, 2\pi] - \{\theta_i, 1 \leq i \leq n\}$ then the pair of disks $(\Theta^{\theta'}_{\theta}(a'_\theta), \Theta^{\theta'}_{\theta}(b'_\theta))$ is properly isotopic to the pair  $(a_\theta, b_\theta)$ in $(S^3, T_\theta)$. 

A shorthand notation for this is $(a_\theta, b_\theta) = (a_{\theta'}, b_{\theta'})$.  
\item $(a_{2\pi}, b_{2\pi})$ = $(a_{0}, b_{0})$.  That is, the pair of disks $(a_{2\pi}, b_{2\pi})$ is properly isotopic to the pair $(a_{0}, b_{0})$ in $(S^3, T_g)$.

\item For every $0 \leq \theta \leq 2\pi$ each of the disks $a_{\theta}$ and $b_{\theta}$ have level boundaries.  That is, each the circles $\bdd a_{\theta}$ and $\bdd b_{\theta}$ lies in a single sphere of the sweep-out by $S_s$.
\end{itemize}

The last property is used for the application of Proposition \ref{prop:heegaddisk} in the proof of Lemma \ref{lemma:globalfix} below.  It follows from the construction described in \cite[Appendix]{FS1}, as applied in \cite[Subsection 4.5]{FS1}. The disks $a_{\theta}$ and $b_{\theta}$ are chosen in \cite{FS1} by how their boundaries lie  in $T_\theta \cap S$, for $S$ one of the level spheres $S_s$. (The argument in \cite{FS1} requires that $\genus(T) \geq 2$.)  Note that, unlike the collection of disks used in disk decomposition sequences that we have been long discussing, it is not part of the construction for the sweep-outs in \cite{FS1} that the entire disk $a_{\theta}$ or $b_{\theta}$ lies in a single sphere of the sweep-out, only that the boundary of each disk does.  

The third property above, that  $(a_{2\pi}, b_{2\pi})$ = $(a_{0}, b_{0})$ (the notation here differs somewhat from that in \cite{FS1}) follows from the fact that $T_{2\pi} = T_0 = T_g$.  Indeed, we might as well take the sweep out by copies of the Heegaard surface used in the construction to be the same in both cases, which means that the weakly reducing pair of disks arising from the construction will be given by the same rule.  

For each $\theta \notin \{\theta_i, 1 \leq i \leq n\}$, define the flagged chamber complex $\calC_{\theta}$ as that obtained from $T_{\theta}$ by weak reduction along the pair of disks $(a_{\theta}, b_{\theta})$. 

Observe some properties:
\begin{itemize}
\item By Proposition \ref{prop:weaknottiny} no chamber complex $\calC_\theta$ is tiny.  
\item  Suppose $\theta, \theta'$ are in the same interval of $[0, 2\pi] - \{\theta_i, 1 \leq i \leq n\}$.  Since the pair of disks $(\Theta^{\theta'}_{\theta}(a_\theta'), \Theta^{\theta'}_{\theta}(b_\theta'))$ is properly isotopic to the pair  $(a_\theta, b_\theta)$, the chamber complex $\Theta^{\theta'}_{\theta}(\calC_{\theta'})$ is also properly isotopic to the chamber complex $\calC_{\theta}$ in $(S^3, T_\theta)$.  
\item Since $(a_{2\pi}, b_{2\pi})$ = $(a_{0}, b_{0})$ we have $ \calC_{2\pi} = \calC_0$.
\end{itemize}

Finally, for each $\theta \notin \{\theta_i, 1 \leq i \leq n\}$ define $h_\theta: (S^3, T_{\theta}) \to (S^3, T_g)$ to be  $h_{\calC_{\theta}}$ as defined following Corollary \ref{cor:graphicisotopy}. (Technically, $h_\theta$ is only defined up to eyeglass equivalence, since that is the case for $h_{\calC_{\theta}}$.)   We have 

\begin{lemma} \label{lemma:intervalfix} Suppose $\theta, \theta'$ are in the same interval of $[0, 2\pi] - \{\theta_i, 1 \leq i \leq n\}$. Then $h_{\theta} \sim h_{\theta'}{\Theta^{\theta}_{\theta'}}$.  
\end{lemma}

\begin{proof} As just observed  the chamber complex $\Theta^{\theta'}_{\theta}(\calC_{\theta'})$ is properly isotopic to the chamber complex $\calC_{\theta}$ in $(S^3, T_\theta)$ so $h_\theta = h_{\calC_{\theta}} = h_{\Theta^{\theta'}_{\theta}(\calC_{\theta'})}$.  Now apply Corollary \ref{cor:calCcert2}: $ h_{\Theta^{\theta'}_{\theta}(\calC_{\theta'})}  \Theta^{\theta'}_{\theta} \sim h_{\calC_{\theta'}}$ so $h_\theta \sim h_{\calC_{\theta'}}(\Theta^{\theta'}_{\theta})^{-1} = h_{\theta'}{\Theta^{\theta}_{\theta'}} $
\end{proof}

Further following \cite{FS1}, for each $1 \leq i \leq n$ there are three disjoint disks, either 
\begin{itemize}
\item $a_{\theta_i} \subset A_{\theta_i}, b^\pm_{\theta_i} \subset B_{\theta_i}$  with the pairs $(a_{\theta_i},b^+_{\theta_i})$ and $(a_{\theta_i},b^-_{\theta_i})$ each weakly reducing or
\item symmetrically $a^\pm_{\theta_i} \subset A_{\theta_i}, b_{\theta_i} \subset B_{\theta_i}$  with the pairs $(a^+_{\theta_i},b_{\theta_i})$ and $(a^-_{\theta_i}, b_{\theta_i})$ each weakly reducing.  
\end{itemize}

with the property that (respectively for the two cases) for small $\epsilon$ 
\begin{itemize}
\item  $\Theta^{\theta_i}_{\theta_i \pm \epsilon}(a_{\theta_i}) = a_{\theta_i \pm \epsilon}$ and $\Theta^{\theta_i}_{\theta_i \pm \epsilon}(b^{\pm}_{\theta_i}) = b_{\theta_i \pm \epsilon}$ 
\item symmetrically  $\Theta^{\theta_i}_{\theta_i \pm \epsilon}(a^{\pm}_{\theta_i}) = a_{\theta_i \pm \epsilon}$ and $\Theta^{\theta_i}_{\theta_i \pm \epsilon}(b_{\theta_i}) = b_{\theta_i \pm \epsilon}$
\end{itemize}

\begin{lemma} \label{lemma:globalfix} Suppose $\theta, \theta'$ are any two points in $[0, 2\pi] - \{\theta_i, 1 \leq i \leq n\}$. Then $h_{\theta} \sim h_{\theta'}{\Theta^{\theta}_{\theta'}}$.  
\end{lemma}

\begin{proof}  It suffices to prove the case in which $\theta, \theta'$ are in adjacent intervals, for then we can just proceed around the circle.  (See Figure \ref{fig:Gchamber} for a highly schematic picture.)  So suppose that they are on adjacent intervals sharing the end point $\theta_i \in (0, 2\pi)$.  Following Lemma \ref{lemma:intervalfix} it suffices to prove the case in which $\theta = \theta_i - \epsilon$ and $\theta' = \theta_i + \epsilon$, for small $\epsilon$.  

With no loss of generality assume that the three weakly reducing disks in $(S^3, T_i)$ are $a_{\theta_i} \subset A_{\theta_i}, b^\pm_{\theta_i} \subset B_{\theta_i}$ and let $\calC_i$ be the (not tiny) chamber complex obtained by weakly reducing $T_{\theta_i}$ along the three disks $a_{\theta_i}, b^\pm_{\theta_i}$.  Let $\calC^+_i$ be that obtained by weakly reducing along just the pair $a_{\theta_i}, b^+_{\theta_i}$ and $\calC^-_i$ be that obtained by weakly reducing along just  $a_{\theta_i}, b^-_{\theta_i}$.  Apply Proposition \ref{prop:heegaddisk} to deduce that $\calC^-_i \sim \calC_i$.  Similarly, $\calC^+_i \sim \calC_i$, so as a result $\calC^+_i \sim \calC^-_i$.

The argument of Lemma \ref{lemma:intervalfix} further shows that $h_{\calC^+_i} \sim h_{\theta_i + \epsilon}{\Theta^{\theta_i}_{\theta_i + \epsilon}}$ and $h_{\calC^-_i} \sim h_{\theta_i - \epsilon}{\Theta^{\theta_i}_{\theta_i - \epsilon}}$ so, combining, $h_{\theta_i - \epsilon}{\Theta^{\theta_i}_{\theta_i - \epsilon}} \sim h_{\theta_i + \epsilon}{\Theta^{\theta_i}_{\theta_i + \epsilon}}$.  Hence $h_{\theta_i - \epsilon} \sim h_{\theta_i + \epsilon}{\Theta^{\theta_i}_{\theta_i + \epsilon}}{\Theta^{\theta_i-\epsilon}_{\theta_i}} \sim   h_{\theta_i + \epsilon}{\Theta^{\theta_i-\epsilon}_{\theta_i + \epsilon}}$ as required.  
\end{proof}

\begin{figure}[ht] 
  \centering
  \begin{tikzpicture}
   \draw (1.9, 0) -- (2.1, 0);
\draw (0,0) circle (2cm);
\draw (30:1.9)--(30:2.1);
\node at (30:1.6) {$\theta_1$};
\draw (80:1.9)--(80:2.1);
\node at (80:1.6) {$\theta_2$};
\draw (-30:1.9)--(-30:2.1);
\node at (-30:1.6) {$\theta_n$};
\draw (-80:1.9)--(-80:2.1);
\node at (-80:1.6) {$\theta_{n-1}$};
\draw (120:1.9)--(120:2.1);
\node at (120:1.6) {$\theta_{3}$};
\node at (2.6, 0) {$0, 2\pi$};
\node at (-15:3) {$\calC_{\theta_{n} + \epsilon} \sim \calC_{2\pi}$};
\node at (-45:3) {$\calC_{\theta_{n-1} + \epsilon} \sim \calC_{\theta_n - \epsilon}$};
\node at (15:3) {$\calC_0 \sim \calC_{\theta_1 - \epsilon}$};
\node at (40:3) {$\calC_{\theta_1 + \epsilon} \sim \calC_{\theta_2 - \epsilon}$};
\node at (100:2.5) {$\calC_{\theta_2 + \epsilon} \sim \calC_{\theta_3 - \epsilon}$};
\node at (190:2.5) {$\cdot$};
\node at (195:2.5) {$\cdot$};
\node at (200:2.5) {$\cdot$};
\node at (185:1.5) {$\cdot$};
\node at (195:1.5) {$\cdot$};
\node at (205:1.5) {$\cdot$};

  \end{tikzpicture}
  \caption{Schematic for Lemmas \ref{lemma:intervalfix} 
  and \ref{lemma:globalfix}} \label{fig:Gchamber}
\end{figure}

\begin{cor} $\tau \in \calE$
\end{cor}

\begin{proof}  Choose, in Lemma \ref{lemma:globalfix} the values $\theta = 0, \theta' = 2\pi$.  Then, per that lemma, $h_0 \sim h_{2\pi} \Theta_{2\pi}\Theta_0^{-1}$. Furthermore, since $\calC_0 = \calC_{2\pi}$, Corollary \ref{cor:heegcertify} has $h_0 \sim h_{\calC_0} \sim h_{\calC_{2\pi}} \sim h_{2\pi}$, and so  $\tau = \Theta_{2\pi} \sim \Theta_0 = id_{(S^3, T_g)}$.
\end{proof}

{\bf Remark:}  It is perhaps not surprising that it is possible to position $T_g$ with respect to the height function on $S^3$ so that $h_0$ and $h_{2\pi}: (S^3, T_g) \to (S^3, T_g)$ are not just eyeglass equivalent, but in fact are both the identity (under the standard inductive Assumption \ref{ass:inductive}.  We do not need this here.

\begin{cor} \label{cor:finale} $G(S^3, T) = \calE$.  
\end{cor}

\printindex

\end{document}